\documentclass[11pt]{article}
\usepackage{mathrsfs}
\usepackage{amssymb}
\usepackage{color}
\usepackage{amsmath}
\allowdisplaybreaks[4]

\usepackage{amsthm}
\usepackage{amsmath}
\usepackage{graphicx}

 \newtheorem{thm}{Theorem}[section]
 
 \newtheorem{lem}[thm]{Lemma}
 
 \theoremstyle{definition}
 
 \newtheorem{rem}[thm]{Remark}
 \numberwithin{equation}{section}
 \def\bR{\mathbb{R}}

\newtheorem{theorem}{Theorem}[section]

\theoremstyle{definition}

\def\div{\mathop{\rm div}\nolimits}
\def\curl{\mathop{\rm curl}\nolimits}

\newcommand{\newcom}{\newcommand}

\newcom{\al}{\alpha}
\newcom{\be}{\beta}
\newcom{\eps}{\epsilon}
\newcom{\veps}{\varepsilon}
\newcom{\e}{\varepsilon}
\newcom{\ga}{\gamma}
\newcom{\Ga}{\Gamma}
\newcom{\ka}{\kappa}
\newcom{\La}{\Lambda}
\newcom{\lam}{\lambda}
\newcom{\Om}{\Omega}
\newcom{\om}{\omega}
\newcom{\Si}{\Sigma}
\newcom{\si}{\sigma}
\newcom{\tht}{\theta}
\newcom{\dtri}{\nabla}
\newcom{\tri}{\triangle}
\newcom{\oo}{\infty}
\newcom{\vphi}{\varphi}
\newcom{\cB}{{\mathcal B}}
\newcom{\cC}{{\mathcal C}}
\newcom{\cD}{{\mathcal D}}
\newcom{\cF}{{\mathcal F}}
\newcom{\cH}{{\mathcal H}}
\newcom{\cL}{{\mathcal L}}
\newcom{\cM}{{\mathcal M}}
\newcom{\cN}{{\mathcal N}}
\newcom{\cP}{{\mathcal P}}
\newcom{\cS}{{\mathcal S}}
\newcom{\cQ}{{\mathcal Q}}
\newcom{\cT}{{\mathcal T}}
\newcom{\cY}{{\mathcal Y}}
\newcom{\cZ}{{\mathcal Z}}
\newcom{\R}{\mathbb R}
\newcom{\T}{\mathbb T}
\newcom{\BT}{{\mathbb{R}^2}}
\newcom{\Z}{\mathbb Z}
\newcom{\C}{\mathbb C}
\newcom{\E}{\mathbb E}

\newcommand{\vc}[1]{{\bf #1}}
\newcom{\ve}{\vc{e}}
\newcom{\vN}{\vc{N}}
\newcom{\vn}{\vc{n}}
\newcom{\vG}{\vc{G}}
\newcom{\vF}{\vc{F}}
\newcom{\vZ}{\vc{Z}}
\newcom{\vf}{\vc{f}}
\newcom{\vg}{\vc{g}}
\newcom{\vq}{\vc{q}}
\newcom{\vu}{\vc{u}}
\newcom{\vv}{\vc{v}}
\newcom{\vw}{\vc{w}}
\newcom{\vb}{\vc{b}}
\newcom{\vh}{\vc{h}}
\newcom{\vz}{\vc{z}}
\newcom{\vup}{\vu^{+}}
\newcom{\vum}{\vu^{-}}
\newcom{\vvp}{\vv^{+}}
\newcom{\vvm}{\vv^{-}}
\newcom{\vbp}{\vb^{+}}
\newcom{\vbm}{\vb^{-}}
\newcom{\vhp}{\vh^{+}}
\newcom{\vhm}{\vh^{-}}
\newcom{\Omp}{{\Om^+}}
\newcom{\Omm}{{\Om^-}}
\newcom{\vupm}{{\vu^{\pm}}}
\newcom{\vvpm}{{\vv^{\pm}}}
\newcom{\vbpm}{{\vb^{\pm}}}
\newcom{\vhpm}{{\vh^{\pm}}}
\newcom{\vwp}{{\vc{w}^+}}
\newcom{\vwm}{{\vc{w}^-}}
\newcom{\vwpm}{{\vc{w}^{\pm}}}
\newcom{\Ompm}{{\Omega^{\pm}}}
\newcom{\vom}{\boldsymbol{\omega}}
\newcom{\vvap}{\boldsymbol{\varpi}}
\newcom{\vop}{\vom^{+}}
\newcom{\vnu}{\boldsymbol{\nu}}
\newcom{\vopm}{\vom^{\pm}}
\newcom{\vjp}{\vj^+}
\newcom{\vjm}{\vj^-}
\newcom{\vjpm}{\vj^{\pm}}
\newcom{\vj}{\boldsymbol{\xi}}
\newcom{\Ds} {\langle\nabla\rangle^{s-\f12}}
\newcom{\vcY}{\ud{\mathcal{Y}}}
\newcom{\vcZ}{\ud{\mathcal{Z}}}
\newcom{\p}{\partial}
\newcom{\f}{\frac}
\newcommand{\dx}{{\rm d} {x}}

\newcommand{\dt}{{\rm d} t }

\newcom{\ds}{{\rm d} s}
\newcom{\ud}{\underline}
\newcom{\wu}{\ud{w}}
\newcom{\Wu}{\ud{W}}
\newcom{\Wo}{\overline{W}}
\newcom{\BR}{\mathbb{R}}
\newcom{\wt}{\widetilde}
\newcom{\sigu}{\ud{\sigma}}
\newcom{\qu}{\ud{q}}
\def\d{\mathrm{d}}
\def\dx{\mathrm{d}x}

\def\div{\mathop{\rm div}\nolimits}
\def\curl{\mathop{\rm curl}\nolimits}
\providecommand{\abs}[1]{\left\vert#1\right\vert}
\providecommand{\norm}[1]{\left\Vert#1\right\Vert}

\theoremstyle{remark}
\newtheorem{remark}[theorem]{Remark}

\begin{document}
\title{
%Global nonlinear stability of the plasma-vacuum interface
%problem for the 2D ideal incompressible MHD}
Global nonlinear stability of the plasma-vacuum free boundary problem for the 2D ideal incompressible MHD}

\author{Yuan Cai\footnote{School of Mathematical Sciences; LMNS and Shanghai Key Laboratory for Contemporary Applied Math, Fudan University, Shanghai 200433, P. R. China. Email: caiy@fudan.edu.cn
}
}
\date{}
\maketitle

\begin{abstract}
We address the plasma-vacuum interface problem in the two dimensional space.
The plasma region is  governed by
the ideal incompressible magnetohydrodynamic equations. % and
The vacuum region is described by the pre-Maxwell dynamics,
with the unknowns set to zero for simplicity.
This is a free boundary problem of ideal Euler type equations.
 Under the strong horizontal magnetic field in the plasma region,
we prove the global in time nonlinear stability of the two dimensional plasma-vacuum interface.
The proof relies on understanding the interplay between the dynamics of the fluids inside the domain and those on the free interface, as well as the inherent structures of the problem.
\end{abstract}

% With AMS-LaTeX, \maketitle follows the abstract
\maketitle

%%      ---------------------------------------------------------------------
%%      ------------------- TABLE OF CONTENTS (OPTIONAL) --------------------
%%      ---------------------------------------------------------------------

 \tableofcontents

%%      ---------------------------------------------------------------------
%%      ---------------------------- BODY OF PAPER --------------------------
%%      ---------------------------------------------------------------------

%%      Please input or insert the body of your paper here.

\section{Introduction}
\subsection{Presentation of the problem}

The plasma-vacuum interface problem concerns the free boundary between a magnetized plasma and an outer vacuum region. Such an interface is a characteristic feature in magnetic confinement devices (Tokamaks, Stellarators), where strong magnetic fields isolate the plasma from the surrounding vessel walls. The same free-boundary formulation also arises in astrophysical contexts, notably in modeling stellar motions and the solar corona when magnetic effects are included.

In this paper, we consider the plasma-vacuum free boundary problem in the two dimensional domain
$\Om=\BR\times [-1,\infty)$ separated by the free surface
\[
\Ga_f:= \left\{x=(x_1,x_2) | x_2=f(t,x_1), x_1 \in \mathbb{R} \right\}.
\]
The domain $\Om$ is thus divided into the plasma region $\Om_f$
 and the vacuum region $\Om_f^c$ where
\begin{align*}
\Om_f:&=\left\{ x=(x_1,x_2) | -1<x_2 < f(t,x_1),\, x_1\in\BR \right\},\\
\Om_f^c:&=\left\{ x=(x_1,x_2) | x_2 > f(t,x_1),\, x_1\in\BR \right\}.
\end{align*}
In the lower domain $\Om_f$, the evolution of plasma satisfies the following
two dimensional incompressible ideal  magnetohydrodynamics (MHD)
\begin{equation}\label{A0}
\begin{cases}
\p_t u + u\cdot\nabla u - h\cdot\nabla h + \nabla p= 0, \\
\p_t h + u\cdot\nabla h - h\cdot\nabla u = 0,\\
\div u = 0,\quad \div h=0,
\end{cases}
\text{in}\,\, \Om_f.
\end{equation}
Here $u=(u^1, u^2)$ denotes the velocity field, $h=(h^1,h^2)$ is the magnetic field,
$p=q + \f{1}{2}|h|^2$ is the total pressure and $q$ is the fluid pressure.
We consider the perfectly conducting magnetic fluids where
both the viscosity and resistivity vanish. This case is referred to as ideal MHD.

In the domain $\Omega^c_f$ above the free boundary, we assume the so-called {\it pre-Maxwell dynamics}. In such a case, the magnetic field $\hat h=(\hat{h}_1,\hat{h}_2)$ is determined by the div-curl system:
\begin{equation} \label{A3}
\div \hat{h}=0, \quad \curl \hat{h}=0\quad \text{in}\,\, \Omega^c_f.
\end{equation}
The physical quantities of the plasma and the vacuum region are related by the pressure balance condition on the interface $\Ga_f$:
\begin{equation}\label{A4}
p=\frac12|\hat{h}|^2\quad \text{on}\,\, \Ga_f
\end{equation}
as well as
\begin{equation}\label{A5}
h\cdot N_f = 0, \quad \hat h\cdot N_f = 0\quad \text{on}\,\, \Ga_f .
\end{equation}
Here
\[
N_f = (-\p_1 f, 1)
\]
is the outward normal to $\Ga_f$.
The free surface moves along the normal component of the velocity field,
thus  it satisfies the kinetic boundary condition:
\begin{equation}\label{A7}
\p_t f = u \cdot N_f\quad \text{on}\,\, \Ga_f .
\end{equation}
Moreover, on the artificial boundary $\Gamma^-=\BR\times \{-1\}$, we prescribe the following boundary conditions:
\begin{eqnarray}\label{A8}
u\cdot e_2=0,\,\, h\cdot e_2=0
\quad \text{on}\,\, \Gamma^- .
\end{eqnarray}

In this paper,
we consider the trivial solution $\hat{h}=0$ in the vacuum for simplicity.
This trivial solution obviously satisfies the {\it pre-Maxwell dynamics} \eqref{A3}.
Thus the physical balance conditions \eqref{A4} and \eqref{A5} on the interface $\Ga_f$ reduce to
\begin{align}
\label{A9}
 p&=0 \quad \text{on}\,\, \Ga_f,\\
\label{A10}
 h\cdot N_f &= 0 \quad \text{on}\,\, \Ga_f.
\end{align}
The system is supplemented with the initial data
\begin{equation} \label{A11}
\begin{cases}
u(0,x) = u^{0}(x), \quad
 h(0,x)=h^{0}(x)\quad\text{ in }\,\, \Omega_{f^0},\\
f(0,x_1) = f^0(x_1), \quad
\p_tf(0,x_1) = v^0(x_1)\quad\text{ in }\,\, \bR
\end{cases}
\end{equation}
which satisfy the following compatibility conditions:
\begin{equation}\label{A12}
\left\{\begin{aligned}
&\div u^{0} = 0, \,\, \div h^{0} = 0\quad \text{ in } \Om_{f^0},\,\,  \\
&h^{0}\cdot N_{f^0} = 0,\,\, v^0=u^0\cdot N_{f^0}\quad \text{ on }\,\, \Ga_{f^0},\\
&  u^0\cdot e_2=h^0\cdot e_2=0
\quad\text{ on }\,\, \Gamma^- ,
\end{aligned}\right.
\end{equation}
where $N_{f^0}=(-\p_1 f^0,1)$.

Here $u^{0}$ and $h^{0}$ denote the initial velocity field and magnetic field in the plasma region, while
$f^0, v^0$ denote the initial amplitude and velocity  of the free surface, respectively.
The initial data $u^{0}, h^{0}, f^0, v^0$ are assumed to be smooth and decay fast
enough at spatial infinity.

\begin{remark}
The divergence free restriction on $h$ is automatically
 preserved whenever $\nabla\cdot h^0=0$, thanks to the transport property
\begin{equation*}
\p_t (\nabla\cdot h)+u\cdot\nabla  (\nabla\cdot h)=0\,\,\,\text{in}\,\, \Om_f .
\end{equation*}
A similar argument can  also be applied to yield $h\cdot N_f=0$ on $\Ga_f$ if $h^0\cdot N_{f^0}=0$ holds on $\Ga_{f^0}$.
\end{remark}

In this paper, we establish the global nonlinear stability for the plasma-vacuum free boundary
problem for the two dimensional ideal incompressible MHD in the presence of a strong magnetic background.
Precisely, we assume the background magnetic field is aligned with the horizontal direction and set it to be $e_1=(1,0)$.
We prove that the equilibrium state $(u, h, f) = (\textbf{0}, e_1,0)$ is nonlinearly stable under sufficiently small initial perturbations. Such kinds of solutions will produce Alfv\'en  waves \cite{Alfven42}.
The equation of the interface 
derived in \cite{SWZ18, SWZ19} and the stability effect of strong magnetic field \cite{BSS, CL, HeXuYu,WZ}
played a crucial role in our current work. Apart from those important existing observations,
our proof highly relies on the understanding of the interplay between the dynamics of the fluids inside the domain and on the free interface, a design of multiple-level energy estimates with different weights, as well as the inherent structures of the system.
We will explain them in more detail in Section \ref{MaI}.

\subsection{Review of the closely related works}

There is a large body of mathematical literature on free boundary problems in fluid mechanics.
It would be impossible to give exhaustive references.
We only  briefly review some closely related historic works.
Concerning the free boundary problem for the incompressible Euler equations, the early literature on irrotational flows focuses on small perturbations of an initial interface, including Nalimov \cite{N}, Yosihara \cite{Yo}, Kano-Nishida \cite{KN}, and Craig \cite{Craig}.
The  local well-posedness for general initial data was established by the breakthrough work of Wu \cite{Wu1,Wu2}.
She showed that in dimensions two and three,
the strong Rayleigh-Taylor sign condition
\begin{align*}
-\f{\p p}{\p n}\geq c_0 >0
\end{align*}
always holds for the infinite depth water wave problem and that local-in-time solutions
can be constructed in Sobolev spaces $H^s$, $s \geq 4$ for  initial data of arbitrary size.
Further results include Beyer and G\"unther \cite{BG} on surface tension, and Lannes \cite{Lannes} on non-trivial bottom topography.
For the general free boundary problem in incompressible Euler equations, the local well-posedness %in 3D
 was independently obtained by Lindblad \cite{Lindblad05}
(see Christodoulou and Lindblad \cite{CL_00} for the a priori estimates),
 Coutand and Shkoller \cite{CS07},
Shatah and Zeng \cite{SZ1},
 Zhang and Zhang \cite{ZZ}.
For the study on singularities
 of free boundary problems,
we refer to Castro, C\'{o}rdoba, Fefferman, Gancedo and G\'{o}mez-Serrano
\cite{CCFG1, CCFG2}, Fefferman, Ionescu and Lie \cite{Feff} and Coutand \cite{Coutand} among others.
We also refer to Alazard, Burq and Zuily \cite{ABZ-D, ABZ}, Ambrose and Masmoudi \cite{Am05, Am07, Am09},
  Christianson, Hur and Staffilani \cite{CHS},
Coutand and Shkoller \cite{DS_10}, Gu and Lei \cite{Gu_2011, Gu_2016},
Ifrim, Pineau, Tataru and Taylor \cite{IPTT},
Masmoudi and Rousset \cite{MasRou},  Wang and Xin \cite{Wang_15} %, Wu \cite{Wu5},
 etc. for more works on the free boundary problem of inviscid fluids.

For irrotational inviscid fluids with gravity or surface tension, dispersive effects yield long-time existence for small data. Global three dimensional gravity water waves were obtained by Germain, Masmoudi, Shatah \cite{GMS1} and Wu \cite{Wu4}, with the capillary case later treated by Germain, Masmoudi, Shatah \cite{GMS2}. In the more challenging two dimensional setting, Wu \cite{Wu3} proved almost global existence, while Ionescu, Pusateri \cite{IP} and Alazard, Delort \cite{AD} established global results (see also Hunter, Ifrim, Tataru \cite{HIT,IT} for alternative proofs). For the two dimensional capillary waves, global existence is due to Ifrim, Tataru \cite{IT17} and Ionescu, Pusateri \cite{IP2}. Additional global results include Deng, Ionescu, Pausader, Pusateri \cite{DIPP} for the three dimensional gravity-capillary system and Wang \cite{W,W1} for the three dimensional finite-depth case.

For ideal MHD free boundary problems,
the fact that the magnetic field has a stabilizing effect for the current-vortex
sheet was found  long before by physicists Syrovatskij \cite{Sy} and Axford \cite{Ax}.
For the local in time stability of the incompressible current-vortex sheets which describe velocity and magnetic field discontinuities in the ideal MHD,  Coulombel, Morando, Secchi and Trebeschi \cite{CMST} obtained the \emph{a priori} estimate for the nonlinear problem.
 The nonlinear stability was solved by Sun, Wang and Zhang \cite{SWZ18}.
For more related studies on the local in time
stability of the compressible current-vortex sheets, we refer to Trakhinin \cite{Tra1, Trak_09}, Chen and Wang \cite{Chen_08}, Wang and Yu \cite{WY}.

Trakhinin \cite{Trak_10} introduced the following non-collinearity condition for the linearized compressible plasma-vacuum interface problem:
\begin{equation} \label{st}
\big| h\times \hat{h} \big|>0\ \ \text{on}\ \ \Ga_f.
\end{equation}
Under this condition, Secchi and Trakhinin \cite{Secchi_13, Secchi_14} proved the well-posedness of the compressible plasma-vacuum interface problem. For the incompressible plasma-vacuum free boundary problem,  the linear local well-posedness was
 proved by Morando, Trakhinin and Trebeschi \cite{Mo_14}, and the nonlinear stability was established by Sun, Wang and Zhang \cite{SWZ19}. The two dimensional case was treated by Liu and Luo \cite{LiuLuo} under the stability condition
 \begin{equation} \label{st1}
|h|+ |\hat{h}|\geq \lambda_0>0\ \ \text{on}\ \ \Ga_f.
\end{equation}
Further results include Gu \cite{Gu_17} for the corresponding three dimensional axially symmetric problem,
Liu and Xin \cite{LiuXin} for the non-graph nonlinear result.
 On the other hand, under the Rayleigh-Taylor sign condition,
 Hao and Luo \cite{Hao_13}  established
a priori estimates for the incompressible plasma-vacuum interface problem.
The local in time well-posedness was proved  by Gu and Wang \cite{GW_16}
for the case where the vacuum magnetic field is zero.
For more results, we refer to Hao and Luo \cite{Hao_20},
Zhao 
 \cite{ZhaoWB}, Ifrim, Pineau, Tataru and Taylor \cite{IPTT-MHD} among others.

For the global in time results on the MHD free boundary problem,
 Wang and Xin \cite{Wang_20}  proved the global existence of solutions for the plasma-vacuum and plasma-plasma interface problem for the incompressible inviscid resistive MHD
in a horizontally periodic slab impressed by a uniform non-horizontal magnetic field
and with surface tension on the free surface.
The author of this paper and Lei \cite {CL-24} proved the global in time nonlinear stability of the
two dimensional current-vortex sheets in the ideal incompressible MHD.

\subsection{Els\"{a}sser variables }
To study the global solutions of the plasma-vacuum free boundary problem, we need to explore the propagation of Alfv\'{e}n waves. %More importantly, the global stability effect of the free boundary.
To this end, we use the following  Els\"{a}sser variables %to study the propagation of Alfv\'{e}n waves:
\begin{equation*} %\label{A13}
\begin{cases}
Z_+=u + h,\\
Z_-=u - h,
\end{cases}
\textrm{in}\,\,\Omega_f,
\end{equation*}
In terms of the Els\"{a}sser variables, the MHD system \eqref{A0} is written as follows
\begin{equation} \label{MHD}
\begin{cases}
\p_t Z_+ +Z_-\cdot\nabla Z_+ + \nabla p =0, \\
\p_t Z_- +Z_+\cdot\nabla Z_- + \nabla p =0,\\
\div Z_+ = 0,\quad \div Z_-=0,
\end{cases}
\textrm{in}\,\, \Omega_f.
\end{equation}
The stability of the free surface is achieved under strong background magnetic field.
Inspired by the stability condition \eqref{st} and \eqref{st1},
we set the background magnetic field in the plasma domain to be
$$e_1=(1,0).$$
Then we introduce the perturbation variables:
\begin{equation*}
\Lambda_+=Z_+ - e_1, \quad \Lambda_-= Z_- + e_1 \quad \textrm{in}\,\, \Omega_f.
\end{equation*}
In terms of $\Lambda_\pm$, the ideal incompressible MHD system \eqref{MHD} is written as follows
\begin{equation} \label{MHD1}
\begin{cases}
\p_t \Lambda_+ - e_1\cdot \nabla \Lambda_+
+\Lambda_-\cdot\nabla \Lambda_+ + \nabla p =0,\\
\p_t \Lambda_- + e_1\cdot \nabla \Lambda_-
+\Lambda_+\cdot\nabla \Lambda_- + \nabla p =0,\\
\div \Lambda_+ = 0,\quad \div \Lambda_-=0,
\end{cases}
\textrm{in}\,\, \Omega_f.
\end{equation}
Then the boundary conditions \eqref{A8} become
\begin{align} \label{A18}
\Lambda_+\cdot e_2=0,\,\, \Lambda_-\cdot e_2=0
\quad \text{on}\,\, \Gamma^- .
\end{align}
And \eqref{A7}, \eqref{A9}-\eqref{A10} become
\begin{align}
\label{A19}
p&=0 \quad \text{on}\,\, \Ga_f,\\
\label{A21}
\p_t f &= Z_\pm \cdot N_f\quad \text{on}\,\, \Ga_f .
\end{align}
The initial data \eqref{A11} become
\begin{equation*}%\label{A24}
\begin{cases}
\Lambda_{\pm}(0,x)= \Lambda_{\pm}^0(x) \,\,\,\, \text{in}\,\, \Om_{f_0},\\
f(0,x_1)=f^0(x_1),\quad\p_tf(0,x_1)=v^0(x_1),\,\, x_1\in\bR,
\end{cases}
\end{equation*}
where $\Lambda_{\pm}^0=u^{0}\pm (h^{0}-e_1)$.
The compatibility conditions \eqref{A12} of the initial data become
\begin{equation}\label{com}
\left\{\begin{aligned}
&\div \Lambda_+^0 = 0, \,\, \div \Lambda_-^0 = 0\,\, \text{ in } \Om_{f^0},\,\,  \\
& v^0=(\Lambda_\pm^0 \pm e_1)\cdot N_{f^0} \,\, \text{ on } \Ga_{f^0},\\
& \Lambda_+^0\cdot e_2=0,\,\, \Lambda_-^0\cdot e_2=0 \,\,\text{ on }\Gamma^- .
\end{aligned}\right.
\end{equation}
We will prove that for sufficiently small initial disturbance,
the weighted Sobolev norms of $\Lambda_\pm$ and $f$
will be uniformly bounded for all time.
This, combined with the local existence theory, will yield the global solutions of the free boundary problem for the plasma-vacuum
near the equilibrium.

\subsection{Main theorem}
The weighted energy $E_s(t)$, the ghost weight energy $G_s(t)$ and other notations appearing in the following theorem will be explained in Section 2. Note that
\begin{align*}
E_s(0) =
&\sum_{|\alpha|\leq s}
\big( \| \langle x_1\rangle^{2\mu} \nabla^\alpha\La_+^0(x)\|^2_{L^2(\Om_{f^0})}
+\|\langle x_1\rangle^{2\mu} \nabla^\alpha\La_-^0(x)\|^2_{L^2(\Om_{f^0})}\big)\\
&+\sum_{|a|\leq s-1}\big( \| \langle x_1 \rangle^{2\mu} \langle \p_1 \rangle^{\f12 }\p_1^a v^0(x_1)\|^2_{L^2(\bR)}
 +\|\langle x_1 \rangle^{2\mu} \langle \p_1 \rangle^{\f12 } \p_1^{a+1}f^0(x_1)\|^2_{L^2(\bR)}\big).
\end{align*}
We state the main result of this paper  as follows.
\begin{thm}\label{thm}
Let $s\geq 4$ be an integer, $1/2< \mu \leq 3/4$, $\delta<\f12$. Assume $(\Lambda_+^0, \Lambda_-^0) \in H^s(\Omega_{f^0})$,
$f^0 \in  \dot{H}^1(\BR) \cap \dot{H}^{s+\frac12}(\BR) \cap L^{\infty}(\BR)$ and
$v^0\in H^{s-\f12}(\BR)$. Suppose that $\Lambda_+^0$, $\Lambda_-^0$, $f^0$
 and $v^0$ satisfy the compatibility conditions \eqref{com}
and
\begin{align*}
E_s(0)\leq \epsilon, \quad \| f^0(x_1)\|_{L^\infty(\BR)}\leq 1-2\delta.
\end{align*}
There exists a positive constant $\epsilon_0$ which
depends only on $s,\ \mu,\,\delta$ such that, if $\epsilon
\leq \epsilon_0$, then the free boundary problem for the plasma-vacuum of the two dimensional incompressible MHD  \eqref{MHD1}-\eqref{A21}  with  the following initial data
\begin{align*}
&\Lambda_+(0,x)= \Lambda_+^0(x),\quad \Lambda_-(0, x)= \Lambda_-^0(x),\\
& f(0,x_1)=f^0(x_1),\quad
\p_tf(0,x_1)=v^0(x_1)
\end{align*}
has a unique global classical solution which satisfies
\begin{eqnarray}\nonumber
E_s(t) +\int_0^t G_s(\tau )\d\tau \leq C_0E_s(0),
\quad
\| f(t,\cdot)\|_{L^\infty}\leq 1-\delta
\end{eqnarray}
for some $C_0 > 1$ uniformly for all $t\geq 0$.
\end{thm}

\begin{remark}
The ideas from \cite{LiuLuo,SWZ18,SWZ19} can be adapted to prove local well-posedness for \eqref{MHD1}--\eqref{A21}. Unlike those works, we take Sobolev norms on $(u,\,h-e_1,\,f)$ instead of on $(u,\,h,\,f)$.
In order to present a self-contained proof, we will sketch the proof of the local well-posedness in the last section.
Actually, due to the inherent structure,
we could work on the unknowns $(\La_\pm,\,f)$.
\end{remark}

\begin{remark}
The method from the author of the paper and Lei \cite{CL-24} for proving the global current-vortex sheets in the two-dimensional ideal
incompressible MHD are applicable to the plasma-vacuum free boundary problem in this paper.
However, the method would necessitate the weights of the form $\langle w^\pm\rangle^{5\mu}$ ($\mu>\f12$) rather than $\langle w^\pm\rangle^{2\mu}$. We refer to Remark \ref{rem-hw} for the detailed discussion.
In contrast, this paper adopts a different strategy.
We utilize weights of $\langle w^\pm\rangle^{2\mu}$ on $\La^\pm$ and $(\p_t\pm Z^1_\pm\p_1)f$, which are found to be   sufficient.
We work directly on $\La^\pm$ instead of employing the vorticity formulation in \cite{CL-24}.
Moreover, we can obtain $H^{s+1}-$order estimate of the pressure rather than only $H^s-$order estimate in \cite{CL-24}. Consequently, the proof becomes more streamlined.
\end{remark}
\begin{remark}
The weighted energy $E_s$ is small for all time, implying that the derivatives of $f$ are also small.
The amplitude of the free surface $\| f \|_{L^\infty}$ does not necessarily need to be small.
The condition $\|f^0\|_{L^\infty}<1-2\delta$ is imposed to ensure that the free surface will not intersect with the  bottom fixed boundary.
\end{remark}
\begin{remark}
For the case where both the upper and the lower boundaries are moving free surfaces, the methods in this paper still work and the proof is the same.
\end{remark}
\begin{remark}
The global solutions obtained in this paper work in the two dimensional setting.
Whether the global nonlinear stability result holds for the general three dimensional case is not clear.
Under strong magnetic field in the plasma region and zero magnetic field in the vacuum region, the local in time stability result is not clear for the corresponding three dimensional problem.
Actually, this setting seems do not satisfy the
non-collinearity stability condition
\eqref{st} in three dimensional case.
\end{remark}

\subsection{Main ideas}
There are two classical methods in proving the long time existence of solutions in the water wave problems:
the vector field method
and the method of space-time resonances.
Wu \cite{Wu3, Wu4}, Alazard, Delort \cite{AD}, Hunter, Ifrim, Tataru \cite{HIT, IT} and some experts use the (variant) vector field method.
Germain, Shatah, Masmoudi \cite{GMS1, GMS2},  Ionescu, Pusateri \cite{IP, IP2} and some other experts  use the method of space-time resonances.
Most of these works rely on the irrotational condition which reduce the system to the free boundary.
By the expansion of the Dirichlet-Neumann operator, it reduces to the study of the quasilinear dispersive equations.
When the vorticity is nontrivial, we need a new strategy.
Our work on global solutions for the two dimensional plasma-vacuum free boundary problem does not need the irrotational condition.
It relies on the strong magnetic field in the plasma region which will produce the propagation of Alfv\'{e}n waves \cite{Alfven42}.
This is  inspired by the work of Bardos, Sulem, Sulem \cite{BSS}, Cai, Lei \cite{CL}, He, Xu, Yu \cite{HeXuYu} and Wei, Zhang \cite{WZ} on global small solutions of ideal incompressible MHD in $\BR^n$. Since we are working on nontrivial vorticity fluids, the problem cannot be reduced to the boundary.
We need to explore the dynamics inside the domain, the interplay between inside the domain and the free boundary.
Hence the analysis becomes rather involved.

Now let us show  how the basic strategy goes and the key ideas.
Inspired by the work of Sun-Wang-Zhang  \cite{SWZ18, SWZ19} on the local solutions
 for plasma-vacuum interface problem for ideal incompressible MHD,
we derive the second order evolution equation for the free surface in terms of the Els\"{a}sser variables.
Under the strong background magnetic field assumption along
$e_1$ direction, the equation of the free surface \eqref{B13} turns out to be one dimensional nonlocal quasilinear wave equation.
Moreover, $ (\p_t+\p_1) f$ and $(\p_t-\p_1) f$  propagate
along the background magnetic field
 in opposite directions. Due to the kinetic boundary condition \eqref{A21}, heuristically, at the linearization level, $ (\p_t+\p_1) f\sim \Lambda_+^2$,
$ (\p_t-\p_1) f\sim \Lambda_-^2$.
On the other hand, one observes $\Lambda_{+}$ and $\Lambda_{-}$  propagate
along the background magnetic field
in opposite directions.
For the nonlinearities of \eqref{MHD1}, they  are always the combination of the interaction of
the left Alfv\'en waves and the right Alfv\'en waves.
Furthermore,
for the nonlinearities
of  both equations in the bulk (see \eqref{B11}, \eqref{B14}) and in the free surface (see \eqref{B13}), they are always the combination of the interaction of left Alfv\'en waves $\Lambda_+, (\p_t+\p_1)f$
and the right Alfv\'en waves $\Lambda_-, (\p_t-\p_1)f$.
In other words, both equations in the bulk and in the free surface satisfy the
same type of null conditions.
This structure makes it reasonable to expect the nonlinearities will finally be negligible in the perturbation regime.

In order to capture these  structures,
we design suitable weight functions to obtain the directional decay of solutions.
At the same time, we
apply the ghost weight energy method by Alinhac \cite{Alinhac00} to perform energy estimates with weights. This yields the fact that energies of good unknowns with appropriate weights are always integrable in time.
This method is philosophically similar to the space-time resonance.
The left and the right Alfv\'en waves reflect the separation of the characteristics and the null condition suggests the nonresonance.
On the other hand, it also shares many similarities to the vector field method: appropriate weight functions are applied, then we can expect
some form of temporal decay.

The method developed by the author of this paper and Lei \cite{CL-24} for establishing the global current-vortex sheets in the two-dimensional ideal
incompressible MHD is also applicable to the plasma-vacuum free boundary
problem for the two dimensional ideal incompressible MHD. In this paper, however, we pursue a different strategy.
Working directly on $(\p_t\pm\p_1)f$ with weights $\langle w^\pm\rangle^{2\mu}$ ($\mu>1/2$) leads to the fact that the nonlinearities are too strong: they will require increased power of weights $\langle w^\pm\rangle^{5\mu}$ for $\p^{\leq 1}\La^\pm$.
Our key idea in this paper is to move some of the nonlinear terms of the free surface equation to
the left hand side and incorporate them into the ``linear terms".
Specifically, for the free surface equation, we will
use the form in Lemma \ref{lemf} instead of
\eqref{B13}. Thus we perform the energy estimate for the free surface on  $(\p_t+Z_\pm^1 \p_1)f$ rather than on $(\p_t\pm\p_1)f$.
Part of these ideas can be traced to the seminal work of Christodoulou-Klainerman where they proved the global nonlinear stability of Minkowski space-time in general relativity \cite{CK}. In order to study global solutions,  we will treat energy estimate with weight functions.
%In order to do so, we need pass the
These weight functions commute with the operators
 $(\p_t+Z_\pm^1 \p_1)$.
 We therefore define them as solutions of suitable transport equations.
It turns out that by using this method, the weights $\langle w^\pm\rangle^{2\mu}$
applied on $\La^\pm$ and $(\p_t+Z_\pm^1 \p_1)f$ suffice.

We will treat $s$-order weighted energy estimate in the bulk of the region and
 will treat $(s+\f12)$-order weighted energy for the free surface.
Moreover, we can obtain the $(s+1)$-order weighted energy estimate of pressure.
The choice of these orders relies
on the vanishing condition of pressure on the free boundary and the structure of the ideal MHD system.
The $(s+1)$-order estimate of pressure enables us to work on the velocity field and the magnetic field, or say their equivalent Els\"{a}sser variables and perturbation variables instead of on vorticity formulation.
The estimate of pressure is carried out in a fixed domain with flattened boundary .
Thus the weighted tangential derivative estimate is carried out directly by energy method due to the straightening of the free boundary.

The remaining part of this paper is organized as follows: In
Section 2, we will derive the second order evolution equation of the free surface.
Then, at an intuitive level, we exhibit the Alfv\'en waves and the null condition. The various energy functionals and the ansatz for the continuity method are also introduced.
In Section 3,  the properties
 of the weight function will be studied.
Section 4 is devoted to the weighted Sobolev inequalities, weighted trace theorem and weighted commutator estimate.
Section 5 is devoted to the weighted estimate of pressure.
In Section 6, we treat the weighted energy estimate of the plasma.
Then we treat the $L^\infty$ estimate and weighted energy estimate of the free surface in Section 7.
In the last section, we will show the local well-posedness of the plasma-vacuum free boundary
problem for the two dimensional ideal incompressible MHD.

\section{Equations for the free surface, null condition and energy estimate ansatz}\label{MaI}
In this section, we first present the second-order evolution equation for the free surface in Lemma \ref{lemB2}.
The key point is that this equation turns out to be a strictly hyperbolic system under some stability condition.
Then by collecting the evolution equations for $\Lambda_\pm$ in the plasma region, the evolution equation for the free surface, the elliptic equation for the pressure and their boundary conditions, we briefly look into the structure of the systems and explain that they satisfy the null condition. Then we introduce various energy notations and explain the ansatz for the
continuity method.

Throughout the whole paper, we denote $\p=\nabla=(\p_1,\p_2)$, $\p^\alpha=\p_1^{\alpha_1}\p_2^{\alpha_2}$ for $\alpha\in\mathbb{N}^2$. For $1\leq p\leq +\infty$, we use $\|g \|_{L^p(\Omega_f)}$  to denote the $L^p$ norm in $\Omega_f$.
 %$\|g \|_{L^p(\mathbb{R})}$ is abbreviated as $\|g \|_{L^p}$.
 We use Japanese brackets to represent $\langle a\rangle=(1+a^2)^{\frac{1}{2}}$. For $m\in\mathbb{N}$, denote $|\p^m h|=\sum_{|\alpha|=m} |\p^\alpha h|$, $|\p^{\leq m} h|=\sum_{|\alpha|\leq m} |\p^\alpha h|$. Similarly, $\|w\p^{\leq m}g\|_{L^p}=\sum_{|\alpha|\leq m}\|w\p^\alpha g\|_{L^p}$.
%Throughout the whole paper,
For function $g:\Om_f\rightarrow\BR$, we use $\ud{g}$ to denote the trace of $g$ on $\Ga_f.$
The summation convention over repeated indices is always used.
Throughout this paper, we use $A\lesssim B$ to denote $A \leq C B$ for some positive constant $C$,
 whose meaning is clear from the context.
 %may change from line to line. Without
%specification, the constant $C$ depends only on $\mu$, $s$,  but not on $t$.

\subsection{Evolution equation of the free surface}

We copy the moving equation of the free surface \eqref{A21} as follows
\begin{align}\label{B2}
\p_t f =\ud{Z_\pm} \cdot N_f .
\end{align}
where $N_f$ is the normal vector field of the surface
$$ %N=-\nabla f(t,x_1)+e_2
N_f=-\nabla f(t,x_1)+e_2 .$$
Note that $f(t,x_1)$ depends only on $(t,x_1)$, hence
\begin{equation*}
\nabla f(t,x_1)=(\p_1,\p_2) f(t,x_1)=
(\p_1f(t,x_1),0) .
\end{equation*}
In some occasions, this notation will simplify the presentation of the argument.

Now we derive the second order evolution equation for the free surface
in terms of the Els\"{a}sser variables.
\begin{lem}\label{lemB2}
For the  free surface, there holds
\begin{align} \label{B3}
\p_t^2 f=
&-(\ud{Z_+^1}+\ud{Z_-^1})
\cdot\p_1 \p_t f
-\ud{Z_-^1} \ud{Z_+^1}\p_1^2 f
- N_f \cdot\ud{\nabla p}.
\end{align}
\end{lem}
\begin{proof}
%Now, we turn to $\p_t^2 f(t,x_1)$.
Employing the expression of $\p_tf$, $N_f$ and the first equation of \eqref{MHD}, by chain rules, we calculate
%\begin{align*}
%%\p_t^2 f
%&(\p_t Z_+ +\p_tf\p_2 Z_+)\big|_{x_2=f(t,x_1)}\\
%&=\big(-Z_-\cdot\nabla Z_+ -\nabla p + Z_-\cdot(-\nabla f+e_2) \p_2 Z_+\big)\big|_{x_2=f(t,x_1)}\nonumber\\\nonumber
%&=\big(-Z_-^h \cdot\nabla_h Z_+ -\nabla p -Z_-^h \cdot \nabla_h f \p_2 Z_+\big)\big|_{x_2=f(t,x_1)}\nonumber\\\nonumber
%&=\big(-Z_-\cdot\nabla \ud{Z_+} -\nabla p \big)\big|_{x_2=f(t,x_1)} .
%\end{align*}
\begin{align*}
%\p_t^2 f
&(\p_t Z_+ +\p_tf\p_2 Z_+)\big|_{x_2=f(t,x_1)}\\
&=\big(-Z_-\cdot\nabla Z_+ -\nabla p + Z_-\cdot(-\nabla f+e_2) \p_2 Z_+\big)\big|_{x_2=f(t,x_1)}\nonumber\\\nonumber
&=\big(-Z_-^1 \cdot\p_1 Z_+ -\nabla p -Z_-^1 \cdot \p_1 f \p_2 Z_+\big)\big|_{x_2=f(t,x_1)}\nonumber\\\nonumber
&=\big(-Z_-\cdot\nabla \ud{Z_+} -\nabla p \big)\big|_{x_2=f(t,x_1)} .
\end{align*}
Consequently, taking the derivative of \eqref{B2} in $t$,
%by integration by parts,
 we get
\begin{align}\label{B4}
\p_t^2 f=
&\big( (\p_t Z_+ +\p_2 Z_+\p_tf)\cdot N_f+Z_+\cdot\p_t N_f \big)\big|_{x_2=f(t,x_1)}\nonumber \\
=&\big(-(Z_- \cdot\nabla \ud{Z_+}) \cdot N_f -\nabla p\cdot N_f \big)\big|_{x_2=f(t,x_1)} +Z_+\cdot\p_t N_f\big|_{x_2=f(t,x_1)}\nonumber\\
=&-\ud{Z_-}\cdot\nabla \p_t f-\ud{Z_+}\cdot\nabla \p_t f -N_f\cdot\ud{\nabla p}
- \ud{Z_-^i}\ud{Z_+^j}\p_i\p_j f
%(\ud{Z_-}\otimes \ud{Z_+}):(\nabla\otimes \nabla) f.
\end{align}
%Here $(\ud{Z_-}\otimes \ud{Z_+}):(\nabla\otimes \nabla) f=\ud{Z_-^i}\ud{Z_+^j}\p_i\p_j f$.
Note that $\ud{Z_\pm}$ and $f$ depend only on $(t,x_1)$.
Therefore, \eqref{B4} yields \eqref{B3}.
%\begin{align}\label{B5}
%\p_t^2 f=
%&-(\ud{Z_+^1}+\ud{Z_-^1})
%\cdot\p_1 \p_t f
%-\ud{Z_-^1} \ud{Z_+^1}\p_1^2 f
%- N_f \cdot\ud{\nabla p}.
%\end{align}
\end{proof}
We rewrite the free-surface equation into a form that is better suited to the weighted energy estimate.
\begin{lem}\label{lemf}
For the free surface, there hold
\begin{align*} %\label{B6}
&(\p_t + \ud{Z_+^1}\p_1)(\p_t + \ud{Z_-^1}\p_1)f
=-\ud{\p_2 p},\\\nonumber
&(\p_t + \ud{Z_-^1}\p_1)(\p_t + \ud{Z_+^1}\p_1)f
=-\ud{\p_2 p} .
\end{align*}
\end{lem}
\begin{proof}
The proof of the first and the second expression of Lemma \ref{lemf} is the same. Hence in the sequel, we only present the details
for the first one.

Firstly, there naturally holds
\begin{align}\label{B7}
&(\p_t + \ud{Z_+^1}\p_1)(\p_t + \ud{Z_-^1}\p_1)f \\ \nonumber
&=\p_t^2 f+(\ud{Z_+^1}+\ud{Z_-^1})
\p_1\p_t f
+(\ud{Z_-^1}\ud{Z_+^1})\p_1^2 f\\ \nonumber
&\quad+(\p_t + \ud{Z_+^1}\p_1)\ud{Z_-^1}\p_1f.
\end{align}
On the other hand, note that \eqref{B2} can be organized as follows:
\begin{align*}
(\p_t +\ud{Z_\pm^1}\p_1)f=\ud{Z_\pm^2}.
\end{align*}
Hence, %restricting \eqref{MHD} to the free surface,
 by chain rules, we derive
%\begin{align*}
%(\p_t +\ud{Z_-^1}\p_1)\ud{Z}^+ +\ud{\nabla p}=0 .
%\end{align*}
\begin{align*}
(\p_t +\ud{Z_+^1}\p_1)\ud{Z_-}
&=\ud{(\p_t +{Z_+^1}\p_1)Z_-}
+\ud{\p_2 Z_-}  (\p_t +\ud{Z_+^1}\p_1)f \\\nonumber
&=\ud{(\p_t +{Z_+^1}\p_1 +{Z_+^2}\p_2 )Z_-}
=-\ud{\nabla p}.
\end{align*}
Consequently, combined with Lemma \ref{lemB2}, \eqref{B7} yields
\begin{align*} %\label{B8}
(\p_t + \ud{Z_+^1}\p_1)(\p_t + \ud{Z_-^1}\p_1)f
=-\p_1f\ud{\p_1 p}
- N_f\cdot\ud{\nabla p}=-\ud{\p_2 p}.
\end{align*}
This finishes the proof of the lemma.
\end{proof}

\subsection{Alfv\'en waves and the null condition}
Recalling the evolution equations for  $\Lambda_+$ and $\Lambda_-$ of \eqref{MHD1} as follows:
\begin{equation} \label{B11}
\begin{cases}
\p_t \Lambda_+ - e_1\cdot \nabla \Lambda_+
+\Lambda_-\cdot\nabla \Lambda_+ + \nabla p =0,\\
\p_t \Lambda_- + e_1\cdot \nabla \Lambda_-
+\Lambda_+\cdot\nabla \Lambda_- + \nabla p =0,\\
\div \Lambda_+ = 0,\quad \div \Lambda_-=0,
\end{cases}
\textrm{in} \quad \Omega_f.
\end{equation}
In terms of $\Lambda_+$ and $\Lambda_-$, the moving of free surface \eqref{B2} is written as follows
%\begin{align*}
%&\p_tf=\ud{Z_\pm} \cdot N
%= \mp\p_1 f+\ud{\Lambda_\pm} \cdot N.
%\end{align*}
%We write the above equations in another way
\begin{align*}
&(\p_t+\p_1) f=\ud{\Lambda_+} \cdot N_f,\\
&(\p_t-\p_1) f=\ud{\Lambda_-} \cdot N_f.
\end{align*}
On the other hand, in terms of $\Lambda_\pm$, the equation of the free surface \eqref{B3} is written as follows:
\begin{align}\label{B12}
\p_t^2 f-\p_1^2 f
=&- (\ud{\Lambda_+^1} +\ud{\Lambda_-^1})
\p_t\p_1 f
-\ud{\Lambda_-^1}\p^2_1 f + \ud{\Lambda_+^1}\p^2_1 f
- \ud{\Lambda_-^1}  \ud{\Lambda_+^1}  \p_1^2 f
- N_f\cdot\ud{\nabla p}.
\end{align}
To show the null structure, the above \eqref{B12} can be organized as follows:
\begin{align}\label{B13}
\p_t^2 f-\p_1^2 f=&- \ud{\Lambda_+^1}\p_1(\p_t-\p_1)f
  - \ud{\Lambda_-^1}\p_1(\p_t+\p_1)f
-\ud{\Lambda_-^1}  \ud{\Lambda_+^1} \p_1^2 f
- N_f\cdot\ud{\nabla p}.
\end{align}
As for the pressure, applying divergence operator to the first equation of \eqref{B11},
we obtain
\begin{equation*}
-\Delta  p =\div \nabla\cdot(\Lambda_- \otimes \Lambda_+) \quad \mathrm{in}\,\, \Om_f.
\end{equation*}
For the boundary condition of the pressure on $\Ga_-$, restricting $\eqref{B11}_1$ on $\Ga_-$ in the trace sense, there holds the following Neumann boundary condition for the pressure
\begin{align*}%\label{BB12}
\p_2 p=0\,\, \textrm{on} \,\,  \Ga_-.
\end{align*}
Combined with \eqref{A19}, we collect the elliptic equations and the boundary conditions for the pressure as follows
\begin{equation}\label{B14}
\begin{cases}
-\Delta  p %&=\nabla_j\Lambda_-^i\nabla_i\Lambda_+^j\\
=\div \nabla\cdot(\Lambda_- \otimes \Lambda_+) \quad \mathrm{in}\,\, \Om_f, \\
p =0 \quad  \text{on}\,\, \Ga_f  ,\\
\p_2 p=0 \quad \text{on}\,\, \Gamma_-.
\end{cases}
\end{equation}

From the linearized equations of \eqref{B11} and \eqref{B13},
it can be seen that $\Lambda_{+}$, $ (\p_t+\p_1) f$ and $\Lambda_{-}$, $(\p_t-\p_1) f$
propagate along the background magnetic field in opposite directions.
For the nonlinear terms of \eqref{B11} and \eqref{B13},
they are always the combination of the interaction of the left Alfv\'en waves
$\Lambda_+$, $(\p_t+\p_1)f$
and the right Alfv\'en waves $\Lambda_-$, $(\p_t-\p_1)f$.
On the other hand, we see from \eqref{B14} that the pressure term
is also the nonlinear interaction of
$\Lambda_+$ and $\Lambda_-$ modulo the pseudo-differential operator.
In other words, both equations  of the plasma \eqref{B11} and of the free surface \eqref{B13} satisfy the
same type of null conditions.

\subsection{Some notations and ansatz for the continuity argument}\label{energy}
Inspired by Christodoulou-Klainerman \cite{CK} and
He-Xu-Yu \cite{HeXuYu}, we introduce the weight functions $w^\pm(t,x)$ through the following equations
\begin{equation}\label{C0}
\begin{cases}
(\p_t + Z_+^1\p_1)w^-(t,x)=0,\quad w^-(0,x)=x_1,\\
(\p_t + Z_-^1\p_1)w^+(t,x)=0,\quad w^+(0,x)=x_1.
\end{cases}
\end{equation}

Employing the weight functions, for positive integer $s$ and $1/2<\mu \leq 3/4$, we define the weighted energy of the plasma as follows:
\begin{align*}
E_s^b (t)=\sum_{+,-}\sum_{ |\alpha|\leq s}
\| \langle w^\pm \rangle^{2\mu}\nabla^\alpha \La_\pm\|_{L^2(\Om_f)}^2 \,.
\end{align*}
%Here $\alpha\in\mathbb{N}^2$ is multi-index.
%Note that here the index for weight function applied on unknowns with derivative and without derivative are different.
The ghost weight energy is defined as follows:
\begin{align*}
G_s^b (t)
=\sum_{+,-}\sum_{ |\alpha|\leq s}
\Big\| \frac{\langle w^\pm \rangle^{2\mu}\nabla^\alpha \Lambda_\pm}{ \langle w^\mp \rangle^{\mu} }\Big\|_{L^2(\Om_f)}^2.
%G_s^b (t)
%=
%\sum_{ |\alpha|\leq s} \int_{\Omega_f}
%\frac{|\langle w^+ \rangle^{2\mu}\nabla^\alpha \Lambda_+ |^2}{ \langle w^- \rangle^{2\mu} }\dx
%+\int_{\Omega_f}
%\frac{|\langle w^- \rangle^{2\mu}\nabla^\alpha \Lambda_- |^2}{\langle w^+ \rangle^{2\mu} } \dx.
\end{align*}
In order to be compatible with the weight functions inside the domain occupied by plasma,
the weight functions on the free surface is defined as the trace of $w^\pm$ on $\Ga_f$:
%we define the weight functions on the free surface as follows
$$\wu^\pm (t,x_1)=w^\pm(t,x)\big|_{x_2=f(t,x_1)}.$$
Employing the weight functions on $\Ga_f$,
%for positive integer $s$ and $1/2<\mu \leq 1$,
we define the weighted energy of the free surface as follows:
\begin{align*}
E^f_{s+\f12} (t)= \sum_{+,-} \sum_{|a|\leq s-1} \big\| \langle\wu^\pm\rangle^{2\mu}(\p_t \pm\p_1)\langle\p_1\rangle^{\f12}\p_1^a f \big\|_{L^2(\bR)}^2\,.
%+\big\| \langle\wu^+\rangle^{2\mu}(\p_t +\p_1)\langle\p_1\rangle^{\f12}\p_1^a f \big\|_{L^2(\bR)}^2\big).
\end{align*}
%\begin{align*}
%&E^f_{s+\f12} (t)=  \sum_{|a|\leq s-1} \int_{\BR}
%\big|\langle\wu^-\rangle^{2\mu}(\p_t -\p_1)\langle\p_1\rangle^{\f12}\p_1^a f \big|^2
%+\big|\langle\wu^+\rangle^{2\mu}(\p_t +\p_1)\langle\p_1\rangle^{\f12}\p_1^a f \big|^2
%\dx_1.
%\end{align*}
The ghost weight energy of the free surface is defined by:
\begin{align*}
G^f_{s+\f12} (t)= \sum_{+,-}\sum_{|a|\leq s-1}
\Big\|  \f{\langle\wu^\pm\rangle^{2\mu}(\p_t \pm \p_1)\langle\p_1\rangle^{\f12}\p_1^a f }
{\langle\wu^\mp\rangle^{\mu}} \Big\|_{L^2(\bR)}^2\,.
%&\qquad+\f{|\langle\wu^+\rangle^{2\mu}(\p_t +\p_1)\langle\p_1\rangle^{\f12}\p_1^a f |^2}
%{\langle\wu^-\rangle^{2\mu}} \dx_1.
\end{align*}
%\begin{align*}
%&G^f_{s+\f12} (t)= \sum_{|a|\leq s-1} \int_{\BR}
% \f{|\langle\wu^-\rangle^{2\mu}(\p_t - \p_1)\langle\p_1\rangle^{\f12}\p_1^a f |^2}
%{\langle\wu^+\rangle^{2\mu}}
%+\f{|\langle\wu^+\rangle^{2\mu}(\p_t +\p_1)\langle\p_1\rangle^{\f12}\p_1^a f |^2}
%{\langle\wu^-\rangle^{2\mu}} \dx_1.
%\end{align*}
Thus the total weighted energy and the weighted ghost weight energy are defined as follows:
\begin{align*}
E_s(t)= E_s^b(t)+E_{s+\f12}^f(t),\\
G_s(t)= G_s^b(t)+G_{s+\f12}^f(t).
\end{align*}

The local well-posedness of the system will be presented in the last section.
The existence of local solutions is achieved by the method of successive approximations.
The unknowns will be in Sobolev spaces without weights.
The estimate is treated for the successive approximations of the unknowns. Then we prove the contraction of the iteration map.
The limiting unknowns $(f,\La_\pm,p)$ will satisfy the plasma-vacuum free boundary problem \eqref{MHD1}-\eqref{A21}.
Precisely, for $s\geq 4$, there exists $T>0$ such that the system  admits a unique solution $(f,\La_\pm,p)$
satisfying
\begin{align*}
\p_t f,\p_1 f\in L^{\infty}([0,T),H^{s-\f12}(\bR)),\,\, f\in L^{\infty}([0,T),L^\infty(\bR)),\,\,
 \La_\pm\in L^{\infty}([0,T),H^{s}(\Om_f)).
\end{align*}
Once the local existence of the solutions is obtained,
we only need to treat the plasma-vacuum free boundary problem \eqref{MHD1}-\eqref{A21} instead of the sequence of the unknowns.
The estimate of the solutions will be conducted in Sobolev spaces with weights.
Hence we will treat two sets of estimates: the weight functions and the weighted Sobolev norms of the solutions.

By the local in time well-posedness result, Theorem \ref{thm2}, the maximal possible time that the solutions exist  depends on the size
of the Sobolev norm of the initial data.
Hence to extend the local solution to be a global one, we can use the continuity argument.
Precisely, we make the following  two sets of  ansatz.
The first one is for the weighted energy, the ghost weight energy and the amplitude of the free surface. For $s\geq 4$, $t\in[0,T^*]$, we assume that
\begin{eqnarray}\label{AA6}
E_s(t) +\int_0^t G_s(\tau)\d\tau\leq 2C_0\epsilon,
\quad
\| f(t,\cdot)\|_{L^\infty}\leq 1-\f12 \delta.
\end{eqnarray}
The second set is about the pointwise estimate for the weight functions:
%\begin{align}
%\label{AA4}
%&|\p_t w^\pm(t,x)\mp 1 |\leq C_1\epsilon,\,\,  |\p_1 w^\pm(t,x)-1 |\leq C_1\epsilon,\,\, |\p_2 w^\pm(t,x) |\leq C_1\epsilon,\\
%\label{AA5}
%&|\nabla^2 w^\pm(t,x)|\leq C_1\epsilon,
%\end{align}
\begin{align}
\label{AA4}
|\p_t w^\pm(t,x)\mp 1 |,\,  |\p_1 w^\pm(t,x)-1 |,\, |\p_2 w^\pm(t,x) |,\,
%\label{AA5}
|\nabla^2 w^\pm(t,x)|\leq 2C_1\epsilon',
\end{align}
Here $C_1=4M$, $M=\int_{-\infty}^{\infty} \langle z  \rangle^{-2\mu} \d z$,
  %is a universal constant which will be determined in Section 3,
  $C_0$ is a constant which will be determined later at the end of this section.
 $E_s(0)\leq \epsilon\ll 1$, $\epsilon'\ll 1$ will be explained later, $T^*$ is the maximal possible time so that the assumptions \eqref{AA6} and \eqref{AA4} hold.

%The continuity argument is as follows: since () and () hold for the initial data, they
To prove Theorem \ref{thm},
it suffices to show that under the bootstrap assumptions \eqref{AA6}, \eqref{AA4}, we can derive stronger estimates:
\begin{eqnarray}\label{AA9}
E_s(t) +\int_0^t G_s(\tau)\d\tau \leq C_0\epsilon,
\quad
\| f(t,\cdot)\|_{L^\infty}\leq 1-\delta,
\end{eqnarray}
and
\begin{align}\label{AA10}
|\p_t w^\pm(t,x)\mp 1 |,\,  |\p_1 w^\pm(t,x)-1 |,\, |\p_2 w^\pm(t,x) |,\,
%\label{AA5}
|\nabla^2 w^\pm(t,x)|\leq C_1\epsilon'.
\end{align}
%The proof of the pointwise estimate of the weight functions \eqref{AA10} will be conducted in the following Section 3.
To avoid the cumbersome exact constant %in the pointwise estimate of the weight functions,
in the proof of \eqref{AA10}, we will replace \eqref{AA6} by another set of assumptions. By \eqref{AA6},
 there exists $\epsilon'\ll 1$, $C_2\geq 0$ depending on $\epsilon$ and $C_0$ such that

%Assume $\| \langle w^\pm \rangle^{2\mu} \Lambda_\pm \|_{H^s( \Omega_f )}\leq \epsilon$ and %$\| \langle \wu^\pm \rangle^{2\mu} f \|_{H^{s-\f12}(\BR )}\leq \epsilon$.
%We Assume
\begin{align}
\label{AA7}
&\sum_{|\alpha|\leq 2}\| \langle w^\pm \rangle^{2\mu} \nabla^\alpha \Lambda_\pm^1(t,\cdot) \|_{L^\infty( \Omega_f )}\leq \epsilon',\,t\in[0,T^*],\\
\label{AA8}
&\| \p_t f(t,\cdot)\|_{L^\infty(\BR )}+\|\p_1 f(t,\cdot)\|_{L^\infty(\BR )}+\|\p_1^2 f(t,\cdot)\|_{L^\infty(\BR )}\leq C_2,\,t\in[0,T^*].
\end{align}
%\eqref{AA4}-\eqref{AA5}.
%Here $\mu> \f12$. %$\epsilon$ is sufficiently small,
%$Z_\pm=\Lambda_\pm \pm e_1$, $\mu> \f12$.
Under the assumptions \eqref{AA4} and \eqref{AA7}-\eqref{AA8}, we will prove
 \eqref{AA10} in the following Section 3.

%we need to show the uniform estimate in time.
%Precisely,
To prove \eqref{AA9}, we will show that there hold the following weighted energy estimates uniformly in time:
\begin{align}
\label{AA1}
%&\f{\d }{\dt} E_s^b(t) +G_s^b (t) \leq C E_s^{\f 12}(t) G_s(t),\\
&E^b_s(t)+\int_0^t G^b_s(\tau)\d\tau \leq CE^b_s(0)+C\int_0^t E_s^{\f12}  G_s (\tau)\d\tau,\\
\label{AA2}
%&\f{\d }{\dt} E_{s+\f12}^f(t)+G_{s+\f12}^f(t) \leq C E_s^{\f 12}(t) G_s(t),\\
&E^f_s(t)+\int_0^t G^f_s(\tau)\d\tau \leq C E^f_s(0)+C\int_0^t E_s^{\f12}  G_s (\tau)\d\tau,\\
\label{AA3}
&\| f(t,\cdot)\|_{L^\infty} \leq \| f(0,\cdot)\|_{L^\infty}
+ C M \sup_{0\leq \tau \leq t} (E_s^b(\tau))^{\f12}.
\end{align}
where $M=\int_{-\infty}^{+\infty} \langle z  \rangle^{-2\mu} \d z$. %$C$ is a positive constant depending on $s,\ \mu$.
The constant $C$ in \eqref{AA1} and \eqref{AA2} polynomially depends on $\sup_{0\leq \tau\leq t}E_s(\tau)$ and $\sup_{0\leq \tau\leq t}\| f(\tau,\cdot)\|_{L^\infty}$.
%where $E_s^b(t)$ is the weighted energy of the plasmas,
%$G_s^b(t)$ the weighted ghost weight energy for the plasmas,
%$E_{s+\f12}^f(t)$ is the weighted energy of the free surface,
%$G_{s+\f12}^f(t)$ is the weighted ghost energy of the free surface. Moreover,
%\begin{align*}
%E_s(t)= E_s^b(t)+E_{s+\f12}^f(t),\\
%G_s(t)= G_s^b(t)+G_{s+\f12}^f(t).
%\end{align*}
%%We refer to ... for the explicit definition.
%We refer to section 6.2 for the definition of $E_s^f(t)$ and $G_{s+\f12}^f(t)$,
%section 5 for the definition of $E_s^b(t)$ and $G_s^b(t)$,
%section 3 for the definition of weight functions in the weighted energy.
The proof of \eqref{AA1} is given in Section \ref{sec-energy-b}.
The proof of \eqref{AA2} and \eqref{AA3} is given in Section \ref{sec-energy-f} and Section \ref{sec-amp-f}.
Once \eqref{AA1}-\eqref{AA3} are obtained, taking the sum of \eqref{AA1} and \eqref{AA2} yields
\begin{align}
\label{AA11}
%&\f{\d }{\dt} E_s^b(t) +G_s^b (t) \leq C E_s^{\f 12}(t) G_s(t),\\
&\sup_{0\leq \tau\leq t}E_s(\tau)+\int_0^t G_s(\tau)\d\tau \leq CE_s(0)+C\int_0^t E_s^{\f12}  G_s (\tau)\d\tau.
%%&\f{\d }{\dt} E_{s+\f12}^f(t)+G_{s+\f12}^f(t) \leq C E_s^{\f 12}(t) G_s(t),\\
%&E^f_s(t)+\int_0^t G^f_s(\tau)\d\tau \leq C E^f_s(0)+C\int_0^t E_s^{\f12}  G_s (\tau)\d\tau,\\
%\label{AA3}
%&\| f(t,\cdot)\|_{L^\infty} \leq \| f(0,\cdot)\|_{L^\infty}
%+ C M \sup_{0\leq \tau \leq t} (E_s^b(\tau))^{\f12}.
\end{align}
The assumption \eqref{AA6} yields $E_s(t)\leq 1$ and $\|f(t,\cdot) \|_{L^\infty}\leq 1$.
Thus the constant $C$ in \eqref{AA11} polynomially depending on  $\sup_{0\leq \tau\leq t}\| f(\tau,\cdot)\|_{L^\infty}$ and
$\sup_{0\leq \tau\leq t}E_s(\tau)$ can be bounded by absolute constant $C'$.
By taking $C_0=2C'$ and taking $\epsilon$ sufficiently small so that $C'\sqrt{4C'\epsilon}\leq \f12$, \eqref{AA11} yields
\begin{align}
\label{AA12}
%&\f{\d }{\dt} E_s^b(t) +G_s^b (t) \leq C E_s^{\f 12}(t) G_s(t),\\
&2\sup_{0\leq \tau\leq t}E_s(\tau)+ \int_0^t G_s(\tau)\d\tau \leq 2C'E_s(0)=C_0\epsilon.
%+C\int_0^t E_s^{\f12}  G_s (\tau)\d\tau.
%%&\f{\d }{\dt} E_{s+\f12}^f(t)+G_{s+\f12}^f(t) \leq C E_s^{\f 12}(t) G_s(t),\\
%&E^f_s(t)+\int_0^t G^f_s(\tau)\d\tau \leq C E^f_s(0)+C\int_0^t E_s^{\f12}  G_s (\tau)\d\tau,\\
%\label{AA3}
%&\| f(t,\cdot)\|_{L^\infty} \leq \| f(0,\cdot)\|_{L^\infty}
%+ C M \sup_{0\leq \tau \leq t} (E_s^b(\tau))^{\f12}.
\end{align}
This  yields the first inequality of \eqref{AA9}.
The second inequality of \eqref{AA9} follows from  \eqref{AA3} and \eqref{AA12} for sufficiently small $\epsilon$.

\section{The weight functions}
In this section, we will introduce the weight functions and study their properties.
Precisely, we will improve \eqref{AA4}
% \eqref{AA4}-\eqref{AA5}
under the bootstrap assumptions  \eqref{AA4}, %\eqref{AA5},
\eqref{AA7}, \eqref{AA8}.

Recalling \eqref{C0} that the weight functions $w^\pm(t,x)$ are  defined through the following equations
\begin{align}\label{C1}
&(\p_t + Z_+^1\p_1)w^-(t,x)=0,\quad w^-(0,x)=x_1,\\ \nonumber
&(\p_t + Z_-^1\p_1)w^+(t,x)=0,\quad w^+(0,x)=x_1.
\end{align}

\begin{rem}
Note that $Z_\pm^1=\Lambda_\pm^1 \pm e_1$ are defined in $\Omega_f$ for
$-1 \leq x_2\leq f(t,x_1)$.
%Since the solutions of  \eqref{C1} depends on $(t,x)$.
However, solving $w^\pm$ may require that $\Lambda_\pm^1$ have a larger domain beyond  $\Omega_f$.
Fortunately, the \textit{a priori} estimate of
$\| \Lambda_\pm^1 \|_{H^s( \Omega_f )}$ %and $\| f \|_{H^{s+\f12}( \mathbb{R} )}$
allows us to extend the domain of $\Lambda_\pm^1$ to the whole $\BR^2$
and the Sobolev norm is still kept.
Hence in the following argument, we do not need to worry about the domain of $\Lambda_\pm^1$.
%the weight functions in \eqref{C1} can be well-defined in $\mathbb{R}^2$.
\end{rem}

To study the property of the weight functions in \eqref{C1}, we draw two backward characteristics
$\phi_\pm(\tau,\sigma_\pm,x_2)$ from $(t,x_1)$ and intersect with the plane $\tau=0$ by two points $\sigma_\pm$:
\begin{align}\label{C2}
\begin{cases}
\frac{\d}{\d \tau} \phi_\pm(\tau,\sigma_\pm,x_2)=Z^1_\pm(\tau,\phi_\pm (\tau,\sigma_\pm,x_2),x_2),\\
\phi_\pm (0,\sigma_\pm,x_2)=\sigma_\pm \in \BR,\quad  \phi_\pm (t,\sigma_\pm,x_2)=x_1.
\end{cases}
\end{align}
Along the two characteristics, the weight functions $w^\mp$ satisfy
\begin{align*}
\f{\d w^\mp(\tau,\phi_\pm(\tau,\sigma_\pm,x_2),x_2 )}{\d\tau}
=(\p_t+Z^1_\pm\p_1)w^\mp(\tau,\phi_\pm(\tau,\sigma_\pm,x_2),x_2 )=0,
\end{align*}
respectively.
Hence
\begin{align}\label{C3}
&w^\mp(t,x )=w^\mp(t,\phi_\pm(t,\sigma_\pm,x_2),x_2 )\\\nonumber
&=w^\mp(0,\phi_\pm(0,\sigma_\pm,x_2),x_2 )
= w^\mp(0, \sigma_\pm,x_2 )=\sigma_\pm.
\end{align}

Let us first study the first derivative estimate of the characteristics $\phi_\pm$.
\begin{lem}\label{lemC1}
Under the assumptions \eqref{AA4} %\eqref{AA5},
and \eqref{AA7}
with $\epsilon'$ sufficiently small,
for the characteristics $\phi_\pm(t,\sigma_\pm,x_2)$ defined by \eqref{C2}, there hold
\begin{equation}\label{C4}
\begin{cases}
%|\p_t \phi_\pm(t,\sigma_\pm,x_2)\mp 1 |\leq C\epsilon,\\
%|\p_{\sigma_\pm} \phi_\pm(t,\sigma_\pm,x_2)-1 |\leq C\epsilon,\\
%|\p_2 \phi_\pm(t,\sigma_\pm,x_2) |\leq C\epsilon.
|\p_t \phi_\pm(t,\sigma_\pm,x_2)\mp 1 |\leq \epsilon',\\
|\p_{\sigma_\pm} \phi_\pm(t,\sigma_\pm,x_2)-1 |\leq 2M\epsilon',\\
|\p_2 \phi_\pm(t,\sigma_\pm,x_2) |\leq 2M\epsilon',
\end{cases}
\end{equation}
where $M=\int_{-\infty}^{\infty} \langle z  \rangle^{-2\mu} \d z.$
\end{lem}
\begin{proof}
%\textbf{Step 1: proof of the first three lines of \eqref{C8}}
%Due to the similarity between $\phi_\pm$ and $\phi_-$, we only present the details for the first one.
Solving \eqref{C2} gives
\begin{align*} %\label{C9}
\phi_\pm(t,\sigma_\pm,x_2)
&=\sigma_\pm +\int_0^t Z_\pm^1(\tau,\phi_\pm(\tau,\sigma_\pm,x_2),x_2) \d\tau \\\nonumber
&=\sigma_\pm \pm t+\int_0^t \Lambda_\pm^1 (\tau,\phi_\pm(\tau,\sigma_\pm,x_2),x_2) \d\tau.
\end{align*}
Taking the derivative in $t$, it is easy to see
\begin{align*}
&\p_t \phi_\pm(t,\sigma_\pm,x_2)
=\pm 1 +\Lambda_\pm^1(t,\phi_\pm(t,\sigma_\pm,x_2),x_2).
\end{align*}
Hence
\begin{align*}
|\p_t \phi_\pm(t,\sigma_\pm,x_2)\mp 1|\leq \| \Lambda_\pm^1\|_{L^\infty(\Om_f)}\leq \epsilon'.
\end{align*}
Next, we calculate
\begin{align}
\label{C10}
&\p_{\sigma_\pm} \phi_\pm(t,\sigma_\pm,x_2)
=1 +\int_0^t \p_{\sigma_\pm}\phi_\pm(\tau,\sigma_\pm,x_2)\cdot (\p_1\Lambda_\pm^1)(\tau,\phi_\pm(\tau,\sigma_\pm,x_2),x_2) \d\tau,\\
\label{C11}
&\p_2 \phi_\pm(t,\sigma_\pm,x_2)
= \int_0^t (\p_2\Lambda_\pm^1)(\tau,\phi_\pm(\tau,\sigma_\pm,x_2),x_2) \\\nonumber
&\qquad\qquad\qquad\qquad+\p_2\phi_\pm(\tau,\sigma_\pm,x_2)\cdot (\p_1\Lambda_\pm^1)(\tau,\phi_\pm(\tau,\sigma_\pm,x_2),x_2) \d\tau.
\end{align}
Thus
\begin{align*}
|\p_{\sigma_\pm} \phi_\pm(t,\sigma_\pm,x_2)-1|
&\leq \int_0^t |\p_{\sigma_\pm} \phi_\pm(\tau,\sigma_\pm,x_2)-1|\cdot| (\p_1\Lambda_\pm^1)(\tau,\phi_\pm(\tau,\sigma_\pm,x_2),x_2) | \d\tau \\
&\quad+\int_0^t |(\p_1\Lambda_\pm^1)(\tau,\phi_\pm(\tau,\sigma_\pm,x_2),x_2) | \d\tau ,\\
|\p_2 \phi_\pm(t,\sigma_\pm,x_2)|
&\leq \int_0^t
|\p_2\phi_\pm(\tau,\sigma_\pm,x_2)| \cdot|(\p_1\Lambda_\pm^1)(\tau,\phi_\pm(\tau,\sigma_\pm,x_2),x_2) | \d\tau\\
&\quad+\int_0^t |(\p_2\Lambda_\pm^1)(\tau,\phi_\pm(\tau,\sigma_\pm,x_2),x_2)| \d\tau.
\end{align*}
By Gronwall's inequality, there hold
\begin{align*}
&|\p_{\sigma_\pm} \phi_\pm(t,\sigma_\pm,x_2)-1|,\quad |\p_2\phi_\pm(t,\sigma_\pm,x_2)|\\
&\leq \int_0^t |(\nabla\Lambda_\pm^1)(\tau,\phi_\pm (\tau,\sigma_\pm,x_2),x_2) | \d\tau
\cdot e^{\int_0^t |(\nabla\Lambda_\pm^1)(\tau,\phi_\pm(\tau,\sigma_\pm,x_2),x_2) | \d\tau}.
\end{align*}
While
\begin{align*}
&\int_0^t |(\nabla\Lambda_\pm^1)(\tau,\phi_\pm(\tau,\sigma_\pm,x_2),x_2) | \d\tau\\
&\leq \| \langle w^\pm \rangle^{2\mu}  \nabla\Lambda_\pm^1 \|_{L^\infty}
\int_0^t  \langle w^\pm (\tau,x) \rangle^{-2\mu} |_{x_1= \phi_\pm (\tau,\sigma_\pm,x_2)} \d\tau\\
%&\leq \| \langle w^\pm \rangle^{2\mu}  \nabla\Lambda_\pm^1 \|_{L^\infty} \int_0^t  \langle w^\pm (\tau,x) \rangle^{-2\mu} |_{x_1= \phi_\pm (\tau,\sigma_\pm,x_2)} \d\tau\\
&\leq \| \langle w^\pm \rangle^{2\mu}  \nabla\Lambda_\pm^1 \|_{L^\infty} \int_{-\infty}^{+\infty} \langle w  \rangle^{-2\mu} \d w
\cdot\sup_{\tau}\frac{1}{\big| \f{\d w^\pm(\tau,\phi_\pm (\tau,\sigma_\pm,x_2),x_2) }{\d\tau}\big| }.
\end{align*}
%By $\eqref{C7}_4$ and $\eqref{C7}_5$, we have
By $\eqref{AA4}$, we have
\begin{align*}
&\Big|\frac{\d w^\pm(\tau,\phi_\pm (\tau,\sigma_\pm,x_2),x_2) }{\d\tau}\Big|\\
&=|\p_tw^\pm(\tau,\phi_\pm (\tau,\sigma_\pm,x_2),x_2)
+\p_1 w^\pm(\tau,\phi_\pm (\tau,\sigma_\pm,x_2),x_2) \cdot (\p_t \phi_\pm)(\tau,\sigma_\pm,x_2)|
\\
&=|\p_tw^\pm(\tau,\phi_\pm (\tau,\sigma_\pm,x_2),x_2)
+\p_1 w^\pm(\tau,\phi_\pm (\tau,\sigma_\pm,x_2),x_2) \cdot \big( \pm 1 +\Lambda_\pm^1(\tau,\phi_\pm(\tau,\sigma_\pm,x_2),x_2) \big)|
\\
&\geq 1.
\end{align*}
Consequently, by \eqref{AA7}, we derive that
\begin{align*}
&|\p_{\sigma_\pm} \phi_\pm(t,\sigma_\pm,x_2)-1|,\quad |\p_2\phi_\pm(t,\sigma_\pm,x_2)|\\
&\leq M \| \langle w^\pm \rangle^{2\mu}  \nabla\Lambda_\pm^1 \|_{L^\infty} e^{M \| \langle w^\pm \rangle^{2\mu}  \nabla\Lambda_\pm^1 \|_{L^\infty}}  \leq M\epsilon' e^{M\epsilon'}\leq 2M\epsilon',
\end{align*}
for sufficiently small $\epsilon'$.
\end{proof}

%Next we present the first derivative estimate of the weight functions.
\begin{lem}\label{lemC2}
%Let $Z_\pm=\Lambda_\pm \pm e_1$, $\mu> \f12$.
%Assume $\sum_{|\alpha|\leq 1}\| \langle w^\pm \rangle^{2\mu}
%\nabla^\alpha\Lambda_\pm^1 \|_{L^\infty( \Omega_f )}\leq \epsilon$
%with $\epsilon$ sufficiently small.
Under the assumptions \eqref{AA4} %\eqref{AA5},
and \eqref{AA7}
with $\epsilon'$ sufficiently small,
for the weight functions defined by \eqref{C1}, there hold
%\begin{align*}
%&|\p_t w^\pm(t,x)\mp 1 |\leq C\epsilon,\\[-5mm]\\
%&|\p_1 w^\pm(t,x)-1 |\leq C\epsilon,\\[-5mm]\\
%&|\p_2 w^\pm(t,x) |\leq C\epsilon.
%\end{align*}
\begin{equation}\label{C14}
\begin{cases}
%|\p_t \phi_\pm(t,\sigma_\pm,x_2)\mp 1 |\leq \epsilon,\\[-4mm]\\
%|\p_{\sigma_\pm} \phi_\pm(t,\sigma_\pm,x_2)-1 |\leq 2M\epsilon,\\[-4mm]\\
%|\p_2 \phi_\pm(t,\sigma_\pm,x_2) |\leq 2M\epsilon, \\
|\p_t w^\pm(t,x)\mp 1 |\leq (1+3M)\epsilon',\\
|\p_1 w^\pm(t,x)-1 |\leq 3M\epsilon'  ,\\
|\p_2 w^\pm(t,x) |\leq 3M\epsilon',
\end{cases}
\end{equation}
where $M=\int_{-\infty}^{\infty} \langle z  \rangle^{-2\mu} \d z.$
\end{lem}
\begin{proof}
By the ansatz  \eqref{C2} and \eqref{C3},
we see $w^\mp(t,x)=\sigma_\pm$.
%Thus to obtain the last three lines of \eqref{C8},
Thus to prove the lemma,
it suffices to estimate $\p_t \sigma_\pm, \p_1 \sigma_\pm, \p_2 \sigma_\pm$.

Let us copy and write the last expression of \eqref{C2} as follows:
% write \eqref{C15} as follows:
\begin{align} \label{C15}
x_1=\phi_\pm(t,\sigma_\pm(t,x),x_2).
\end{align}
Taking the derivative of both sides in $t$ yields
\begin{align*}
 \p_t\phi_+ + \p_{\sigma_+}\phi_+ \p_t \sigma_+
=\p_t\phi_- + \p_{\sigma_-}\phi_- \p_t \sigma_-=0.
\end{align*}
This further yields
\begin{align*}
\p_t \sigma_+ =-\p_t\phi_+ \big(\p_{\sigma_+}\phi_+\big)^{-1},\\
\p_t \sigma_- =-\p_t\phi_- \big(\p_{\sigma_-}\phi_-\big)^{-1}.
\end{align*}
Consequently, by $\eqref{C4}_1$ and $\eqref{C4}_2$, we obtain
\begin{align*}
&|\p_t w^- +1|= |\p_t \sigma_+ +1|\leq (1+\epsilon') (1+2M \epsilon') -1\leq (1+3M)\epsilon',\\
&|\p_t w^+ -1|= |\p_t \sigma_- -1|\leq (1+\epsilon') (1+2M \epsilon') -1\leq (1+3M)\epsilon',
\end{align*}
for sufficiently small $\epsilon'$. This yields $\eqref{C14}_1$.

Next, taking the derivative of both sides of \eqref{C15} in $x_1$ yields
\begin{align*} %\label{C16}
1=\p_{\sigma_\pm} \phi_\pm (t,\sigma_\pm,x_2) \cdot \p_1 \sigma_\pm(t,x).
\end{align*}
Thus,
\begin{align*}
 \p_1 \sigma_\pm(t,x)= \big(  \p_{\sigma_\pm} \phi_\pm (t,\sigma_\pm,x_2) \big)^{-1} .
\end{align*}
Consequently, by   $\eqref{C4}_2$, we obtain
\begin{align}\label{C17}
\p_1w^\mp= \p_1 \sigma_\pm(t,x)\leq \f{1}{1- 2M\epsilon'}\leq 1+3M\epsilon', \\ \nonumber
\p_1w^\mp= \p_1 \sigma_\pm(t,x)\geq \f{1}{1+ 2M\epsilon'}\geq 1-3M\epsilon',
\end{align}
for sufficiently small $\epsilon'$. This yields $\eqref{C14}_2$.

Finally, taking the derivative of both sides of \eqref{C15} in $x_2$ yields
\begin{align*}
0=\p_{\sigma_\pm} \phi_\pm(t,\sigma_\pm,x_2) \cdot \p_2 \sigma_\pm(t,x)+\p_2 \phi_\pm(t,\sigma_\pm,x_2).
\end{align*}
Thus
\begin{align*}
\p_2 \sigma_\pm(t,x)=-\p_2 \phi_\pm(t,\sigma_\pm,x_2)
\cdot \big(\p_{\sigma_\pm} \phi_\pm(t,\sigma_\pm,x_2) \big)^{-1}.
\end{align*}
Consequently, by $\eqref{C4}_2$ and $\eqref{C4}_3$, we obtain
\begin{align}\label{C18}
|\p_2 w^\mp|=|\p_2 \sigma_\pm|\leq  2M\epsilon' \cdot(1+3M\epsilon')\leq 3M\epsilon',
\end{align}
for sufficiently small $\epsilon$.  This yields $\eqref{C14}_3$.
%By Lemma \ref{lemC1}, the estimate for $\f{\p \sigma_\pm}{\p t}$ follows.
%The estimate for $\p_1\sigma_\pm$ and $\p_2 \sigma_\pm$ can be proved similarly.
This ends the proof of the  two lemmas.
\end{proof}

Next, we present the second order derivative estimate of the weight functions.
\begin{lem}\label{lemC4}
Under the assumptions \eqref{AA4} %\eqref{AA5},
and \eqref{AA7}
with $\epsilon'$ sufficiently small,
for the weight functions defined by \eqref{C1}, %there holds
there hold
\begin{align}\label{C20}
&|\p_1^2 w^\pm(t,x)|\leq 3M\epsilon',\\[-5mm]\nonumber\\\nonumber
&|\p_1\p_2 w^\pm(t,x)|\leq 4M\epsilon',\\[-5mm]\nonumber\\\nonumber
&|\p_2^2 w^\pm(t,x) | \leq 3M\epsilon'.
\end{align}
%where $C$ is an absolute constant.
\end{lem}
\begin{proof}
The proof is similar to the argument that we have conducted in Lemma \ref{lemC1} and Lemma \ref{lemC2}.
%Hence we only sketch the proof of this lemma.
%We still use the change of the coordinates and the continuity method.
%Let us work in $(t,x)$ variables and
Let us copy and write \eqref{C15} as follows:
\begin{align} \label{C21}
x_1=\phi_\pm(t,\sigma_\pm(t,x),x_2).
\end{align}
By the characteristics ansatz \eqref{C2} and \eqref{C3}, we have $w^\mp(t,x)=\sigma_\pm$.
%Along the characteristics, we have.
Hence it suffices to estimate $\p_1^2 \sigma_\pm, \p_2^2 \sigma_\pm, \p_1\p_2 \sigma_\pm$.

Taking twice derivative of both sides of \eqref{C21} in $x_1$, we deduce
%\begin{align*}
%0&= \f{\p\phi_+}{\p \sigma_+} \p^2_1\sigma_+
%+\f{\p^2\phi_+}{\p \sigma_+^2} (\p_1 \sigma_+)^2 \\\nonumber
%&=\f{\p\phi_-}{\p \sigma_-} \p^2_1 \sigma_-
%+\f{\p^2\phi_-}{\p \sigma_-^2}  ( \p_1 \sigma_- )^2 .
%\end{align*}
\begin{align*}
0&= \p_{ \sigma_+}\phi_+  \p^2_1 \sigma_+
+ \p_{\sigma_+}^2  \phi_+ (\p_1 \sigma_+)^2 \\\nonumber
&=\p_{ \sigma_-}\phi_-  \p^2_1 \sigma_-
+ \p_{\sigma_-}^2  \phi_- (\p_1 \sigma_-)^2.
\end{align*}
%This further yields
Thus
\begin{align}
\label{C22}
& \p^2_1 \sigma_+
=-\big(\p_{\sigma_+}\phi_+ \big)^{-1} {\p^2_{\sigma_+}\phi_+} (\p_1 \sigma_+)^2, \\
\label{C23}
& \p^2_1 \sigma_-
=-\big(\p_{\sigma_-}\phi_- \big)^{-1} {\p^2_{\sigma_-}\phi_-} (\p_1 \sigma_-)^2 .
\end{align}
On the other side,
taking  derivative of \eqref{C10} in $\sigma_\pm$, we obtain
%\begin{align*}
%\p^2_{\sigma_+} \phi_+(t,\sigma_+,x_2)
%&= \int_0^t \p^2_{\sigma_+}\phi_+(\tau,\sigma_+,x_2) (\p_1\Lambda_+^1)(\tau,\phi_+(\tau,\sigma_+,x_2),x_2)\\
%&\qquad+\big(\p_{\sigma_+}\phi_+(\tau,\sigma_+,x_2) \big)^2 (\p_1^2\Lambda_+^1)(\tau,\phi_+(\tau,\sigma_+,x_2),x_2)
% d\tau.
%\end{align*}
\begin{align*}
\p^2_{\sigma_\pm} \phi_\pm(t,\sigma_\pm,x_2)
&= \int_0^t \p^2_{\sigma_\pm}\phi_\pm(\tau,\sigma_\pm,x_2) (\p_1\Lambda_\pm^1)(\tau,\phi_\pm(\tau,\sigma_\pm,x_2),x_2)\\
&\qquad+\big(\p_{\sigma_\pm}\phi_\pm(\tau,\sigma_\pm,x_2) \big)^2 (\p_1^2\Lambda_\pm^1)(\tau,\phi_\pm(\tau,\sigma_\pm,x_2),x_2)
 \d\tau.
\end{align*}
Similarly to the argument we conducted in the proof of Lemma \ref{lemC1} and Lemma \ref{lemC2}.
 By %the first derivative estimate of characteristics
 \eqref{C4}, \eqref{C14}, %Lemma \ref{lemC1},
the assumption
%$\| \langle w^\pm \rangle^{2\mu} \Lambda_\pm \|_{H^s( \Omega_f )}\leq \epsilon'$
$\sum_{|\alpha|\leq 2}\| \langle w^\pm \rangle^{2\mu} \nabla^\alpha \Lambda_\pm^1 \|_{L^\infty( \Omega_f )}\leq \epsilon'$
and by Gronwall's inequality, we obtain
\begin{align}\label{C24}
|\p^2_{\sigma_\pm} \phi_\pm(t,\sigma_\pm,x_2)|
&\leq \int_0^t \big|\big(\p_{\sigma_\pm}\phi_\pm(\tau,\sigma_\pm,x_2) \big)^2 (\p_1^2\Lambda_\pm^1)(\tau,\phi_\pm(\tau,\sigma_\pm,x_2),x_2) \big|
 \d\tau\\\nonumber
&\qquad \cdot \exp\big( \int_0^t |(\p_1\Lambda_\pm^1)(\tau,\phi_\pm(\tau,\sigma_\pm,x_2),x_2)| \d\tau \big) \\\nonumber
&\leq (1+2M\epsilon')^2 M\epsilon' e^{M\epsilon'}\leq 2M\epsilon',
\end{align}
for sufficiently small $\epsilon'$.
%The proof of
%\begin{align}\label{C25}
%|\p^2_{\sigma_-} \phi_-(t,\sigma_-,x_2)|\leq 2M\epsilon'
%\end{align}
%can be obtained in the same method.
%Thus by  \eqref{C4}, \eqref{C14},
%%\eqref{C22}, \eqref{C23} and \eqref{C24},  and  yield
% \eqref{C22} and \eqref{C23} yield
%%By the argument in the above lemma, we have
Thus by \eqref{C4}, \eqref{C14}, \eqref{C22}, \eqref{C23} and \eqref{C24},  we derive
\begin{align*} %\label{C26}
& |\p^2_1 w^\mp|=|\p^2_1 \sigma_\pm|\leq (1+2M\epsilon') \cdot 2M\epsilon'\cdot (1+3M\epsilon')\leq 3M\epsilon' .
\end{align*}
This yields the first inequality of \eqref{C20}.
%The estimates for the other second derivative of the characteristics
%can be conducted similarly.

Next, taking  derivative of both sides of \eqref{C21} in $x_1$ and then taking the derivative in $x_2$, we obtain
\begin{align*}
0=\p_{\sigma_\pm}^2 \phi_\pm \cdot \p_1 \sigma_\pm \p_2 \sigma_\pm
+ \p_2 \p_{\sigma_\pm} \phi_\pm \cdot \p_1 \sigma_\pm
+ \p_{\sigma_\pm} \phi_\pm \cdot \p_1\p_2 \sigma_\pm .
\end{align*}
This yields
\begin{align}\label{C27}
 \p_1\p_2 \sigma_\pm
= -\big(\p_{\sigma_\pm} \phi_\pm \big)^{-1}
\big(\p_{\sigma_\pm}^2 \phi_\pm \cdot \p_1 \sigma_\pm \p_2 \sigma_\pm
+ \p_2 \p_{\sigma_\pm} \phi_\pm \cdot \p_1 \sigma_\pm\big)  .
\end{align}
%The first line of \eqref{C10} reads as follows
%\begin{align*}
%&\p_{\sigma_+} \phi_+(t,\sigma_+,x_2)
%=1 +\int_0^t \p_{\sigma_+}\phi_+(\tau,\sigma_+,x_2)\cdot (\p_1\Lambda_+^1)(\tau,\phi_+(\tau,\sigma_+,x_2),x_2) d\tau. %\\
%%&\p_2 \phi_+(t,\sigma_+,x_2)
%%= \int_0^t (\p_2\Lambda_+^1)(\tau,\phi_+(\tau,\sigma_+,x_2),x_2) \\
%%&\qquad\qquad\qquad\qquad+\p_2\phi_+(\tau,\sigma_+,x_2)\cdot (\p_1\Lambda_+^1)(\tau,\phi_+(\tau,\sigma_+,x_2),x_2) d\tau.
%\end{align*}
Now taking the derivative of the first line of \eqref{C10} in $x_2$, we obtain
\begin{align*}
&\p_2\p_{\sigma_\pm} \phi_\pm(t,\sigma_\pm,x_2) \\
&= \int_0^t (\p_1\p_2\Lambda_\pm^1)(\tau,\phi_\pm(\tau,\sigma_\pm,x_2),x_2)\cdot \p_{\sigma_\pm}\phi_\pm\\
&\qquad+\p_{\sigma_\pm}\p_2\phi_\pm(\tau,\sigma_\pm,x_2)\cdot (\p_1\Lambda_\pm^1)(\tau,\phi_\pm(\tau,\sigma_\pm,x_2),x_2) \\[-4mm]\\
&\qquad+(\p_{\sigma_\pm}\phi_\pm \p_2\phi_\pm)(\tau,\sigma_\pm,x_2)\cdot (\p_1^2\Lambda_\pm^1)(\tau,\phi_\pm(\tau,\sigma_\pm,x_2),x_2)\d\tau.
\end{align*}
By  \eqref{C4}, \eqref{C14},
 %$\eqref{C8}_2$, $\eqref{C8}_3$,
the assumption $\sum_{|\alpha|\leq 2}\| \langle w^\pm \rangle^{2\mu} \nabla^\alpha \Lambda_\pm \|_{L^\infty( \Omega_f )}\leq \epsilon'$ and Gronwall's inequality, we obtain
\begin{align}\label{C29}
|\p_2\p_{\sigma_\pm} \phi_\pm|
\leq
&\int_0^t \big|(\p_{\sigma_\pm}\phi_\pm \p_2\phi_\pm)(\tau,\sigma_\pm,x_2)\cdot (\p_1^2\Lambda_\pm^1)(\tau,\phi_\pm(\tau,\sigma_\pm,x_2),x_2) \big|\\\nonumber
&\qquad+ \big|(\p_1\p_2\Lambda_\pm^1)(\tau,\phi_\pm(\tau,\sigma_\pm,x_2),x_2)\cdot \p_{\sigma_\pm}\phi_\pm \big| \d\tau \\ \nonumber
&\cdot \exp\big( \int_0^t  \big| (\p_1\Lambda_\pm^1)(\tau,\phi_\pm(\tau,\sigma_\pm,x_2),x_2) \big| \d\tau \big) \\[-4mm]\nonumber\\\nonumber
\leq &(1+2M\epsilon')^2  M\epsilon' \cdot e^{M\epsilon'}\leq 2M\epsilon'
\end{align}
for sufficiently small $\epsilon'$.
Consequently, by \eqref{C4}, \eqref{C14},  \eqref{C17}, \eqref{C18}, \eqref{C24},
\eqref{C27} and  \eqref{C29},   we derive
\begin{align*} %\label{C30}
|\p_1\p_2 w^\mp|=|\p_1\p_2 \sigma_\pm|
&\leq  (1-2M\epsilon')^{-1}
\big( 2M\epsilon'(1+3M\epsilon')3M\epsilon'+2M\epsilon'(1+3M\epsilon')  \big)\\\nonumber
&\leq (1+3M\epsilon')3M\epsilon'\leq 4M\epsilon'
\end{align*}
for sufficiently small $\epsilon'$. This yields the second inequality of \eqref{C20}.

Finally, taking the derivative twice on both sides of \eqref{C21} with respect to $x_2$, we obtain
%\begin{align*}
%0=\p_{\sigma_\pm} \phi_\pm(t,\sigma_\pm,x_2) \cdot \p_2 \sigma_\pm(t,x)+\p_2 \phi_\pm(t,\sigma_\pm,x_2).
%\end{align*}
%Further taking derivative of both sides in $x_2$ yields
\begin{align*} %\label{C31}
0=
&\p_{\sigma_\pm} \phi_\pm(t,\sigma_\pm,x_2) \cdot \p_2^2 \sigma_\pm(t,x)
+\p^2_{\sigma_\pm} \phi_\pm(t,\sigma_\pm,x_2) \cdot \big(\p_2 \sigma_\pm(t,x) \big)^2 \\\nonumber
&+2\p_2\p_{\sigma_\pm} \phi_\pm(t,\sigma_\pm,x_2) \p_2 \sigma_\pm(t,x)
%+\p_{\sigma_\pm}\p_2 \phi_\pm(t,\sigma_\pm,x_2) \p_2 \sigma_\pm
%\\\nonumber
%&
+\p_2^2 \phi_\pm(t,\sigma_\pm,x_2).
\end{align*}
%From $\eqref{C31}$, we get
This yields
\begin{align}\label{C32}
\p_2^2 \sigma_\pm
&= -\big(\p_{\sigma_\pm} \phi_\pm\big)^{-1}
\Big(\p^2_{\sigma_\pm} \phi_\pm(t,\sigma_\pm,x_2) \cdot \big(\p_2 \sigma_\pm(t,x) \big)^2 \\\nonumber
&\qquad +2 \p_2\p_{\sigma_\pm} \phi_\pm(t,\sigma_\pm,x_2) \p_2 \sigma_\pm(t,x)
%+\p_{\sigma_\pm}\p_2 \phi_\pm(t,\sigma_\pm,x_2) \p_2 \sigma_\pm \\\nonumber
%&\qquad
+\p_2^2 \phi_\pm(t,\sigma_\pm,x_2)\Big).
\end{align}
On the other hand,
taking the derivative of \eqref{C11} in $x_2$, we obtain
\begin{align*}
\p_2^2 \phi_\pm(t,\sigma_\pm,x_2)
&= \int_0^t (\p_2^2\Lambda_\pm^1)(\tau,\phi_\pm(\tau,\sigma_\pm,x_2),x_2)  \\[-4mm]\nonumber\\
&\qquad+2\p_2\phi_\pm(\tau,\sigma_\pm,x_2) (\p_1\p_2\Lambda_\pm^1)(\tau,\phi_\pm(\tau,\sigma_\pm,x_2),x_2) \\[-4mm]\nonumber\\
&\qquad+\p_2^2\phi_\pm(\tau,\sigma_\pm,x_2)\cdot (\p_1\Lambda_\pm^1)(\tau,\phi_\pm(\tau,\sigma_\pm,x_2),x_2)  \\[-4mm]\nonumber\\
&\qquad+(\p_2\phi_\pm)^2(\tau,\sigma_\pm,x_2) \cdot (\p_1^2\Lambda_\pm^1)(\tau,\phi_\pm(\tau,\sigma_\pm,x_2),x_2)
 \d\tau.
\end{align*}
By \eqref{C4}, \eqref{C14},
the assumption $\sum_{|\alpha|\leq 2}\| \langle w^\pm \rangle^{2\mu} \nabla^\alpha \Lambda_\pm^1 \|_{L^\infty( \Omega_f )}\leq \epsilon'$ and Gronwall's inequality, we obtain
\begin{align} \label{C33}
&|\p_2^2 \phi_\pm(t,\sigma_\pm,x_2)| \\\nonumber
&\leq  \int_0^t \big| (\p_2^2\Lambda_\pm^1)(\tau,\phi_\pm(\tau,\sigma_\pm,x_2),x_2)  \\[-4mm]\nonumber\\\nonumber
&\qquad+2\p_2\phi_\pm(\tau,\sigma_\pm,x_2) (\p_1\p_2\Lambda_\pm^1)(\tau,\phi_\pm(\tau,\sigma_\pm,x_2),x_2) \\[-4mm]\nonumber\\\nonumber
%&\qquad+\p_2^2\phi_\pm(\tau,\sigma_\pm,x_2)\cdot (\p_1\Lambda_\pm^1)(\tau,\phi_\pm(\tau,\sigma_\pm,x_2),x_2)  \\[-4mm]\nonumber\\
&\qquad+(\p_2\phi_\pm)^2(\tau,\sigma_\pm,x_2) \cdot (\p_1^2\Lambda_\pm^1)(\tau,\phi_\pm(\tau,\sigma_\pm,x_2),x_2) \big|
 \d\tau\\\nonumber
&\quad \cdot \exp\big(  \int_0^t |(\p_1\Lambda_\pm^1)(\tau,\phi_\pm(\tau,\sigma_\pm,x_2),x_2)|  \d\tau \big)\\\nonumber
&\leq\big(M\epsilon'+4M\epsilon'\cdot M\epsilon'+ (2M\epsilon')^2\cdot M\epsilon'  \big)\cdot e^{M\epsilon'}
\leq 2M\epsilon'
\end{align}
for sufficiently small $\epsilon'$.
%Consequently, by \eqref{C24}, \eqref{C18}, \eqref{C29}, \eqref{C33}, \eqref{C32} reduces to
Consequently, by  \eqref{C18}, \eqref{C24}, \eqref{C29}, \eqref{C32} and \eqref{C33}, we obtain
\begin{align*}
|\p_2^2 w^\mp|=|\p_2^2 \sigma_\pm| \leq(1-2M\epsilon')^{-1}\big( 2M\epsilon' (3M\epsilon')^2
+2\cdot 3M\epsilon'\cdot 2M\epsilon'+2M\epsilon'  \big)
\leq 3M\epsilon'.
\end{align*}
This yields the last inequality of \eqref{C20}.
%This ends the proof of the lemma.
\end{proof}

Under the assumptions \eqref{AA4} %\eqref{AA5},
and \eqref{AA7}
with $\epsilon'$ sufficiently small, by Lemma \ref{lemC2} and Lemma \ref{lemC4},  we see
\eqref{AA10} holds by taking $C_1\geq 4M$.

%We also need introduce the weight function in the free surface.
%To be compatible with the weight inside the domain occupied by plasmas,
%we define the weight on the free surface as follows:
%\begin{align*}
%\wu^\pm=w^\pm(t,x)\big|_{x_2=f(t,x_1)}.
%\end{align*}
Next, we study the weight functions on the free surface:
\begin{align*}
\wu^\pm=w^\pm(t,x)\big|_{x_2=f(t,x_1)}.
\end{align*}
%We show the first and second derivative estimate of $\wu^\pm$.
\begin{lem}\label{lemC3}
%Let $Z_\pm=\Lambda_\pm \pm e_1$, $\mu> \f12$.
%Assume $\sum_{|\alpha|\leq 1}\| \langle w^\pm \rangle^{2\mu}
%\nabla^\alpha\Lambda_\pm^1 \|_{L^\infty( \Omega_f )}\leq \epsilon'$,
%$\| \p_t f \|_{L^\infty(\BR)}+\| \p_1 f \|_{L^\infty(\BR)} \lesssim 1$
Under the assumptions \eqref{AA4}, \eqref{AA7}
and \eqref{AA8}
with $\epsilon'$ sufficiently small,
%For the weight function $\wu^\pm$ defined on the free surface,
there exists an absolute constant $C$ such that
\begin{equation}\label{C36}
\begin{cases}
|\p_t \wu^\pm(t,x_1)\mp 1 |\leq C\epsilon',\\
|\p_1 \wu^\pm(t,x_1)-1 |\leq C\epsilon'.
\end{cases}
\end{equation}
\end{lem}
\begin{proof}
By chain rules, there hold
\begin{align*}
&\p_t \wu^\pm(t,x_1)=\ud{\p_t w^\pm}+\ud{\p_2 w^\pm} \p_t f ,\\
&\p_1 \wu^\pm(t,x_1)=\ud{\p_1 w^\pm}+\ud{\p_2 w^\pm} \p_1 f .
\end{align*}
Hence the lemma follows from Lemma \ref{lemC2} and the assumption
%\it{a priori} estimate of
$\| \p_t f \|_{L^\infty(\BR)}+\| \p_1 f \|_{L^\infty(\BR)} \lesssim 1$.
\end{proof}

%Next, we present the second order derivative estimate of the weight functions
%on the free surface.
\begin{lem}\label{lemC5}
%Let $Z_\pm=\Lambda_\pm \pm e_1$, $\mu> \f12$. %, $s>\f52$.
%%Assume $\| \langle w^\pm \rangle^{2\mu} \Lambda_\pm \|_{H^s( \Omega_f )}\leq \epsilon'$ and %$\| \langle \wu^\pm \rangle^{2\mu} f \|_{H^{s-\f12}(\BR )}\leq \epsilon'$.
%Assume $\sum_{|\alpha|\leq 2}\| \langle w^\pm \rangle^{2\mu} \nabla^\alpha \Lambda_\pm^1 \|_{L^\infty( \Omega_f )}\leq \epsilon'$,
%$ \| \p_t f\|_{L^\infty(\BR )}+\|\p_1 f\|_{L^\infty(\BR )}+\|\p_1^2 f\|_{L^\infty(\BR )}\lesssim 1$
Under the assumptions \eqref{AA4}, \eqref{AA7}
and \eqref{AA8}
with $\epsilon'$ sufficiently small,
%For the weight functions $\wu^\pm$ defined on the free surface,
%there holds
there exists an absolute constant $C$ such that
\begin{align}\label{C38}
%&|\p_t \wu^\pm(t,z)\mp 1 |\leq C\epsilon'\leq \f14,\\
|\p_1^2 \wu^\pm(t,x_1)|\leq C\epsilon'.
\end{align}
\end{lem}
\begin{proof}
By chain rules, there holds
\begin{align*}
&\p_1^2 \wu^\pm(t,x_1)=\ud{\p_1^2 w}^\pm+2\ud{\p_1\p_2 w}^\pm \p_1 f
+\ud{\p_2^2 w}^\pm (\p_1 f)^2+\ud{\p_2 w}^\pm \p_1^2 f.
\end{align*}
Hence the lemma follows from Lemma \ref{lemC2}, Lemma \ref{lemC4} and the %\textit{a priori}
assumption of $ \| \p_t f\|_{L^\infty(\BR )}+\|\p_1 f\|_{L^\infty(\BR )}+\|\p_1^2 f\|_{L^\infty(\BR )}\lesssim 1$.

\end{proof}

\section{Preliminary weighted estimates}
In this section, we establish
the weighted Sobolev inequalities, the weighted trace theorem, the weighted commutator estimate,
and the weighted chain rules.
These estimates are shown with the weights $\langle \wu^+ \rangle^{\mu}\langle \wu^- \rangle^{\nu}$ for any $\mu\in \bR,\, \nu\in \bR$.
Here we recall $\wu^\pm=w^\pm(t,x)\big|_{x_2=f(t,x_1)}$, where $w^\pm(t,x)$ are defined by \eqref{C1}.

Frequent use of fractional derivatives is required throughout this work. Various equivalent methods exist; we refer to \cite{Adams,Caf,Di} and the references therein.
Here, we use the inhomogeneous fractional
derivatives $\langle \p_1\rangle^{\mu}$ ($\mu\in\bR$), explained as singular integrals with Bessel potentials on the whole space $\bR$.
This definition coincides with that obtained via difference quotients or the Fourier multiplier approach, so no ambiguity arises.
Moreover, in contrast to the Riesz potentials used for homogeneous fractional derivatives, the Bessel potentials enjoy exponential decay at infinity.

For weights of the form $\langle v^+ \rangle^{\mu}\langle v^- \rangle^{\nu}$ with $v^\pm=x_1\pm t$ and arbitrary $\mu\in \bR,\, \nu\in \bR$, the weighted Sobolev inequalities, trace theorem, commutator estimates and chain rules have all been established in  \cite[Section 3]{CL-24}.
Similar properties hold for the weights $\langle \wu^+ \rangle^{\mu}\langle \wu^- \rangle^{\nu}$
as they satisfy the following conditions:
For the integral domain near the origin:
$\{x_1':|x_1-x_1'|\leq \f12\}$, we have $\langle \wu^\pm(x_1)\rangle
\sim \langle \wu^\pm(x_1')\rangle$.
For the integral domain away from the origin:
$\{x_1':|x_1-x_1'|\geq \f12\}$, there holds $\langle \wu^\pm(x_1)\rangle
\leq 3  |x_1-x_1'|\langle \wu^\pm(x_1')\rangle$.
%Now we explain the main ideas of the proof of these weighted estimates.
%For the weighted trace theorem and the weighted commutator estimate, we transform them into the commutator of the fractional derivatives and the weights $\langle w^+ \rangle^{\mu}\langle w^- \rangle^{\nu}$.
Specifically, these commutators mainly consist of a singular integral with weights, see, for instance, $\eqref{D14}$ and Lemma \ref{lemD0}:
$$\textrm{p.v.} \int_{\BR}  (1-\p_1^2)B_{2-s}(x_1-x_1')
\big(  \langle \wu^+(x_1)\rangle^{\mu}\langle \wu^-(x_1)\rangle^{\nu} -\langle \wu^+(x_1')\rangle^{\mu}\langle \wu^-(x_1')\rangle^{\nu}\big) h(x_1')\d x_1'.$$
%These commutators can be estimated. Actually,
For the integral domain near the origin
$\{x_1':|x_1-x_1'|\leq \f12\}$, since $\langle \wu^\pm(x_1)\rangle$ and
$\langle \wu^\pm(x_1')\rangle$ are equivalent,
we can freely move the weights $\langle \wu^\pm(x_1)\rangle$ and $\langle \wu^\pm(x_1')\rangle$ and keep the structure of the integral kernel near the origin.
For the integral domain away from the origin:
$\{x_1':|x_1-x_1'|\geq \f12\}$,  we see that $\langle \wu^+(x_1)\rangle^{\mu}\langle \wu^-(x_1)\rangle^{\nu}$ can be bounded by
$3^{|\mu|+|\nu|} |x_1-x_1'|^{|\mu|+|\nu|}\langle \wu^+(x_1')\rangle^{\mu}\langle \wu^-(x_1')\rangle^{\nu}$. The growth factor $|x_1-x_1'|^{|\mu|+|\nu|} $
will be absorbed into the exponential decay of the Bessel potentials. Thus the integral kernel is always summable.

We always assume \eqref{AA4}, \eqref{AA7}
and \eqref{AA8} with $\epsilon'$ sufficiently small.
Thus the pointwise estimate \eqref{AA4}, \eqref{C36}, \eqref{C38} for the weight functions hold.
%%Since the estimates are similar to those in \cite[Section 3]{CL-24}, we will sketch the proof of the lemmas in this section.
%As the estimates to be presented is this section
%closely resemble those detailed in \cite[Section 3]{CL-24}, we will provide only a concise outline of the proofs for the lemmas presented in this section.

%In the sequel, we only present some of the details, since t

\subsection{Weighted Sobolev inequalities}
In the sequel, we present several weighted Sobolev inequalities.
%in terms of the weight function $w^\pm$.
 They are the direct consequence of
the  Sobolev imbedding  %$H^{2}(\Omega_f) \hookrightarrow L^\infty(\Omega_f)$
and the properties of the weight functions $w^\pm$.
%Here we recall that $w^\pm$ are defined by \eqref{C1}.

%Hence in this section, we always assume
%$\sum_{|\alpha|\leq 2}\| \langle w^\pm \rangle^{2\mu} \nabla^\alpha \Lambda_\pm^1 \|_{L^\infty( \Omega_f )}\leq \epsilon$,
%$ \| \p_t f\|_{L^\infty(\BR )}+\|\p_1 f\|_{L^\infty(\BR )}+\|\p_1^2 f\|_{L^\infty(\BR )}\lesssim 1$ with $\epsilon$ sufficiently small.
\begin{lem}\label{Sobo1}
Let $g(t,\cdot )\in H^{2}(\Omega_f)$, $2<q<\infty$, $\mu\in \BR$, $\nu\in \BR$.
Under the assumptions \eqref{AA4} and \eqref{AA7}
with $\epsilon'$ sufficiently small,
there hold
\begin{align}
\label{SobWeit1}
\big\| \langle w^+ \rangle^{\mu} \langle w^- \rangle^{\nu} g (t,\cdot)\big\|_{L^\infty(\Omega_f)}
&\lesssim \sum_{\abs{\alpha}\leq 2}
\big\|\langle w^+ \rangle^{\mu}\langle w^- \rangle^{\nu} \nabla^\alpha g (t,\cdot)\big\|_{L^2(\Omega_f)},
\\
%&\Big\| \f{\langle w^\pm\rangle^{\lambda} g (t,\cdot)}
%{\langle w^\mp \rangle^{\mu}}\Big\|_{L^\infty(\Omega_f)}
%\lesssim \sum_{|\alpha|\leq 2}
%\Big\| \f{\langle w^\pm\rangle^{\lambda}
%\nabla^\alpha g(t,\cdot)}{\langle w^\mp \rangle^{\mu}}\Big\|_{L^2(\Omega_f)} \label{SobWeit2},\\
%\end{align}
%and
%\begin{align}
\label{SobWeit3}
\big\|\langle w^+ \rangle^{\mu}\langle w^- \rangle^{\nu} g (t,\cdot)\big\|_{L^q(\Omega_f)}
&\lesssim \sum_{\abs{\alpha}\leq 1}
\big\|\langle w^+ \rangle^{\mu}\langle w^- \rangle^{\nu} \nabla^\alpha g (t,\cdot)\big\|_{L^2(\Omega_f)}.
%\\
%&\Big\| \f{\langle w^\pm\rangle^{\lambda} g (t,\cdot)}
%{\langle w^\mp \rangle^{\mu}}\Big\|_{L^q(\Omega_f)}
%\lesssim \sum_{|\alpha|\leq 1}
%\Big\| \f{\langle w^\pm\rangle^{\lambda}
%\nabla^\alpha g(t,\cdot)}{\langle w^\mp \rangle^{\mu}}\Big\|_{L^2(\Omega_f)} \label{SobWeit4},
\end{align}
provided the right hand sides are finite.

\end{lem}
\begin{proof}
Firstly, by $H^2(\Omega_f)
\hookrightarrow L^\infty( \Omega_f )$, Lemma \ref{lemC2} and Lemma \ref{lemC4},
we have
\begin{align*}
\big\|\langle w^+ \rangle^{\mu} \langle w^- \rangle^{\nu} g (t,x)\big\|_{L^\infty(\Omega_f)}
&\lesssim \sum_{|\alpha|\leq 2}
\big\|\nabla^\alpha \big(\langle w^+ \rangle^{\mu}\langle w^- \rangle^{\nu} g (t,x)\big)\big\|_{ L^2(\Omega_f)} \\
&\lesssim \sum_{|\alpha|\leq 2}\big\|\langle w^+ \rangle^{\mu}\langle w^- \rangle^{\nu} \nabla^\alpha g (t,x)\big\|_{L^2(\Omega_f)},
\end{align*}
which yields \eqref{SobWeit1}.

The proof of \eqref{SobWeit3} is similar. By Sobolev imbedding
$H^{1}(\Omega_f) \hookrightarrow L^q(\Omega_f)$ with $2<q<\infty$, Lemma \ref{lemC2} and Lemma \ref{lemC4},
one has
\begin{align*}
\big\|\langle w^+ \rangle^{\mu}\langle w^- \rangle^{\nu} g (t,x)
\big\|_{L^q (\Omega_f)}
&\lesssim \sum_{|\alpha|\leq 1}
\big\|\nabla^\alpha \big( \langle w^+\rangle^{\mu}\langle w^- \rangle^{\nu}g (t,x) \big)\big\|_{L^2(\Omega_f)}  \\
&\lesssim \sum_{|\alpha|\leq 1}
\big\| \langle w^+ \rangle^{\mu}\langle w^- \rangle^{\nu}
\nabla^\alpha g(t,x) \big\|_{L^2(\Omega_f)}.
\end{align*}
%The proof of \eqref{SobWeit3}, \eqref{SobWeit4} is similar to
%\eqref{SobWeit1}, \eqref{SobWeit2} by Sobolev imbedding
%$H^{1}(\Omega_f) \hookrightarrow L^q(\Omega_f)$ for $2<q<\infty$
%and Lemma \ref{lemC3}.
%This ends the proof of the lemma.
\end{proof}

%In the following, we state two simple weighted imbedding inequalities
%in terms of the weight function $w^\pm$. They are just a consequence of
%the standard Sobolev imbedding theorem $H^{2}(\Omega_f) \hookrightarrow L^\infty(\Omega_f)$
%and the properties of the weight function.
Similar weighted Sobolev inequalities also hold on the free surface.
Here we directly state the following weighted Sobolev inequalities on the real line.
The details of the proof are omitted since they are similar to Lemma \ref{Sobo1}.
%The proof is similar to the above lemma, hence we directly list the lemma without proof.
\begin{lem}\label{Sobo2}
Let $h(t,\cdot )\in H^{1}(\BR )$, $2<q\leq \infty$, $\mu\in \BR$, $\nu\in \BR$.
Under the assumptions \eqref{AA4}, \eqref{AA7}
and \eqref{AA8}
with $\epsilon'$ sufficiently small,
there holds
\begin{align*}
%&\norm{\langle \ud{w}^\pm  \rangle^{\lambda} h (t,\cdot)}_{L^q(\BR )}
%\lesssim \sum_{\abs{a}\leq 1}\norm{\langle \ud{w}^\pm  \rangle^{\lambda} \nabla^a h(t,\cdot)}_{L^2(\BR )},\\
&\big\|\langle \ud{w}^+ \rangle^{\mu}\langle \ud{w}^-  \rangle^{\nu}
 h (t,\cdot) \big\|_{L^q(\BR )}
\lesssim \sum_{|a|\leq 1}
\big\|\langle \ud{w}^+\rangle^{\mu}\langle \ud{w}^-  \rangle^{\nu}
\nabla^a h(t,\cdot)\big\|_{L^2(\BR )} ,
\end{align*}
provided the right hand side is finite.
\end{lem}

\subsection{Weighted trace theorem}
For $0<s<1$, we write
\begin{align}\label{D1}
&\langle \p_{x_1}\rangle^{s} h(x_1)
%=(1-\p_{x_1}^2)^{\f s2} h(x_1)
%= (1-\p_{x_1}^2)^{\f s2 -1} (1-\p_{x_1}^2)h(x_1)\\\nonumber
%&= \int_{\BR} B_{2-s}(x_1-x_1') (1-\p_{x_1'}^2)h(x_1') \dx_1' \\\nonumber
=  \int_{\BR} B_{2-s}(x_1-x_1') h(x_1') \dx_1' \\\nonumber
&\qquad\quad + \text{p.v.} \int_{\BR} \p_{x_1}^2B_{2-s}(x_1-x_1') (h(x_1)-h(x_1')) \dx_1'.
\end{align}
Here p.v. denotes the principal value integral.
 $B_s(x_1)$ are Bessel potentials smooth on $\BR \backslash \{0\}$:
\begin{align*}
B_s(x_1)=\mathcal{F}^{-1}\big( (1+|\xi|^2)^{-\f{s}{2}} \big) (x_1),
\end{align*}
where $\mathcal{F}^{-1}$ stands for the inverse Fourier transform.
%\end{lem}
For the Bessel potentials, we recall the following asymptotic behavior near $0$ and at infinity.
\begin{lem}\label{lemBe}
If $0<s<1$, there hold
%\begin{align}\label{D2}
%&B_{2-s}(x_1) \leq C e^{-\f{|x_1|}{2}},\
%\p_{x_1} B_{2-s}(x_1) \leq C e^{-\f{|x_1|}{2}},\
%\p_{x_1}^2B_{2-s}(x_1) \leq C e^{-\f{|x_1|}{2}},\
%\text{when}\  |x_1|\geq \f12,\\\nonumber
%&B_{2-s}(x_1) \sim 1,\
%\p_{x_1}B_{2-s}(x_1) \sim x_1|x_1|^{-s-1},\
%\p_{x_1}^2 B_{2-s}(x_1) \sim |x_1|^{-s-1},\
% \text{when}\  |x_1|\leq \f12.
%\end{align}
\begin{align*}
&\begin{cases}
B_{2-s}(x_1) \leq C e^{-\f{|x_1|}{2}},\\
\p_{x_1} B_{2-s}(x_1) \leq C e^{-\f{|x_1|}{2}},\\
\p_{x_1}^2B_{2-s}(x_1) \leq C e^{-\f{|x_1|}{2}},\\
\end{cases}
\text{when}\  |x_1|\geq \f12
\end{align*}
and
\begin{align*}
\begin{cases}
B_{2-s}(x_1) \sim 1,\\
\p_{x_1}   B_{2-s}(x_1) \sim  x_1|x_1|^{-s-1},\\
\p_{x_1}^2 B_{2-s}(x_1) \sim |x_1|^{-s-1},\\
\p_{x_1}^3 B_{2-s}(x_1) \sim  x_1|x_1|^{-s-3},\\
\end{cases}
 \text{when}\  |x_1|\leq \f12.
\end{align*}
\end{lem}
\begin{proof}
We refer to Grafakos's book in Chapter 6.1.2 \cite{G_book}, Stein's book in Chapter V \cite{Stein} or Aronszajn-Smith \cite{AS} in Section 4 Chapter II
for the proof of the properties of the Bessel potentials.
\end{proof}
Now we study the commutator of the fractional derivative and the weight functions.
By using \eqref{D1}, we derive that
\begin{align}\label{D14}
&[\langle \p_{x_1}\rangle^{s}, w(x_1)] h(x_1)\\ \nonumber
&=-\text{p.v.} \int_{\BR}  (1-\p_{x_1}^2)B_{2-s}(x_1-x_1')
\big(  w(x_1)-w(x_1')\big) h(x_1')\dx_1',
\end{align}
where we have used the following commutator notation:
\begin{align*}
[\langle\nabla\rangle^{s},g]h
=\langle\nabla\rangle^{s}(g h)
-g\langle\nabla\rangle^{s}h.
\end{align*}
%Here \eqref{D14} holds for any functions $w(x_1)$ and $h(x_1)$ such that the singular integral make sense.

Before the weighted trace theorem and the weighted commutator estimate, we first estimate the second line of \eqref{D14} for appropriate weight functions.
%This technical lemma plays a central role
%in proving weighted trace theorem and the weighted commutator estimate.
\begin{lem}\label{lemD0}
%Assume that the function $w(x_1)$ satisfies the following properties:
%\begin{align*}
%& \langle w(x_1) \rangle \sim \langle w(x_1') \rangle
%\quad \textrm{for all}\quad |x_1-x_1'|\leq \f12,\\
%& \langle w(x_1) \rangle \leq 3|x_1-x_1'|\langle w(x_1') \rangle
%\quad \textrm{for all}\quad |x_1-x_1'|\geq \f12,\\
%&|\p_{x_1} w(x_1)| \lesssim 1 \quad \textrm{for all}\quad z\in \BR.
%\end{align*}
Let $0<s<1$, $1\leq q\leq +\infty$, $\mu\in\BR,\ \nu\in\BR$.
Under the assumptions \eqref{AA4}, \eqref{AA7} and \eqref{AA8}
with $\epsilon'$ sufficiently small,
there holds
\begin{align*} %\label{D15}
%&\big\| \textrm{p.v.} \int_{\BR}  (1-\p_{x_1}^2)B_{2-s}(x_1-x_1')
%\big(  \langle \wu^+(x_1) \rangle^{2\mu}-\langle \wu^+(x_1') \rangle^{2\mu}\big) h(x_1')dx_1' \big\|_{L^q(\BR)}\\
%&\lesssim \| \langle \wu^+ \rangle^{2\mu} h \|_{L^q(\BR )},\\[-4mm]\\
&\big\| \textrm{p.v.} \int_{\BR}  (1-\p_{x_1}^2)B_{2-s}(x_1-x_1')
\big(  \langle \wu^+(x_1) \rangle^{\mu}\langle \wu^-(x_1) \rangle^{\nu} \\
&\qquad\quad-\langle \wu^+(x_1') \rangle^{\mu}\langle \wu^-(x_1') \rangle^{\nu}\big) h(x_1')\dx_1' \big\|_{L^q(\BR)}\\[-4mm]\\
&\lesssim \| \langle \wu^+ \rangle^{\mu} \langle \wu^- \rangle^{\nu} h \|_{L^q(\BR )},
\end{align*}
provided the right hand side is finite.
\end{lem}
%(The lemma also holds if $L^2$ norm is replaced by $L^q$ norm for $1\leq q\leq +\infty$.)
\begin{proof}
%For the first inequality,
To estimate the $L^q$ norm of
\begin{align*}
& \textrm{p.v.} \int_{\BR}  (1-\p_{x_1}^2)B_{2-s}(x_1-x_1')
\big(  \langle \wu^+ (x_1) \rangle^{\mu} \langle \wu^- (x_1) \rangle^{\nu}\\\nonumber
&\qquad\qquad -\langle \wu^+(x_1') \rangle^{\mu} \langle \wu^-(x_1') \rangle^{\nu}\big) h(x_1')\dx_1'.
\end{align*}
We divide the integral domain $\BR$ into two domains
$\{x_1':|x_1-x_1'|\leq \f12\}$ and $\{x_1':|x_1-x_1'|\geq \f12\}$.

Firstly, if $|x_1-x_1'|\leq \f12$, we compute
\begin{align*}%\label{D6}
&\langle \wu^\pm(x_1) \rangle-\langle \wu^\pm(x_1')\rangle\\[-4mm]\nonumber\\\nonumber
&=(x_1-x_1')\cdot
\big( \langle \wu^\pm \rangle^{-1} \wu^\pm \p_{x_1} \wu^\pm \big)( \beta x_1+(1-\beta)x_1'),\ 0<\beta<1.
\end{align*}
By  the  properties of the weight functions in Lemma \ref{lemC3}, one has
\begin{align}\label{D15}
&\langle \wu^\pm(x_1)\rangle
\sim \langle \wu^\pm(x_1')\rangle
\sim \langle \wu^\pm(\beta x_1+(1-\beta)x_1')\rangle,\ 0<\beta<1.
\end{align}
Consequently, by \eqref{D15} and Lemma \ref{lemC3}, we obtain
\begin{align*}
&|\langle \wu^+(x_1) \rangle^{\mu} \langle \wu^-(x_1) \rangle^{\nu}
-\langle \wu^+(x_1')\rangle^{\mu}\langle \wu^-(x_1') \rangle^{\nu}| \\[-4mm]\nonumber\\\nonumber
&\lesssim |x_1-x_1'|\cdot
\big(  \langle \wu^+\rangle^{\mu-1}\langle \wu^-\rangle^{\nu}
+\langle \wu^+\rangle^{\mu}\langle \wu^-\rangle^{\nu-1}  \big)(x_1') .
% \\[-4mm]\nonumber\\\nonumber
%&\lesssim |x_1-x_1'|\cdot
% \langle \wu^+(x_1')\rangle^{\mu}
%\langle \wu^-(x_1')\rangle^{\nu}.
\end{align*}
Thus we deduce that
\begin{align*}
& \ \text{p.v.} \int_{|x_1-x_1'|\leq \f12} (1-\p_{x_1}^2) B_{2-s}(x_1-x_1')
\big[\langle \wu^+(x_1) \rangle^{\mu} \langle \wu^-(x_1) \rangle^{\nu}\\[-5mm]\\
&\qquad\qquad-\langle \wu^+(x_1')\rangle^{\mu}\langle \wu^-(x_1')\rangle^{\nu}\big] h(x_1') \dx_1'\\[-4mm]\\
&\lesssim \int_{|x_1-x_1'|\leq \f12} | (x_1-x_1')(1-\p_{x_1}^2)B_{2-s}(x_1-x_1')|\cdot
 \langle \wu^+(x_1')\rangle^{\mu}\langle \wu^-(x_1')\rangle^{\nu}|h(x_1')| \dx_1' .
\end{align*}
Taking the $L^q$ norm of the above expression,
by the Young inequality, it can be further bounded by
\begin{align*}
&\int_{|x_1'|\leq \f12}  |x_1'(1-\p_{x_1}^2) B_{2-s}(x_1')| \dx_1'
\cdot \| \langle \wu^+\rangle^{\mu}\langle \wu^-\rangle^{\nu}
h \|_{L^q(\BR )}\\
&\lesssim \| \langle \wu^+\rangle^{\mu}\langle \wu^-\rangle^{\nu}
h \|_{L^q(\BR )}.
\end{align*}

Next,  simple calculation gives
\begin{align} \label{D18}
\langle \wu^\pm(x_1)\rangle
&= \langle \wu^\pm(x_1')\rangle +\int_{x_1'}^{x_1} \p_{\tilde{x}_1}\langle \wu^\pm(\tilde{x}_1)\rangle \d\tilde{x}_1 \\\nonumber
&= \langle \wu^\pm(x_1')\rangle +\int_{x_1'}^{x_1} \f{\wu^\pm(\tilde{x}_1)\p_{\tilde{x}_1} \wu^\pm(\tilde{x}_1)}
 {\langle \wu^\pm(\tilde{x}_1)\rangle} \d\tilde{x}_1.
\end{align}
Thus we deduce from Lemma \ref{lemC3} and \eqref{D18} that,  if $|x_1-x_1'|\geq \f12$,
there holds
\begin{align} \label{D19}
\langle \wu^\pm(x_1)\rangle
\leq 3 |x_1-x_1'|\langle \wu^\pm(x_1')\rangle.
\end{align}
Then
\begin{align}\label{D20}
&\text{p.v.}\int_{|x_1-x_1'|\geq \f12}  (1-\p_{x_1}^2)B_{2-s}(x_1-x_1')
\big(  \langle \wu^+(x_1) \rangle^{\mu}\langle \wu^-(x_1) \rangle^{\nu} \\\nonumber
&\qquad\qquad -\langle \wu^+(x_1') \rangle^{\mu}\langle \wu^-(x_1') \rangle^{\nu} \big) h(x_1')\dx_1' \\\nonumber
&\lesssim \int_{|x_1-x_1'|\geq \f12}  |(1-\p_{x_1}^2)B_{2-s}(x_1-x_1')|\\\nonumber
&\qquad\cdot |x_1-x_1'|^{|\mu|+|\nu|} \langle \wu^+(x_1') \rangle^{\mu}
\langle \wu^-(x_1') \rangle^{\nu} |h(x_1')| \dx_1' .
\end{align}
Consequently, by the Young inequality and the properties of Bessel potentials of Lemma \ref{lemBe}, the $L^q$ norm of \eqref{D20} is further bounded by
\begin{align*}
& \int_{|x_1'|\geq \f12} |x_1'|^{|\mu|+|\nu|}\cdot |(1-\p_{x_1}^2)B_{2-s}(x_1')| \dx_1'
\cdot\|  \langle \wu^+ \rangle^{\mu} \langle \wu^- \rangle^{\nu} h \|_{L^q(\BR )} \\
&\lesssim  \| \langle \wu^+ \rangle^{\mu} \langle \wu^- \rangle^{\nu} h\|_{L^q(\BR )}.
\end{align*}
This finishes the proof of the lemma.
\end{proof}

Now we study the weighted trace theorem.
Let  $g=g(x)$  be a function defined in $\Om_f$.
Let $\ud{g}(x_1)=g(x_1,f(x_1))$ be the trace of $g(x)$ on the boundary $\Ga_f$. We start by recalling the  classical trace theorem in the two dimensional setting.
\begin{lem}[Classical trace theorem in the 2D case]
\label{lemD1}
Let $s> \f12$. There holds
\begin{align*}
&\|  \langle \p_{x_1} \rangle^{s-\f12} \ud{g} \|_{L^2(\BR )}
\lesssim  \| \langle \nabla \rangle^{s} g\|_{L^2(\Omega_f)}.
\end{align*}
Moreover, there holds
\begin{align*}
&\|  \ud{g} \|_{L^2(\BR )}
\lesssim  \|   g\|_{H^1(\Omega_f)}.
\end{align*}
\end{lem}
A direct consequence of the second inequality of Lemma \ref{lemD1} and Lemma \ref{lemC2} is the following lemma.
\begin{lem}[Relaxed weighted trace theorem in the 2D case] \label{lemD2}
Under the assumptions \eqref{AA4}, \eqref{AA7} and \eqref{AA8}
with $\epsilon'$ sufficiently small,
for any $\mu\in \BR,\ \nu\in\BR$, there holds
\begin{align*}
\big\| \langle \wu^+ \rangle^{\mu} \langle \wu^- \rangle^{\nu}  \ud{g} \big\|_{L^2(\BR )}
&\lesssim  \sum_{|\alpha|\leq 1} \big\| \langle w^+ \rangle^{\mu}\langle w^- \rangle^{\nu}  \p^{\alpha} g \big\|_{L^2(\Omega_f)}.
\end{align*}
\end{lem}
 Next, we study the weighted trace theorem in fractional derivatives.
 \begin{lem}[Weighted trace theorem in the 2D case]\label{lemD3}
 Under the assumption \eqref{AA4}, \eqref{AA7} and \eqref{AA8}
with $\epsilon'$ sufficiently small.
For any $\mu\in \BR,\ \nu\in \BR$, there holds
%(the derivative on the right hand side is wrong)
\begin{align*}
\big\| \langle \wu^+ \rangle^{\mu}\langle \wu^- \rangle^{\nu}
\langle \p_{x_1} \rangle^{\f12} \ud{g} \big\|_{L^2(\BR )}
&\lesssim  \sum_{|\alpha|\leq 1}\big\|\langle w^+ \rangle^{\mu}\langle w^- \rangle^{\nu} \p^{\alpha} g\big\|_{L^2(\Omega_f)}.
\end{align*}
provided the right hand side is finite.
\end{lem}
\begin{proof}%[Proof by Bessel potential:]
%The estimate of these inequalities is the same.
%In the sequel, we only present the details for the first one.
We start by rewriting \eqref{D14} as follows:
\begin{align}\label{D24}
&\langle \wu^+(x_1) \rangle^{\mu} \langle \wu^-(x_1) \rangle^{\nu} \langle \p_{x_1}\rangle^{\f 12} \ud{g}(x_1)\\\nonumber
&=\langle \p_{x_1}\rangle^{\f 12} \big( \langle \wu^+(x_1) \rangle^{\mu} \langle \wu^-(x_1) \rangle^{\nu} \ud{g}(x_1)\big)
\\\nonumber
&\quad+\text{p.v.} \int_{\BR}  (1-\p_{x_1}^2)B_{\f 32}(x_1-x_1')
\big(   \langle \wu^+(x_1) \rangle^{\mu} \langle \wu^-(x_1) \rangle^{\nu} \\\nonumber
&\qquad-\langle \wu^+(x_1') \rangle^{\mu}\langle \wu^-(x_1') \rangle^{\nu} \big) \ud{g}(x_1')\dx_1'.
\end{align}
%Now we estimate \eqref{D24}.
The first term on the right hand side \eqref{D24} can be controlled from the classical trace theorem of Lemma \ref{lemD1} and Lemma \ref{lemC3}.
For the second term, it is bounded from Lemma \ref{lemD0}
and the relaxed trace theorem Lemma \ref{lemD2}.
This finishes the proof of the lemma.
\end{proof}

\subsection{Weighted commutator estimate}
%Before the commutator estimate, we first give a simple lemma.
Different from the notations in the above subsection,
$g$ and $h$ in this subsection denote functions defined in $\BR$.

Before the weighted commutator estimate,
we first present two lemmas.
%an interpolation estimate for the fractional derivative.
\begin{lem}\label{lemD21}
Let  $0< s< 1$, $1\leq q\leq  \infty$. There holds
\begin{align*}
\| \langle \p_{x_1}\rangle^{s} h\|_{L^q(\BR )}
\lesssim \|   h  \|_{L^q(\BR )}
+\|  \p_{x_1}  h  \|_{L^q(\BR )}.
\end{align*}
\end{lem}
\begin{proof}
We refer to Lemma 3.8 in \cite{CL-24}.
\end{proof}
%Now we study a simple weighted estimate.
\begin{lem}\label{lemD23}
Let $s>0$, $\mu\in\BR,\, \nu\in\BR $, $1\leq q\leq\infty$.
Under the assumptions \eqref{AA4}, \eqref{AA7} and \eqref{AA8}
with $\epsilon'$ sufficiently small,  there hold
\begin{align*}
& \big\|  \langle \wu^+ \rangle^{\mu}\langle \wu^- \rangle^{\nu}
\langle \p_{x_1}\rangle^{-s} g \big\|_{L^q(\BR )}
\lesssim
\big\| \langle \wu^+ \rangle^{\mu} \langle \wu^- \rangle^{\nu}
 g\big\|_{L^q(\BR )} ,\\[-4mm]\\
& \big\|  \langle \wu^+ \rangle^{\mu}\langle \wu^- \rangle^{\nu}
 g \big\|_{L^q(\BR )}
\lesssim
\big\| \langle \wu^+ \rangle^{\mu} \langle \wu^- \rangle^{\nu}
\langle \p_{x_1}\rangle^{s} g\big\|_{L^q(\BR )} ,
\end{align*}
provided the right hand sides are finite.
\end{lem}
%This lemma will save us many convenience.
\begin{proof}
The second inequality is a direct consequence of the first estimate.
%The proof of the first and the second inequality is the same.
%Hence we only present the details for the second one.
Thus it suffices to prove the first inequality.

For $s>0$, by \eqref{D15} and \eqref{D19}, we write
\begin{align}\label{D25}
&\langle \wu^+ \rangle^{\mu}\langle \wu^- \rangle^{\nu}
\langle \p_{x_1}\rangle^{-s} g(x_1)  \\\nonumber
&=\langle \wu^+ (x_1) \rangle^{\mu}\langle \wu^-(x_1) \rangle^{\nu}
\int_{\BR} B_s(x_1-x_1')g(x_1')\dx_1' \\\nonumber
&\lesssim
\int_{|x_1-x_1'|\leq \f12} B_s(x_1-x_1') \langle \wu^+ (x_1') \rangle^{\mu}\langle \wu^-(x_1') \rangle^{\nu} |g(x_1')|\dx_1'  \\\nonumber
&\quad+\int_{|x_1-x_1'|\geq \f12} |x_1-x_1'|^{|\mu|+|\nu|} B_s(x_1-x_1') \langle \wu^+ (x_1') \rangle^{\mu}\langle \wu^-(x_1') \rangle^{\nu} |g(x_1')| \dx_1'.
\end{align}
Here Bessel potentials $B_s(x_1)$ satisfy
\begin{align*}
&B_s(x_1)\lesssim e^{-\f{|x_1|}{2}}\quad \textrm{for}\,\,|x_1|\geq \f12,
\end{align*}
and
\begin{align*}
&B_s(x_1)\sim
\begin{cases}
|x_1|^{s-1}, \quad \textrm{for} \quad 0<s<1 ,\\
\ln|x_1|^{-1}, \quad \textrm{for} \quad s=1,\\
1, \quad \textrm{for} \quad s>1,
\end{cases}
\quad\textrm{for}\,\, |x_1|\leq \f12.
\end{align*}
Taking the $L^q$ norm of \eqref{D25} for $1\leq q\leq +\infty$, by Young's inequality,
we obtain
\begin{align*}
\| \langle \wu^+ \rangle^{\mu}\langle \wu^- \rangle^{\nu}
\langle \p_{x_1}\rangle^{-s} g   \|_{L^q(\BR )}
\lesssim
\big\| \langle \wu^+ \rangle^{\mu} \langle \wu^- \rangle^{\nu}
 g \big\|_{L^q(\BR )} .
\end{align*}
This finishes the proof of the lemma.
\end{proof}

Combining Lemma \ref{lemD0} and  Lemma \ref{lemD23}, we immediately have the following commutator estimate.
\begin{lem}\label{lemD5}
Under the assumptions \eqref{AA4}, \eqref{AA7} and \eqref{AA8}
with $\epsilon'$ sufficiently small,
for  $0< s<1$, $\mu\in \BR,\, \nu\in\BR$, $1\leq q\leq\infty$, there hold
\begin{align*}
& \big\| [\langle \p_{x_1}\rangle^{s}, \langle \wu^+ \rangle^{\mu}\langle \wu^- \rangle^{\nu}] g \big\|_{L^q(\BR )}
\lesssim
\big\| \langle \wu^+ \rangle^{\mu} \langle \wu^- \rangle^{\nu}
 g\big\|_{L^q(\BR )} ,\\[-4mm]\nonumber\\\nonumber
%\end{align*}
%and
%\begin{align*}
& \big\| \langle \p_{x_1}\rangle^{s}\big( \langle \wu^+ \rangle^{\mu}\langle \wu^- \rangle^{\nu}  g\big)  \big\|_{L^q(\BR )}
\lesssim
\big\| \langle \wu^+ \rangle^{\mu} \langle \wu^- \rangle^{\nu}
 \langle \p_{x_1}\rangle^{s} g\big\|_{L^q(\BR )} ,
\end{align*}
provided the right hand sides are finite.
\end{lem}

Next we recall a version of the Kato-Ponce commutator estimate \cite{KP}.
We also refer to Li \cite{Li} for a more general and sharp commutator estimate.
\begin{lem}\label{lemD6}
Let  $0<s,s_1<1$, $\f12 =\f{1}{q_1}+\f{1}{q_2}=\f{1}{q_3}+\f{1}{q_4}$, $2\leq q_1, q_2, q_3, q_4 \leq\infty$, there hold
\begin{align*}
& \big\|[\langle\p_{x_1}\rangle^s,g] h \big\|_{L^2(\BR)}
\lesssim
\| \langle\p_{x_1}\rangle^s  g \|_{L^{q_1}(\BR)}
\| h \|_{L^{q_2}(\BR)},\\
&\big\|[\langle\p_{x_1}\rangle^{s_1}\p_{x_1},g] h \big\|_{L^2(\BR )}
\lesssim
\| \p_{x_1}  g \|_{L^{q_1}(\BR)}
\| \langle\p_{x_1}\rangle^{s_1} h \|_{L^{q_2}(\BR )}
+\| \langle\p_{x_1}\rangle^{s_1}\p_{x_1} g \|_{L^{q_3}(\BR)}
\|  h \|_{L^{q_4}(\BR)} .
\end{align*}
\end{lem}
\begin{proof}
We refer to Corollary 1.4 and Theorem 1.9 in \cite{Li}.
\end{proof}
Now we study the weighted commutator estimate for the energy estimate of the free surface.
We first present the weighted commutator estimate for fractional derivative ($0<s<1$) in one dimensional space.
Similar estimates also hold  in general space dimension and the proof is the same.
\begin{lem}[Weighted commutator estimate, first version] \label{lemD8}
Let  $0<s<1$, $\f12 =\f{1}{q_1}+\f{1}{q_2}$, $2\leq q_1, q_2 \leq\infty$, $\mu,\mu_1,\mu_2,\nu,\nu_1,\nu_2\in\BR$,
 $\mu=\mu_1+\mu_2$, $\nu=\nu_1+\nu_2$.
 Under the assumptions \eqref{AA4}, \eqref{AA7} and \eqref{AA8}
with $\epsilon'$ sufficiently small, there holds
\begin{align*}
 \big\| \langle \wu^+\rangle^{\mu}\langle \wu^-\rangle^{\nu}
[\langle\p_{x_1}\rangle^{s},g] h \big\|_{L^2(\BR )}
&\lesssim  \| \langle \wu^+\rangle^{\mu_1}\langle \wu^-\rangle^{\nu_1} \langle \p_{x_1}\rangle^{s} g \|_{L^{q_1}(\BR)} \| \langle \wu^+\rangle^{\mu_2}\langle \wu^-\rangle^{\nu_2}  h  \|_{L^{q_2}(\BR )} ,
\end{align*}
provided the right hand side is finite.
\end{lem}
\begin{proof}
We begin by rearrangement
%By commutator notation, we write
\begin{align}\label{D30}
&\langle \wu^+\rangle^{\mu}\langle \wu^-\rangle^{\nu}   [\langle\p_{x_1}\rangle^{s},g] h \\\nonumber
&=-[\langle \p_{x_1}\rangle^{s},\langle \wu^+\rangle^{\mu}\langle \wu^-\rangle^{\nu}] (gh) \\[-4mm]\nonumber\\\nonumber
&\quad+\big[\langle \p_{x_1}\rangle^{s},
\langle \wu^+\rangle^{\mu_1}\langle \wu^-\rangle^{\nu_1} g\big] (\langle \wu^+\rangle^{\mu_2}\langle \wu^-\rangle^{\nu_2}h)  \\[-4mm]\nonumber\\\nonumber
&\quad+\langle \wu^+\rangle^{\mu_1}\langle \wu^-\rangle^{\nu_1} g
\cdot
[\langle \p_{x_1}\rangle^{s}, \langle \wu^+\rangle^{\mu_2}\langle \wu^-\rangle^{\nu_2} ] h.
\end{align}
%\begin{align}\label{D30}
%&\langle \wu^+ \rangle^{2\mu} \langle \wu^- \rangle^{\mu} [\langle\p_{x_1}\rangle^{s},g] h \\[-4mm]\nonumber\\\nonumber
%%&=\langle \wu^+ \rangle^{2\mu} \langle \wu^- \rangle^{\mu}
%%\langle \p_{x_1}\rangle^{s} ( gh)
%%-\langle \p_{x_1}\rangle^{s}
%%\big( \langle \wu^+ \rangle^{2\mu} \langle \wu^- \rangle^{\mu}  g h \big)\\[-4mm]\nonumber\\\nonumber
%%&\quad +\langle \p_{x_1}\rangle^{s}
%%\big( \langle \wu^+ \rangle^{2\mu} \langle \wu^- \rangle^{\mu}  g h \big)
%%-\f{\langle \wu^+ \rangle^{2\mu}}{\langle \wu^- \rangle^{\mu}}g
%%\langle \p_{x_1}\rangle^{s}\big( \langle \wu^- \rangle^{2\mu}  h \big)\\\nonumber
%%&\quad+\f{\langle \wu^+ \rangle^{2\mu}}{\langle \wu^- \rangle^{\mu}}g
%%\cdot \big(\langle \p_{x_1}\rangle^{s} ( \langle \wu^- \rangle^{2\mu}  h  )
%%- \langle \wu^- \rangle^{2\mu} \langle\p_{x_1}\rangle^{s} h\big) \\\nonumber
%&=-[\langle \p_{x_1}\rangle^{s},\langle \wu^+ \rangle^{2\mu} \langle \wu^- \rangle^{\mu}]gh \\[-4mm]\nonumber\\\nonumber
%&\quad+\big[\langle \p_{x_1}\rangle^{s},
%\f{\langle \wu^+ \rangle^{2\mu}}{\langle \wu^- \rangle^{\mu}}g\big]
%\langle \wu^- \rangle^{2\mu}  h  \\[-4mm]\nonumber\\\nonumber
%&\quad+\f{\langle \wu^+ \rangle^{2\mu}}{\langle \wu^- \rangle^{\mu}}g
%\cdot
%[\langle \p_{x_1}\rangle^{s}, \langle \wu^- \rangle^{2\mu}] h.
%\end{align}
The second line on the right hand side of \eqref{D30} can be controlled by the classical commutator estimate Lemma \ref{lemD6} and the properties of the weight functions Lemma \ref{lemC3}.
The first and the third line can be estimated by Lemma \ref{lemD5} and the H\"{o}lder inequality.
%It remains to estimate the first line on the right hand side of \eqref{D30}.
%%By \eqref{D14}, we have the following expression:
%%\begin{align}\label{D32}
%%&\langle \wu^+(x_1) \rangle^{2\mu} \langle \wu^-(x_1) \rangle^{\mu} \langle \p_{x_1}\rangle^{\f 12} \big( g(x_1)h(x_1)\big) %\\[-4mm]\nonumber\\\nonumber
%%-\langle \p_{x_1}\rangle^{\f 12} \big( \langle \wu^+(x_1) \rangle^{2\mu} \langle \wu^-(x_1) \rangle^{\mu}  g(x_1)h(x_1) \big) \\\nonumber
%%&= \ \text{p.v.} \int_{\BR} (1-\p_{x_1}^2) B_{\f 32}(x_1-x_1')
%%\big[\langle \wu^+(x_1) \rangle^{2\mu} \langle \wu^-(x_1) \rangle^{\mu}\\%[-4mm]\nonumber\\
%%&\qquad\qquad\quad-\langle \wu^+(x_1')\rangle^{2\mu}\langle \wu^-(x_1')\rangle^{\mu}\big] g(x_1')h(x_1') \dx_1'.\nonumber
%%\end{align}
%%The second term on the left  hand side of \eqref{D32} is estimated
%%using
%By Lemma... and  the classical commutator estimate Lemma \ref{lemD6}, we obtain...
%
%(and furthermore, we need fractional interpolation in $L^\infty$ framework).
Thus the lemma is proved.
\end{proof}

\subsection{Weighted commutator estimate: higher order derivative}

In this subsection, we study the weighted commutator estimate with higher order fractional derivatives
\begin{align*}
& \big\| \langle \wu^+\rangle^{\mu} \langle \wu^-\rangle^{\nu} [\langle\p_{x_1}\rangle^{s}\p_{x_1},g] h \big\|_{L^2(\BR )}
\end{align*}
for $0<s<1$.
Before that, we study the commutator estimate
%of $\langle \p_{x_1}\rangle^{s} \p_{x_1}  $
% $(0<s<1)$ and the weight functions $\wu^\pm$.
\begin{align}\label{D31}
&\big\| [\langle \p_{x_1}\rangle^{s}\p_{x_1}, \langle \wu^+\rangle^{\mu} \langle \wu^-\rangle^{\nu}] h \big\|_{L^2(\BR )}.
\end{align}

For $h=h(x_1)$ defined in $\BR$,
%Now we study the commutator estimate for $\langle \p_{x_1}\rangle^{s} \p_{x_1}  $
% $(0<s<1)$.
 by \eqref{D1}, we have
\begin{align*}
&\langle \p_{x_1}\rangle^{s} \p_{x_1} h(x_1)
%= \int_{\BR} B_{2-s}(x_1-x_1') (1-\p_{x_1'}^2) \p_{x_1'}h(x_1') \dx_1' \\\nonumber
= \int_{\BR} \p_{x_1} B_{2-s}(x_1-x_1') h(x_1') \dx_1' \\\nonumber
&\quad + \text{p.v.} \int_{\BR} \p_{x_1}^2B_{2-s}(x_1-x_1') (\p_{x_1} h(x_1)-\p_{x_1'} h(x_1')) \dx_1'.
\end{align*}
%Here p.v. denotes the principal value integral.
%To estimate \eqref{D31},
Thus for $h(x_1)$ and $w(x_1)=\langle \wu^+\rangle^{\mu} \langle \wu^-\rangle^{\nu}$ defined on $\BR$, we have the symmetric commutator \cite{P}
\begin{align}\label{D32}
&\langle \p_{x_1}\rangle^{s}\p_{x_1} \big( w(x_1)  h(x_1)\big)
-w(x_1) \langle \p_{x_1}\rangle^{s}\p_{x_1} h(x_1)
-h(x_1) \langle \p_{x_1}\rangle^{s}\p_{x_1} w(x_1)
\\\nonumber
%&= \int_{\BR} \p_{x_1} B_{2-s}(x_1-x_1') \big(w(x_1') h(x_1')- w(x_1) h(x_1')-w(x_1') h(x_1) \big)\dx_1' \\\nonumber
%&\quad+ \ \text{p.v.} \int_{\BR} \p_{x_1}^2B_{2-s}(x_1-x_1') \big(w(x_1)-w(x_1')\big)\p_{x_1'}\big(h(x_1)-h(x_1')\big) \dx_1' \\\nonumber
%%\end{align}
%%%For the last line of \eqref{D32}, by integration by parts, we obtain
%%\begin{align}\label{D23}
&=\int_{\BR} \p_{x_1} B_{2-s}(x_1-x_1') \big(w(x_1') h(x_1')- w(x_1) h(x_1')-w(x_1') h(x_1) \big)\dx_1' \\\nonumber
%& \text{p.v.} \int_{\BR} \p_{x_1}^2B_{2-s}(x_1-x_1') \big(w(x_1)-w(x_1')\big)\p_{x_1'}\big(h(x_1)-h(x_1')\big) \dx_1' \\\nonumber
%&\quad +\text{p.v.} \int_{\BR} \p_{x_1}^2B_{2-s}(x_1-x_1')
%\p_{x_1'}  w(x_1') \big(h(x_1)-h(x_1')\big) \dx_1' \\\nonumber
&\quad - \text{p.v.} \int_{\BR} \p_{x_1}^3B_{2-s}(x_1-x_1') \big(w(x_1)-w(x_1')\big)\big(h(x_1)-h(x_1')\big) \dx_1' .
\end{align}
Thus in order to estimate \eqref{D31}, we need to estimate $\langle \p_1\rangle^{s}\p_1 \langle \wu^+\rangle^{\mu} \langle \wu^-\rangle^{\nu}$
and the right hand side of \eqref{D32}.

%In the sequel, we estimate them in several lemmas.

%The estimate of the second the third line is crazy.
%Now let us focus on the estimate of the third line.
Firstly, we estimate $\langle \p_1\rangle^{s}\p_1 \langle \wu^+\rangle^{\mu} \langle \wu^-\rangle^{\nu}$.
\begin{lem}\label{lemD24}
Under the assumptions \eqref{AA4}, \eqref{AA7} and \eqref{AA8}
with $\epsilon'$ sufficiently small,
for $0<s<1$, $\mu \in \BR,\ \nu\in\BR$,
there holds
\begin{align*}
%&| \langle \p_{x_1}\rangle^{s}\p_{x_1}
% \langle \wu^+(x_1)\rangle^{2\mu}|
%\lesssim \langle \wu^+(x_1)\rangle^{2\mu-1},
%\\[-4mm]\nonumber\\
&| \langle \p_{x_1}\rangle^{s}\p_{x_1}
\big(\langle \wu^+(x_1)\rangle^{\mu} \langle \wu^-(x_1)\rangle^{\nu} \big)| \\[-4mm]\nonumber\\
&\lesssim
\langle \wu^+(x_1)\rangle^{\mu-1}\langle \wu^-(x_1)\rangle^{\nu}
+\langle \wu^+(x_1)\rangle^{\mu}\langle \wu^-(x_1)\rangle^{\nu-1}.
\end{align*}
\end{lem}
%Note for the weight functions $\wu^\pm$, there hold
%\begin{align*}
%&\langle \wu^\pm(x_1)\rangle
%\sim \langle \wu^\pm(x_1')\rangle,\quad \textrm{if}\ \ |x_1-x_1'|\leq \f12, \\
%&\langle \wu^\pm(x_1')\rangle
%\leq 3  |x_1-x_1'|
%\langle \wu^\pm(x_1)\rangle ,\quad \textrm{if}\ \ |x_1-x_1'|\geq \f12.
%\end{align*}
%The proof of Lemma \ref{lemD24} is exactly the same to Lemma 3.14 in \cite{CL-24}.
%The details are omitted.
\begin{proof}
%We only show
%\begin{align*}
%&| \langle \p_{x_1}\rangle^{s}\p_{x_1}
%\big(\langle \wu^+(x_1)\rangle^{\mu} \langle \wu^-(x_1)\rangle^{\nu} \big)| \\[-4mm]\nonumber\\
%&\lesssim \langle \wu^+(x_1)\rangle^{\mu-1}\langle \wu^-(x_1)\rangle^{\nu}
%+\langle \wu^+(x_1)\rangle^{\mu}\langle \wu^-(x_1)\rangle^{\nu-1}.
%\end{align*}
%The other cases can be estimated similarly.
Denote $v(x_1)=\p_{x_1}(\langle \wu^+(x_1)\rangle^{\mu}\langle \wu^-(x_1)\rangle^{\nu})$.
By Lemma \ref{lemC3}, one has
\begin{align*}
v(x_1)\lesssim \langle \wu^+(x_1)\rangle^{\mu-1}\langle \wu^-(x_1)\rangle^{\nu}
+\langle \wu^+(x_1)\rangle^{\mu}\langle \wu^-(x_1)\rangle^{\nu-1}.
\end{align*}
By \eqref{D1}, we write
\begin{align}\label{D41}
&\langle \p_{x_1}\rangle^{s}\p_{x_1} (\langle \wu^+(x_1)\rangle^{\mu}\langle \wu^-(x_1)\rangle^{\nu})
=\langle \p_{x_1}\rangle^{s} v(x_1) \\\nonumber
&=\int_{\BR} B_{2-s}(x_1-x_1')v(x_1')  \dx_1' \\\nonumber
&\quad  +\text{p.v.}\int_{\BR} \p_{x_1}^2B_{2-s}(x_1-x_1') \big(v(x_1)-v(x_1')\big) \dx_1'.
\end{align}
Recalling that for the weight functions $\wu^\pm$, there hold
\begin{align*}
&\langle \wu^\pm(x_1)\rangle
\sim \langle \wu^\pm(x_1')\rangle,\quad \textrm{if}\ \ |x_1-x_1'|\leq \f12, \\
&\langle \wu^\pm(x_1')\rangle
\leq 3  |x_1-x_1'|
\langle \wu^\pm(x_1)\rangle ,\quad \textrm{if}\ \ |x_1-x_1'|\geq \f12.
\end{align*}
Thus for the second line of \eqref{D41},
we estimate similar to the proof of Lemma \ref{lemD0}:
\begin{align*}
& \int_{\BR} B_{2-s}(x_1-x_1')
v(x_1') \dx_1' \\\nonumber
&\lesssim
\big( \langle \wu^+(x_1)\rangle^{\mu-1}\langle \wu^-(x_1)\rangle^{\nu}
+\langle \wu^+(x_1)\rangle^{\mu}\langle \wu^-(x_1)\rangle^{\nu-1} \big) \\[-5mm]\nonumber\\
&\quad \Big(\int_{|x_1-x_1'|\leq \f12} |B_{2-s}(x_1-x_1')|  \dx_1'
+ \int_{|x_1-x_1'|\geq \f12} |x_1-x_1'|^{|\mu|+|\nu| }\cdot |B_{2-s}(x_1-x_1')|\dx_1'\Big)\\[-4mm]\nonumber\\
&\lesssim
 \langle \wu^+(x_1)\rangle^{\mu-1}\langle \wu^-(x_1)\rangle^{\nu}
+\langle \wu^+(x_1)\rangle^{\mu}\langle \wu^-(x_1)\rangle^{\nu-1} .
\end{align*}
For the last line of \eqref{D41},
if $|x_1-x_1'|\geq \f12$, by the exponential decay
of the Bessel potentials at infinity, we obtain
\begin{align*}
& \text{p.v.}
\int_{|x_1-x_1'|\geq \f12} \p_{x_1}^2B_{2-s}(x_1-x_1') \big(
v(x_1)-v(x_1') \big) \dx_1' \\\nonumber
&\lesssim
 \langle \wu^+(x_1)\rangle^{\mu-1}\langle \wu^-(x_1)\rangle^{\nu}
+\langle \wu^+(x_1)\rangle^{\mu}\langle \wu^-(x_1)\rangle^{\nu-1}   .
\end{align*}
If $|x_1-x_1'|\leq \f12$, by Lemma \ref{lemC3} and Lemma \ref{lemC5}, we obtain
%for some $0<\beta<1$,
%we write
\begin{align*}
|v(x_1)-v(x_1')|
%&\lesssim |x_1-x_1'|\cdot |\p_{x_1}v(\beta x_1+(1-\beta)x_1') |\\[-4mm]\nonumber\\
\lesssim |x_1-x_1'|
\big( \langle \wu^+(x_1)\rangle^{\mu-1}\langle \wu^-(x_1)\rangle^{\nu}
+\langle \wu^+(x_1)\rangle^{\mu}\langle \wu^-(x_1)\rangle^{\nu-1} \big) .
\end{align*}
Consequently,   we obtain
\begin{align*}
& \text{p.v.}
\int_{|x_1-x_1'|\leq \f12} \p_{x_1}^2B_{2-s}(x_1-x_1') \big(v(x_1)-v(x_1')\big) \dx_1'\\\nonumber
%&\lesssim \langle \wu^+(x_1)\rangle^{2\mu-2}\langle \wu^-(x_1)\rangle^{\mu}
%+\langle \wu^+(x_1)\rangle^{2\mu}\langle \wu^-(x_1)\rangle^{\mu-2}
%+\langle \wu^+(x_1)\rangle^{2\mu-1}\langle \wu^-(x_1)\rangle^{\mu-1} \\[-4mm]\nonumber\\\nonumber
&\lesssim \langle \wu^+(x_1)\rangle^{\mu-1}\langle \wu^-(x_1)\rangle^{\nu}
+\langle \wu^+(x_1)\rangle^{\mu}\langle \wu^-(x_1)\rangle^{\nu-1}.
\end{align*}
This finishes the proof of the lemma.
\end{proof}

For the second line of \eqref{D32},
by a similar argument to the proof of Lemma \ref{lemD0}, we have
\begin{lem}\label{lemD25}
Let $0<s<1$, $\mu \in\BR,\, \nu\in\BR $. Under the assumptions \eqref{AA4}, \eqref{AA7} and \eqref{AA8}
with $\epsilon'$ sufficiently small, there holds
\begin{align*}
%&\big\| \int_{\BR} \p_{x_1} B_{2-s}(x_1-x_1') \big(\langle \wu^+(x_1') \rangle^{2\mu} h(x_1')
% - \langle \wu^+(x_1) \rangle^{2\mu} h(x_1')\\\nonumber
%&\qquad -\langle \wu^+(x_1') \rangle^{2\mu} h(x_1) \big)dx_1'  \big\|_{L^2(\BR)}
%\lesssim \big\| \langle \wu^+ \rangle^{2\mu}  h  \big\|_{L^2(\BR)},\\\nonumber
&\big\| \int_{\BR} \p_{x_1} B_{2-s}(x_1-x_1') \big(\langle \wu^+(x_1') \rangle^{\mu}
\langle \wu^-(x_1') \rangle^{\nu} h(x_1')
 - \langle \wu^+(x_1) \rangle^{\mu}\langle \wu^-(x_1) \rangle^{\nu} h(x_1')\\\nonumber
&\qquad\quad -\langle \wu^+(x_1') \rangle^{\mu}\langle \wu^-(x_1') \rangle^{\nu} h(x_1) \big)\dx_1'  \big\|_{L^2(\BR)}
\lesssim \big\| \langle \wu^+ \rangle^{\mu} \langle \wu^- \rangle^{\nu}  h  \big\|_{L^2(\BR)},
\end{align*}
provided the right hand side is finite.
\end{lem}
For the last line of \eqref{D32}, we have the following estimate.
%we estimate it in the following lemma.
%\begin{align}
%\text{p.v.} \int_{\BR} \p_{x_1}^3B_{2-s}(x_1-x_1') \big(w(x_1)-w(x_1')\big)\big(h(x_1)-h(x_1')\big) \dx_1' .
%\end{align}
\begin{lem}\label{lemD26}
Let $0<s<1$, $\mu\in\BR,\, \nu\in\BR $.
Under the assumptions \eqref{AA4}, \eqref{AA7} and \eqref{AA8}
with $\epsilon'$ sufficiently small, there holds
\begin{align*} %\label{D51}
%&\big\| \text{p.v.} \int_{\BR} \p_{x_1}^3B_{2-s}(x_1-x_1')
%\big(\langle \wu^+(x_1) \rangle^{2\mu}-\langle \wu^+(x_1') \rangle^{2\mu} \big) \\\nonumber
%&\qquad \cdot\big(h(x_1)-h(x_1')\big) \dx_1' \big\|_{L^2(\BR)}
%\lesssim \big\| \langle \wu^+ \rangle^{2\mu} \langle \p_{x_1}\rangle^s h  \big\|_{L^2(\BR)},\\\nonumber
&\big\| \text{p.v.} \int_{\BR} \p_{x_1}^3B_{2-s}(x_1-x_1')
\big(\langle \wu^+(x_1) \rangle^{\mu}\langle \wu^-(x_1) \rangle^{\nu}
-\langle \wu^+(x_1') \rangle^{\mu}\langle \wu^-(x_1') \rangle^{\nu} \big) \\\nonumber
&\qquad\quad \cdot\big(h(x_1)-h(x_1')\big) \dx_1' \big\|_{L^2(\BR)}
 \lesssim \big\| \langle \wu^+ \rangle^{\mu}\langle \wu^- \rangle^{\nu} \langle \p_{x_1}\rangle^s h  \big\|_{L^2(\BR)},
\end{align*}
provided the right hand side is finite.
\end{lem}
%The proof of Lemma \ref{lemD26} is exactly the same to Lemma 4.15 in \cite{CL-24}.
%The details are omitted.
\begin{proof}
%We only present the details for
%\begin{align*}
%&\big\| \text{p.v.} \int_{\BR} \p_{x_1}^3B_{2-s}(x_1-x_1')
%\big(\langle \wu^+(x_1) \rangle^{\mu}\langle \wu^-(x_1) \rangle^{\nu}
%-\langle \wu^+(x_1') \rangle^{\mu}\langle \wu^-(x_1') \rangle^{\nu} \big) \\\nonumber
%&\qquad\quad \cdot\big(h(x_1)-h(x_1')\big) \dx_1' \big\|_{L^2(\BR)}
% \lesssim \big\| \langle \wu^+ \rangle^{\mu}\langle \wu^- \rangle^{\nu} \langle \p_{x_1}\rangle^s h  \big\|_{L^2(\BR)}.
%\end{align*}
%The other cases can be treated in the same method.
We still divide the integral domain $\BR$ into two domains
$\{x_1':|x_1-x_1'|\leq \f12\}$ and $\{x_1':|x_1-x_1'|\geq \f12\}$.
Recalling that if $|x_1-x_1'|\geq \f12$, there hold
\begin{align*}
\langle \wu^\pm(x_1)\rangle
\leq 3 |x_1-x_1'|\langle \wu^\pm(x_1')\rangle,
\quad  \langle \wu^\pm(x_1')\rangle
\leq 3 |x_1-x_1'|\langle \wu^\pm(x_1)\rangle.
\end{align*}
We can estimate similar to Lemma \ref{lemD0} by the exponential decay
of the Bessel potentials. Precisely, by Young's inequality, we deduce
\begin{align*} %\label{D52}
&\big\| \int_{|x_1-x_1'|\geq \f12} \p_{x_1}^3B_{2-s}(x_1-x_1')
\big(\langle \wu^+(x_1) \rangle^{\mu}\langle \wu^-(x_1) \rangle^{\nu} \\\nonumber
&\qquad\qquad-\langle \wu^+(x_1') \rangle^{\mu}\langle \wu^-(x_1') \rangle^{\nu} \big)\big(h(x_1)-h(x_1')\big) \dx_1' \big\|_{L^2(\BR)} \\[-5mm]\nonumber\\\nonumber
&\lesssim \big\| \int_{|x_1-x_1'|\geq \f12} |\p_{x_1}^3B_{2-s}(x_1-x_1')|\cdot
|x_1-x_1'|^{|\mu|+|\nu|}\langle \wu^+(x_1) \rangle^{\mu}\langle \wu^-(x_1) \rangle^{\nu}
|h(x_1)|   \dx_1' \big\|_{L^2(\BR)} \\\nonumber
&\quad+\big\| \int_{|x_1-x_1'|\geq \f12} |\p_{x_1}^3 B_{2-s}(x_1-x_1')|
\cdot |x_1-x_1'|^{|\mu|+|\nu|}\langle \wu^+(x_1') \rangle^{\mu}\langle \wu^-(x_1') \rangle^{\nu}
|h(x_1')|  \dx_1' \big\|_{L^2(\BR)} \\\nonumber
&\lesssim \|  \langle \wu^+ \rangle^{\mu}\langle \wu^- \rangle^{\nu} h \|_{L^2(\BR)}.
\end{align*}

Next we estimate the case of $|x_1-x_1'|\leq \f 12$.
To simplify the presentation, denote $v(x_1)=\langle \wu^+(x_1) \rangle^{\mu} \langle \wu^-(x_1) \rangle^{\nu} $.
We write
\begin{align*}
v(x_1)=v(x_1')+(x_1-x_1')\p_{x_1}v(x_1')+\f12 (x_1-x_1')^2\p_{x_1}^2v(\beta x_1+ (1-\beta)x_1')
,\,\, 0<\beta<1 .
\end{align*}
By Lemma \ref{lemC3} and Lemma \ref{lemC5}, there hold
\begin{align*}
&|\p_{x_1}v(x_1)| \lesssim \langle \wu^+(x_1)\rangle^{\mu}\langle \wu^-(x_1)\rangle^{\nu-1}
+\langle \wu^+(x_1)\rangle^{\mu-1}\langle \wu^-(x_1)\rangle^{\nu} ,\\[-4mm]\nonumber\\
&|\p_{x_1}^2v(x_1)| \lesssim \langle \wu^+(x_1)\rangle^{\mu-2}\langle \wu^-(x_1)\rangle^{\nu}
+\langle \wu^+(x_1)\rangle^{\mu-1}\langle \wu^-(x_1)\rangle^{\nu-1}
+\langle \wu^+(x_1)\rangle^{\mu}\langle \wu^-(x_1)\rangle^{\nu-2} .
\end{align*}
Thus
\begin{align}\label{D54}
&\text{p.v.} \int_{|x_1-x_1'|\leq \f12} \p_{x_1}^3B_{2-s}(x_1-x_1')
\big(v(x_1) -v(x_1') \big)\big(h(x_1)-h(x_1')\big) \dx_1' \\\nonumber
&=\f12\text{p.v.} \int_{|x_1-x_1'|\leq \f12} (x_1-x_1')^2 \p_{x_1}^3B_{2-s}(x_1-x_1')
 \, \p_{x_1}^2v(\beta x_1+ (1-\beta)x_1') \big(h(x_1)-h(x_1')\big) \dx_1' \\\nonumber
&\quad+\text{p.v.} \int_{|x_1-x_1'|\leq \f12} (x_1-x_1')\p_{x_1}^3 B_{2-s}(x_1-x_1')
 \p_{x_1}v(x_1') \big(h(x_1)-h(x_1')\big) \dx_1' .
\end{align}
For the second line of \eqref{D54}, we have the following bound:
\begin{align*}
&\f12\text{p.v.} \int_{|x_1-x_1'|\leq \f12} (x_1-x_1')^2 \p_{x_1}^3B_{2-s}(x_1-x_1')
 \, \p_{x_1}^2v(\beta x_1+ (1-\beta)x_1') \big(h(x_1)-h(x_1')\big) \dx_1' \\\nonumber
&\lesssim  \int_{|x_1-x_1'|\leq \f12} |x_1-x_1'|^{-s }
\big( \langle \wu^+(x_1')\rangle^{\mu}\langle \wu^-(x_1')\rangle^{\nu}  \big)\cdot |h(x_1)-h(x_1')| \dx_1' .
\end{align*}
Taking the $L^2$ norm of the above expression, by the Young inequality,
it is further controlled by
\begin{align*}
%\big\| (\langle \wu^+ \rangle^{2\mu}+\langle \wu^- \rangle^{\mu}) h\big\|_{L^2(\BR)}
%\lesssim
\big\| \langle \wu^+ \rangle^{\mu}\langle \wu^- \rangle^{\nu} h\big\|_{L^2(\BR)} .
\end{align*}
For the last line of \eqref{D54}, the integral kernel is still too singular. We further organize that
\begin{align*}
&\p_{x_1}v(x_1') \big(h(x_1)-h(x_1')\big) \\\nonumber
&=\big(\p_{x_1}v(x_1)\cdot h(x_1)-\p_{x_1}v(x_1')\cdot h(x_1')\big) \\\nonumber
&\quad-\big( \p_{x_1}v(x_1)-\p_{x_1}v(x_1')\big)  h(x_1)  .
\end{align*}
Consequently, we organize the last line of \eqref{D54} as follows
\begin{align}\label{D56}
&\text{p.v.} \int_{|x_1-x_1'|\leq \f12} (x_1-x_1')\p_{x_1}^3 B_{2-s}(x_1-x_1')
 \p_{x_1}v(x_1') \big(h(x_1)-h(x_1')\big) \dx_1' \\\nonumber
&=\text{p.v.} \int_{|x_1-x_1'|\leq \f12} (x_1-x_1')\p_{x_1}^3 B_{2-s}(x_1-x_1')
 \big(\p_{x_1}v(x_1)\cdot h(x_1)-\p_{x_1}v(x_1')\cdot h(x_1')\big)  \dx_1'\\\nonumber
&\quad-\text{p.v.} \int_{|x_1-x_1'|\leq \f12} (x_1-x_1')\p_{x_1}^3 B_{2-s}(x_1-x_1')
\big(\p_{x_1}v(x_1)-\p_{x_1}v(x_1')\big)  h(x_1) \dx_1' .
\end{align}
To estimate the last line of \eqref{D56}, by Lemma \ref{lemC3} and Lemma \ref{lemC5}, we get
\begin{align*}
&|\p_{x_1}v(x_1)-\p_{x_1}v(x_1')|
%&\lesssim |x_1-x_1'|\cdot |\p_{x_1}^2v(\beta x_1+(1-\beta)x_1')| \\[-4mm]\nonumber\\
\lesssim |x_1-x_1'| \big(\langle \wu^+(x_1)\rangle^{\mu-2}\langle \wu^-(x_1)\rangle^{\nu} \\[-4mm]\nonumber\\
&\qquad+\langle \wu^+(x_1)\rangle^{\mu-1}\langle \wu^-(x_1)\rangle^{\nu-1}
+\langle \wu^+(x_1)\rangle^{\mu}\langle \wu^-(x_1)\rangle^{\nu-2} \big) .
\end{align*}
Hence the last line of \eqref{D56} has the following bound
\begin{align*}
&\big|\text{p.v.} \int_{|x_1-x_1'|\leq \f12} (x_1-x_1')\p_{x_1}^3 B_{2-s}(x_1-x_1')
\big(\p_{x_1}v(x_1)-\p_{x_1}v(x_1')\big)  h(x_1) \dx_1' \big|\\\nonumber
&\lesssim  \int_{|x_1-x_1'|\leq \f12} (x_1-x_1')^2|\p_{x_1}^3 B_{2-s}(x_1-x_1')|
\big(\langle \wu^+(x_1)\rangle^{\mu}\langle \wu^-(x_1)\rangle^{\nu} \big) |h(x_1)| \dx_1'.
\end{align*}
Note that the integral kernel $x_1^2\p_{x_1}^3B_{2-s}(x_1)$ belongs to $L^1$.
Thus taking the $L^2$ norm of the above expression, % by the Young inequality,
it is controlled by
\begin{align*}
%\big\| (\langle \wu^+ \rangle^{2\mu}+\langle \wu^- \rangle^{\mu}) h\big\|_{L^2(\BR)}
%\lesssim
\big\| \langle \wu^+ \rangle^{\mu}\langle \wu^- \rangle^{\nu} h\big\|_{L^2(\BR)} .
\end{align*}
For the second line of \eqref{D56}, we write
%\begin{align}\label{D58}
%&\big\| \text{p.v.} \int_{|x_1-x_1'|\leq \f12} (x_1-x_1')\p_{x_1}^3B_{2-s}(x_1-x_1') \\\nonumber
%&\qquad\qquad\cdot \big(\p_{x_1}v(x_1)\cdot h(x_1)-\p_{x_1}v(x_1')\cdot h(x_1')\big)  \dx_1'
% \big\|_{L^2(\BR)}\\[-4mm]\nonumber\\\nonumber
%&\leq  \big\| \text{p.v.} \int_{\BR} (x_1-x_1')\p_{x_1}^3B_{2-s}(x_1-x_1')
% \big(\p_{x_1}v(x_1)\cdot h(x_1)-\p_{x_1}v(x_1')\cdot h(x_1')\big)  \dx_1'
% \big\|_{L^2(\BR)}\\\nonumber
%&\quad +\big\| \text{p.v.} \int_{|x_1-x_1'|\geq \f12} (x_1-x_1')\p_{x_1}^3 B_{2-s}(x_1-x_1')
% \big(\p_{x_1}v(x_1)\cdot h(x_1)-\p_{x_1}v(x_1')\cdot h(x_1')\big)  \dx_1'
% \big\|_{L^2(\BR)}.
%\end{align}
%The last line of \eqref{D58} can be estimated by the exponential decay of Bessel potentials at infinity.
\begin{align} \label{D58}
& \text{p.v.} \int_{|x_1-x_1'|\leq \f12} (x_1-x_1')\p_{x_1}^3 B_{2-s}(x_1-x_1') \\\nonumber
&\qquad\cdot \big(\p_{x_1}v(x_1)\cdot h(x_1)-\p_{x_1}v(x_1')\cdot h(x_1')\big)  \dx_1'\\\nonumber
&\sim \text{p.v.} \int_{|x_1-x_1'|\leq \f12} |x_1-x_1'|^{-1-s}
 \big(\p_{x_1}v(x_1)\cdot h(x_1)-\p_{x_1}v(x_1')\cdot h(x_1')\big)  \dx_1' \\[-4mm]\nonumber\\\nonumber
&\sim |\p_{x_1}|^s \big(\p_{x_1}v(x_1)\cdot h(x_1)\big)\\[-4mm]\nonumber\\\nonumber
&\quad- \text{p.v.} \int_{|x_1-x_1'|\geq \f12} |x_1-x_1'|^{-1-s}
 \big(\p_{x_1}v(x_1)\cdot h(x_1)-\p_{x_1}v(x_1')\cdot h(x_1')\big)  \dx_1'  .
\end{align}
Thus taking the $L^2$ norm of \eqref{D58}, for the fourth line of \eqref{D58}, repeating the argument of the proof of Lemma \ref{lemD5}, it is further bounded by
\begin{align*}
&\big\| |\p_{x_1}|^s \big(\p_{x_1}v \cdot h \big)\big\|_{L^2(\BR)}
\leq \| \langle \p_{x_1}\rangle^s \big(\p_{x_1}v \cdot h \big)\|_{L^2(\BR)}\\[-4mm]\\
&\lesssim \big\|
\langle \wu^+ \rangle^{\mu}\langle \wu^- \rangle^{\nu} \langle \p_{x_1}\rangle^s h\big\|_{L^2(\BR)}.
\end{align*}
For the last term of \eqref{D58}, by the Young inequality, we have
\begin{align*}
&\big\| \text{p.v.} \int_{|x_1-x_1'|\geq \f12} |x_1-x_1'|^{-1-s}
 \big(\p_{x_1}v(x_1)\cdot h(x_1)-\p_{x_1}v(x_1')\cdot h(x_1')\big)  \dx_1'  \big\|_{L^2(\BR)} \\
&\lesssim \|  \p_{x_1}v \cdot h  \|_{L^2(\BR)}
\lesssim \big\| \langle \wu^+ \rangle^{\mu}\langle \wu^- \rangle^{\nu}   h\big\|_{L^2(\BR)} .
\end{align*}
This finishes the proof of the lemma.
\end{proof}
Now we present the bound for \eqref{D31}.
Combining \eqref{D32}, Lemma \ref{lemD23}, Lemma \ref{lemD24}, Lemma \ref{lemD25} and
 Lemma \ref{lemD26},
%Lemma \ref{lemD27},
we have the following commutator estimate.
\begin{lem}\label{lemD9}
For $0< s<1$, $\mu \in \BR,\, \nu\in\BR$.  Under the assumptions \eqref{AA4}, \eqref{AA7} and \eqref{AA8}
with $\epsilon'$ sufficiently small, there holds
\begin{align*}
%& \big\| [\langle \p_{x_1}\rangle^{s}\p_{x_1}, \langle \wu^\pm \rangle^{2\mu}] h \big\|_{L^2(\BR )}
%\lesssim
%\big\| \langle \wu^\pm \rangle^{2\mu}
%\langle \p_{x_1} \rangle^{s } h \big\|_{L^2(\BR )} , \\[-4mm]\\
& \big\| [\langle \p_1\rangle^{s} \p_1,
\langle \wu^+ \rangle^{\mu}\langle \wu^- \rangle^{\nu}] h \big\|_{L^2(\BR )}
\lesssim
\big\| \langle \wu^+ \rangle^{\mu} \langle \wu^- \rangle^{\nu}
\langle \p_1 \rangle^s h\big\|_{L^2(\BR )} ,
\end{align*}
provided the right hand side is finite.
\end{lem}

Next, we study the  weighted commutator estimate for higher order fractional derivative.
The estimate also holds for the general dimensional case and the proof is similar.
\begin{lem}[Weighted commutator estimate, second version] \label{lemD10}
Let  $0<s<1$, $\f12 =\f{1}{q_1}+\f{1}{q_2}=\f{1}{q_3}+\f{1}{q_4}$, $2\leq q_1, q_2, q_3, q_4 \leq\infty$, $\mu,\mu_1,\mu_2,\nu,\nu_1,\nu_2\in\BR$,
 $\mu=\mu_1+\mu_2$, $\nu=\nu_1+\nu_2$. Under the assumptions \eqref{AA4}, \eqref{AA7} and \eqref{AA8}
with $\epsilon'$ sufficiently small, there holds
\begin{align*}
& \big\| \langle \wu^+\rangle^{\mu}\langle \wu^-\rangle^{\nu} [\langle\p_1\rangle^{s}\p_1,g] h \big\|_{L^2(\BR )} \\
&\lesssim \| \langle \wu^+\rangle^{\mu_1}\langle \wu^-\rangle^{\nu_1} \langle \p_1\rangle^s\p_1^{\leq 1} g \|_{L^{q_1}}
 \| \langle \wu^+\rangle^{\mu_2}\langle \wu^-\rangle^{\nu_2}  h \|_{L^{q_2}} \\
&\quad+ \| \langle \wu^+\rangle^{\mu_1}\langle \wu^-\rangle^{\nu_1}\p_1^{\leq 1} g \|_{L^{q_3}}
\| \langle \wu^+\rangle^{\mu_2}\langle \wu^-\rangle^{\nu_2} \langle \p_1\rangle^s h \|_{L^{q_4}} ,
\end{align*}
provided the right hand side is finite.
\end{lem}
\begin{proof}
Similar to \eqref{D30}, we have
\begin{align}\label{D61}
\langle \wu^+\rangle^{\mu}\langle \wu^-\rangle^{\nu}   [\langle\p_1\rangle^{s}\p_1,g] h
&=-[\langle \p_1\rangle^{s}\p_1,\langle \wu^+\rangle^{\mu}\langle \wu^-\rangle^{\nu}] (gh) \\[-4mm]\nonumber\\\nonumber
&\quad+\big[\langle \p_1\rangle^{s}\p_1,\langle \wu^+\rangle^{\mu_1}\langle \wu^-\rangle^{\nu_1} g\big] (\langle \wu^+\rangle^{\mu_2}\langle \wu^-\rangle^{\nu_2}h)  \\[-4mm]\nonumber\\\nonumber
&\quad+\langle \wu^+\rangle^{\mu_1}\langle \wu^-\rangle^{\nu_1} g\cdot [\langle \p_1\rangle^{s}\p_1, \langle \wu^+\rangle^{\mu_2}\langle \wu^-\rangle^{\nu_2} ] h.
\end{align}
For the first term  on the right hand side of \eqref{D61}, by Lemma \ref{lemD9}, Lemma \ref{lemD8} and the H\"{o}lder inequality, it is bounded by
\begin{align*}
&\| [\langle \p_1\rangle^{s}\p_1,\langle \wu^+\rangle^{\mu}\langle \wu^-\rangle^{\nu}] (gh) \|_{L^2(\bR)} \lesssim \| \langle \wu^+\rangle^{\mu}\langle \wu^-\rangle^{\nu}   \langle\p_1\rangle^{s} (gh) \|_{L^2(\BR )} \\
&\lesssim \| \langle \wu^+\rangle^{\mu_1}\langle \wu^-\rangle^{\nu_1} \langle \p_1\rangle^{s} g \|_{L^{q_1}(\BR)} \| \langle \wu^+\rangle^{\mu_2}\langle \wu^-\rangle^{\nu_2}  h  \|_{L^{q_2}(\BR )}\\
&\quad+\| \langle \wu^+\rangle^{\mu_1}\langle \wu^-\rangle^{\nu_1}  g \|_{L^{q_3}(\BR)} \| \langle \wu^+\rangle^{\mu_2}\langle \wu^-\rangle^{\nu_2} \langle \p_1\rangle^{s} h  \|_{L^{q_4}(\BR )}.
\end{align*}
For the second line of \eqref{D61},
 by applying Lemma \ref{lemD21}, Lemma \ref{lemD23}, Lemma \ref{lemD5} and  Lemma \ref{lemD6}, we have
\begin{align*}
&\big\| \big[\langle \p_1\rangle^{s}\p_1,\langle \wu^+\rangle^{\mu_1}\langle \wu^-\rangle^{\nu_1} g\big] (\langle \wu^+\rangle^{\mu_2}\langle \wu^-\rangle^{\nu_2}h) \big\|_{L^2} \\[-4mm]\nonumber\\\nonumber
&\lesssim  \| \langle \p_1\rangle^s\p_1\big( \langle \wu^+\rangle^{\mu_1}\langle \wu^-\rangle^{\nu_1} g\big)
\|_{L^{q_1}} \| \langle \wu^+\rangle^{\mu_2}\langle \wu^-\rangle^{\nu_2}  h \|_{L^{q_2}}\\\nonumber
&\quad+ \| \p_1 \big( \langle \wu^+\rangle^{\mu_1}\langle \wu^-\rangle^{\nu_1} g\big)  \|_{L^{q_3}}
\| \langle \p_1\rangle^s\big( \langle \wu^+\rangle^{\mu_2}\langle \wu^-\rangle^{\nu_2} h \big) \|_{L^{q_4}}\\\nonumber
&\lesssim  \| \langle \wu^+\rangle^{\mu_1}\langle \wu^-\rangle^{\nu_1} \langle \p_1\rangle^s\p_1^{\leq 1} g \|_{L^{q_1}}
 \| \langle \wu^+\rangle^{\mu_2}\langle \wu^-\rangle^{\nu_2}  h \|_{L^{q_2}}\\\nonumber
&\quad+ \| \langle \wu^+\rangle^{\mu_1}\langle \wu^-\rangle^{\nu_1}\p_1^{\leq 1} g \|_{L^{q_3}}
\| \langle \wu^+\rangle^{\mu_2}\langle \wu^-\rangle^{\nu_2} \langle \p_1\rangle^s h \|_{L^{q_4}} .
\end{align*}
For the third line of \eqref{D61}, by Lemma \ref{lemD9} and the H\"{o}lder inequality, it is bounded by
\begin{align*}
\| \langle \wu^+\rangle^{\mu_1}\langle \wu^-\rangle^{\nu_1}  g \|_{L^{q_3}(\BR)} \| \langle \wu^+\rangle^{\mu_2}\langle \wu^-\rangle^{\nu_2} \langle \p_1\rangle^{s} h  \|_{L^{q_4}(\BR )}.
\end{align*}
This finishes the proof of the lemma.
\end{proof}

Combining the above estimate, we have the following more general commutator estimate for higher order derivatives.
\begin{lem}[Weighted commutator estimate, higher order version] \label{lemD11}
Let $a\geq 2$ be an integer,  $0<s<1$, $\mu,\mu_1,\mu_2,\nu,\nu_1,\nu_2\in\BR$.
 $\mu=\mu_1+\mu_2$, $\nu=\nu_1+\nu_2$. Under the assumptions \eqref{AA4}, \eqref{AA7} and \eqref{AA8}
with $\epsilon'$ sufficiently small, there holds
\begin{align*}
&\| \langle \wu^+\rangle^{\mu}\langle \wu^-\rangle^{\nu} \big[\langle\p_1\rangle^s\p_1^a,g \big] \p_1h\|_{L^2(\BR)} \\\nonumber
&\lesssim \| \langle \wu^+\rangle^{\mu_1}\langle \wu^-\rangle^{\nu_1} \langle\p_1\rangle^s\p_1^{\leq a} g\|_{L^2}
          \| \langle \wu^+\rangle^{\mu_2}\langle \wu^-\rangle^{\nu_2} \langle\p_1\rangle^s\p_1^{\leq a} h\|_{L^2},
\end{align*}
provided the right hand side is finite.
\end{lem}

\subsection{Weighted chain rules}
Before the energy estimate, we give a weighted chain rule estimate.
\begin{lem}\label{lemG3}
Let $g=g(x):\, \Om\rightarrow \BR$, $f=f(x_1):\, \BR\rightarrow \BR$,
$\ud{g}(x_1)=g(x_1, f(x_1))$.
Let $k$ be a positive integer, $0<s<1$, $\mu\in\BR,\, \nu\in\BR $.
%Assume $\|f\|_{\dot{H}^1(\BR)}+\|f\|_{\dot{H}^{k}(\BR)}\leq 1$.
Under the assumptions \eqref{AA4}, \eqref{AA7} and \eqref{AA8}
with $\epsilon'$ sufficiently small,
 there hold
\begin{align*}
&\sum_{|a|\leq k}\big\| \langle\wu^+\rangle^{\mu}\langle\wu^-\rangle^{\nu} \p_1^a \ud{g}
\big\|_{L^2(\BR)}   \\
&\lesssim \sum_{|\alpha|\leq k}\big\| \langle\wu^+\rangle^{\mu}\langle\wu^-\rangle^{\nu}  \ud{\p^\alpha g}\big\|_{L^2(\BR)}
\cdot \mathcal{F}\big(1+ \|f\|_{\dot{H}^1(\BR)}+\|f\|_{\dot{H}^{k}(\BR)}\big),\\
&\sum_{|a|\leq k}\big\| \langle\wu^+\rangle^{\mu}\langle\wu^-\rangle^{\nu} \langle\p_1\rangle^{s} \p_1^a \ud{g}\big\|_{L^2(\BR)}   \\
&\lesssim
\sum_{|\alpha|\leq k}\big\| \langle\wu^+\rangle^{\mu}\langle\wu^-\rangle^{\nu} \langle\p_1\rangle^{s}  \ud{\p^\alpha g}\big\|_{L^2(\BR)}
\cdot \mathcal{F}\big(1+ \|f\|_{\dot{H}^1(\BR)}+\|\langle\p_1\rangle^{s}   f\|_{\dot{H}^{k}(\BR)} \big),
\end{align*}
%provided the right hand side is finite.
where $\mathcal{F}: \BR^+\rightarrow\BR^+$ is a generic non-decreasing function.
\end{lem}
\begin{proof}
The proof of the first and the second inequality is the same.
In the sequel, we only present the details for the first one.

The proof is conducted  by an induction argument.
The case of $k=1,\, 2$ is an easy consequence of chain rules.
Assume for $k\geq 2$, there holds
\begin{align}\label{G0}
&\sum_{|a|\leq k}\big\| \langle\wu^+\rangle^{\mu}\langle\wu^-\rangle^{\nu} \p_1^a \ud{g}
\big\|_{L^2(\BR)} \\\nonumber
&\lesssim \sum_{|\alpha|\leq k}
\big\| \langle\wu^+\rangle^{\mu}\langle\wu^-\rangle^{\nu}  \ud{\p^\alpha g}\big\|_{L^2(\BR)}
\cdot \mathcal{F}\big(1+\|f\|_{\dot{H}^1(\BR)}+\|f\|_{\dot{H}^k(\BR)}\big)  .
\end{align}
Next, we calculate
\begin{align*}
&\p_1^{k+1} \ud{g}=\p_1^k
  \big( \ud{\p_1 g}+ \ud{\p_2 g}\p_1 f\big) ,\\\nonumber
&=\p_1^k\ud{\p_1 g}+
\sum_{a+b=k}C_k^a \p_1^a \ud{\p_2 g}\p_1^{b+1} f ,
\end{align*}
where $ C_k^a=\f{k!}{a!(k-a)!}$ is the binomial coefficient.
Then by the H\"older inequality, the Sobolev inequality and \eqref{G0}, we have
\begin{align*}
& \big\| \langle\wu^+\rangle^{\mu}\langle\wu^-\rangle^{\nu} \p_1^{k+1} \ud{g} \big\|_{L^2(\BR)} \\
&\leq \big\| \langle\wu^+\rangle^{\mu}\langle\wu^-\rangle^{\nu} \p_1^k\ud{\p_1 g} \big\|_{L^2(\BR)}
+C \sum_{a+b=k} \big\| \langle\wu^+\rangle^{\mu}\langle\wu^-\rangle^{\nu} \p_1^a \ud{\p_2 g}\p_1^{b+1} f \big\|_{L^2(\BR)} \\
&\leq \big\| \langle\wu^+\rangle^{\mu}\langle\wu^-\rangle^{\nu} \p_1^k\ud{\p_1 g} \big\|_{L^2(\BR)}
+
\sum_{|a|\leq k}\big\| \langle\wu^+\rangle^{\mu}\langle\wu^-\rangle^{\nu} \p_1^a \ud{\p_2 g}\|_{L^2(\BR)}
\cdot\big( \|f\|_{\dot{H}^1(\BR)}+\|f\|_{\dot{H}^{k+1}(\BR)}\big)\\
&\lesssim
\sum_{|\alpha|\leq k+1}\big\| \langle\wu^+\rangle^{\mu}\langle\wu^-\rangle^{\nu}  \ud{\p^\alpha g}\big\|_{L^2(\BR)}
(1+\|f\|_{\dot{H}^1(\BR)}+\|f\|_{\dot{H}^{k+1}(\BR)}) \\
&\quad \cdot \mathcal{F}\big(1+\|f\|_{\dot{H}^1(\BR)}+\|f\|_{\dot{H}^{k+1}(\BR)}\big) .
\end{align*}
%Combined the assumption for the Sobolev norm of $f$,
Thus the case for $k+1$ is proved.
This finishes the proof of the lemma.
\end{proof}

\section{Weighted estimate of the pressure}
In this section, we treat the weighted estimate of pressure.
The estimate is conducted by straightening of the free boundary, thus
we will work in a fixed domain.
%``Lagrangian coordinates".

The weighted estimate of the pressure is conducted under the assumptions \eqref{AA6} and \eqref{AA4}.
Hence we always assume that for the free surface $x_2=f(t,x_1)$,
there hold
%\begin{align} \label{F1}
%&\|f(t,\cdot)\|_{\dot{H}^m(\BR )}
%\leq C\epsilon,\quad \textrm{for}\ 1\leq m\leq s+\f12,\\\nonumber
%&\|f(t,\cdot)\|_{L^\infty(\BR )}
%\leq C\epsilon,
%\end{align}
\begin{align} \label{F1}
&\|f(t,\cdot)\|_{L^\infty(\BR )}
\leq 1-\f12\delta,\,\,\|\p_1f(t,\cdot)\|_{H^{s-\f12}}
\leq 2C_0\epsilon,
%\quad \textrm{for}\ 1\leq m\leq s+\f12,
\end{align}
where $0<\delta<\f12$, $C_0$ is a universal constant, $\epsilon$ is sufficiently small.

\subsection{ Straightening of the free boundary}\label{Lag-stra}
%Flatten
We make the following change of variables $\Om_L \mapsto\Om_f$ to flatten the boundary of  $\Omega_f$ by a regularized mapping \cite{Lannes, ABZ, CMST}:
\begin{align*}
%(x_1,x_2)\in \BR\times [-1,0]=\Om_L
 \BR\times [-1,0]=\Om_L \ni (x_1,x_2)
\mapsto  \Psi(t,x_1,x_2)=(x_1,x_2+\psi(t,x_1,x_2)) \in \Omega_f,
\end{align*}
where
\begin{align*}
%&\psi:\mathbb{R}_+\times \mathbb{R} \times [-1,0] \mapsto  [-1, f(t,x_1,x_2)],\\
&\psi(t,x)=(1+x_2) \varphi(\rho x_2|\p_1|) f(t,x_1).
\end{align*}
Here $\varphi \in C_c^\infty(\BR)$ is a smooth cut-off function:
%($\varphi$ need to be define in $\BR$ since we will
%calculate the Fourier transform,
% and need check this cut-off and the above construction)
\begin{equation*}
\varphi(x_2)=
\begin{cases}
1,\quad \textrm{if}\quad -\f31 \leq x_2 \leq 0,\\
0,\quad \textrm{if}\quad  -1\leq x_2 \leq  -\f23.
\end{cases}
\end{equation*}
A small positive parameter $\rho$ is chosen such that
 the mapping $\Psi$ is a diffeomorphism.
This allows one to work in a fixed domain $\Om_L=\BR\times [-1,0]$
instead of on the moving domain $\Om_f$.
%Instead of working on the moving domain $\Om_f,$
%we could work in a flat finite channel $\Om_L=\BR\times [-1,0]$.
Denote the upper boundary of $\Om_L$ by $\Gamma=\BR\times\{0\}$
and the lower boundary by $\Gamma_-=\BR\times\{-1\}$,
there hold
% on the boundary
\begin{align*}
& \psi\big|_{\Gamma}=   f(t,x_1),
\quad \p_1\psi\big|_{\Gamma}=  \p_1f(t,x_1),
\quad\p_2\psi\big|_{\Gamma}= f(t,x_1),\\
&\psi\big|_{\Gamma_-}= 0,
\quad \p_1\psi\big|_{\Gamma_-}=0,
\quad \p_2\psi\big|_{\Gamma_-}= \varphi (- \rho |\p_1|)  f(t,x_1).
\end{align*}
Note that if $|\alpha|\geq 2$, $\p^\alpha\psi\big|_{\Gamma_-}$ and $\p^\alpha\psi\big|_{\Gamma}$ do not vanish.
%Note that we do not have vanishing property for $\p^\alpha\psi\big|_{\Gamma_-}$
% if $|\alpha|\geq 2$. Also $\p^\alpha\psi\big|_{\Gamma}$ does not vanish
% if $|\alpha|\geq 2$.

We summarize the estimate of $\psi$ in $L^2$ and $L^\infty$ norm in the following lemma.
%the $L^\infty$ and $L^2$ norm estimate of $\psi$ and its derivatives in the following lemma.
\begin{lem}\label{lemTR}
There holds the following $L^\infty$ estimate of $\psi$ and its derivatives.

\noindent
For $\psi$ and $\p_2\psi$ in  $L^\infty$ norm:
%\begin{align*}
%\|  \psi(t,\cdot)\|_{L^\infty(\Om_L)}
%&\lesssim \| f(t,\cdot)\|_{L^\infty(\BR )}.
%\end{align*}
\begin{align*}
\|\psi(t,\cdot )\|_{L^\infty}, \|\p_2\psi(t,\cdot )\|_{L^\infty}
&\leq \|f(t,\cdot)\|_{L^\infty(\bR)}+\rho \tilde{C} \|\p_1f(t,\cdot)\|_{H^1(\bR)}
\end{align*}
For $\p_1 \psi$ and $\nabla^2 \psi $ in  $L^\infty$ norm:
\begin{align*}
\|  \p_1 \psi(t,\cdot)\|_{L^\infty(\Om_L)}
&\lesssim \| \p_1 f(t,\cdot)\|_{L^\infty(\BR )}, \\
%\|  \p_2 \psi(t,\cdot)\|_{L^\infty(\Om_L)}
%&\lesssim \| |\p_1| f(t,\cdot)\|_{L^\infty(\BR )}
%+\| f(t,\cdot)\|_{L^\infty(\BR )} \\
%&\lesssim \| \p_1 f(t,\cdot)\|_{L^2(\BR )}^{\f12}
%\| \p_1^2 f(t,\cdot)\|_{L^2(\BR )}^{\f12}
%+\| f(t,\cdot)\|_{L^\infty(\BR )}.
%,\\
%\end{align*}
%Second derivative estimate of $\psi$ in  $L^\infty$ norm:
%\begin{align*}
\|  \p_1^2 \psi(t,\cdot)\|_{L^\infty(\Om_L)}
&\lesssim \| \p_1^2 f(t,\cdot)\|_{L^\infty(\BR )}, \\
\|  \p_1\p_2 \psi(t,\cdot)\|_{L^\infty(\Om_L)}
&\lesssim \| \p_1|\p_1| f(t,\cdot)\|_{L^\infty(\BR )}
+\| \p_1f(t,\cdot)\|_{L^\infty(\BR )}
\lesssim \| \p_1 f(t,\cdot)\|_{H^2(\BR )}  ,\\
\|  \p_2^2 \psi(t,\cdot)\|_{L^\infty(\Om_L)}
&\lesssim \| |\p_1|^2 f(t,\cdot)\|_{L^\infty(\BR )}
+\| |\p_1|f(t,\cdot)\|_{L^\infty(\BR )}
\lesssim \| \p_1 f(t,\cdot)\|_{H^2(\BR )}.
\end{align*}
%Various order $L^2$ norm of $\psi$:
%Moreover, there hold the following $L^2$ estimate of $\psi$ with its derivatives:
%\noindent
%$0$ order $L^2$ norm of $\psi$:
%For $\psi$ in  $L^2$ norm:
%\begin{align*}
%\|  \psi(t,\cdot)\|_{L^2(\Om_L)}
%&\lesssim \| \langle \p_1\rangle^{-\f12} f(t,\cdot)\|_{L^2(\BR )} .
%\end{align*}

\noindent
%$1$ order $L^2$ norm of $\psi$:
For $\p_1\psi$ in  $L^2$ norm:
\begin{align*}
\|  \p_1 \psi(t,\cdot)\|_{L^2(\Om_L)}
&\lesssim \| \langle \p_1\rangle^{-\f12} \p_1 f(t,\cdot)\|_{L^2(\BR )}.
%\|  \p_2 \psi(t,\cdot)\|_{L^2(\Om_L)}
%&\lesssim \| \langle \p_1\rangle^{-\f12} f(t,\cdot)\|_{H^1(\BR )}
%+\| \langle \p_1\rangle^{-\f12} f(t,\cdot)\|_{L^2(\BR )}.
\end{align*}
\noindent
%$2$ order $L^2$ norm of $\psi$:
For $\nabla^2 \psi$ in  $L^2$ norm:
\begin{align*}
\|  \p_1^2 \psi(t,\cdot)\|_{L^2(\Om_L)}
&\lesssim \| \langle \p_1\rangle^{-\f12} \p_1^2 f(t,\cdot)\|_{L^2(\BR )},\\
\|  \p_1\p_2 \psi(t,\cdot)\|_{L^2(\Om_L)}
&\lesssim \| \langle \p_1\rangle^{-\f12} \p_1 f(t,\cdot)\|_{H^1(\BR )},\\
\|  \p_2^2 \psi(t,\cdot)\|_{L^2(\Om_L)}
&\lesssim \| \langle \p_1\rangle^{-\f12}\p_1 f(t,\cdot)\|_{H^1(\BR )} .
\end{align*}
For $\nabla^k \psi$ in  $L^2$ norm ($2\leq k\leq s+1$):
\begin{align*}
\|  \nabla^k \psi(t,\cdot)\|_{L^2(\Om_L)}
&\lesssim %\| f(t,\cdot)\|_{\dot{H}^{k-\f12}(\BR )}
 \| \langle \p_1\rangle^{-\f12} \p_1^{k-1} f(t,\cdot)\|_{H^1(\BR )} .
\end{align*}
\end{lem}
\begin{rem}
Since $\| f(t,\cdot)\|_{L^2(\BR)}$ may not be bounded,
the estimate of $\psi$ can only rely on
$\|f(t,\cdot)\|_{L^\infty(\BR )}$ and $\|\p_1f(t,\cdot)\|_{H^{s-\f12}}$.
%$\|  \p_2 \psi(t,\cdot)\|_{L^2(\Om_L)}$ requires the bound of $\| f(t,\cdot)\|_{L^2(\BR)}$.
%Hence, throughout this paper, we must rely on the $L^\infty$ norm for $  \p_2\psi$.
\end{rem}
\begin{proof}[Proof of Lemma \ref{lemTR}]
Firstly, we write
\begin{align*}
\psi(t,x)&= (1+x_2) \varphi(\rho x_2|\p_1|)f(t,x_1) \\
&=(1+x_2)\big( f(t,x_1)
+\rho x_2\int_0^1 \varphi'(\rho \tau x_2|\p_1|) |\p_1|f(t,x_1)\textrm{d}\tau \big)\,.
\end{align*}
Next, we calculate $\p^\alpha \psi$. The first derivative of $\psi$ is given by
\begin{align*}
\p_1\psi(t,x)&=(1+x_2) \varphi(\rho x_2|\p_1|) \p_1f(t,x_1),\\
\p_2\psi(t,x)&=\rho (1+x_2) \varphi'(\rho x_2|\p_1|) |\p_1|f(t,x_1)
+ \varphi (\rho x_2|\p_1|)  f(t,x_1)\\
&= \rho (1+x_2) \varphi'(\rho x_2|\p_1|) |\p_1|f(t,x_1)+f(t,x_1)
+\rho x_2\int_0^1 \varphi'(\rho \tau x_2|\p_1|) |\p_1|f(t,x_1)\textrm{d}\tau\,.
\end{align*}
The second derivative of $\psi$ is given by
\begin{align*}
\p_1^2\psi(t,x)
=&(1+x_2)\varphi(\rho x_2|\p_1|) \p_1^2 f(t,x_1),\\
\p_1\p_2\psi(t,x)
=&\rho(1+x_2) \varphi' (\rho x_2|\p_1|) \p_1|\p_1| f(t,x_1)
+\varphi(\rho x_2|\p_1|) \p_1 f(t,x_1),
\\
\p_2^2\psi(t,x)
=&
\rho^2 (1+x_2) \varphi''(\rho x_2|\p_1|) |\p_1|^2 f(t,x_1)
+2\rho \varphi'(\rho x_2|\p_1|) |\p_1| f(t,x_1).
\end{align*}
For $|\alpha|=k\geq 3$, we have
\begin{align*}
\p^\alpha\psi=
\sum_{k-1 \leq |\beta|\leq k} \sum_{|\gamma|\leq k} \rho^\gamma
C_{\beta,\gamma}(x_2) \varphi^{(\gamma)} ( x_2|\p_1|) A^{\beta}(\p_1) f ,
\end{align*}
where $A^{\beta}(\p_1)$ is a homogeneous differential operator of order $|\beta|$,
$C_{\beta,\gamma}(x_2)$ is a linear function which is obviously bounded if $x_2\in[-1,0]$.
%and it depends on $\delta$.

Let $\hat{\psi}$ stand for the Fourier transform of $\psi$  in $x_1$ variable.
Then
\begin{align*}
\hat{\psi}(t,\xi_1,x_2)&=(1+x_2) \varphi( \rho x_2|\xi_1|) \hat{f}(t,\xi_1),\\
%\end{align*}
%Thus
%\begin{align*}
\psi(t,x)&=(1+x_2) \mathcal{F}^{-1}\big( \varphi(\rho x_2|\xi_1|)  \big)\ast f(t,x_1).
\end{align*}
Here $\ast$ stands for a convolution in $x_1$ variable.
Since $\varphi$ is a Schwartz function, its Fourier transform is also a Schwartz function.
Then by H\"older inequality, we have
\begin{align*}
\|  \p_1 \psi(t,\cdot)\|_{L^\infty(\Om_L)}
\lesssim \| \p_1 f(t,\cdot)\|_{L^\infty(\BR )} \cdot \|\mathcal{F}^{-1} \varphi \|_{L^1(\BR)}
\lesssim \| \p_1 f(t,\cdot)\|_{L^\infty(\BR )} .
\end{align*}
The estimate of $\| \nabla^2 \psi(t,\cdot)\|_{L^\infty(\Om_L)}$ follows from the same procedure and the details are omitted.
%The estimate of the $L^\infty$ norm of $\psi$ with derivatives is the same.
%By the expression of $\p^\alpha \psi$,
%%$\psi$ with derivatives,
%by a similar argument,
%we see the $L^\infty$ bounds in the lemma holds.
%There holds

Next,
\begin{align*}
\|\psi(t,\cdot )\|_{L^\infty}
&\leq \|f(t,\cdot)\|_{L^\infty(\bR)}
+\rho \int_0^1 \|\varphi'(\rho \tau x_2|\p_1|) |\p_1|f(t,x_1)\|_{L^\infty(\bR)}\textrm{d}\tau\\
&\leq \|f(t,\cdot)\|_{L^\infty(\bR)}+\rho \tilde{C} \|\p_1f(t,\cdot)\|_{H^1(\bR)}
\end{align*}
Similarly,% estimate yields
\begin{align*}
%\|\p_1\psi(t,x)\|_{L^\infty}
%&\leq \|\varphi(\rho x_2|\p_1|) \p_1f(t,x_1)\|_{L^\infty}
%\lesssim \|\p_1f(t,\cdot)\|_{L^\infty},\\
\|\p_2\psi(t,\cdot)\|_{L^\infty}
%&=\rho (1+x_2) \varphi'(\rho x_2|\p_1|) |\p_1|f(t,x_1)
%+ \varphi (\rho x_2|\p_1|)  f(t,x_1)\\
%&\leq \rho\| (1+x_2) \varphi'(\rho x_2|\p_1|) |\p_1|f(t,x_1)\|_{L^\infty(\bR)}\\
%&\quad +\|f(t,\cdot)\|_{L^\infty(\bR)}
%+\rho \int_0^1 \|\varphi'(\rho \tau x_2|\p_1|) |\p_1|f(t,x_1)\|_{L^\infty(\bR)}\textrm{d}\tau\\
&\leq \|f(t,\cdot)\|_{L^\infty(\bR)}+\rho \tilde{C} \|\p_1f(t,\cdot)\|_{H^1(\bR)}
\end{align*}

Now we turn to the $L^2$ estimate of $\psi$.
%To estimate $\|\psi\|_{L^2(\Om_L)}$,
Note
%(there might be several choice)
\begin{align*}
\int_{-1}^{0} |(1+x_2) \varphi( x_2|\xi_1|) |^2 \dx_2
\leq C (|\xi_1|+1)^{-1}.
\end{align*}
Then by the Parseval-Plancherel identity, there holds
\begin{align*}
&\|\p_1\psi(t,\cdot)\|_{L^2(\Om_L)}^2
%=\|\hat{\psi}(t,\xi_1,\cdot)\|_{L^2(\Om_L)}^2
= \int_{\BR}\|\widehat{\p_1\psi}(t,\xi_1,\cdot)\|_{L^2([-1,0])}^2 \d\xi_1 \\
&=\int_{\BR} |\widehat{\p_1f}(t,\xi_1)|^2
\int_{-1}^{0} |(1+x_2) \varphi( x_2|\xi_1|) |^2 \dx_2 \d\xi_1\\
&\lesssim \| \langle \p_1\rangle^{-\f12} \p_1 f(t,\cdot)\|_{L^2(\BR )}.
\end{align*}
%The estimate of the $L^2$ norm of $\psi$ with derivatives is the same.
%A similar argument on  $\psi$ with its derivatives yields the $L^2$ bounds in the lemma.
The estimate of $\nabla^k \psi$ in  $L^2$ norm ($2\leq k\leq s+1$) is the same and the details are omitted.
This finishes the proof of Lemma \ref{lemTR}.
\end{proof}

\subsection{Quantities in fixed domain}\label{Sec-Lag}
The Jacobian matrix of the change of variables is given by
\begin{align*}
\nabla \Psi =
\left(
\begin{matrix}
1 \qquad\ 0\\
\p_1\psi \quad 1+\p_2\psi
\end{matrix}
\right).
\end{align*}
The Jacobian determinant is given by $J=\textrm{det}(\nabla \Psi)=1+\p_2\psi$.
By the $L^\infty$ estimate of $\p_2\psi$ in Lemma \ref{lemTR} and \eqref{F1}, we choose positive $\rho$ sufficiently small such that
$\rho \tilde{C} \|\p_1f(t,\cdot)\|_{H^1(\bR)}\leq \f14 \delta$. Hence
\begin{align*}
\f14\delta \leq J\leq 2-\f14 \delta,
\end{align*}
for $\epsilon$ sufficiently small.
This ensures that $\Psi(t,\cdot)$ is a diffeomorphism.

The inverse of $\nabla\Psi$ is denoted by $A$
\begin{align*}
A=[\nabla \Psi]^{-1}=\frac{1}{J}
\left(
\begin{matrix}
1+\p_2\psi \quad 0\\
-\p_1\psi \qquad 1
\end{matrix}
\right).
\end{align*}
It is easy to see that $\textrm{det}(A)=J^{-1}=(1+\p_2\psi)^{-1}$.

In the fixed domain $\Om_L$,
the $L^q$ norm %in the Lagrangian coordinates
can be defined as follows:
\begin{align*}
\| g\|_{L^q(\Om_L)}=\big(\int_{\Om_L} |g(x)|^q J \dx\big)^{\f1q}.
\end{align*}
Note that $\f14\delta\leq J\leq 2-\f14\delta$, hence the above definition is equivalent to the $L^q$ norm.
In the remainder of this section,  we simply write $L^q$ ($1\leq q\leq+\infty$)
 for $L^q(\Om_L)$ unless otherwise stated.

By setting
\begin{align*}
%&\wt{Z}_\pm (t,x)=Z_\pm  (t,\Psi(t,x)),\quad
\wt{\Lambda}_\pm (t,x)=\Lambda_\pm (t,\Psi(t,x) ),\quad
\wt{p} (t,x)=p (t,\Psi(t,x) ),
\end{align*}
The problem is reduced to the fixed domain $\Om_L$.
In order to simplify the presentation, let us denote
\begin{align*}
\nabla^\psi=A^\top \nabla.
\end{align*}

In the fixed domain $\Om_L$,
the MHD equations \eqref{MHD1} are equivalent to the following form:
%\begin{align*}
%\p_t \wt{\Lambda}_\pm(t,x)
%&= (\p_t \Lambda_\pm) (t,\Psi(t,x))
%+\p_t\Psi\cdot (\nabla \Lambda_\pm) (t,\Psi(t,x)) \\
%&=(\pm e_1\cdot\nabla\Lambda_\pm-\Lambda_\mp\cdot\nabla \Lambda_\pm) (t,\Psi(t,x))-(\nabla p)(t,\Psi(t,x)) \\
%&\quad+(0,\p_t\psi)\cdot (\nabla \Lambda_\pm) (t,\Psi(t,x)),
%\end{align*}
%which can be further simplified as follows:
\begin{equation}\label{J20}
\begin{cases}
\p_t \wt{\Lambda}_\pm
+(\wt{\Lambda}_\mp  \mp e_1 )\cdot (A^\top\nabla) \wt{\Lambda}_\pm  +A^\top \nabla \wt{p}
-(0,\p_t\psi/J)\cdot \nabla \wt{\Lambda}_\pm =0 ,\\
\div^\psi \wt{\Lambda}_\pm=0,
\end{cases}
 \textrm{in}\,\, \Om_L.
\end{equation}
Due to the boundary conditions \eqref{A18}, %and \eqref{A19},
 we deduce
\begin{align} \label{J21}
%u_2=0,\,\, h_2=0
\wt{\Lambda}_+\cdot e_2&=0,\,\, \wt{\Lambda}_-\cdot e_2=0
\quad \text{on}\,\, \Gamma^- .
% ,\\
%\label{J21}
%\wt{p}&=0
%,\quad
% h\cdot N = 0
% \quad \text{on}\,\, \Gamma.
\end{align}
Note that on the upper boundary $\Gamma$, we have
\begin{align*}
A\big|_{\Gamma}
=\frac{1}{1+f}
\left(
\begin{matrix}
1+f \quad 0\\
-\p_1 f  \quad 1
\end{matrix}
\right).
\end{align*}
On the lower boundary $\Gamma_-$, we obtain
\begin{align*}
A\big|_{\Gamma_-}
=\frac{1}{1+\p_2\psi\big|_{\Gamma_-}}
\left(
\begin{matrix}
1+\p_2\psi\big|_{\Gamma_-} \quad 0\\
\qquad 0  \qquad\quad 1
\end{matrix}
\right).
\end{align*}

Applying $\div^\psi$ to $\eqref{J20}_1$,
the equation of the pressure in the fixed domain reads as follows
\begin{align} \label{H1}
& -\nabla^\psi\cdot \nabla^\psi \wt{p}
=\nabla^\psi\cdot (\wt{\Lambda}_-\cdot \nabla^\psi \wt{\Lambda}_+)
=\p^\psi_j \wt{\Lambda}^i_- \p^\psi_i \wt{\Lambda}^j_+  \quad  {\rm in} \; \Omega_L.
\end{align}
Restricting the second equation of $\eqref{J20}_1$ to $\Gamma_-$ in the trace sense, due to the boundary condition  \eqref{A19}, \eqref{J21} and the expression of $A\big|_{\Gamma_-}$,
we obtain
\begin{align}\label{H2}
\wt{p} \big|_{\Gamma}=0,\quad  \p_2 \wt{p} \big|_{\Gamma_-}=0.
\end{align}
Let us organize \eqref{H1}  in another way in order conveniently to treat the
energy estimate of the pressure. Set $a=A J$, there holds the so-called Piola's identity:
\begin{align}\label{H3}
\nabla\cdot a^\top=0.
\end{align}
Then multiplying \eqref{H1} by $J$, we obtain
\begin{align} \label{H4}
& -a^\top\nabla\cdot \nabla^\psi \wt{p}
=a^\top\nabla\cdot (\wt{\Lambda}_-\cdot \nabla^\psi \wt{\Lambda}_+).
\end{align}

%\subsection{Boundary conditions}
%There holds on the boundary that
%$$\wt{p} \big|_{\Gamma}=0. $$
In the sequel, we frequently use the value of $a_{2,i}\p_i^\psi\wt{p}$ on $\Gamma$ and $\Gamma_-$.
Here,
% we calculate them by the boundary condition of pressure \eqref{H2}, %$A\big|_{\Gamma}$ and $A\big|_{\Gamma_-}$,
we calculate
%On the upper boundary $\Gamma$, by \eqref{J21} and $A\big|_{\Gamma}$, we have
\begin{align*}
&a_{2,i}\p_i^\psi\wt{p}\big|_{\Gamma}
=J A_{2,i} A_{j,i} \p_j \wt{p}\big|_{\Gamma}
%\\
%&=J^{-1}(1+|\p_1 f|^2) \, \p_2 \wt{p} \big|_{\Gamma}
%-\p_1 f \, \p_1 \wt{p} \big|_{\Gamma}
=(1+f)^{-1} (1+|\p_1 f|^2) \, \p_2 \wt{p} \big|_{\Gamma}\, ,\\
%\end{align*}
%In the above deduction, we have used
%\begin{align*}
%\wt{p} \big|_{\Gamma}=0 .
%\end{align*}
%and
%$$\p_2\wt{p} \big|_{\Gamma_-}=0. $$
%On the lower boundary $\Gamma_-$, by \eqref{J22} and $A\big|_{\Gamma_-}$, there holds
%\begin{align*}
&a_{2,i}\p_i^\psi\wt{p}\big|_{\Gamma_-}
=J A_{2,i} A_{j,i} \p_j \wt{p}\big|_{\Gamma_-}
=J^{-1} \p_2 \wt{p} \big|_{\Gamma_-}=0.
\end{align*}

%\subsection{Weight functions in Lagrangian coordinates}
%In the sequel, we are going to treat the weighted estimate of pressure in Lagrangian coordinates.
Due to the change of coordinates, the weight functions in the fixed domain $\Om_L$ are thus given as follows:
$$W^\pm(t,x)=w^\pm (t,\Psi(t,x)).$$
By chain rules, there hold
\begin{align*}
&\p_1 W^\pm(t,x)=(\p_1w^\pm) (t,\Psi(t,x))+ (\p_2 w^\pm) (t,\Psi(t,x)) \cdot(1+\p_1\psi),\\
&\p_2 W^\pm(t,x)=(\p_2w^\pm) (t,\Psi(t,x))\cdot (1+\p_2\psi).
\end{align*}
Consequently, by Lemma  \ref{lemC2}, the estimate for $\nabla \psi$ in Lemma \ref{lemTR} and \eqref{F1}, there exists $C$ such that
%\begin{equation}  \label{F11}
%\begin{cases}
%|\p_1W^\pm-1|\leq C\epsilon ,\\
%|\p_2W^\pm|\leq C\epsilon.
%\end{cases}
%\end{equation}
\begin{equation}  \label{F11}
|\p_1W^\pm-1|\leq C\epsilon' ,\quad
|\p_2W^\pm|\leq C\epsilon'.
\end{equation}

Before the weighted estimate for pressure,
we first present a simple Poincar\'e inequality.
\begin{lem}\label{lemHardy}
For all $g(x)\in \dot{H}^1(\Om_L)$ with $g|_{\Gamma}=0$, there holds
\begin{align}\label{HH5}
&\| g\|_{L^2}\lesssim \| \p_2 g\|_{L^2}.
\end{align}
Moreover, under the assumption \eqref{F1},
for any $\nu\in \BR,\ \mu\in \BR$, there hold
%\eqref{H5} reduces to
\begin{align}
\label{HH6}
\| g\|_{L^2}
&\leq 2\| \p_2^\psi g\|_{L^2},\\
%\end{align}
%Note that $\langle W^\pm\rangle^{\nu} \wt{p}$ also vanishes on $\Gamma$,
%by Lemma \ref{lemH1}, \eqref{H5} and \eqref{H6}, there holds
%\begin{align}
\label{HH7}
\| \langle W^\pm\rangle^{\nu} g\|_{L^2}
&\lesssim \| \langle W^\pm\rangle^{\nu} \p_2^\psi g\|_{L^2}
,\\[-4mm]\nonumber\\
\label{HH8}
\| \langle W^\pm\rangle^{\nu} \langle W^\mp\rangle^{\mu} g\|_{L^2}
&\lesssim \| \langle W^\pm\rangle^{\nu}
\langle W^\mp\rangle^{\mu}\p_2^\psi g\|_{L^2}
\\[-4mm]\nonumber\\\nonumber
&\quad
+\| \langle W^\pm\rangle^{\nu-1} \langle W^\mp\rangle^{\mu} g\|_{L^2}
+\| \langle W^\pm\rangle^{\nu} \langle W^\mp\rangle^{\mu-1} g\|_{L^2}.
\end{align}
Furthermore, if $g(x)\in \dot{H}^1(\Om_L) \cap \dot{H}^2(\Om_L)$, there holds
\begin{align}\label{HH9}
&\| g\|_{L^\infty}\lesssim \| \p_2 g\|^{\f12}_{L^2} \| \p_2\p_1 g\|^{\f12}_{L^2}.
\end{align}
\end{lem}
\begin{proof}
We write
\begin{align*}
&g(x)=g(x)-g(x_1,x_2=0)
=-\int^0_{x_2} \p_{\bar{x}_2} g(x_1,\bar{x}_2) \d\bar{x}_2\\
&\leq |x_2|^{\f12}
\Big(\int^0_{-1} |\p_{\bar{x}_2} g(x_1,\bar{x}_2)|^2 \d\bar{x}_2\Big)^{\f12}.
\end{align*}
Further taking $L^2(\Om_L)$ norm of the above expression yields \eqref{HH5}.

Next, note $\p_2^\psi=J^{-1}\p_2$ and $\f14\delta\leq J\leq 2-\f14\delta$,
then the estimates \eqref{HH6}, \eqref{HH7}, \eqref{HH8} follow from
\eqref{HH5}.
% and the assumption
%$\| \nabla\psi \|_{L^\infty}\leq C \epsilon $.

%\begin{align*}
%\| g\|_{L^2}\leq  \| \p_2 g\|_{L^2}.
%\end{align*}
Finally, we calculate
\begin{align*}
|g(x)|\leq
\Big(\int^0_{-1} |\p_{\bar{x_2}} g(x_1,\bar{x}_2)|^2 \d\bar{x}_2\Big)^{\f12}
\leq   \| \p_2 g\|^{\f12}_{L^2} \| \p_1\p_2 g\|^{\f12}_{L^2}.
\end{align*}
This yields \eqref{HH9}.
\end{proof}

\subsection{Weighted estimate of the pressure}
\begin{lem}\label{lemPreL}
Let $s\geq 3$ be an integer, $1/2< \mu \leq 3/4$. Under the assumption \eqref{F1}.
For the pressure satisfying \eqref{H1}-\eqref{H2}, there holds
\begin{align*}
&\| \langle W^\pm\rangle^{2\mu}
\langle W^\mp\rangle^{\mu}  \wt{p}\|^2_{L^2(\Om_L)}
+\sum_{|\alpha|\leq s}\| \langle W^\pm\rangle^{2\mu}
\langle W^\mp\rangle^{\mu} \nabla^\psi\p^\alpha \wt{p}\|^2_{L^2(\Om_L)}\\
&\lesssim %\sum_{|\alpha|\leq s-1,\ \pm} \| \langle W^\pm\rangle^{2\mu}  \p^\alpha \nabla^\psi\wt{\Lambda}_\pm\|_{L^2(\Om_L)}
%      \sum_{|\alpha|\leq s-1,\ \mp} \Big\| \f{\langle W^\mp\rangle^{2\mu}  \p^\alpha \nabla^\psi\wt{\Lambda}_\mp}{\langle W^\pm\rangle^{\mu}}\Big\|_{L^2(\Om_L)}\\
\mathcal{E}_s \mathcal{G}_s
%&\qquad
\cdot
\mathcal{F}(1+\| f \|_{L^\infty(\BR)}+\| f \|_{\dot{H}^1(\BR)}+\| f \|_{\dot{H}^{s+\f12}(\BR)}),
\end{align*}
where $\mathcal{F}(\cdot)$ is some non-decreasing function, $\mathcal{E}_s,\ \mathcal{G}_s$ denote
 %one dimensional polynomial.
\begin{align*}
\mathcal{E}_s &=\sum_{|\alpha|\leq s-1,\ |\beta|\leq 1,\ \pm} \| \langle W^\pm\rangle^{2\mu}  \p^\alpha (\nabla^\psi)^\beta \wt{\Lambda}_\pm\|_{L^2(\Om_L)}^2,\\
\mathcal{G}_s &=\sum_{|\alpha|\leq s-1,\ |\beta|\leq 1,\ \mp} \Big\| \f{\langle W^\mp\rangle^{2\mu}  \p^\alpha (\nabla^\psi)^\beta \wt{\Lambda}_\mp}{\langle W^\pm\rangle^{\mu}}\Big\|_{L^2(\Om_L)}^2.
\end{align*}
\end{lem}
%An immediate consequence of the above lemma is the weighted estimate of pressure in Eulerian coordinates.
From the above weighted estimate of pressure in the fixed domain $\Om_L$,
we obtain the corresponding weighted estimate of pressure in Eulerian coordinates.
\begin{lem}\label{lemPre}
Let $s\geq 3$ be an integer, $1/2< \mu \leq 3/4$. Under the assumption \eqref{F1}.
For the pressure satisfying \eqref{B14},
there holds
\begin{align*}
&\sum_{|\alpha|\leq s+1}\| \langle w^\pm\rangle^{2\mu}
\langle w^\mp\rangle^{\mu}  \p^\alpha p\|_{L^2(\Om_f)}^2\\
%&\leq \sum_{|\alpha|\leq s,\ \pm} \| \langle w^\pm\rangle^{2\mu}  \p^\alpha  \Lambda_\pm\|_{L^2(\Om_f)}
%      \sum_{|\alpha|\leq s,\ \mp} \Big\| \f{\langle w^\mp\rangle^{2\mu}  \p^\alpha  \Lambda_\mp}{\langle w^\pm\rangle^{\mu}}\Big\|_{L^2(\Om_f)}\\
%&\qquad
&\lesssim
E_s^b G_s^b
\cdot \mathcal{F}(1+\| f \|_{L^\infty(\BR)}+\| f \|_{\dot{H}^1(\BR)}+\| f \|_{\dot{H}^{s+\f12}(\BR)}),
\end{align*}
where $\mathcal{F}(\cdot)$ is some non-decreasing function.
\end{lem}

We postpone the proof of Lemma \ref{lemPreL} and first
 prove Lemma \ref{lemPre} assuming Lemma \ref{lemPreL}.
%Before the proof of Lemma \ref{}
Before that, we show the weighted norm relation between $\nabla$ and $\nabla^\psi$.
\begin{lem}\label{lemCo}
Let $\mu\in\BR $, $\nu\in\BR$, $s\geq 1$. Under the assumption \eqref{F1}, there hold
\begin{align*}
%&\sum_{2\leq |\alpha|+|\beta|\leq s,|\alpha|\geq 1,|\beta|\geq1}
%\| [\p^\alpha, (\nabla^\psi)^\beta] g \|_{L^2}\\
%&\leq
%\sum_{1 \leq |\alpha|\leq s} \| (\p^\psi)^\alpha  g \|_{L^2}
% \cdot \mathcal{F}(1+\| f \|_{L^\infty(\BR)}+\| f \|_{\dot{H}^1(\BR)}+\| f \|_{\dot{H}^{s-\f12}(\BR)}),\\
%&\sum_{2\leq |\alpha|+|\beta|\leq s,|\alpha|\geq 1,|\beta|\geq1}
%\| [\p^\alpha, (\nabla^\psi)^\beta] g \|_{L^2}\\
%&\leq \sum_{1 \leq |\alpha|\leq s} \| \p^\alpha  g \|_{L^2}
% \cdot \mathcal{F}(1+\| f \|_{L^\infty(\BR)}+\| f \|_{\dot{H}^1(\BR)}+\| f \|_{\dot{H}^{s-\f12}(\BR)}),\\
& \sum_{|\alpha|\leq s}\|  \langle W^\pm\rangle^{\mu}
\langle W^\mp\rangle^{\nu} \p^\alpha  g \|_{L^2}\\
&\leq \sum_{ |\alpha|\leq s}\| \langle W^\pm\rangle^{\mu}
\langle W^\mp\rangle^{\nu} (\p^\psi)^\alpha  g \|_{L^2}
 \cdot \mathcal{F}(1+\| f \|_{L^\infty(\BR)}+\| f \|_{\dot{H}^1(\BR)}+\| f \|_{\dot{H}^{s-\f12}(\BR)}),\\
& \sum_{|\alpha|\leq s}\| \langle W^\pm\rangle^{\mu}
\langle W^\mp\rangle^{\nu} (\p^\psi)^\alpha    g \|_{L^2}\\
&\leq \sum_{|\alpha|\leq s}\| \langle W^\pm\rangle^{\mu}
\langle W^\mp\rangle^{\nu} \p^\alpha g \|_{L^2}
 \cdot \mathcal{F}(1+\| f \|_{L^\infty(\BR)}+\| f \|_{\dot{H}^1(\BR)}+\| f \|_{\dot{H}^{s-\f12}(\BR)}).
 \end{align*}
\end{lem}
\begin{proof}

Firstly, we write down the expression of $\nabla^\psi$:
\begin{align}\label{FF0}
%\p_i^\psi
%&=\p_i-\p_i\psi \p_2^\psi,\, i=0,1,
\p_1^\psi
&=\p_1-\p_1\psi \p_2^\psi,
\quad  \p_2^\psi=\f{1}{1+\p_2\psi}\p_2=J^{-1}\p_2 .
%,\\\nonumber
%[\p_1^\psi,\p_2^\psi]&=0.
\end{align}
%Then

{\bf{First derivative estimate:}}

By \eqref{FF0}, we calculate
\begin{align}
\label{FF1}
&\|  \langle W^\pm\rangle^{\mu}
\langle W^\mp\rangle^{\nu} \nabla  g \|_{L^2}
\lesssim (1+\| \nabla\psi\|_{L^\infty}) \|\langle W^\pm\rangle^{\mu}
\langle W^\mp\rangle^{\nu} \nabla^\psi g\|_{L^2},\\
\label{FF2}
&\|  \langle W^\pm\rangle^{\mu}
\langle W^\mp\rangle^{\nu} \nabla^\psi  g \|_{L^2}
\lesssim (1+\| \nabla\psi\|_{L^\infty}) \|\langle W^\pm\rangle^{\mu}
\langle W^\mp\rangle^{\nu} \nabla g\|_{L^2}.
\end{align}

{\bf{Second derivative and commutator estimate:}}

We calculate
\begin{align}
\label{FF3}
[\nabla,\p^\psi_2]
&=-\nabla\p_2\psi J^{-2}\p_2= -\nabla\p_2\psi J^{-1}\p_2^\psi ,\\
\label{FF4}
[\nabla,\p^\psi_1]
&=-\nabla \p_1\psi\p_2^\psi-\p_1\psi [\nabla,\p^\psi_2]
=(-\nabla \p_1\psi +\p_1\psi \nabla\p_2\psi J^{-1})\p_2^\psi.
\end{align}
Thus
%\begin{align*}
%\| [\nabla,\nabla^\psi] g \|_{L^2}
%\leq (1+\| \p_1\psi\|_{L^\infty}) \| \nabla^2\psi\|_{L^\infty}\|\p_2^\psi g\|_{L^2}
%\end{align*}
%Or
%\begin{align*}
%\| [\nabla,\nabla^\psi] g \|_{L^2}
%\leq (1+\| \p_1\psi\|_{L^\infty}) \| \nabla^2\psi\|_{L^2}\|\p_2^\psi g\|_{L^\infty} .
%\end{align*}
%Or
\begin{align}\label{FF5}
&\| \langle W^\pm\rangle^{\mu}
\langle W^\mp\rangle^{\nu} [\nabla,\nabla^\psi] g \|_{L^2}\\\nonumber
&\lesssim (1+\| \nabla \psi\|_{L^\infty}) \| \nabla^2\psi\|_{L^4}
\|\langle W^\pm\rangle^{\mu}\langle W^\mp\rangle^{\nu}\p_2^\psi g\|_{L^4} \\\nonumber
%&\lesssim (1+\| \nabla \psi\|_{L^\infty})^2\big( \| \nabla^2\psi\|_{L^2}+\| \nabla^3\psi\|_{L^2} \big)
%\big( \|\langle W^\pm\rangle^{\mu}\langle W^\mp\rangle^{\nu} \p_2^\psi g\|_{L^2}
%+ \|\langle W^\pm\rangle^{\mu} \langle W^\mp\rangle^{\nu} \nabla^\psi \p_2^\psi g\|_{L^2} \big)\\\nonumber
&\leq \mathcal{F}(1+\| \nabla\psi\|_{L^\infty}+\| \nabla^2\psi\|_{L^2}+\| \nabla^3\psi\|_{L^2}  )
\sum_{1\leq |\alpha|\leq 2} \| \langle W^\pm\rangle^{\mu}
\langle W^\mp\rangle^{\nu} (\nabla^\psi)^\alpha g\|_{L^2}.
\end{align}
Consequently, by \eqref{FF1} and \eqref{FF5}, we obtain
\begin{align}\label{FF6}
&\| \langle W^\pm\rangle^{\mu}\langle W^\mp\rangle^{\nu} \nabla^2  g \|_{L^2} \\\nonumber
&\lesssim (1+\| \nabla\psi\|_{L^\infty})
\|\langle W^\pm\rangle^{\mu}\langle W^\mp\rangle^{\nu} \nabla^\psi\nabla g\|_{L^2}\\\nonumber
&\leq (1+\| \nabla\psi\|_{L^\infty})
\big(\|\langle W^\pm\rangle^{\mu}\langle W^\mp\rangle^{\nu} \nabla\nabla^\psi g\|_{L^2}
+\|\langle W^\pm\rangle^{\mu}\langle W^\mp\rangle^{\nu} [\nabla,\nabla^\psi] g\|_{L^2} \big)\\\nonumber
%&\lesssim (1+\| \nabla\psi\|_{L^\infty})^2
%\big(\| (\nabla^\psi)^2 g\|_{L^2}%(1+\| \nabla\psi\|_{L^\infty})
%+\|[\nabla,\nabla^\psi] g\|_{L^2} \big)\\\nonumber
&\leq \mathcal{F}(1+\| \nabla\psi\|_{L^\infty}+\| \nabla^2\psi\|_{L^2}+\| \nabla^3\psi\|_{L^2}  )
\sum_{1\leq |\alpha|\leq 2} \| \langle W^\pm\rangle^{\mu}
\langle W^\mp\rangle^{\nu} (\nabla^\psi)^\alpha g\|_{L^2}.
\end{align}
On the other hand, by \eqref{FF3} and \eqref{FF4}, we calculate
\begin{align} \label{FF7}
&\| \langle W^\pm\rangle^{\mu}\langle W^\mp\rangle^{\nu}[\nabla,\nabla^\psi] g \|_{L^2}\\\nonumber
&\lesssim (1+\| \nabla \psi\|_{L^\infty}) \| \nabla^2\psi\|_{L^4}
\|\langle W^\pm\rangle^{\mu}\langle W^\mp\rangle^{\nu}\p_2 g\|_{L^4} \\\nonumber
%&\leq (1+\| \nabla \psi\|_{L^\infty})\big( \| \nabla^2\psi\|_{L^2}+\| \nabla^3\psi\|_{L^2} \big)
%\big( \|\nabla g\|_{L^2}+ \|\nabla^2 g\|_{L^2} \big)\\\nonumber
&\leq \mathcal{F}(1+\| \nabla\psi\|_{L^\infty}+\| \nabla^2\psi\|_{L^2}+\| \nabla^3\psi\|_{L^2}  )
\sum_{1\leq |\alpha|\leq 2} \| \langle W^\pm\rangle^{\mu}\langle W^\mp\rangle^{\nu} \nabla^\alpha g\|_{L^2}.
\end{align}
Consequently, by \eqref{FF2} and \eqref{FF7}, we deduce
\begin{align}\label{FF8}
&\| \langle W^\pm\rangle^{\mu}\langle W^\mp\rangle^{\nu} (\nabla^\psi)^2  g \|_{L^2}\\\nonumber
&\lesssim (1+\| \nabla\psi\|_{L^\infty})
\|\langle W^\pm\rangle^{\mu}\langle W^\mp\rangle^{\nu}\nabla\nabla^\psi g\|_{L^2}\\\nonumber
&\leq (1+\| \nabla\psi\|_{L^\infty})
\big(\|\langle W^\pm\rangle^{\mu}\langle W^\mp\rangle^{\nu}\nabla^\psi\nabla g\|_{L^2}
+\|\langle W^\pm\rangle^{\mu}\langle W^\mp\rangle^{\nu} [\nabla,\nabla^\psi] g\|_{L^2} \big)\\\nonumber
%&\lesssim (1+\| \nabla\psi\|_{L^\infty})^2
%\big(\| \nabla^2 g\|_{L^2}%(1+\| \nabla\psi\|_{L^\infty})
%+\|[\nabla,\nabla^\psi] g\|_{L^2} \big)\\\nonumber
&\leq \mathcal{F}(1+\| \nabla\psi\|_{L^\infty}+\| \nabla^2\psi\|_{L^2}+\| \nabla^3\psi\|_{L^2}  )
\sum_{1\leq |\alpha|\leq 2} \| \langle W^\pm\rangle^{\mu}\langle W^\mp\rangle^{\nu} \nabla^\alpha g\|_{L^2}.
\end{align}
{\bf{Third derivative and commutator estimate:}}

By \eqref{FF0}, we calculate
\begin{align}
\label{FF16}
[\nabla^2,\p^\psi_2]
&=\nabla [\nabla,\p^\psi_2] + [\nabla,\p^\psi_2] \nabla
=-\nabla(\nabla\p_2\psi J^{-2}\p_2)-\nabla\p_2\psi J^{-2}\p_2\nabla \\\nonumber
%&=-\nabla^2\p_2\psi J^{-2}\p_2
%+2\nabla\p_2\psi J^{-3}\nabla\p_2\psi\p_2 -\nabla\p_2\psi J^{-2}\p_2\nabla \\\nonumber
%&=\big(-\nabla^2\p_2\psi J^{-1}
%+2(\nabla\p_2\psi)^2 J^{-2}\big)\p_2^\psi -\nabla\p_2\psi J^{-1}\p_2^\psi\nabla\\\nonumber
%&=\big(-\nabla^2\p_2\psi J^{-1}
%+2(\nabla\p_2\psi)^2 J^{-2} \big)\p_2^\psi -\nabla\p_2\psi J^{-1}
%\big(\nabla \p_2^\psi +[ \p_2^\psi,\nabla]\big)\\\nonumber
&=\big(-\nabla^2\p_2\psi J^{-1}
+2(\nabla\p_2\psi)^2 J^{-2} \big)\p_2^\psi \\\nonumber
&\quad-\nabla\p_2\psi J^{-1}\big(\nabla \p_2^\psi +\nabla\p_2\psi J^{-1}\p_2^\psi \big),\\
\label{FF17}
[\nabla^2,\p^\psi_1]
&=\nabla [\nabla,\p^\psi_1] + [\nabla,\p^\psi_1] \nabla \\\nonumber
&=\nabla(-\nabla \p_1\psi +\p_1\psi \nabla\p_2\psi J^{-1})\p_2^\psi
+(-\nabla \p_1\psi +\p_1\psi \nabla\p_2\psi J^{-1})\nabla \p_2^\psi \\\nonumber
&\quad +(-\nabla \p_1\psi +\p_1\psi \nabla\p_2\psi J^{-1}) \p_2^\psi\nabla \\\nonumber
%&=\nabla(-\nabla \p_1\psi +\p_1\psi \nabla\p_2\psi J^{-1})\p_2^\psi
%+2(-\nabla \p_1\psi +\p_1\psi \nabla\p_2\psi J^{-1})\nabla \p_2^\psi \\\nonumber
%&\quad +(-\nabla \p_1\psi +\p_1\psi \nabla\p_2\psi J^{-1}) [\p_2^\psi,\nabla] \\\nonumber
%&=\nabla(-\nabla \p_1\psi +\p_1\psi \nabla\p_2\psi J^{-1})\p_2^\psi
%+2(-\nabla \p_1\psi +\p_1\psi \nabla\p_2\psi J^{-1})\nabla \p_2^\psi \\\nonumber
%&\quad +(-\nabla \p_1\psi +\p_1\psi \nabla\p_2\psi J^{-1}) \nabla\p_2\psi J^{-1}\p_2^\psi \\\nonumber
&=\big(-\nabla^2 \p_1\psi
+\nabla \p_1\psi \nabla\p_2\psi J^{-1}
+\p_1\psi \nabla^2\p_2\psi J^{-1}
-\p_1\psi (\nabla\p_2\psi)^2 J^{-2}
\big)\p_2^\psi\\\nonumber
&\quad+2(-\nabla \p_1\psi +\p_1\psi \nabla\p_2\psi J^{-1})\nabla \p_2^\psi
 +(-\nabla \p_1\psi +\p_1\psi \nabla\p_2\psi J^{-1}) \nabla\p_2\psi J^{-1}\p_2^\psi.
%&=\nabla(-\nabla \p_1\psi +\p_1\psi \nabla\p_2\psi J^{-1})\p_2^\psi
%+2(-\nabla \p_1\psi +\p_1\psi \nabla\p_2\psi J^{-1})\p_2^\psi\nabla\\
%&\quad+(-\nabla \p_1\psi +\p_1\psi \nabla\p_2\psi J^{-1})[\nabla,\p_2^\psi],\\
\end{align}
Thus by the first and the second order estimate \eqref{FF1}, \eqref{FF6}, we derive that
\begin{align}\label{FF19}
&\| \langle W^\pm\rangle^{\mu}\langle W^\mp\rangle^{\nu} [\nabla^2,\nabla^\psi] g \|_{L^2} \\\nonumber
&\lesssim (1+\| \nabla\psi\|_{L^\infty})
\Big((\| \nabla^3\psi\|_{L^2}+\| \nabla^2\psi\|^2_{L^4})
\|\langle W^\pm\rangle^{\mu}\langle W^\mp\rangle^{\nu} \nabla^\psi g\|_{L^\infty} \\\nonumber
&\quad+ \| \nabla^2\psi\|_{L^4} \| \langle W^\pm\rangle^{\mu}\langle W^\mp\rangle^{\nu} \nabla\nabla^\psi g\|_{L^4}
\Big) \\\nonumber
&\leq (1+\| \nabla\psi\|_{L^\infty})
\Big((\| \nabla^3\psi\|_{L^2}+\| \nabla^3\psi\|_{L^2} \| \nabla^2\psi\|_{L^2})
\sum_{|\alpha|\leq 2}\|  \langle W^\pm\rangle^{\mu}\langle W^\mp\rangle^{\nu} \nabla^\alpha \nabla^\psi g\|_{L^2} \\\nonumber
&\quad+ \|\nabla^3\psi\|_{L^2}^{\f12} \|\nabla^2\psi\|_{L^2}^{\f12}
\sum_{|\alpha|\leq 1}\|  \langle W^\pm\rangle^{\mu}\langle W^\mp\rangle^{\nu} \nabla^\alpha \nabla\nabla^\psi g\|_{L^2}
\Big) \\\nonumber
%\end{align}
%By the first and the second order estimate \eqref{FF1}, \eqref{FF4} and \eqref{FF5},
%\eqref{FF8} yields
%\begin{align}\label{}
%&\| [\nabla^2,\nabla^\psi] g \|_{L^2} \\\nonumber
&\leq \mathcal{F}(1+\| \nabla\psi\|_{L^\infty}+\| \nabla^2\psi\|_{L^2}+\| \nabla^3\psi\|_{L^2}  )
\cdot \sum_{1\leq |\alpha| \leq 3} \| \langle W^\pm\rangle^{\mu}\langle W^\mp\rangle^{\nu} (\nabla^\psi)^\alpha g\|_{L^2}.
%+\| \nabla^\psi g\|_{L^2}\big).
\end{align}
Consequently, by \eqref{FF1}, \eqref{FF6} and \eqref{FF19}, we obtain
\begin{align*} %\label{FF20}
&\| \langle W^\pm\rangle^{\mu}\langle W^\mp\rangle^{\nu} \nabla^3 g \|_{L^2} \\\nonumber
&\lesssim (1+\| \nabla \psi\|_{L^\infty})
\| \langle W^\pm\rangle^{\mu}\langle W^\mp\rangle^{\nu} \nabla^\psi \nabla^2 g \|_{L^2} \\\nonumber
&\leq (1+\| \nabla\psi\|_{L^\infty})
\Big( \| \langle W^\pm\rangle^{\mu}\langle W^\mp\rangle^{\nu} \nabla^2\nabla^\psi g \|_{L^2}
+\| \langle W^\pm\rangle^{\mu}\langle W^\mp\rangle^{\nu} [\nabla^2,\nabla^\psi] g \|_{L^2}
\Big) \\\nonumber
&\leq \mathcal{F}(1+\| \nabla\psi\|_{L^\infty}+\| \nabla^2\psi\|_{L^2}+\| \nabla^3\psi\|_{L^2}  )
\cdot \sum_{1\leq |\alpha| \leq 3} \| \langle W^\pm\rangle^{\mu}\langle W^\mp\rangle^{\nu} (\nabla^\psi)^\alpha g\|_{L^2}.
\end{align*}
%On the other hand,
%\begin{align}\label{}
%&\| \langle W^\pm\rangle^{\mu}\langle W^\mp\rangle^{\nu} [\nabla^2,\nabla^\psi] g \|_{L^2} \\\nonumber
%&\lesssim (1+\| \nabla\psi\|_{L^\infty})
%\Big((\| \nabla^3\psi\|_{L^2}+\| \nabla^2\psi\|^2_{L^4})
%\| \langle W^\pm\rangle^{\mu}\langle W^\mp\rangle^{\nu} \nabla g\|_{L^\infty} \\\nonumber
%&\quad+ \| \nabla^2\psi\|_{L^4} \| \langle W^\pm\rangle^{\mu}\langle W^\mp\rangle^{\nu} \nabla^2 g\|_{L^4}
%\Big) \\\nonumber
%%&\leq (1+\| \nabla\psi\|_{L^\infty})
%%\Big((\| \nabla^3\psi\|_{L^2}+\| \nabla^3\psi\|_{L^2} \| \nabla^2\psi\|_{L^2})
%%\|\nabla^\psi g\|_{L^2}^{\f12}\|\nabla^2\nabla^\psi g\|_{L^2}^{\f12} \\\nonumber
%%&\quad+ \|\nabla^3\psi\|_{L^2}^{\f12} \|\nabla^2\psi\|_{L^2}^{\f12}
%%\| \nabla^2\nabla^\psi g\|_{L^2}^{\f12}\| \nabla\nabla^\psi g\|_{L^2}^{\f12}
%%\Big).
%&\leq \mathcal{F}(1+\| \nabla\psi\|_{L^\infty}+\| \nabla^2\psi\|_{L^2}+\| \nabla^3\psi\|_{L^2}  )
%\cdot \sum_{1\leq |\alpha| \leq 3} \| \langle W^\pm\rangle^{\mu}\langle W^\mp\rangle^{\nu} \nabla^\alpha g\|_{L^2}.
%\end{align}
On the other side, by \eqref{FF0}, \eqref{FF3}, \eqref{FF4}, \eqref{FF16}, \eqref{FF17}, we deduce that
\begin{align*} %\label{FF19}
&\| \langle W^\pm\rangle^{\mu}\langle W^\mp\rangle^{\nu} [\nabla^2,\nabla^\psi] g \|_{L^2} \\\nonumber
&\lesssim (1+\| \nabla\psi\|_{L^\infty})
\Big((\| \nabla^3\psi\|_{L^2}+\| \nabla^2\psi\|^2_{L^4})
\|\langle W^\pm\rangle^{\mu}\langle W^\mp\rangle^{\nu} \nabla^\psi g\|_{L^\infty} \\\nonumber
&\quad+ \| \nabla^2\psi\|_{L^4}
\big(\| \langle W^\pm\rangle^{\mu}\langle W^\mp\rangle^{\nu} \nabla^\psi \nabla g\|_{L^4}
+\| \langle W^\pm\rangle^{\mu}\langle W^\mp\rangle^{\nu} [\nabla,\nabla^\psi] g\|_{L^4}
\big)
\Big) \\\nonumber
&\leq \mathcal{F}(1+\| \nabla\psi\|_{L^\infty}+\| \nabla^2\psi\|_{L^2}+\| \nabla^3\psi\|_{L^2}  )
\cdot \sum_{1\leq |\alpha| \leq 3} \| \langle W^\pm\rangle^{\mu}\langle W^\mp\rangle^{\nu} \nabla^\alpha g\|_{L^2}.
\end{align*}
Consequently, by \eqref{FF2} and \eqref{FF8}, we derive that
\begin{align*} %\label{FF21}
&\| \langle W^\pm\rangle^{\mu}\langle W^\mp\rangle^{\nu} (\nabla^\psi)^3 g \|_{L^2} \\\nonumber
%&\lesssim (1+\| \nabla \psi\|_{L^\infty})\| \nabla^\psi \nabla^2 g \|_{L^2} \\\nonumber
&\leq \mathcal{F}(1+\| \nabla\psi\|_{L^\infty}+\| \nabla^2\psi\|_{L^2}+\| \nabla^3\psi\|_{L^2}  )\\\nonumber
&\qquad \cdot\Big( \| \langle W^\pm\rangle^{\mu}\langle W^\mp\rangle^{\nu} \nabla^\psi\nabla^2 g \|_{L^2}
+\| \langle W^\pm\rangle^{\mu}\langle W^\mp\rangle^{\nu} [\nabla^2,\nabla^\psi] g \|_{L^2}
\Big) \\\nonumber
&\leq \mathcal{F}(1+\| \nabla\psi\|_{L^\infty}+\| \nabla^2\psi\|_{L^2}+\| \nabla^3\psi\|_{L^2}  )
\cdot \sum_{1\leq |\alpha| \leq 3} \| \langle W^\pm\rangle^{\mu}\langle W^\mp\rangle^{\nu} \nabla^\alpha g\|_{L^2}.
\end{align*}
The remaining higher order derivative estimates are the same. The details are omitted.
This finishes the proof of the lemma.
\end{proof}
Now we show Lemma \ref{lemPre} by using Lemma \ref{lemPreL} and Lemma \ref{lemCo}.
\begin{proof}[Proof of Lemma \ref{lemPre}]
By the coordinate transform, it suffices to prove
\begin{align*}
&\sum_{0\leq |\alpha|\leq s+1}\| \langle W^\pm\rangle^{2\mu}
\langle W^\mp\rangle^{\mu} (\nabla^\psi)^\alpha \wt{p}\|^2_{L^2}\\
&\lesssim E_s^b G_s^b\cdot
\mathcal{F}(1+\| f \|_{L^\infty(\BR)}+\| f \|_{\dot{H}^1(\BR)}+\| f \|_{\dot{H}^{s+\f12}(\BR)}).
\end{align*}
%The case for $|\alpha|=0$ is trivial.
By Lemma \ref{lemPreL} and Lemma \ref{lemCo}, we deduce that
\begin{align*}
&\sum_{ |\alpha|\leq s+1}\| \langle W^\pm\rangle^{2\mu}\langle W^\mp\rangle^{\mu} (\nabla^\psi)^\alpha \wt{p}\|^2_{L^2}\\
&\leq \sum_{|\alpha|\leq s+1}\| \langle W^\pm\rangle^{2\mu}\langle W^\mp\rangle^{\mu} \p^\alpha\wt{p}\|^2_{L^2}
\cdot \mathcal{F}(1+\| f \|_{L^\infty(\BR)}+\| f \|_{\dot{H}^1(\BR)}+\| f \|_{\dot{H}^{s+\f12}(\BR)})\\
&\leq \big(\sum_{|\alpha|\leq s}\| \langle W^\pm\rangle^{2\mu}\langle W^\mp\rangle^{\mu}\nabla^\psi \p^\alpha\wt{p}\|^2_{L^2}
+\| \langle W^\pm\rangle^{2\mu}\langle W^\mp\rangle^{\mu}\wt{p}\|^2_{L^2}
\big)\\\nonumber
&\quad\cdot \mathcal{F}(1+\| f \|_{L^\infty(\BR)}+\| f \|_{\dot{H}^1(\BR)}+\| f \|_{\dot{H}^{s+\f12}(\BR)})\\
&\lesssim
\sum_{|\alpha|\leq s-1,\ |\beta|\leq 1,\ \pm} \| \langle W^\pm\rangle^{2\mu}  \p^\alpha (\nabla^\psi)^\beta \wt{\Lambda}_\pm\|_{L^2}^2
\cdot\sum_{|\alpha|\leq s-1,\ |\beta|\leq 1,\ \mp} \Big\| \f{\langle W^\mp\rangle^{2\mu}  \p^\alpha (\nabla^\psi)^\beta \wt{\Lambda}_\mp}{\langle W^\pm\rangle^{\mu}}\Big\|_{L^2}^2\\
&\qquad\quad\cdot \mathcal{F}(1+\| f \|_{L^\infty(\BR)}+\| f \|_{\dot{H}^1(\BR)}+\| f \|_{\dot{H}^{s+\f12}(\BR)})\\
&\lesssim \sum_{|\alpha|\leq s,\ \pm} \| \langle W^\pm\rangle^{2\mu}  (\nabla^\psi)^\alpha \wt{\Lambda}_\pm\|_{L^2}^2
\cdot     \sum_{|\beta|\leq s,\ \mp} \Big\| \f{\langle W^\mp\rangle^{2\mu}  (\nabla^\psi)^\beta \wt{\Lambda}_\mp}{\langle W^\pm\rangle^{\mu}}\Big\|_{L^2}^2\\
&\qquad\quad \cdot\mathcal{F}(1+\| f \|_{L^\infty(\BR)}+\| f \|_{\dot{H}^1(\BR)}+\| f \|_{\dot{H}^{s+\f12}(\BR)})\\
&\lesssim E_s^b G_s^b \cdot\mathcal{F}(1+\| f \|_{L^\infty(\BR)}+\| f \|_{\dot{H}^1(\BR)}+\| f \|_{\dot{H}^{s+\f12}(\BR)}).
\end{align*}
Thus Lemma \ref{lemPre} is proved.
\end{proof}

\begin{proof}[Proof of Lemma \ref{lemPreL}]
~\\
\textbf{First derivative and zero-order estimate of pressure}:

\textbf{Step one} (first derivative estimate of pressure without weights):

%In the Lagrangian coordinates,
%the equations for pressure reads as follows:
%\begin{equation}
%\begin{cases}
%-a^\top\nabla\cdot \nabla^\psi \wt{p}
%=a^\top \nabla \cdot (\wt{\Lambda}_-\cdot \nabla^\psi \wt{\Lambda}_+)
%\, \,& {\rm in} \; \Omega_L
%\, ,\\
%\wt{p}\big|_{\Gamma} =0,\quad \p_2\wt{p}\big|_{\Gamma_-} =0.
%\end{cases}
%\end{equation}
Taking the $L^2(\Om_L,dx)$ inner product of \eqref{H4} with $\wt{p}$,
we obtain
\begin{align}\label{H8}
&-\int_{\Om_L} a^\top\nabla\cdot \nabla^\psi \wt{p} \cdot \wt{p}\, \dx\\\nonumber
&=\int_{\Om_L}
a^\top \nabla \cdot (\wt{\Lambda}_-\cdot \nabla^\psi \wt{\Lambda}_+)\cdot \wt{p}\, \dx\, .
\end{align}
For the first line of \eqref{H8},
by Piola's identity \eqref{H3} and the boundary condition for pressure \eqref{H2},
employing integration by parts, we deduce that
\begin{align}\label{H9}
&-\int_{\Om_L} a^\top\nabla\cdot \nabla^\psi \wt{p} \cdot \wt{p}\, \dx\\\nonumber
&=\int_{\Om_L} |\nabla^\psi \wt{p}|^2 J \dx
+\int_{\Om_L} (\nabla \cdot  a^\top)\cdot \nabla^\psi \wt{p}\cdot \wt{p}\, \dx\\\nonumber
&\quad -\int_{\Gamma} a_{2i}\p_i^\psi \wt{p}\cdot \wt{p}\ \dx_1
+\int_{\Gamma_-} a_{2i}\p_i^\psi \wt{p}\cdot \wt{p}\ \dx_1 \\\nonumber
&
%=\int_{\Om_L} |\nabla^\psi \wt{p}|^2 \dx
%+\int_{\Gamma} (1+|\p_1 f|^2)\p_2 \wt{p}\cdot \wt{p}\ \dx_1
%-\int_{\Gamma_-} \p_2 \wt{p}\cdot \wt{p}\ \dx_1
=\int_{\Om_L} |\nabla^\psi \wt{p}|^2 \dx .
\end{align}
For the second line of \eqref{H8},
similarly by Piola's identity \eqref{H3}, the boundary condition for pressure \eqref{H2} and by integration by parts, we derive that
\begin{align}\label{H10}
&\int_{\Om_L}
a^\top \nabla \cdot (\wt{\Lambda}_-\cdot \nabla^\psi \wt{\Lambda}_+)\cdot \wt{p}\, \dx\\\nonumber
&=-\int_{\Om_L}
 (\wt{\Lambda}_-\cdot \nabla^\psi \wt{\Lambda}_+)\cdot \nabla^\psi \wt{p}\, J\, \dx\\\nonumber
&\quad -\int_{\Om_L}
(\nabla \cdot a^\top)\cdot (\wt{\Lambda}_-\cdot \nabla^\psi \wt{\Lambda}_+)\cdot \wt{p}\, \dx\\\nonumber
&\quad +\int_{\Gamma}
a_{2,\cdot}\cdot (\wt{\Lambda}_-\cdot \nabla^\psi \wt{\Lambda}_+)\cdot \wt{p}\, \dx_1\\\nonumber
&\quad -\int_{\Gamma_-}
a_{2,\cdot}\cdot (\wt{\Lambda}_-\cdot \nabla^\psi \wt{\Lambda}_+)\cdot \wt{p}\, \dx_1\\\nonumber
&=-\int_{\Om_L}
 \f{\wt{\Lambda}_-}{\langle W^+\rangle^{\mu}}
 \cdot \langle W^+\rangle^{\mu} \nabla^\psi \wt{\Lambda}_+\cdot \nabla^\psi \wt{p}\, J\, \dx \\\nonumber
&\leq \| \nabla^\psi \wt{p} \|_{L^2} \| \langle W^+\rangle^{\mu} \nabla^\psi \wt{\Lambda}_+ \|_{L^4}
\Big\|   \f{\wt{\Lambda}_-}{\langle W^+\rangle^{\mu}} \Big\|_{L^4} \\\nonumber
%&\leq \| \nabla^\psi \wt{p} \|_{L^2}
%\sum_{|\alpha|\leq 1}\| \langle W^+\rangle^{\mu} \nabla^\alpha\nabla^\psi \wt{\Lambda}_+ \|_{L^2}
%\sum_{|\alpha|\leq 1}\Big\|   \f{\p^\alpha \wt{\Lambda}_-}{\langle W^+\rangle^{\mu}} \Big\|_{L^2} \\\nonumber
&\leq \| \nabla^\psi \wt{p} \|_{L^2} \mathcal{E}_2^{\f12} \mathcal{G}_1^{\f12}  (1+\|\nabla\psi\|_{L^\infty}).
\end{align}
Combining \eqref{H9} and \eqref{H10},  \eqref{H8} yields
\begin{align}\label{H11}
\| \nabla^\psi \wt{p} \|^2_{L^2}
\lesssim \mathcal{E}_2  \mathcal{G}_1\cdot\mathcal{F}(1+\|\nabla\psi\|_{L^\infty}).
\end{align}

\noindent \textbf{Step two} (first derivative estimate of pressure with small power weights):

For $0\leq \nu\leq \mu$,
taking the $L^2(\Om_L,dx)$ inner product of \eqref{H4} with
$ \langle W^\pm\rangle^{4\nu}  \wt{p}$,
we obtain
\begin{align}\label{H12}
&-\int_{\Om_L} a^\top\nabla\cdot \nabla^\psi \wt{p} \cdot \wt{p}
\,\langle W^\pm\rangle^{4\nu} \, \dx\\\nonumber
&= \int_{\Om_L}
a^\top \nabla \cdot (\wt{\Lambda}_-\cdot \nabla^\psi \wt{\Lambda}_+)\cdot \wt{p}
\,\langle W^\pm\rangle^{4\nu} \, \dx.
\end{align}
For the first line of \eqref{H12},
by Piola's identity \eqref{H3}, the boundary condition for pressure % \eqref{H2}
and integration by parts,
we calculate that
\begin{align}\label{H13}
&-\int_{\Om_L} a^\top\nabla\cdot \nabla^\psi \wt{p} \cdot \wt{p}
\,\langle W^\pm\rangle^{4\nu} \, \dx\\\nonumber
%&=\int_{\Om_L} |\nabla^\psi \wt{p}|^2 \,\langle W^\pm\rangle^{4\nu} \, J \dx
%+\int_{\Om_L} (\nabla \cdot  a^\top)\cdot \nabla^\psi \wt{p}\cdot \wt{p}\,\langle W^\pm\rangle^{4\nu} \, \dx\\\nonumber
%&\quad+\int_{\Om_L} (\nabla \langle W^\pm\rangle^{4\nu} \cdot  a^\top)\cdot \nabla^\psi \wt{p}\cdot \wt{p} \, \dx\\\nonumber
%&\quad -\int_{\Gamma} a_{2i}\p_i^\psi \wt{p}\cdot \wt{p}
%\,\langle \Wu^\pm\rangle^{4\nu} \, \dx_1
%+\int_{\Gamma_-} a_{2i}\p_i^\psi \wt{p}\cdot \wt{p}
%\,\langle \Wo^\pm\rangle^{4\nu} \, \dx_1 \\\nonumber
&=\int_{\Om_L} |\nabla^\psi \wt{p}|^2 \,\langle W^\pm\rangle^{4\nu} \, J \dx
+\int_{\Om_L} (\nabla \langle W^\pm\rangle^{4\nu} \cdot  a^\top)\cdot \nabla^\psi \wt{p}\cdot \wt{p} \, \dx\, \\\nonumber
%\end{align}
%For the last term of \eqref{H13}, we have
%\begin{align}
%&\int_{\Om_L} (\nabla \langle W^\pm\rangle^{4\nu} \cdot  a^\top)\cdot \nabla^\psi \wt{p}\cdot \wt{p} \, \dx\\\nonumber
&\geq
\int_{\Om_L} |\nabla^\psi \wt{p}|^2 \,\langle W^\pm\rangle^{4\nu} \, J \dx \\\nonumber
&\qquad -C \| \langle W^\pm\rangle^{2\nu} \nabla^\psi \wt{p}\|_{L^2}
\| \langle W^\pm\rangle^{2\nu-1}  \wt{p}\|_{L^2}
(1+\|\nabla \psi\|_{L^\infty})\,.
\end{align}
While for the second line of \eqref{H12}, by \eqref{F11},  Piola's identity \eqref{H3},
 the boundary condition \eqref{J21}, \eqref{H2} and the integration by parts,
we have
\begin{align}\label{H15}
&\int_{\Om_L}
a^\top \nabla \cdot (\wt{\Lambda}_-\cdot \nabla^\psi \wt{\Lambda}_+)\cdot \wt{p}
\,\langle W^\pm\rangle^{4\nu} \, \dx\\\nonumber
%&=-\int_{\Om_L}
% (\wt{\Lambda}_-\cdot \nabla^\psi \wt{\Lambda}_+)\cdot \nabla^\psi \wt{p}
% \,\langle W^\pm\rangle^{4\nu} J\, \dx\\\nonumber
%&\quad -\int_{\Om_L}
%(\nabla \cdot a^\top)\cdot (\wt{\Lambda}_-\cdot \nabla^\psi \wt{\Lambda}_+)\cdot \wt{p}
%\,\langle W^\pm\rangle^{4\nu} \, \dx\\\nonumber
%&\quad -\int_{\Om_L}
%(\nabla \langle W^\pm\rangle^{4\nu} \cdot a^\top)\cdot (\wt{\Lambda}_-\cdot \nabla^\psi \wt{\Lambda}_+)\cdot \wt{p}\, \dx\\\nonumber
%&\quad +\int_{\Gamma}
%a_{2,\cdot}\cdot (\wt{\Lambda}_-\cdot \nabla^\psi \wt{\Lambda}_+)\cdot \wt{p}
%\,\langle \Wu^\pm\rangle^{4\nu} \, \dx_1\\\nonumber
%&\quad -\int_{\Gamma_-}
%a_{2,\cdot}\cdot (\wt{\Lambda}_-\cdot \nabla^\psi \wt{\Lambda}_+)\cdot \wt{p}
%\,\langle \Wo^\pm\rangle^{4\nu} \, \dx_1\\\nonumber
&=-\int_{\Om_L}
 (\wt{\Lambda}_-\cdot \nabla^\psi \wt{\Lambda}_+)\cdot \nabla^\psi \wt{p}
 \langle W^\pm\rangle^{4\nu} J\, \dx\\\nonumber
&\quad -\int_{\Om_L}
(\nabla \langle W^\pm\rangle^{4\nu} \cdot a^\top)\cdot (\wt{\Lambda}_-\cdot \nabla^\psi \wt{\Lambda}_+)\cdot \wt{p}\, \dx \\\nonumber
%\end{align}
%%The above \eqref{H15} can be further bounded by
%\begin{align*}
%&=\int_{\Om_L}
% \langle W^\mp\rangle^{\mu} \wt{\Lambda}_-\cdot
% \f{\langle W^\pm\rangle^{2\nu}\nabla^\psi \wt{\Lambda}_+}{\langle W^\mp\rangle^{\mu}}
% \cdot \nabla^\psi \wt{p}\langle W^\pm\rangle^{2\nu} J\, \dx\\
%&\quad -\int_{\Om_L}
%(\nabla \langle W^\pm\rangle^{4\nu} \cdot a^\top)\cdot (\wt{\Lambda}_-\cdot \nabla^\psi \wt{\Lambda}_+)\cdot \wt{p}\, \dx \\
&\lesssim \| \langle W^\pm\rangle^{2\nu} \nabla^\psi \wt{p} \|_{L^2}
\| \langle W^\pm\rangle^{2\nu} \wt{\Lambda}_-\cdot \nabla^\psi \wt{\Lambda}_+\|_{L^2}
\\[-4mm]\nonumber \\\nonumber
&\quad +\| \langle W^\pm\rangle^{2\nu-1}  \wt{p} \|_{L^2}
\| \langle W^\pm\rangle^{2\nu} \wt{\Lambda}_-\cdot \nabla^\psi \wt{\Lambda}_+\|_{L^2}
(1+\|\nabla\psi\|_{L^\infty}) \\[-4mm]\nonumber\\\nonumber
&\lesssim \big(\| \langle W^\pm\rangle^{2\nu} \nabla^\psi \wt{p} \|_{L^2}
+\| \langle W^\pm\rangle^{2\nu-1} \wt{p} \|_{L^2} )
\mathcal{E}_2^{\f12}  \mathcal{G}_1^{\f12} (1+\|\nabla\psi\|_{L^\infty}).
\end{align}
Consequently, combining \eqref{H13}-\eqref{H15}, for sufficiently small $\epsilon$, \eqref{H12} yields
\begin{align}\label{H17}
&\int_{\Om_L} J \langle W^\pm\rangle^{4\nu}
\big|\nabla^\psi \wt{p}\big|^2
  \dx \\\nonumber
&\lesssim  \big(\mathcal{E}_2 \mathcal{G}_1
+\|  \langle W^\pm\rangle^{2\nu-1} \wt{p} \|^2_{L^2}\big)
\cdot \mathcal{F}(1+\|\nabla\psi\|_{L^\infty}).
\end{align}

For \eqref{H17}, set $\nu=\mu-\f12$. Note $\mu\leq \f34$, thus $\langle W^\pm\rangle^{2\nu-1}=\langle W^\pm\rangle^{2\mu-2}\lesssim 1$. Combined with Poincar\'e inequality Lemma \ref{lemHardy} and \eqref{H11}, \eqref{H17} yields
\begin{align}\label{H18}
&\int_{\Om_L} J \langle W^\pm\rangle^{4\mu-2}
\big| \nabla^\psi \wt{p}\big|^2\,  \dx
\lesssim  \mathcal{E}_2 \mathcal{G}_1 \cdot \mathcal{F}(1+\|\nabla\psi\|_{L^\infty}) .
\end{align}
Next for \eqref{H17}, set $\nu=\mu$. Combined with Lemma \ref{lemHardy},
%\eqref{H6}, \eqref{H7},
\eqref{H11} and \eqref{H18},  \eqref{H17} yields
\begin{align}\label{H19}
&\int_{\Om_L} J \langle W^\pm\rangle^{4\mu}
\big| \nabla^\psi \wt{p}\big|^2\,  \dx
\lesssim  \mathcal{E}_2 \mathcal{G}_1 \cdot \mathcal{F}(1+\|\nabla\psi\|_{L^\infty}).
\end{align}
Furthermore, applying  Lemma \ref{lemHardy} to \eqref{H19},
%\eqref{H6}, \eqref{H7},
by \eqref{H11} and \eqref{H19},
% combining \eqref{E1}, \eqref{E2}, \eqref{E4} and \eqref{E8},
we deduce that
\begin{align}\label{H20}
&\int_{\Om_L} \langle W^\pm\rangle^{4\mu} |\wt{p}|^2
 \dx\lesssim  \mathcal{E}_2 \mathcal{G}_1 \cdot \mathcal{F}(1+\|\nabla\psi\|_{L^\infty}).
\end{align}

\noindent \textbf{Step three } (first derivative estimate of pressure with weights):
For $0\leq \nu\leq \mu$,
taking the $L^2(\Om_L,dx)$ inner product of \eqref{H4} with
$ \langle W^+\rangle^{4\nu}\langle W^+\rangle^{2\nu}  \wt{p}$,
we obtain
\begin{align}\label{H21}
&-\int_{\Om_L} a^\top\nabla\cdot \nabla^\psi \wt{p} \cdot \wt{p}
\,\langle W^\pm\rangle^{4\nu}\langle W^\mp\rangle^{2\nu} \, \dx\\\nonumber
&= \int_{\Om_L}
a^\top \nabla \cdot (\wt{\Lambda}_-\cdot \nabla^\psi \wt{\Lambda}_+)\cdot \wt{p}
\,\langle W^\pm\rangle^{4\nu}\langle W^\mp\rangle^{2\nu} \, \dx.
\end{align}
Using a similar trick in Step two,
\eqref{H21} yields
\begin{align}\label{H22}
&\int_{\Om_L} J\langle W^+\rangle^{4\nu}\langle W^-\rangle^{2\nu}
\big| \nabla^\psi \wt{p}\big|^2\, \dx \\\nonumber
&\lesssim \big( \mathcal{E}_2 \mathcal{G}_1
+\|  \langle W^+\rangle^{2\nu} \wt{p} \|^2_{L^2}
+\|  \langle W^+\rangle^{2\nu-1}\langle W^-\rangle^{2\nu} \wt{p} \|^2_{L^2}\big)\\[-4mm]\nonumber\\\nonumber
&\quad \cdot \mathcal{F}(1+\|\nabla\psi\|_{L^\infty}).
\end{align}
Let
 $\nu\leq \f12$. By \eqref{H20}, \eqref{H22} yields
\begin{align*} %\label{F27}
\int_{\Om_L} J\langle W^+\rangle^{4\nu}\langle W^-\rangle^{2\nu}
\big|\nabla^\psi \wt{p}\big|^2\, \dx \lesssim  \mathcal{E}_2 \mathcal{G}_1 \cdot \mathcal{F}(1+\|\nabla\psi\|_{L^\infty}).
\end{align*}
Furthermore by Poincar\'e inequality Lemma \ref{lemHardy},
this implies that %for $\nu\leq \f12$
\begin{align}\label{H24}
&\int_{\Om_L} \langle W^+\rangle^{4\nu}\langle W^-\rangle^{2\nu} |\wt{p}|^2
J\, \dx\lesssim  \mathcal{E}_2 \mathcal{G}_1 \cdot \mathcal{F}(1+\|\nabla\psi\|_{L^\infty}) .
\end{align}
for $0\leq \nu\leq \f12$.
Now set $\nu$ in \eqref{H22} to be $\mu$, noting that $\mu \leq 3/4$, then combining \eqref{H20}, \eqref{H22}, \eqref{H24} and Poincar\'e inequality Lemma \ref{lemHardy} yields
\begin{align} \label{H25}
\int_{\Om_L} J\langle W^+\rangle^{4\nu}\langle W^-\rangle^{2\nu}
\big(\big|\nabla^\psi \wt{p}\big|^2 +|\wt{p}|^2
\big)\, \dx \lesssim  \mathcal{E}_2 \mathcal{G}_1 \cdot \mathcal{F}(1+\|\nabla\psi\|_{L^\infty}).
\end{align}

\textbf{Second tangential derivative estimate of pressure}:
%Before the application the tangential derivative $\p_1$. We first compute the commutators:
%\begin{align*}
%[\p_1,\nabla^\psi]=
%\end{align*}
%For the equation,
%\begin{align*}
%-\nabla^\psi \cdot \nabla^\psi \wt{p}
%=\nabla^\psi \cdot (\wt{\Lambda}_-\cdot \nabla^\psi \wt{\Lambda}_+).
%\end{align*}

We first write \eqref{H1} in the following component form
\begin{align}\label{H26}
-A_{kj}\p_k (A_{lj}\p_l\wt{p})
=\p^\psi_j \wt{\Lambda}^i_- \p^\psi_i \wt{\Lambda}^j_+.
\end{align}
The commutator of $\p_1$ and the Laplacian in the fixed domain $\Om_L$ is given by:
\begin{align}\label{H27}
[\p_1,\nabla^\psi\cdot \nabla^\psi] \wt{p}
=\p_1 A_{kj} \, \p_k (A_{l j} \, \p_l \wt{p} )
+A_{kj} \, \p_k ( \p_1 A_{l j} \, \p_l \wt{p}) .
\end{align}
Applying tangential derivative $\p_1$ to \eqref{H1}, we have
\begin{align}\label{H28}
&-\nabla^\psi \cdot \nabla^\psi \p_1\wt{p}
-[\p_1,\nabla^\psi\cdot \nabla^\psi] \wt{p}\\\nonumber
&=\p_1\nabla^\psi \cdot (\wt{\Lambda}_-\cdot \nabla^\psi \wt{\Lambda}_+).
\end{align}
Taking the $L^2(\Om_L,dx)$ inner product of  \eqref{H28} with
$ J \langle W^+\rangle^{4\mu}\langle W^-\rangle^{2\mu}  \p_1\wt{p}$, by \eqref{H27},
we obtain
%\begin{align}\label{H30}
%&-\int_{\Om_L} \nabla^\psi\cdot \nabla^\psi \p_1\wt{p} \cdot \p_1\wt{p}
%\,\langle W^+\rangle^{4\mu}\langle W^-\rangle^{2\mu} \, \dx\\\nonumber
%&-\int_{\Om_L} [\p_1,\nabla^\psi\cdot \nabla^\psi] \wt{p}
% \cdot \p_1\wt{p}
%\,\langle W^+\rangle^{4\mu}\langle W^-\rangle^{2\mu} \, \dx\\\nonumber
%&= \int_{\Om_L}
%\p_1 \nabla^\psi \cdot (\wt{\Lambda}_-\cdot \nabla^\psi \wt{\Lambda}_+)\cdot \p_1\wt{p}
%\,\langle W^+\rangle^{4\mu}\langle W^-\rangle^{2\mu} \, \dx.
%\end{align}
%Or
\begin{align}\label{H30}
&-\int_{\Om_L} a^\top\nabla \cdot \nabla^\psi \p_1\wt{p} \cdot \p_1\wt{p}
\,\langle W^+\rangle^{4\mu}\langle W^-\rangle^{2\mu} \, \dx\\\nonumber
&= \int_{\Om_L}
J \p_1 \nabla^\psi \cdot (\wt{\Lambda}_-\cdot \nabla^\psi \wt{\Lambda}_+)\cdot \p_1\wt{p}
\,\langle W^+\rangle^{4\mu}\langle W^-\rangle^{2\mu} \, \dx\\\nonumber
&\quad+\int_{\Om_L} J[\p_1,\nabla^\psi\cdot \nabla^\psi] \wt{p}
 \cdot \p_1\wt{p}
\,\langle W^+\rangle^{4\mu}\langle W^-\rangle^{2\mu} \, \dx\\\nonumber
&= \int_{\Om_L}
J \p_1 \nabla^\psi \cdot (\wt{\Lambda}_-\cdot \nabla^\psi \wt{\Lambda}_+)\cdot \p_1\wt{p}
\,\langle W^+\rangle^{4\mu}\langle W^-\rangle^{2\mu} \, \dx\\\nonumber
&\quad+\int_{\Om_L} J \p_1 A_{kj} \, \p_k (A_{l j} \, \p_l \wt{p} )
 \cdot \p_1\wt{p}
\,\langle W^+\rangle^{4\mu}\langle W^-\rangle^{2\mu} \, \dx \\\nonumber
&\quad+\int_{\Om_L} J\, A_{kj} \, \p_k ( \p_1 A_{l j} \, \p_l \wt{p}) \cdot \p_1\wt{p}
\,\langle W^+\rangle^{4\mu}\langle W^-\rangle^{2\mu} \, \dx .
\end{align}
For the first line of \eqref{H30},
by \eqref{F11}, Piola's identity \eqref{H3}, the boundary condition for pressure \eqref{H2} and integration by parts,
we have
\begin{align}%\label{}
&-\int_{\Om_L} a^\top\nabla\cdot \nabla^\psi \p_1\wt{p} \cdot \p_1\wt{p}
\,\langle W^+\rangle^{4\mu}\langle W^-\rangle^{2\mu} \, \dx\\\nonumber
&=\int_{\Om_L} |\nabla^\psi\p_1 \wt{p}|^2 \,\langle W^+\rangle^{4\mu}\langle W^-\rangle^{2\mu} \, J \dx
+\int_{\Om_L}
 a^\top\nabla (\langle W^+\rangle^{4\mu}\langle W^-\rangle^{2\mu})
\cdot \nabla^\psi \p_1\wt{p}\cdot \p_1\wt{p} \, \dx\\\nonumber
&\geq \int_{\Om_L} |\nabla^\psi \p_1\wt{p}|^2 \,\langle W^+\rangle^{4\mu}\langle W^-\rangle^{2\mu} \, J \dx \\\nonumber
&\quad-C(1+\|\nabla\psi \|_{L^\infty})
\| \langle W^+\rangle^{2\mu}\langle W^-\rangle^{\mu} \nabla^\psi\p_1 \wt{p} \|_{L^2}
\| \langle W^+\rangle^{2\mu}\langle W^-\rangle^{\mu} \p_1\wt{p} \|_{L^2}.
\end{align}
For the first line on the right hand side of \eqref{H30},
by \eqref{F11} and the Sobolev inequalities, we
%use integration by parts to
derive that
\begin{align}
&\int_{\Om_L}
J \p_1 \nabla^\psi \cdot (\wt{\Lambda}_-\cdot \nabla^\psi \wt{\Lambda}_+)\cdot \p_1\wt{p}
\,\langle W^+\rangle^{4\mu}\langle W^-\rangle^{2\mu} \, \dx \\\nonumber
&= -\int_{\Om_L}
\p^\psi_j \wt{\Lambda}^i_- \p^\psi_i \wt{\Lambda}^j_+
\cdot
\p_1\big( J\p_1\wt{p}
\,\langle W^+\rangle^{4\mu}\langle W^-\rangle^{2\mu} \big) \, \dx \\\nonumber
&\leq  \| \langle W^-\rangle^{2\mu} \nabla^\psi \wt{\Lambda}_-\|_{L^4}
   \Big\|\f{\langle W^+\rangle^{2\mu} \nabla^\psi \wt{\Lambda}_+}{\langle W^-\rangle^{\mu}} \Big\|_{L^4}
\| \langle W^+\rangle^{2\mu}\langle W^-\rangle^{\mu}\, \p^2_1 \wt{p} \|_{L^2} \|J\|_{L^\infty} \\\nonumber
&\quad+
\| \langle W^-\rangle^{2\mu} \nabla^\psi \wt{\Lambda}_-\|_{L^8}
   \Big\|\f{\langle W^+\rangle^{2\mu} \nabla^\psi \wt{\Lambda}_+}{\langle W^-\rangle^{\mu}} \Big\|_{L^8}
\| \langle W^+\rangle^{2\mu}\langle W^-\rangle^{\mu}\, \p_1 \wt{p}  \|_{L^4}
\|\p_1J\|_{L^2}\\\nonumber
%&\leq
%\big\| \langle W^-\rangle^{2\mu} |\nabla^\psi \wt{\Lambda}_-|
%\f{\langle W^+\rangle^{2\mu}|\nabla^\psi \wt{\Lambda}_+|}
%{\langle W^-\rangle^{\mu}} \big\|_{L^2} \\\nonumber
%&\quad\cdot \| \langle W^+\rangle^{2\mu}\langle W^-\rangle^{\mu}\,
%(|\p_1 \wt{p}|+|\p^2_1 \wt{p}| )\|_{L^2}
%(\|J \|_{L^\infty}+\|\p_1 J \|_{L^\infty}) \\\nonumber
%&\lesssim \sum_{|\alpha|\leq 1}\| \langle W^-\rangle^{2\mu} \p^\alpha \nabla^\psi \wt{\Lambda}_-\|_{L^2}
%\sum_{|\alpha|\leq 1} \Big\|\f{\langle W^+\rangle^{2\mu}\p^\alpha \nabla^\psi \wt{\Lambda}_+}
%{\langle W^-\rangle^{\mu}} \Big\|_{L^2} \\\nonumber
%&\quad \cdot
%\sum_{|\alpha|\leq 1}\| \langle W^+\rangle^{2\mu}\langle W^-\rangle^{\mu}\,
%\p_1\p^\alpha \wt{p}\|_{L^2}
%(1+\|\nabla\psi \|_{L^\infty}+\|\nabla^2\psi \|_{L^2}) \\[-5mm]\nonumber\\\nonumber
&\leq \mathcal{E}_2^{\f12} \mathcal{G}_2^{\f12}\cdot
\sum_{|\alpha|\leq 1}\| \langle W^+\rangle^{2\mu}\langle W^-\rangle^{\mu}\,
\p_1\p^\alpha \wt{p}\|_{L^2}
(1+\|\nabla\psi \|_{L^\infty}+\|\nabla^2\psi \|_{L^2})  .
\end{align}
For the second line on the right hand side of \eqref{H30},
by \eqref{F11} and the Sobolev inequalities, it is bounded by
\begin{align}
&-\int_{\Om_L} J \p_1 A_{kj} \, \p_k (A_{l j} \, \p_l \wt{p} )
 \cdot \p_1\wt{p}
\,\langle W^+\rangle^{4\mu}\langle W^-\rangle^{2\mu} \, \dx \\\nonumber
&\lesssim
\|\nabla A\|_{L^4}^2\| \langle W^+\rangle^{2\mu}\langle W^-\rangle^{\mu}\, \nabla\wt{p}\|^2_{L^4} \\[-4mm]\nonumber\\\nonumber
&\quad+\|A\|_{L^\infty}\|\nabla A\|_{L^4}\| \langle W^+\rangle^{2\mu}\langle W^-\rangle^{\mu}\, \nabla\wt{p}\|_{L^4}
\| \langle W^+\rangle^{2\mu}\langle W^-\rangle^{\mu}\, \nabla^2\wt{p}\|_{L^2}  \\[-4mm]\nonumber\\\nonumber
&\lesssim
%\| \langle W^+\rangle^{2\mu}\langle W^-\rangle^{\mu}\, \p_1\wt{p}\|_{L^2}
\| \langle W^+\rangle^{2\mu}\langle W^-\rangle^{\mu}\,
(|\nabla \wt{p}|+|\nabla^2 \wt{p}| )\|^2_{L^2} \\[-4mm]\nonumber\\\nonumber
&\quad\cdot(\|\nabla^2\psi \|_{L^2}+\|\nabla^3 \psi   \|_{L^2})
(1+\|\nabla\psi \|_{L^\infty}+\|\nabla^2\psi \|_{L^2}+\|\nabla^3 \psi   \|_{L^2}) .
\end{align}
For the third line on the right hand side of \eqref{H30},
note the derivative applied on $\p_1A$ is of third order, which is too high.
Hence we use integration by parts to move one derivative.
Note by the expression of $A$, $A|_{\Gamma_-}$ and \eqref{H2}, there hold
\begin{align} \label{H34}
&\p_1 \wt{p}\big|_{\Gamma}=0,\quad  %\\
%&
(J\, A_{2j} \,   \p_1 A_{l j} \, \p_l \wt{p})\big|_{\Gamma_-}
%=(  \p_1 A_{l 2} \, \p_l \wt{p})\big|_{\Gamma_-}
%=(  \p_1 A_{1 2} \, \p_1 \wt{p})\big|_{\Gamma_-}
=0 .
\end{align}
Consequently, by \eqref{F11}, Piola's identity \eqref{H3}, the boundary condition for pressure \eqref{H2},
the Sobolev inequalities
%matrix $A$
and the integration by parts,
we have
% (what about boundary term?)
\begin{align}\label{H35}
&-\int_{\Om_L} J\, A_{kj} \, \p_k ( \p_1 A_{l j} \, \p_l \wt{p}) \cdot \p_1\wt{p}
\,\langle W^+\rangle^{4\mu}\langle W^-\rangle^{2\mu} \, \dx \\\nonumber
&=\int_{\Om_L} J\, A_{kj} \,   \p_1 A_{l j} \, \p_l \wt{p} \cdot
\p_k\big( \p_1\wt{p}
\,\langle W^+\rangle^{4\mu}\langle W^-\rangle^{2\mu} \big)\, \dx \\\nonumber
&\quad-\int_{\Gamma} J\, A_{2j} \,  ( \p_1 A_{l j} \, \p_l \wt{p}) \cdot \p_1\wt{p}
\,\langle W^+\rangle^{4\mu}\langle W^-\rangle^{2\mu} \, \dx_1 \\\nonumber
%\qquad \textrm{note}\,\, \p_l \wt{p}\big|_{\Gamma}=0 \\
&\quad+\int_{\Gamma_-} J\, A_{2j} \,  ( \p_1 A_{l j} \, \p_l \wt{p}) \cdot \p_1\wt{p}
\,\langle W^+\rangle^{4\mu}\langle W^-\rangle^{2\mu} \, \dx_1 \\\nonumber
%\end{align}
%
%The last two lines of \eqref{H35} involving the boundary terms are zero due to the boundary condition.
%Consequently, \eqref{H35} can be bounded as follows:
%\begin{align*}
&=\int_{\Om_L} J\, A_{kj} \,   \p_1 A_{l j} \, \p_l \wt{p} \cdot
\p_k\big( \p_1\wt{p}
\,\langle W^+\rangle^{4\mu}\langle W^-\rangle^{2\mu} \big)\, \dx\\\nonumber
&\lesssim \| |A|  \,   |\p_1 A| \|_{L^4}
\| \langle W^+\rangle^{2\mu}\langle W^-\rangle^{\mu} \nabla \wt{p} \|_{L^4}
\|\langle W^+\rangle^{2\mu}\langle W^-\rangle^{\mu} ( |\nabla\p_1\wt{p}|+|\p_1\wt{p}|)
\|_{L^2}\\[-4mm]\nonumber\\\nonumber
&\lesssim  (1+\|\nabla\psi \|_{L^\infty})(\|\nabla^2\psi \|_{L^2}+\|\nabla^3 \psi   \|_{L^2}) %\\[-4mm]\nonumber\\\nonumber
%&\quad \cdot %\| \langle W^+\rangle^{2\mu}\langle W^-\rangle^{\mu} \nabla \wt{p} \|_{L^2}
\|\langle W^+\rangle^{2\mu}\langle W^-\rangle^{\mu} ( |\nabla^2\wt{p}|+|\nabla \wt{p}|)
\|^2_{L^2}.
\end{align}
Combining \eqref{H30}-\eqref{H35}, we obtain
\begin{align}\label{H37}
&\norm{\langle W^+\rangle^{2\mu}\langle W^-\rangle^{\mu}
\nabla^\psi\p_1 \wt{p}}^2_{L^2} \\\nonumber
&\lesssim  \Big(\sum_{|\alpha|\leq 1}\big\|\langle W^+\rangle^{2\mu}\langle W^-\rangle^{\mu} \nabla^\psi\nabla^\alpha \wt{p}\big\|^2_{L^2}
\cdot(\|\nabla\psi\|_{L^\infty}+\|\nabla^2\psi \|_{L^2}+\|\nabla^3 \psi   \|_{L^2})\\\nonumber
&\quad  +\|\langle W^+\rangle^{2\mu}\langle W^-\rangle^{\mu} \nabla^\psi  \wt{p}\big\|^2_{L^2}
+ \mathcal{E}_2\mathcal{G}_2
\Big)\cdot \mathcal{F}(1+\|\nabla\psi\|_{L^\infty}+\|\nabla^2\psi \|_{L^2}+\|\nabla^3 \psi   \|_{L^2}).
\end{align}

\textbf{Second normal derivative estimate of pressure}:

The remaining second order derivative $\p_2^2 \wt{p}$ is estimated directly from the
equation \eqref{H26} by using the explicit expression of the coefficients $A_{kj}$.
More precisely, \eqref{H26} can be written as follows
%We write the equations for the pressure as follows
\begin{equation}\label{H41}
A_{kj}\, A_{lj}\, \p_k \p_l \wt{p}
= -\p^\psi_j \wt{\Lambda}^i_- \p^\psi_i \wt{\Lambda}^j_+
 -A_{kj} \, \p_k A_{lj} \, \p_l \wt{p} \, .
\end{equation}
On the other hand, simple calculation yields
\begin{equation} \label{H42}
A_{kj}\, A_{lj}\, \p_k \p_l \wt{p}
=\f{1+|\p_1 \psi|^2}{(1+\p_2 \psi)^2} \, \p_2^2 \wt{p}
+\p_1^2 \wt{p}-2\f{\p_1 \psi \, \p_1 \p_2 \wt{p}}{1+\p_2 \psi}.
\end{equation}
Combining \eqref{H41} and \eqref{H42}, we obtain
\begin{align*} %\label{H43}
\p_2^2 \wt{p}
&=\f{(1+\p_2 \psi)^2 }{1+|\p_1 \psi|^2}
\Big( -\p^\psi_j \wt{\Lambda}^i_- \p^\psi_i \wt{\Lambda}^j_+
-A_{kj} \, \p_k A_{lj} \, \p_l \wt{p}
-\p_1^2 \wt{p}+2\f{\p_1 \psi \, \p_1 \p_2 \wt{p}}{1+\p_2 \psi}\Big).
\end{align*}
%By \eqref{H43}
Consequently, by \eqref{F11} and the Sobolev inequalities,
we obtain
%\begin{align*}
%&\|\langle W^+\rangle^{2\mu}\langle W^-\rangle^{\mu} \p_2^2 \wt{p} \|_{L^2}\\
%&\leq \|\langle W^+\rangle^{2\mu}\langle W^-\rangle^{\mu} \p_1^2 \wt{p} \|_{L^2}
%+\|\langle W^+\rangle^{2\mu}\langle W^-\rangle^{\mu}\, A\,  \p A \, \nabla \wt{p} \|_{L^2}\\
%&\quad +\|\langle W^+\rangle^{\mu}\langle W^-\rangle^{\mu}\, \p_1\psi \p_1\p_2 \wt{p} \|_{L^2}
%+\|\langle W^+\rangle^{2\mu}\langle W^-\rangle^{\mu}
%\p^\psi_j \wt{\Lambda}^i_- \p^\psi_i \wt{\Lambda}^j_+ \|_{L^2}.
%\end{align*}
%Or we write it as follows:
\begin{align}\label{H44}
&\f12\|\langle W^+\rangle^{2\mu}\langle W^-\rangle^{\mu} \p_2^\psi \p_2\wt{p} \|_{L^2}\leq
\|\langle W^+\rangle^{2\mu}\langle W^-\rangle^{\mu} \p_2^2 \wt{p} \|_{L^2}\\[-4mm]\nonumber\\\nonumber
&\leq (1+\|\p_2\psi\|_{L^\infty})^2
\cdot\big\{
\|\langle W^+\rangle^{2\mu}\langle W^-\rangle^{\mu}
\p^\psi_j \wt{\Lambda}^i_- \p^\psi_i \wt{\Lambda}^j_+ \|_{L^2}\\[-4mm]\nonumber\\\nonumber
&\quad+\|\langle W^+\rangle^{2\mu}\langle W^-\rangle^{\mu}\, |A|\,  |\nabla A| \, |\nabla \wt{p}| \|_{L^2}
+\|\langle W^+\rangle^{2\mu}\langle W^-\rangle^{\mu} \p_1^2 \wt{p} \|_{L^2}\\[-4mm]\nonumber\\\nonumber
&\quad +2(1+2\|\p_2\psi\|_{L^\infty})
\|\langle W^+\rangle^{2\mu}\langle W^-\rangle^{\mu}\, \p_1\psi \p_1\p_2 \wt{p} \|_{L^2}
\big\}\\\nonumber
&\leq (1+\|\p_2\psi\|_{L^\infty})^2
\cdot\big\{ \sum_{|\alpha|\leq 1}\| \langle W^-\rangle^{2\mu} \p^\alpha \nabla^\psi \wt{\Lambda}_-\|_{L^2}
\sum_{|\alpha|\leq 1} \Big\|\f{\langle W^+\rangle^{2\mu}\p^\alpha \nabla^\psi \wt{\Lambda}_+}
{\langle W^-\rangle^{\mu}} \Big\|_{L^2}\\\nonumber
&\quad+ (1+\|\nabla\psi\|_{L^\infty})(\|\nabla^2\psi\|_{L^2}+\|\nabla^3\psi\|_{L^2})
   \sum_{|\alpha|\leq 1}\|\langle W^+\rangle^{2\mu}\langle W^-\rangle^{\mu} \, \nabla \p^\alpha \wt{p} \|_{L^2}
   %\\[-6mm]\nonumber
   \\\nonumber
&\quad +\|\langle W^+\rangle^{2\mu}\langle W^-\rangle^{\mu} \p_1^2 \wt{p} \|_{L^2}
+2(1+2\|\p_2\psi\|_{L^\infty})\|\p_1\psi\|_{L^\infty}
\|\langle W^+\rangle^{2\mu}\langle W^-\rangle^{\mu}\, \p_1\p_2 \wt{p} \|_{L^2}\big\}.
%\\[-4mm]\nonumber\\\nonumber
%&\lesssim
% \|\langle W^+\rangle^{2\mu}\langle W^-\rangle^{\mu} \p_1^2 \wt{p} \|_{L^2}
%+\|\langle W^+\rangle^{2\mu}\langle W^-\rangle^{\mu}\, \p_1\p_2 \wt{p} \|_{L^2} \\[-4mm]\nonumber\\\nonumber
%&+\|\langle W^+\rangle^{2\mu}\langle W^-\rangle^{\mu}\, \nabla \wt{p} \|_{L^2}
%+\mathcal{E}_2\mathcal{G}_2.
\end{align}
Combining the second order tangential derivative estimate \eqref{H37} and
the normal derivative estimate \eqref{H44}, we deduce
%\begin{align}
%&\|\langle W^+\rangle^{2\mu}\langle W^-\rangle^{\mu} \p_2^2 \wt{p} \|_{L^2}\\[-4mm]\nonumber\\\nonumber
%&\lesssim \mathcal{E}_2\mathcal{G}_2+.
%\end{align}
\begin{align}\label{H45}
&\sum_{|\alpha|= 1}\big\|\langle W^+\rangle^{2\mu}\langle W^-\rangle^{\mu}
\nabla^\psi\p^\alpha \wt{p} \big\|^2_{L^2} \\\nonumber
&\lesssim   \mathcal{F}(1+\|\nabla\psi\|_{L^\infty}+\|\nabla^2\psi \|_{L^2}+\|\nabla^3 \psi   \|_{L^2})\\[-4mm]\nonumber\\\nonumber
&\quad \cdot \Big(\sum_{|\alpha|\leq 1}\big\|\langle W^+\rangle^{2\mu}\langle W^-\rangle^{\mu} \nabla^\psi\nabla^\alpha \wt{p}\big\|^2_{L^2}
\cdot(\|\nabla\psi\|_{L^\infty}+\|\nabla^2\psi \|_{L^2}+\|\nabla^3 \psi   \|_{L^2})\\\nonumber
&\qquad\quad  +\|\langle W^+\rangle^{2\mu}\langle W^-\rangle^{\mu} \nabla^\psi  \wt{p}\big\|^2_{L^2}
+\mathcal{G}_2 \mathcal{E}_2
\Big).
\end{align}
By \eqref{F1} and Lemma \ref{lemTR}, for sufficiently small $\epsilon$,
the terms containing the second order derivative of pressure in the third line of \eqref{H45}
will be absorbed by the left hand side.
Thus by \eqref{H25}, \eqref{H37} yields
\begin{align*}%\label{H38}
&\sum_{|\alpha|\leq 1}\big\|\langle W^+\rangle^{2\mu}\langle W^-\rangle^{\mu}
\nabla^\psi\p^\alpha \wt{p} \big\|^2_{L^2} \\\nonumber
&\lesssim   \mathcal{E}_2\mathcal{G}_2 \cdot \mathcal{F}(1+\|\nabla\psi\|_{L^\infty}+\|\nabla^2\psi \|_{L^2}+\|\nabla^3 \psi   \|_{L^2}).
\end{align*}

\textbf{$(r+1)$-order tangential derivative estimate of pressure ($2\leq r\leq s$)}
%For the equation,
%\begin{align*}
%-\nabla^\psi \cdot \nabla^\psi \wt{p}
%=\nabla^\psi \cdot (\wt{\Lambda}_-\cdot \nabla^\psi \wt{\Lambda}_+).
%\end{align*}
%Written in component form, we have
%\begin{align*}
%-A_{kj}\p_k (A_{lj}\p_l\wt{p})
%=\p^\psi_j \wt{\Lambda}^i_- \p^\psi_i \wt{\Lambda}^j_+.
%\end{align*}
%Applying derivative $\p_1^r$,

Firstly, there naturally holds the following identity:
\begin{align*}
&\p_1^r \big(A_{kj}\p_k (A_{lj}\p_l\wt{p}) \big)\\
&=\sum_{m<r} C_r^m \p_1^{r-m} A_{kj}\p_1^m\p_k (A_{lj}\p_l\wt{p})
+  A_{kj}\p_1^r\p_k (A_{lj}\p_l\wt{p})\\
&=\sum_{m<r} C_r^m \p_1^{r-m} A_{kj}\p_1^m\p_k (A_{lj}\p_l\wt{p})
+\sum_{m<r} C_r^m A_{kj}\p_k (\p_1^{r-m} A_{lj}\p_l\p_1^m\wt{p})\\
&\quad +A_{kj}\p_k ( A_{lj}\p_l\p_1^r\wt{p}).
\end{align*}
The above identity can be written in another form:
\begin{align}\label{H46}
&[\p_1^r,\nabla^\psi\cdot \nabla^\psi] \wt{p} \\\nonumber
&=\sum_{m<r} C_r^m \p_1^{r-m} A_{kj}\p_1^m\p_k (A_{lj}\p_l\wt{p})
+\sum_{m<r} C_r^m A_{kj}\p_k (\p_1^{r-m} A_{lj}\p_l\p_1^m\wt{p}) .
\end{align}
Applying derivative $\p_1^r$ onto \eqref{H1}, we have
\begin{align}\label{H48}
-\nabla^\psi \cdot \nabla^\psi \p_1^r\wt{p}
-[\p_1^r,\nabla^\psi\cdot \nabla^\psi] \wt{p}
=\p_1^r\nabla^\psi \cdot (\wt{\Lambda}_-\cdot \nabla^\psi \wt{\Lambda}_+).
\end{align}
Taking the $L^2(\Om_L,dx)$ inner product of \eqref{H48} with
$ J\langle W^+\rangle^{4\mu}\langle W^-\rangle^{2\mu}  \p_1^r\wt{p}$, by \eqref{H46},
we obtain
\begin{align}\label{H49}
&-\int_{\Om_L} a^\top\nabla \cdot \nabla^\psi \p_1^r\wt{p} \cdot \p_1^r\wt{p}
\,\langle W^+\rangle^{4\mu}\langle W^-\rangle^{2\mu} \, \dx\\\nonumber
&= \int_{\Om_L}
J \p_1^r \nabla^\psi \cdot (\wt{\Lambda}_-\cdot \nabla^\psi \wt{\Lambda}_+)\cdot \p_1^r\wt{p}
\,\langle W^+\rangle^{4\mu}\langle W^-\rangle^{2\mu} \, \dx\\\nonumber
&\quad +\int_{\Om_L} J[\p_1^r,\nabla^\psi\cdot \nabla^\psi] \wt{p}
 \cdot \p_1^r\wt{p}
\,\langle W^+\rangle^{4\mu}\langle W^-\rangle^{2\mu} \, \dx \\\nonumber
&= \int_{\Om_L}
J \p_1^r \nabla^\psi \cdot (\wt{\Lambda}_-\cdot \nabla^\psi \wt{\Lambda}_+)\cdot \p_1^r\wt{p}
\,\langle W^+\rangle^{4\mu}\langle W^-\rangle^{2\mu} \, \dx\\\nonumber
&\quad +\int_{\Om_L} J \sum_{m<r} C_r^m \p_1^{r-m} A_{kj}\p_1^m\p_k (A_{lj}\p_l\wt{p})
 \cdot \p_1^r\wt{p}
\,\langle W^+\rangle^{4\mu}\langle W^-\rangle^{2\mu} \, \dx \\\nonumber
&\quad +\int_{\Om_L} J\sum_{m<r} C_r^m A_{kj}\p_k (\p_1^{r-m} A_{lj}\p_l\p_1^m\wt{p})
 \cdot \p_1^r\wt{p} \,\langle W^+\rangle^{4\mu}\langle W^-\rangle^{2\mu} \, \dx .
\end{align}
For the first line of \eqref{H49},
by \eqref{F11}, Piola's identity \eqref{H3}, the boundary condition for pressure \eqref{H2} and the integration by parts,
we have
\begin{align}%\label{}
&-\int_{\Om_L} a^\top\nabla\cdot \nabla^\psi \p_1^r\wt{p} \cdot \p_1^r\wt{p}
\,\langle W^+\rangle^{4\mu}\langle W^-\rangle^{2\mu} \, \dx\\\nonumber
&=\int_{\Om_L} |\nabla^\psi \p_1^r \wt{p}|^2 \,\langle W^+\rangle^{4\mu}\langle W^-\rangle^{2\mu} \, J \dx
+\int_{\Om_L}
 a^\top\nabla (\langle W^+\rangle^{4\mu}\langle W^-\rangle^{2\mu})
\cdot \nabla^\psi \p_1^r\wt{p}\cdot \p_1^r\wt{p} \, \dx\\\nonumber
&\geq \int_{\Om_L} |\nabla^\psi \p_1^r \wt{p}|^2 \,\langle W^+\rangle^{4\mu}\langle W^-\rangle^{2\mu} \, J \dx \\\nonumber
&\quad-C(1+\|\nabla\psi \|_{L^\infty})
\| \langle W^+\rangle^{2\mu}\langle W^-\rangle^{\mu} \nabla^\psi\p_1^r \wt{p} \|_{L^2}
\| \langle W^+\rangle^{2\mu}\langle W^-\rangle^{\mu} \p_1^r \wt{p} \|_{L^2}.
\end{align}
For the first line on the right hand side of \eqref{H49},
by \eqref{F11} and the Sobolev inequalities
and integration by parts, we derive
\begin{align}
&\int_{\Om_L}
J \p_1\p_1^{r-1} \nabla^\psi \cdot (\wt{\Lambda}_-\cdot \nabla^\psi \wt{\Lambda}_+)\cdot \p_1^r\wt{p}
\,\langle W^+\rangle^{4\mu}\langle W^-\rangle^{2\mu} \, \dx\\\nonumber
&= -\int_{\Om_L}
\p_1^{r-1} \big(\p^\psi_j \wt{\Lambda}^i_- \p^\psi_i \wt{\Lambda}^j_+\big)
\cdot
\p_1\big( J \p_1^r\wt{p}
\,\langle W^+\rangle^{4\mu}\langle W^-\rangle^{2\mu} \big) \, \dx\\\nonumber
&\leq C\sum_{a+b= r-1}
\big\| \langle W^-\rangle^{2\mu} |\p_1^a\nabla^\psi \wt{\Lambda}_-|
\f{\langle W^+\rangle^{2\mu}|\p_1^b\nabla^\psi \wt{\Lambda}_+|}
{\langle W^-\rangle^{\mu}} \big\|_{L^2} \\\nonumber
&\quad\cdot \| \langle W^+\rangle^{2\mu}\langle W^-\rangle^{\mu}\,
(|\p_1^{r+1}\wt{p}|+|\p_1^r \wt{p}| )\|_{L^2}
%\big\| |J|+|\p_1 J| \big\|_{L^\infty} .
(\|J \|_{L^\infty}+\|\p_1 J \|_{L^\infty}) \\\nonumber
&\lesssim \mathcal{E}_s^{\f12} \mathcal{G}_s^{\f12} \| \langle W^+\rangle^{2\mu}\langle W^-\rangle^{\mu}\,
(|\p_1^{r+1}\wt{p}|+|\p_1^r \wt{p}| )\|_{L^2}
(1+\|\nabla\psi \|_{L^\infty}+\|\nabla^2 \psi   \|_{L^\infty}) .
\end{align}
%For the second line of \eqref{H49}, by \eqref{H46}, we obtain
%\begin{align}\label{H52}
%&-\int_{\Om_L} J [\p_1^r,\nabla^\psi\cdot \nabla^\psi] \wt{p}
% \cdot \p_1^r\wt{p}
%\,\langle W^+\rangle^{4\mu}\langle W^-\rangle^{2\mu} \, \dx\\\nonumber
%&=-\int_{\Om_L} J \big\{
%\sum_{m<r} C_r^m \p_1^{r-m} A_{kj}\p_1^m\p_k (A_{lj}\p_l\wt{p}) \\\nonumber
%&\qquad+\sum_{m<r} C_r^m A_{kj}\p_k (\p_1^{r-m} A_{lj}\p_l\p_1^m\wt{p})\big\}
%\cdot \p_1^r\wt{p}
%\,\langle W^+\rangle^{4\mu}\langle W^-\rangle^{2\mu} \, \dx.
%\end{align}
For the second line on the right hand side of \eqref{H49}, it is bounded by
\begin{align}\label{H50}
&-\int_{\Om_L} J
\sum_{m<r} C_r^m \p_1^{r-m} A_{kj}\p_1^m\p_k (A_{lj}\p_l\wt{p})
\cdot \p_1^r\wt{p}
\,\langle W^+\rangle^{4\mu}\langle W^-\rangle^{2\mu} \, \dx\\\nonumber
&\lesssim
\sum_{m<r} \|\langle W^+\rangle^{2\mu}\langle W^-\rangle^{\mu} J \p_1^{r-m} A_{kj}\p_1^m\p_k (A_{lj}\p_l\wt{p})
\|_{L^2}
\| \langle W^+\rangle^{2\mu}\langle W^-\rangle^{\mu}
\, \p_1^r\wt{p}
\|_{L^2}\\\nonumber
%&\lesssim \mathcal{F} (1+\| \nabla \psi \|_{L^\infty}+\| \nabla^2 \psi \|_{L^\infty}
%+\sum_{2\leq |\alpha|\leq r+1}\| \nabla^\alpha \psi \|_{L^2}) \\\nonumber
%&\quad\cdot \mathcal{F} (\| \nabla \psi \|_{L^\infty}+\| \nabla^2 \psi \|_{L^\infty}
%+\sum_{2\leq |\alpha|\leq r+1}\| \nabla^\alpha \psi \|_{L^2}) \\\nonumber
%&\quad\cdot\sum_{m\leq r-1,\ 1\leq|\alpha|\leq 2}
%\|\langle W^+\rangle^{2\mu}\langle W^-\rangle^{\mu}  \p_1^m \nabla^{\alpha} \wt{p}
%\|_{L^2}  \| \langle W^+\rangle^{2\mu}\langle W^-\rangle^{\mu} \, \p_1^r\wt{p}\|_{L^2} .
%\end{align}
%Or the last bound is give by
%\begin{align}
&\lesssim (1+\| \nabla \psi\|_{L^\infty})\| \nabla^2 \psi \|_{L^\infty}
\|\langle W^+\rangle^{2\mu}\langle W^-\rangle^{\mu}\,  \p_1^{r-1} \nabla^2 \wt{p}\|_{L^2}
\| \langle W^+\rangle^{2\mu}\langle W^-\rangle^{\mu}\, \p_1^r \wt{p}\|_{L^2} \\[-4mm]\nonumber\\\nonumber
&\quad+\sum_{1\leq|\alpha|\leq r}
\|\langle W^+\rangle^{2\mu}\langle W^-\rangle^{\mu}  \nabla^{\alpha} \wt{p}\|_{L^2}
\| \langle W^+\rangle^{2\mu}\langle W^-\rangle^{\mu} \, \p_1^r\wt{p}\|_{L^2} \\\nonumber
&\qquad\cdot\mathcal{F} (1+\| \nabla \psi \|_{L^\infty}%+\| \nabla^2 \psi \|_{L^\infty}
+\sum_{2\leq |\alpha|\leq q}\| \nabla^\alpha \psi \|_{L^2}),
\end{align}
where
\begin{equation*}
q=
\begin{cases}
4,\quad \text{if}\,\, r=2,\\
r+1,\quad \text{if}\,\, 3\leq r\leq s.
\end{cases}
\end{equation*}

In the sequel, we will give a detailed calculation on how \eqref{H50} is achieved. Before that, we calculate the pointwise estimate
of the matrix $A$.
%Let us now give the details for
Recalling
\begin{align*}
A=[\nabla \Psi]^{-1}=\frac{1}{J}
\left(
\begin{matrix}
1+\p_2\psi \quad 0\\
-\p_1\psi \qquad 1
\end{matrix}
\right)
=\left(
\begin{matrix}
1  \qquad\quad 0\\
-\p_1\psi J^{-1} \quad J^{-1}
\end{matrix}
\right)
\end{align*}
where  $J=1+\p_2\psi$.
Hence
\begin{align*}
\nabla A=\left(
\begin{matrix}
0  \qquad\qquad 0\\
-\nabla(\p_1\psi J^{-1}) \quad \nabla(J^{-1})
\end{matrix}
\right),
\end{align*}
where
\begin{align*}
\nabla(J^{-1})
&=-J^{-2}\nabla\p_2\psi,\\
-\nabla(\p_1\psi J^{-1})
&=-\nabla\p_1\psi J^{-1}+\p_1\psi \nabla\p_2\psi J^{-2} .
\end{align*}
Next, we calculate
\begin{align*}
\nabla^2 A=\left(
\begin{matrix}
0  \qquad\qquad 0\\
-\nabla^2(\p_1\psi J^{-1}) \quad \nabla^2(J^{-1})
\end{matrix}
\right),
\end{align*}
where
\begin{align*}
\nabla^2(J^{-1})
&=2J^{-3} (\nabla\p_2\psi)^2-J^{-2}\nabla^2\p_2\psi,\\
-\nabla^2(\p_1\psi J^{-1})
&=-\nabla^2\p_1\psi J^{-1}+ 2\nabla\p_1\psi\nabla\p_2\psi J^{-2} \\
&\quad + \p_1\psi\nabla^2\p_2\psi J^{-2}
-2 \p_1\psi(\nabla\p_2\psi)^2 J^{-3}.
\end{align*}
Hence
\begin{align*}
|A|        &\lesssim 1+|\p_1 \psi|,\\
|\nabla A| &\lesssim (1+|\p_1 \psi|)|\nabla^2 \psi|,\\
|\nabla^2 A|&\lesssim (1+|\nabla \psi|) \big( |\nabla^3 \psi|+|\nabla^2\psi|^2 \big).
\end{align*}
By a similar calculation, we obtain
\begin{align*}
|\nabla^3 A|&\lesssim (1+|\nabla \psi|) \big( |\nabla^4 \psi|+|\nabla^3\psi|\cdot|\nabla^2\psi|+|\nabla^2\psi|^3\big).
\end{align*}
%By a similar calculation,
In summary, by the assumption of \eqref{F1},
for integer $1\leq k\leq s$, $s\geq 4$,
there holds
\begin{align*} %\label{H51}
|\nabla^k A|&\lesssim (1+|\nabla \psi|) \sum_{ 2\leq i\leq k+1} |\nabla^i \psi| .
\end{align*}

%\begin{remark}
%For the term
%\begin{align*}
%\sum_{m<r} \|\langle W^+\rangle^{2\mu}\langle W^-\rangle^{\mu} J \p_1^{r-m} A_{kj}\p_1^m\p_k (A_{lj}\p_l\wt{p})
%\|_{L^2}
%\end{align*}
%Let us give more accurate estimate to clearly express the index.
Now let us show \eqref{H50}. To simplify the presentation,
let us just forget about the weight $\langle W^+\rangle^{2\mu}\langle W^-\rangle^{\mu}$.
The calculation with the weight function is the same since we just need to put the weight function and the pressure together.

If $r=2$, by the estimate for matrix $A$ and the Sobolev inequalities, we calculate
%($3$-order estimate of pressure):
\begin{align*}
& \big\|\sum_{m=0,1}\p_1^{r-m} A_{kj}\p_1^m\p_k (A_{lj}\p_l\wt{p})\big\|_{L^2} \\
& \lesssim \|\nabla^2 A \nabla(A\nabla \wt{p})+\nabla A \nabla^2 (A\nabla \wt{p})\|_{L^2} \\[-4mm]\\
%& \lesssim \underbrace{\nabla^2 A}_{L^4} \underbrace{\nabla A}_{L^\infty}  \underbrace{\nabla \wt{p}}_{L^4}
%       +\underbrace{\nabla^2 A}_{L^4} \, \cdot \underbrace{A}_{L^\infty}  \underbrace{\nabla^2 \wt{p}}_{L^4}\\[-4mm]\\
%&\quad +\underbrace{\nabla A}_{L^\infty} \underbrace{\nabla^2 A}_{L^4}  \underbrace{\nabla \wt{p}}_{L^4}
%+\underbrace{\nabla A}_{L^\infty} \underbrace{\nabla A}_{L^\infty}  \underbrace{\nabla^2 \wt{p}}_{L^2}
%       +\underbrace{\nabla A}_{L^\infty} \cdot \underbrace{A}_{L^\infty} \underbrace{\nabla^3 \wt{p}}_{L^2}\\
& \lesssim \|\nabla^2 A\|_{L^4} \|\nabla A\|_{L^\infty}  \|\nabla \wt{p}\|_{L^4}
       +\|\nabla^2 A\|_{L^4}  \|A\|_{L^\infty}  \|\nabla^2 \wt{p}\|_{L^4}\\[-4mm]\\
&\quad +\|\nabla A\|_{L^\infty} \|\nabla^2 A\|_{L^4}  \|\nabla \wt{p}\|_{L^4}
+\|\nabla A\|^2_{L^\infty}  \|\nabla^2 \wt{p}\|_{L^2}
       +\|\nabla A\|_{L^\infty}  \|A\|_{L^\infty} \|\nabla^3 \wt{p}\|_{L^2}\\[-4mm]\nonumber\\\nonumber
&\lesssim \mathcal{F}(1+\| \nabla\psi\|_{L^\infty}+\| \nabla^2\psi\|_{L^2}+\| \nabla^4\psi\|_{L^2} )
\big( \| \nabla^3 \wt{p}\|_{L^2}+\| \nabla \wt{p} \|_{L^2}  \big) .
\end{align*}
%It can be seen that the above involves $\| \nabla^4\psi\|_{L^2}$, thus cannot be closed.

If $r=3$, %by \eqref{H51},
similarly by the estimate for matrix $A$ and the Sobolev inequalities,
 we calculate
\begin{align*}
& \big\|\sum_{m=0,1,2}  \p_1^{r-m} A_{kj}\p_1^m\p_k (A_{lj}\p_l\wt{p}) \big\|_{L^2} \\
%& \sim \nabla^3 A \nabla(A\nabla \wt{p})+\nabla^2 A \nabla^2 (A\nabla \wt{p})
%       +\nabla A \nabla^3 (A\nabla \wt{p})  \\
%& \sim \big(\nabla^3 A \nabla A + \nabla^2 A \nabla^2 A \big)       \nabla \wt{p}
%       +\big(\nabla^3 A \cdot  A + \nabla^2 A \nabla  A \big)       \nabla^2 \wt{p} \\
%&\quad +\big(\nabla^2 A \cdot  A + \nabla A \nabla  A \big)       \nabla^3 \wt{p}
%       +\nabla A \cdot  A        \nabla^4 \wt{p}  \\
& \lesssim \big(\|\nabla^3 A\|_{L^2} \|\nabla A\|_{L^\infty}
           + \|\nabla^2 A\|_{L^4} \|\nabla^2 A\|_{L^4} \big)    \|\nabla \wt{p}\|_{L^\infty}\\[-4mm]\nonumber\\\nonumber
&\quad +\big(\|\nabla^3 A\|_{L^2}  \|A\|_{L^\infty}
             + \|\nabla^2 A\|_{L^4} \|\nabla  A\|_{L^4} \big)  \|\nabla^2 \wt{p}\|_{L^2} \\[-4mm]\nonumber\\\nonumber
&\quad +\big(\|\nabla^2 A \|_{L^4}  \|A\|_{L^\infty}
         + \|\nabla A\|_{L^4} \|\nabla A\|_{L^\infty}  \big) \|\nabla^3 \wt{p}\|_{L^4}
       +\|\nabla A\|_{L^\infty}  \|A\|_{L^\infty}        \|\nabla^4 \wt{p}\|_{L^2}  \\[-4mm]\nonumber\\\nonumber
%& \lesssim \big(\underbrace{\nabla^3 A}_{L^2} \underbrace{\nabla A}_{L^\infty}
%           + \underbrace{\nabla^2 A}_{L^4} \underbrace{\nabla^2 A}_{L^4} \big)    \underbrace{\nabla \wt{p}}_{L^\infty}
%       +\big(\underbrace{\nabla^3 A}_{L^2} \cdot  \underbrace{A}_{L^\infty}
%             + \underbrace{\nabla^2 A}_{L^4} \underbrace{\nabla  A}_{L^4} \big)      \underbrace{ \nabla^2 \wt{p}}_{L^2} \\
%&\quad +\big(\underbrace{\nabla^2 A \cdot}_{L^4}  \underbrace{A}_{L^\infty}
%         + \underbrace{\nabla A}_{L^4} \underbrace{\nabla A}_{L^\infty}  \big)\underbrace{\nabla^3 \wt{p}}_{L^4}
%       +\underbrace{\nabla A}_{L^\infty}  \cdot  \underbrace{A}_{L^\infty}        \underbrace{\nabla^4 \wt{p}}_{L^2}  \\
&\lesssim \mathcal{F}(1+\| \nabla\psi\|_{L^\infty}+\| \nabla^2\psi\|_{L^2}+\| \nabla^4\psi\|_{L^2} )
\big( \| \nabla^4 \wt{p}\|_{L^2}+\| \nabla \wt{p} \|_{L^2}  \big) .
\end{align*}
For $r\geq 4$, by a similar calculation, we get \eqref{H50}.
%In the above, we have used
%\begin{equation*}
%\| \nabla\psi\|_{L^\infty}+\| \nabla^2\psi\|_{L^2}+\| \nabla^4\psi\|_{L^2} \lesssim 1.
%\end{equation*}

%\end{remark}

Next we estimate the third line on the right hand side of \eqref{H49}.
%Note
%\begin{align*}
%&\p_1^r \wt{p}\big|_{\Gamma}=0,\quad
%(J\, A_{2j} \,   \p_1^{r-m} A_{l j} \, \p_l \wt{p})\big|_{\Gamma_-}=0 .
%\end{align*}
% we need use integration by parts to derive that
Similar to \eqref{H50}, by integration by parts, we move away one derivative applied on $\p^r A$ and by \eqref{F11}, \eqref{H34}, Piola's identity \eqref{H3} and the boundary condition for pressure \eqref{H2}, we have
\begin{align}\label{H54}
&-\int_{\Om_L} J \sum_{m<r} C_r^m A_{kj}\p_k (\p_1^{r-m} A_{lj}\p_l\p_1^m\wt{p})
\cdot \p_1^r\wt{p}
\,\langle W^+\rangle^{4\mu}\langle W^-\rangle^{2\mu} \, \dx \\\nonumber
%&=\int_{\Om_L} J \sum_{m<r} C_r^m A_{kj}\p_1^{r-m} A_{lj}\p_l\p_1^m\wt{p}
%\cdot \p_k \big( \p_1^r\wt{p}
%\,\langle W^+\rangle^{4\mu}\langle W^-\rangle^{2\mu}\big) \, \dx \\\nonumber
%&\quad-\int_{\Gamma} J \sum_{m<r} C_r^m A_{2j} \p_1^{r-m} A_{lj}\p_l\p_1^m\wt{p}
%\cdot \p_1^r\wt{p}
%\,\langle W^+\rangle^{4\mu}\langle W^-\rangle^{2\mu} \, \dx_1 \\\nonumber
%&\quad+\int_{\Gamma_-} J \sum_{m<r} C_r^m A_{2j} \p_1^{r-m} A_{lj}\p_l\p_1^m\wt{p}
%\cdot \p_1^r\wt{p}
%\,\langle W^+\rangle^{4\mu}\langle W^-\rangle^{2\mu} \, \dx_1\\\nonumber
&=\int_{\Om_L} J \sum_{m<r} C_r^m A_{kj}\p_1^{r-m} A_{lj}\p_l\p_1^m\wt{p}
\cdot \p_k \big( \p_1^r\wt{p}
\,\langle W^+\rangle^{4\mu}\langle W^-\rangle^{2\mu}\big) \, \dx \\\nonumber
&\lesssim \sum_{m<r} \| A_{kj}\p_1^{r-m} A_{lj}\p_l\p_1^m\wt{p}
\langle W^+\rangle^{2\mu}\langle W^-\rangle^{\mu} \|_{L^2}
\| \langle W^+\rangle^{2\mu}\langle W^-\rangle^{\mu}\, (|\p_k\p_1^r\wt{p}|+|\p_1^r\wt{p}|)
\|_{L^2}\\\nonumber
%&\lesssim (1+\| \nabla \psi \|_{L^\infty})\sum_{2\leq |\alpha|\leq r+1}\| \nabla^\alpha \psi \|_{L^2}\\
&\lesssim
(1+\| \nabla \psi\|_{L^\infty})
\mathcal{F} \big(\| \nabla\psi\|_{L^\infty}+\sum_{2\leq |\alpha|\leq q} \| \nabla^\alpha \psi\|_{L^2} \big) \\\nonumber
%&\quad \cdot \big(\| \nabla\psi\|_{L^\infty}+ \| \nabla^2 \psi\|_{L^\infty}+\sum_{2\leq |\alpha|\leq r+1} \| \nabla^\alpha \psi\|_{L^2} \big)\\
&\quad\cdot\sum_{m\leq r-1} \|\langle W^+\rangle^{2\mu}\langle W^-\rangle^{\mu}  \nabla \p_1^m \wt{p}\|_{L^2}
\big(\| \langle W^+\rangle^{2\mu}\langle W^-\rangle^{\mu} \, \p_1^r\wt{p}\|_{L^2}
+\| \langle W^+\rangle^{2\mu}\langle W^-\rangle^{\mu} \, \nabla\p_1^r\wt{p}\|_{L^2}\big) ,
%\\
%&\lesssim  (1+\| f \|_{L^\infty(\BR)}+\| f \|_{\dot{H}^1(\BR)}+\| f \|_{\dot{H}^2(\BR)})
%(\| f \|_{\dot{H}^1(\BR)}+\| f \|_{\dot{H}^{r+\f12}(\BR)})
%\\
%&\quad \cdot
%\sum_{1\leq |\alpha|\leq r+1} \|\langle W^+\rangle^{2\mu}\langle W^-\rangle^{\mu}  \p^\alpha \wt{p})\|_{L^2}
%\| \langle W^+\rangle^{2\mu}\langle W^-\rangle^{\mu} \, \p_1^r\wt{p}\|_{L^2}.
\end{align}
where
\begin{equation*}
q=
\begin{cases}
4,\quad \text{if}\,\, r=2,\\
r+1,\quad \text{if}\,\, 3\leq r\leq s.
\end{cases}
\end{equation*}

Combining \eqref{H49}-\eqref{H54}, by \eqref{F1} and Lemma \ref{lemTR}, we obtain
\begin{align}\label{H55}
&\big\|\langle W^+\rangle^{2\mu}\langle W^-\rangle^{\mu}
\nabla^\psi\p_1^r \wt{p}\big\|^2_{L^2} \\[-4mm]\nonumber\\\nonumber
&\lesssim
%(1+\| \nabla \psi\|_{L^\infty})
\| \nabla^2 \psi \|_{L^\infty}
\|\langle W^+\rangle^{2\mu}\langle W^-\rangle^{\mu}\,  \p_1^{r-1} \nabla^2 \wt{p}\|_{L^2}^2 \\[-4mm]\nonumber\\\nonumber
%\|\langle W^+\rangle^{2\mu}\langle W^-\rangle^{\mu} \p_1^{r-1} \nabla^2 \wt{p}\|^2_{L^2} \\\nonumber
%&\quad\cdot \mathcal{F}(\|\nabla\psi\|_{L^\infty}+\|\nabla^2\psi\|_{L^\infty}
%+\sum_{2\leq |\alpha|\leq r+1}\| \nabla^\alpha \psi \|_{L^2}) \\\nonumber
&\quad +\Big(\sum_{1\leq|\alpha|\leq r}\|\langle W^+\rangle^{2\mu}\langle W^-\rangle^{\mu} \p^\alpha \wt{p}\|^2_{L^2}+ \mathcal{E}_s \mathcal{G}_s
\Big)  \\\nonumber
&\qquad\cdot\mathcal{F}(1+\|\nabla\psi\|_{L^\infty}
+\sum_{2\leq |\alpha|\leq q}\| \nabla^\alpha \psi \|_{L^2}),
\end{align}
where
\begin{equation*}
q=
\begin{cases}
4,\quad \text{if}\,\, r=2,\\
r+1,\quad \text{if}\,\, 3\leq r\leq s.
\end{cases}
\end{equation*}

\textbf{$(r+1)$-order non-tangential derivative estimate of pressure}:

For multi-index $\alpha$
%derivative $\p^\alpha$
 with $|\alpha|=r-1$, there naturally holds the following identity:
\begin{align}\label{H60}
&\p^\alpha \big(A_{kj}\p_k (A_{lj}\p_l\wt{p}) \big)\\\nonumber
&=\sum_{\beta\neq \alpha } C_\alpha^\beta \p^{\alpha-\beta} A_{kj}\p^\beta\p_k (A_{lj}\p_l\wt{p})
+  A_{kj}\p^\alpha\p_k (A_{lj}\p_l\wt{p}) \\\nonumber
&=\sum_{\beta\neq \alpha} C_\alpha^\beta \p^{\alpha-\beta} A_{kj}\p^\beta\p_k (A_{lj}\p_l\wt{p})
+\sum_{\beta\neq \alpha} C_\alpha^\beta A_{kj}\p_k (\p^{\alpha-\beta} A_{lj}\p_l \p^\beta \wt{p}) \\\nonumber
&\quad +A_{kj}\p_k ( A_{lj}\p_l\p^\alpha\wt{p}).
\end{align}
Combining \eqref{H26} and \eqref{H60}, we obtain
%The above \eqref{H60} can be written in the following form
\begin{align}\label{H61}
[\p^\alpha,\nabla^\psi\cdot \nabla^\psi] \wt{p}
&=\sum_{\beta\neq \alpha} C_\alpha^\beta \p^{\alpha-\beta} A_{kj}\p^\beta\p_k (A_{lj}\p_l\wt{p})
+\sum_{\beta\neq \alpha} C_\alpha^\beta A_{kj}\p_k (\p^{\alpha-\beta} A_{lj}\p_l \p^\beta \wt{p}) \\\nonumber
&=-\p^\alpha(\p^\psi_j \wt{\Lambda}^i_- \p^\psi_i \wt{\Lambda}^j_+)
-A_{kj}\p_k ( A_{lj}\p_l\p^\alpha\wt{p})\\\nonumber
&=-\p^\alpha(\p^\psi_j \wt{\Lambda}^i_- \p^\psi_i \wt{\Lambda}^j_+)
-A_{kj} A_{lj}\p_l\p_k\p^\alpha\wt{p}
-A_{kj}\p_k A_{lj}\p_l\p^\alpha\wt{p}.
\end{align}
On the other hand, simple calculation yields the identity
\begin{equation} \label{H62}
%A_{ji}\, A_{ki}
A_{kj}\, A_{lj}
\, \p_k \p_l \p^\alpha \wt{p}
=\f{1+|\p_1 \psi|^2}{(1+\p_2 \psi)^2} \, \p_2^2 \p^\alpha\wt{p}
+\p_1^2 \p^\alpha\wt{p}
-2\f{\p_1 \psi \, \p_1 \p_2 \p^\alpha \wt{p}}{1+\p_2 \psi}.
\end{equation}
Combining \eqref{H61} and \eqref{H62},
% and write the above equations as follows:
we obtain
\begin{align*} %\label{H63}
\p_2^2 \p^\alpha\wt{p}
&=\f{(1+\p_2 \psi)^2 }{1+|\p_1 \psi|^2}
\Big( -\p^\alpha\big(\p^\psi_j \wt{\Lambda}^i_- \p^\psi_i \wt{\Lambda}^j_+\big)
-A_{kj} \, \p_k A_{lj} \, \p_l \p^\alpha\wt{p}  \\\nonumber
&\qquad
-[\p^\alpha,\nabla^\psi\cdot \nabla^\psi] \wt{p}
-\p_1^2 \p^\alpha \wt{p} +2\f{\p_1 \psi \, \p_1\p_2 \p^\alpha\wt{p}}{1+\p_2 \psi} \Big) .
\end{align*}
Consequently, we deduce
\begin{align}\label{H64}
&\f12\|\langle W^+\rangle^{2\mu}\langle W^-\rangle^{\mu} \p_2^\psi \p_2 \p^\alpha \wt{p} \|_{L^2}
\lesssim \|\langle W^+\rangle^{2\mu}\langle W^-\rangle^{\mu} \p_2^2 \p^\alpha \wt{p} \|_{L^2} \\\nonumber
&\leq (1+\|\p_2\psi\|_{L^\infty})^2
\cdot %\big\{
\Big\|\langle W^+\rangle^{2\mu}\langle W^-\rangle^{\mu}
\big\{\p^\alpha\big(\p^\psi_j \wt{\Lambda}^i_- \p^\psi_i \wt{\Lambda}^j_+\big)
-A_{kj} \, \p_k A_{lj} \, \p_l \p^\alpha \wt{p}  \\\nonumber
&\qquad
-[\p^\alpha,\nabla^\psi\cdot \nabla^\psi] \wt{p}
-\p_1^2 \p^\alpha \wt{p} +2\f{\p_1 \psi \, \p_1\p_2 \p^\alpha\wt{p}}{1+\p_2 \psi}  \big\}\Big\|_{L^2}.
\end{align}
Note $r\geq 2$, $|\alpha|=r-1$. For \eqref{H64}, we first estimate
%(if $r=2$, there is something wrong with the index)
\begin{align}\label{H65}
&\big\|\langle W^+\rangle^{2\mu}\langle W^-\rangle^{\mu}
\p^\alpha\big(\p^\psi_j \wt{\Lambda}^i_- \p^\psi_i \wt{\Lambda}^j_+\big)\|_{L^2}
\\\nonumber
&\lesssim\sum_{|\alpha|\leq r-1} \big\| \langle W^-\rangle^{2\mu} \p^\alpha\nabla^\psi \wt{\Lambda}_- \big\|_{L^2}
\sum_{|\alpha|\leq r-1} \Big\|\f{\langle W^+\rangle^{2\mu}\p^\alpha\nabla^\psi \wt{\Lambda}_+}
{\langle W^-\rangle^{\mu}} \Big\|_{L^2} \\\nonumber
&\leq \mathcal{E}_q^{\f12} \mathcal{G}_q^{\f12},
\end{align}
where
%\begin{align*}
%q&=3,\quad \text{if}\,\, r=2,\\
%q&=r,\quad \text{if}\,\, 3\leq r\leq s.
%\end{align*}
\begin{equation*}
q=
\begin{cases}
3,\quad \text{if}\,\, r=2,\\
r,\quad \text{if}\,\, 3\leq r\leq s.
\end{cases}
\end{equation*}
Next,
\begin{align}
&\big\|\langle W^+\rangle^{2\mu}\langle W^-\rangle^{\mu}
A_{kj} \, \p_k A_{lj} \, \p_l \p^\alpha \wt{p} \|_{L^2} \\[-4mm]\nonumber\\\nonumber
&\leq  \|\, |A|\, |\nabla A| \|_{L^\infty}
 \|\langle W^+\rangle^{2\mu}\langle W^-\rangle^{\mu}\nabla\p^\alpha \wt{p} \|_{L^2} \\[-4mm]\nonumber\\\nonumber
&\lesssim  \sum_{|\alpha|=r-1}\|\langle W^+\rangle^{2\mu}\langle W^-\rangle^{\mu}\nabla\p^\alpha \wt{p} \|_{L^2}
\big(1+\| \nabla\psi\|_{L^\infty}) \| \nabla^2 \psi\|_{L^\infty} .
%&\lesssim  \sum_{|\alpha|=r-1} \|\langle W^+\rangle^{2\mu}\langle W^-\rangle^{\mu}\nabla\p^\alpha \wt{p} \|_{L^2}
%\big(1+\| \nabla\psi\|_{L^\infty})( \| \nabla^2 \psi\|_{L^2}+\| \nabla^3 \psi\|_{L^2}) .
\end{align}
By \eqref{H61}, we estimate
\begin{align}
&\big\|\langle W^+\rangle^{2\mu}\langle W^-\rangle^{\mu}
[\p^\alpha,\nabla^\psi\cdot \nabla^\psi] \wt{p}\big\|_{L^2} \\[-4mm]\nonumber\\\nonumber
&\leq
C \sum_{\beta\neq \alpha}
\big\|\langle W^+\rangle^{2\mu}\langle W^-\rangle^{\mu}
\big(|\p^{\alpha-\beta} A_{kj}\p^\beta\p_k (A_{lj}\p_l\wt{p})|
+|A_{kj}\p_k (\p^{\alpha-\beta} A_{lj}\p_l \p^\beta \wt{p})|\big) \big\|_{L^2} \\\nonumber
&\leq \mathcal{F}\big(1+\| \nabla\psi\|_{L^\infty}+\sum_{2\leq |\alpha|\leq q} \| \nabla^\alpha \psi\|_{L^2} \big) \\\nonumber
&\quad \cdot \big(\| \nabla\psi\|_{L^\infty}+ \| \nabla^2 \psi\|_{L^\infty}+\sum_{2\leq |\alpha|\leq q} \| \nabla^\alpha \psi\|_{L^2} \big) \sum_{ 1\leq|\alpha|\leq r}\|\langle W^+\rangle^{2\mu}\langle W^-\rangle^{\mu}\p^\alpha \wt{p} \|_{L^2},
%&+\big(1+\| \nabla\psi\|_{L^\infty}\big)
%\big(\| \nabla\psi\|_{L^\infty}+ \| \nabla^2 \psi\|_{L^\infty}+\sum_{2\leq |\alpha|\leq r+1} \| \nabla^\alpha \psi\|_{L^2} \big)
%\sum_{ 1\leq|\alpha|\leq r-1}\|\langle W^+\rangle^{2\mu}\langle W^-\rangle^{\mu}\nabla\p^\alpha \wt{p} \|_{L^2} .
\end{align}
where
\begin{equation*}
q=
\begin{cases}
4,\quad \text{if}\,\, r=2,\\
r+1,\quad \text{if}\,\, 3\leq r\leq s.
\end{cases}
\end{equation*}

Finally, we estimate
\begin{align}\label{H67}
&\Big\|\langle W^+\rangle^{2\mu}\langle W^-\rangle^{\mu}
\f{\p_1 \psi \, \p_1\p_2 \p^\alpha\wt{p}}{1+\p_2 \psi} \Big\|_{L^2} \\\nonumber
&\lesssim (1+2\|\p_2\psi\|_{L^\infty}) \|\p_1\psi\|_{L^\infty}
\|\langle W^+\rangle^{2\mu}\langle W^-\rangle^{\mu}\,
 \p_1\p_2\p^\alpha \wt{p} \|_{L^2}.
\end{align}
%Combining all the estimates, we obtain
Consequently, combining \eqref{H65}-\eqref{H67}, \eqref{H64} yields
\begin{align}\label{H68}
&\|\langle W^+\rangle^{2\mu}\langle W^-\rangle^{\mu} \p_2^\psi \p_2 \p^\alpha \wt{p} \|_{L^2} \\ \nonumber
&\lesssim \mathcal{F}\big(1+\| \nabla\psi\|_{L^\infty}+ \| \nabla^2 \psi\|_{L^\infty}+\sum_{2\leq |\alpha|\leq r+1} \| \nabla^\alpha \psi\|_{L^2} \big) \\\nonumber
&\quad \Big(
\big(\| \nabla\psi\|_{L^\infty}+ \| \nabla^2 \psi\|_{L^\infty}+\sum_{2\leq |\alpha|\leq r+1} \| \nabla^\alpha \psi\|_{L^2} \big) \sum_{ 1\leq|\alpha|\leq r}\|\langle W^+\rangle^{2\mu}\langle W^-\rangle^{\mu}\p^\alpha \wt{p} \|_{L^2} \\\nonumber
&\quad+\|\langle W^+\rangle^{2\mu}\langle W^-\rangle^{\mu} \p_1^2\p^\alpha \wt{p} \|_{L^2}+
\| \p_1\psi\|_{L^\infty} \|\langle W^+\rangle^{2\mu}\langle W^-\rangle^{\mu} \p_1\p_2\p^\alpha \wt{p} \|_{L^2}
\Big).
\end{align}
Then combining the tangential derivative estimate \eqref{H55} and the normal derivative estimate \eqref{H68},
for $|\alpha|=r-1$, we obtain %(need update)
\begin{align}\label{H70}
& \|\langle W^+\rangle^{2\mu}\langle W^-\rangle^{\mu} \nabla^\psi \p_1^r \wt{p} \|^2_{L^2}
+\|\langle W^+\rangle^{2\mu}\langle W^-\rangle^{\mu} \p_2^\psi \p_2 \p^\alpha \wt{p} \|^2_{L^2} \\[-4mm] \nonumber\\\nonumber
&\lesssim
\|\langle W^+\rangle^{2\mu}\langle W^-\rangle^{\mu} \p_1^2\p^\alpha \wt{p} \|^2_{L^2}
+\| \p_1\psi\|_{L^\infty} \|\langle W^+\rangle^{2\mu}\langle W^-\rangle^{\mu} \p_1\p_2\p^\alpha \wt{p} \|^2_{L^2}\\[-4mm]\nonumber\\\nonumber
&\quad+\| \nabla^2 \psi \|_{L^\infty}
\|\langle W^+\rangle^{2\mu}\langle W^-\rangle^{\mu}\,  \p_1^{r-1} \nabla^2 \wt{p}\|_{L^2}^2 \\[-4mm]\nonumber\\\nonumber
&\quad +\Big(\sum_{ 1\leq|\alpha|\leq r}\|\langle W^+\rangle^{2\mu}\langle W^-\rangle^{\mu}\p^\alpha \wt{p} \|^2_{L^2}
+\mathcal{E}_s \mathcal{G}_s \Big)\\\nonumber
&\qquad \cdot\mathcal{F}\big(1+\| \nabla\psi\|_{L^\infty}+ \| \nabla^2 \psi\|_{L^\infty}+\sum_{2\leq |\alpha|\leq s+1} \| \nabla^\alpha \psi\|_{L^2} \big).
\end{align}
Now let us show Lemma \ref{lemPreL} by induction.
Firstly, we assume  for
%the bound for
%\begin{align*}
%\sum_{|\alpha|\leq r-1}\| \langle W^\pm\rangle^{2\mu}
%\langle W^\mp\rangle^{\mu} \nabla^\psi\p^\alpha \wt{p}\|_{L^2}
%\end{align*}
 $1\leq r\leq s-1$, there holds
\begin{align}\label{H71}
&\sum_{|\alpha|\leq r-1}\| \langle W^\pm\rangle^{2\mu}
\langle W^\mp\rangle^{\mu} \nabla^\psi\p^\alpha \wt{p}\|_{L^2}^2\\\nonumber
&\leq %\sum_{|\alpha|\leq r-2,\ \pm} \| \langle W^\pm\rangle^{2\mu}  \p^\alpha \nabla^\psi\wt{\Lambda}_\pm\|_{L^2}
%      \sum_{|\alpha|\leq r-2,\ \mp} \Big\| \f{\langle W^\mp\rangle^{2\mu}  \p^\alpha \nabla^\psi\wt{\Lambda}_\mp}{\langle W^\pm\rangle^{\mu}}\Big\|_{L^2}\\\nonumber
\mathcal{E}_s  \mathcal{G}_s
%&\qquad
\cdot
\mathcal{F}(1+\| \nabla\psi\|_{L^\infty}+ \| \nabla^2 \psi\|_{L^\infty}+\sum_{2\leq |\alpha|\leq s+1} \| \nabla^\alpha \psi\|_{L^2}).
\end{align}
Next, we show the control for
\begin{align*}
\sum_{|\alpha|= r}\| \langle W^\pm\rangle^{2\mu}
\langle W^\mp\rangle^{\mu} \nabla^\psi\p^\alpha \wt{p}\|^2_{L^2} .
\end{align*}
For \eqref{H70}, taking $\p^\alpha=\p_1^{r-1}$. Then by \eqref{F1}, \eqref{H55} and \eqref{H71},  \eqref{H70} yields
\begin{align*}%\label{H73}
& \|\langle W^+\rangle^{2\mu}\langle W^-\rangle^{\mu} \nabla^\psi \p_1^r \wt{p} \|^2_{L^2}
+\|\langle W^+\rangle^{2\mu}\langle W^-\rangle^{\mu} \p_2^\psi \p_2 \p^{r-1}_1 \wt{p} \|^2_{L^2}
\\\nonumber
&\lesssim \mathcal{E}_s  \mathcal{G}_s
\cdot \mathcal{F}(1+\| \nabla\psi\|_{L^\infty}+ \| \nabla^2 \psi\|_{L^\infty}+\sum_{2\leq |\alpha|\leq s+1} \| \nabla^\alpha \psi\|_{L^2}).
\end{align*}
Next, we take one by one $\p^\alpha=\p_1^{r-2}\p_2$, $\p^\alpha=\p_1^{r-3}\p_2^2$, ...
$\p^\alpha=\p_2^{r-1}$,
%\eqref{H70} yields
we obtain
\begin{align*}
&\sum_{|\alpha|= r}\| \langle W^\pm\rangle^{2\mu}
\langle W^\mp\rangle^{\mu} \nabla^\psi\p^\alpha \wt{p}\|_{L^2} \\\nonumber
&\lesssim
\mathcal{E}_s^{\f12} \mathcal{G}_s^{\f12}
\cdot \mathcal{F}(1+ \| \nabla\psi\|_{L^\infty}+ \| \nabla^2 \psi\|_{L^\infty}+\sum_{2\leq |\alpha|\leq s+1} \| \nabla^\alpha \psi\|_{L^2}).
\end{align*}

In summary, we obtain that for $s\geq 3$, $\frac12< \mu \leq \frac34$, there holds
%there holds Let $s\geq 4$ be an integer,
% Under the condition of \eqref{F1}. For pressure,
%there holds
\begin{align*}
&\| \langle W^\pm\rangle^{2\mu}
\langle W^\mp\rangle^{\mu}  \wt{p}\|^2_{L^2}
+\sum_{|\alpha|\leq s}\| \langle W^\pm\rangle^{2\mu}
\langle W^\mp\rangle^{\mu} \nabla^\psi\p^\alpha \wt{p}\|^2_{L^2}\\
&\lesssim\mathcal{E}_s \mathcal{G}_s
\cdot\mathcal{F}(1+ \| \nabla\psi\|_{L^\infty}+ \| \nabla^2 \psi\|_{L^\infty}+\sum_{2\leq |\alpha|\leq s+1} \| \nabla^\alpha \psi\|_{L^2}).
\end{align*}
By the estimate of $\psi$ in Lemma \ref{lemTR}, the above can be further controlled by
\begin{align*}
\mathcal{E}_s \mathcal{G}_s
\cdot \mathcal{F}(1+\| f \|_{L^\infty(\BR)}+\| f \|_{\dot{H}^1(\BR)}+\| f \|_{\dot{H}^{s+\f12}(\BR)}).
\end{align*}
This finishes the proof of Lemma \ref{lemPreL}.
\end{proof}

\section{Weighted energy estimate of the plasma}\label{sec-energy-b}

This section is devoted  to the weighted energy estimate of unknowns in the bulk of the plasma.
%In this section, we are going to estimate the amplitude of the free surface and
%deal with the weighted energy estimate of the free surface.
%The estimation the amplitude of the free surface is carried out using the characteristic integral method.
%Here, we only  require the estimation of the amplitude of the free surface without weights.
Under the assumption \eqref{AA4}, %\eqref{AA5},
\eqref{AA7} and \eqref{AA8}, the energy estimate is performed by applying the weights $\langle w^\pm\rangle^{2\mu}$
and by utilizing the ghost weight method. Here, we recall that the weight functions are defined by \eqref{C0}.

Let $s\geq 4$ be an integer, $1/2<\mu \leq 3/4$ and $q(z)=\int_0^z \langle\tau\rangle^{-2\mu} d\tau$,
so that $|q(z)|\lesssim 1$.
Let $\sigma^- = -w^+$, $\sigma^+ = w^-$, $|\alpha|\leq s$.
Denote $e^{q^+}=e^{q(\sigma^+)}=e^{q(w^-)}$, $e^{q^-}=e^{q(\sigma^-)}=e^{q(-w^+)}$.
Due to the boundary condition on $\Gamma_-$ and $\Ga_f$, we calculate
\begin{align}\label{E2}
&\f12\f{\d}{\dt}  \int_{\Omega_f}
| \langle w^+ \rangle^{2\mu}\nabla^\alpha \Lambda_+ |^2 e^{q^+} \dx\\\nonumber
&=\f12  \int_{\Omega_f}\p_t\big(
| \langle w^+ \rangle^{2\mu}\nabla^\alpha \Lambda_+ |^2 e^{q^+} \big)\dx
+\f12  \int_{\Ga_f}
| \langle w^+ \rangle^{2\mu}\nabla^\alpha \Lambda_+ |^2 e^{q^+}
Z_- \cdot n_f \d\sigma \\\nonumber
&=\f12  \int_{\Omega_f}( \p_t+Z_-\cdot\nabla) \big(
| \langle w^+ \rangle^{2\mu}\nabla^\alpha \Lambda_+ |^2 e^{q^+} \big)\dx \\\nonumber
&=\f12  \int_{\Omega_f}( \p_t+Z_-\cdot\nabla) \big(
| \langle w^+ \rangle^{2\mu}\nabla^\alpha \Lambda_+|^2\big) e^{q^+}  \dx \\\nonumber
&\quad +\f12  \int_{\Omega_f}\big(
| \langle w^+ \rangle^{2\mu}\nabla^\alpha \Lambda_+ |^2\big) ( \p_t+Z_-^1\p_1)  e^{q^+} \dx \\\nonumber
&\quad +\f12  \int_{\Omega_f}\big(
| \langle w^+ \rangle^{2\mu}\nabla^\alpha \Lambda_+ |^2\big) ( Z_-^2 \p_2)  e^{q^+} \dx,
\end{align}
where $n_f=N_f/|N_f|$.
On the other side, for any multi-index $\alpha\in\mathbb{N}^2$, $0\leq |\alpha|\leq s$.
Applying derivative $\nabla^\alpha$ onto \eqref{MHD1} gives
\begin{equation} \label{E3}
\begin{cases}
(\p_t +Z_-\cdot \nabla) \nabla^\alpha \Lambda_+
+[\nabla^\alpha,\Lambda_-\cdot\nabla] \Lambda_+ + \nabla \nabla^\alpha p =0,\\
(\p_t +Z_+\cdot \nabla) \nabla^\alpha \Lambda_-
+[\nabla^\alpha,\Lambda_+\cdot\nabla] \Lambda_- + \nabla \nabla^\alpha p =0,
\end{cases}
\textrm{in} \quad \Omega_f,
\end{equation}
%\begin{equation} \label{E2}
%\begin{cases}
%(\p_t +Z_-^1\p_1) \nabla^\alpha \Lambda_+ +\Lambda_-^2\p_2 \nabla^\alpha \Lambda_+
%+[\nabla^\alpha,\Lambda_-\cdot\nabla] \Lambda_+ + \nabla \nabla^\alpha p =0 ,\\[-4mm]\\
%(\p_t +Z_+^1\p_1) \nabla^\alpha \Lambda_- +\Lambda_+^2\p_2 \nabla^\alpha \Lambda_-
%+[\nabla^\alpha,\Lambda_+\cdot\nabla] \Lambda_- + \nabla \nabla^\alpha p =0 \end{cases}
% \textrm{in} \quad \Omega_f ,
%\end{equation}
where we have used the following commutator notation
$$
[\nabla^\alpha,\Lambda_-\cdot\nabla] \Lambda_+
=\nabla^\alpha(\Lambda_-\cdot\nabla \Lambda_+)
-\Lambda_-\cdot\nabla \nabla^\alpha\Lambda_+.
$$
%Here \eqref{E2} is organized in such way in order to conveniently apply the weight functions $w^\pm$.
%the above equations
%Applying the weight fun $\langle w^\pm \rangle^{2\mu} $ onto \eqref{E2},
Now multiplying the first equation of \eqref{E3} by $\langle w^+ \rangle^{4\mu}\nabla^\alpha\Lambda_+ $
and multiplying the second equation of \eqref{E3} by
$\langle w^- \rangle^{4\mu}\nabla^\alpha\Lambda_-$,
% onto the first and the second equation of \eqref{E2},
respectively,
one has
\begin{equation} \label{E4}
\begin{cases}
\f12(\p_t+Z_-\cdot\nabla )(\langle w^+ \rangle^{2\mu}\nabla^\alpha \Lambda_+)^2
-\f12\p_2\langle w^+\rangle^{4\mu}|\nabla^\alpha \Lambda_+|^2  \Lambda_-^2  \\[-5mm]\\
\quad
+\langle w^+\rangle^{4\mu}\nabla^\alpha \Lambda_+ \big([\nabla^\alpha,\Lambda_-\cdot\nabla] \Lambda_+ + \nabla \nabla^\alpha p \big) =0 ,\\[-5mm]\\
\f12(\p_t+Z_+\cdot\nabla )(\langle w^- \rangle^{2\mu}\nabla^\alpha \Lambda_-)^2
-\f12\p_2\langle w^-\rangle^{4\mu}|\nabla^\alpha \Lambda_-|^2  \Lambda_+^2 \\[-5mm]\\
\quad+\langle w^-\rangle^{4\mu} \nabla^\alpha \Lambda_- \big(
[\nabla^\alpha,\Lambda_+\cdot\nabla] \Lambda_- + \nabla \nabla^\alpha p \big) =0,
\end{cases}
\textrm{in} \ \ \Omega_f.
\end{equation}
%\end{remark}
Plugging \eqref{E4} into \eqref{E2}, one has
\begin{align}\label{E5}
&\f12\frac{\d}{\dt}  \int_{\Omega_f}
| \langle w^+ \rangle^{2\mu}\nabla^\alpha \Lambda_+|^2 e^{q^+}
+ |\langle w^- \rangle^{2\mu}\nabla^\alpha \Lambda_-|^2 e^{q^-} \dx \\
&-\f12\int_{\Omega_f}
|\langle w^+ \rangle^{2\mu}\nabla^\alpha \Lambda_+ |^2
(\p_t + Z_-^1\p_1)e^{q^+}  \dx
-\frac12\int_{\Omega_f}
|\langle w^- \rangle^{2\mu}\nabla^\alpha \Lambda_- |^2
(\p_t + Z_+^1\p_1)e^{q^-}  \dx\nonumber\\
&=
-\int_{\Omega_f}
\langle w^+ \rangle^{4\mu}\nabla^\alpha \Lambda_+ e^{q^+} \cdot
 [\nabla^\alpha,\Lambda_-\cdot\nabla] \Lambda_+
+\langle w^- \rangle^{4\mu}\nabla^\alpha \Lambda_- e^{q^-} \cdot
[\nabla^\alpha,\Lambda_+\cdot\nabla] \Lambda_-  \big) \dx \nonumber\\\nonumber
&\quad -\int_{\Omega_f}
\langle w^+ \rangle^{4\mu}\nabla^\alpha \Lambda_+ e^{q^+} \cdot
 \nabla \nabla^\alpha p
+\langle w^- \rangle^{4\mu}\nabla^\alpha \Lambda_- e^{q^-} \cdot
\nabla \nabla^\alpha p\,  \dx \nonumber\\\nonumber
&\quad
+\frac12\int_{\Omega_f}
|\nabla^\alpha \Lambda_+ |^2\Lambda_-^2\p_2\big(\langle w^+ \rangle^{4\mu} e^{q^+}\big)  \dx
+\frac12\int_{\Omega_f}
|\nabla^\alpha \Lambda_- |^2
\Lambda_+^2\p_2\big(\langle w^- \rangle^{4\mu} e^{q^-}\big)  \dx.
\end{align}
The second line  of \eqref{E5} corresponds to the flux energy of the ghost weight,
%Same to the estimate of the second line of \eqref{E3}, we have the lower bound
by Lemma \ref{lemC2} and the bootstrap assumption $|\Lambda_\pm|\leq C\epsilon$, we compute:
\begin{align*}
&-(\p_t + Z_-^1\p_1) e^{q^+}
=-\frac{e^{q^+}}{\langle w^- \rangle^{2\mu}}
 (\p_t + Z_-^1\p_1) w^-
%= -\frac{e^{q^+}}{\langle w^- \rangle^{2\mu}}
%\big( (\p_t -\p_1) w^-  +\Lambda_-^1\p_1 w^- \big)
\geq \frac32 \frac{e^{q^+}}{\langle w^- \rangle^{2\mu}} ,\\
&-(\p_t + Z_+^1\p_1) e^{q^-}
=\frac{e^{q^-}}{\langle w^+ \rangle^{2\mu}}
 (\p_t + Z_+^1\p_1) w^+
 %=\frac{e^{q^-}}{\langle w^+ \rangle^{2\mu}}
%\big( (\p_t + \p_1) w^+ + \Lambda_+^1\p_1 w^+
%\big)
\geq \frac32\frac{e^{q^-}}{\langle w^+ \rangle^{2\mu}}.
\end{align*}
%While
%\begin{align*}
%(\p_t + Z_-^1\p_1) w^-
%&= (\p_t + Z_+^1\p_1) w^-
%+(\Lambda_-^1 - \Lambda_+^1)\p_1 w^-   -2e_1\cdot\nabla w^+\\[-4mm]\\
%&=(\Lambda_-^1 - \Lambda_+^1)\p_1 w^-   -2e_1\cdot\nabla w^+ ,\\[-4mm]\\
%(\p_t + Z_+^1\p_1) w^+
%&= (\p_t + Z_-^1\p_1) w^+
%+(\Lambda_+^1 - \Lambda_-^1)\p_1 w^+   +2e_1\cdot\nabla w^+\\[-4mm]\\
%&=(\Lambda_+^1 - \Lambda_-^1)\p_1 w^+   +2e_1\cdot\nabla w^+ .
%\end{align*}
%Hence by Lemma \ref{lemC2} and the bootstrap assumption $|\Lambda_\pm|\leq C\epsilon$, there hold
%\begin{align*}
%-(\p_t + Z_-^1\p_1) e^{q^+}
%&\geq \frac32 \frac{e^{q^+}}{\langle w^- \rangle^{2\mu}} ,\\
%-(\p_t + Z_+^1\p_1) e^{q^-}
%&\geq \frac32\frac{e^{q^-}}{\langle w^+ \rangle^{2\mu}}.
%\end{align*}
Consequently, the second line  of \eqref{E5} has the following lower bound
\begin{align} \label{E9}
&-\f12\int_{\Omega_f}
|\langle w^+ \rangle^{2\mu}\nabla^\alpha \Lambda_+ |^2
(\p_t + Z_-^1\p_1)e^{q^+}  \dx
\geq  \f34\int_{\Omega_f}
\f{|\langle w^+ \rangle^{2\mu}\nabla^\alpha \Lambda_+ |^2}{\langle w^- \rangle^{2\mu}}
e^{q^+}  \dx,\\ %\nonumber
&-\frac12\int_{\Omega_f}
|\langle w^- \rangle^{2\mu}\nabla^\alpha \Lambda_- |^2
(\p_t + Z_+^1\p_1)e^{q^-}  \dx
\geq\f34\int_{\Omega_f}
\f{|\langle w^- \rangle^{2\mu} \nabla^\alpha \Lambda_- |^2}{\langle w^+ \rangle^{2\mu}}
e^{q^-}  \dx .
\end{align}
For the first commutator term on the right hand side of \eqref{E5}, by Lemma \ref{Sobo1}, one has
\begin{align}
&-\int_{\Omega_f}
\langle w^+ \rangle^{4\mu}\nabla^\alpha \Lambda_+ e^{q^+} \cdot
 [\nabla^\alpha,\Lambda_-\cdot\nabla] \Lambda_+ \dx \\\nonumber
&=- \sum_{\beta+\gamma=\alpha,\, \gamma\neq \alpha}C_\alpha^\beta\int_{\Omega_f}
\langle w^+ \rangle^{4\mu}\nabla^\alpha \Lambda_+ e^{q^+} \cdot
 (\nabla^\beta \Lambda_-\cdot\nabla \nabla^\gamma\Lambda_+) \dx \\\nonumber
&\lesssim \sum_{\beta+\gamma=\alpha,\, \gamma\neq \alpha}\int_{\Omega_f}
\frac{|\langle w^+ \rangle^{2\mu}\nabla^\alpha \Lambda_+|}{\langle w^-\rangle^{\mu}}
\frac{|\langle w^+ \rangle^{2\mu}\nabla\nabla^\gamma \Lambda_+|}{\langle w^- \rangle^{\mu}}
|\langle w^- \rangle^{2\mu}\nabla^\beta \Lambda_-| e^{q^+} \dx \\\nonumber
&\lesssim E_s^{\f12} G_s.
\end{align}
%In the above inequaity, we require $s\geq 3$.
The second commutator term on the right hand side of \eqref{E5} can be estimated similarly.
As for the first term in last line of \eqref{E5}, by Lemma \ref{lemC2} and Lemma \ref{Sobo1}, we have
\begin{align}
&\frac12\int_{\Omega_f}
|\nabla^\alpha \Lambda_+ |^2
\Lambda_-^2\p_2\big(\langle w^+ \rangle^{4\mu} e^{q^+}\big)  \dx\\\nonumber
&=\f12\int_{\Omega_f}
|\nabla^\alpha \Lambda_+ |^2
\Lambda_-^2\cdot \Big(4\mu \langle w^+ \rangle^{4\mu-2}w^+ \p_2w^+
+\f{\langle w^+ \rangle^{4\mu}}{\langle w^-\rangle^{2\mu}} \p_2w^- \Big)e^{q^+}  \dx\\\nonumber
%&\leq \int_{\Omega_f}
%\f{|\langle w^- \rangle^{2\mu} \nabla^\alpha \Lambda_- |}{\langle w^+ \rangle^{\mu}}
%|\langle w^- \rangle^{2\mu}|\nabla^\alpha \Lambda_- |
%\f{|\langle w^+ \rangle^{\mu} \Lambda_+^2|}{\langle w^- \rangle^{\mu}}
%\langle w^- \rangle^{\mu-1}  e^{q^-}  \dx\\
&\lesssim \int_{\Omega_f}
\f{|\langle w^+ \rangle^{2\mu} \nabla^\alpha \Lambda_+ |^2}{\langle w^- \rangle^{2\mu}}
%|\langle w^- \rangle^{2\mu}\nabla^\alpha \Lambda_- |
|\langle w^- \rangle^{2\mu} \Lambda_-^2|  e^{q^+}  \dx\\\nonumber
%&\quad+\int_{\Omega_f}
%\f{|\langle w^- \rangle^{2\mu} \nabla^\alpha \Lambda_- |^2}{\langle w^+ \rangle^{2\mu}}
%%|\langle w^- \rangle^{2\mu}|\nabla^\alpha \Lambda_- |
%|\Lambda_+^2|    e^{q^-}  \dx.
%\end{align*}
%By Lemma \ref{Sobo1}, the above is bounded by
%\begin{align*}
&\lesssim E_s^{\frac12} G_s.
\end{align}
The second term in the last line of \eqref{E5} has the same bound.
As for the last second line of \eqref{E5}, by Lemma \ref{lemPre}, we obtain
\begin{align}\label{E13}
&-\int_{\Omega_f}(
 e^{q^+}\langle w^+ \rangle^{4\mu}\nabla^\alpha \Lambda_+
+e^{q^-}\langle w^- \rangle^{4\mu}\nabla^\alpha \Lambda_- )
\cdot
\nabla \nabla^\alpha  p\,  \dx \\\nonumber
&\lesssim \int_{\Omega_f}
\f{\langle w^+ \rangle^{2\mu}|\nabla^\alpha \Lambda_+|}{\langle w^- \rangle^{\mu}}
\cdot \langle w^+ \rangle^{2\mu}\langle w^- \rangle^{\mu}|\nabla \nabla^\alpha  p| \,  \dx \\\nonumber
&\quad +\int_{\Omega_f} \f{\langle w^- \rangle^{2\mu}|\nabla^\alpha \Lambda_-|}{\langle w^+ \rangle^{\mu}}
\cdot \langle w^- \rangle^{2\mu}\langle w^+ \rangle^{\mu}|\nabla \nabla^\alpha  p| \,  \dx \\\nonumber
&\lesssim E_s^{\frac12} G_s.
\end{align}
In summary, combining \eqref{E9}-\eqref{E13}, \eqref{E5} yields
\begin{align*}
&\f12\frac{\d}{\dt}  \int_{\Omega_f}
| \langle w^+ \rangle^{2\mu}\nabla^\alpha \Lambda_+|^2 e^{q^+}
+ |\langle w^- \rangle^{2\mu}\nabla^\alpha \Lambda_-|^2 e^{q^-} \dx \\
&+\f34\int_{\Omega_f}
\f{|\langle w^+ \rangle^{2\mu}\nabla^\alpha \Lambda_+ |^2}{\langle w^- \rangle^{2\mu}}
e^{q^+}+
\f{|\langle w^- \rangle^{2\mu} \nabla^\alpha \Lambda_- |^2}{\langle w^+ \rangle^{2\mu}}
e^{q^-}  \dx \\
&\lesssim  E_s^{\f12}  G_s.
\end{align*}
%%ANd
%%\begin{align}
%%&\frac12\frac{d}{dt}\int_{\mathbb{R}} e^{\qu^-}
%%\big[\langle\wu^-\rangle^{2\mu} (\p_t + \ud{Z_-^1}\p_1)\langle\p_1\rangle^{\f12}\p_1^a f \big]^2 \dx_1
%%\lesssim E_s  G_s^{\f12}.
%%\end{align}
%Taking the sum for all $|\alpha|\leq s$. Note $e^{q^\pm}\sim 1$, we obtain
%\begin{align*}
%\frac{d}{dt} E^b_s+G^b_s \lesssim E_s^{\f12}  G_s.
%\end{align*}
Summing over all $|\alpha|\leq s$ and taking the integral over $[0,t]$,
and noting that the same estimate also holds for $\La_-$ and $e^{q^\pm}\sim 1$, we obtain
\begin{align*}
E^b_s(t)+\int_0^t G^b_s(\tau)\d\tau \lesssim E^b_s(0)+\int_0^t E_s^{\f12}  G_s (\tau)\d\tau.
\end{align*}
This yields %the first \emph{a priori} estimate
\eqref{AA1}.

\section{Weighted estimate for the free surface}

\subsection{Estimate of the amplitude of the free surface}\label{sec-amp-f}

In this subsection, we are going to show
\begin{align}\label{G3}
\| f(t,\cdot)\|_{L^\infty(\BR)} \leq
\| f(0,\cdot)\|_{L^\infty(\BR)}
+ M\sup_{0\leq \tau\leq t} \| \langle w^\pm \rangle^{2\mu}   \Lambda_\pm(\tau,\cdot) \|_{L^\infty},
%(E_s^b(\tau))^{\f12} \int_{-\infty}^{\infty} \langle z  \rangle^{-2\mu} \d z.
\end{align}
where $M=\int_{\bR} \langle z \rangle^{-2\mu} \d z$.

Firstly, the kinetic boundary equation $\p_t f=\ud{Z_\pm}\cdot N_f$ can be written as follows:
\begin{align*}
%(\p_t +\ud{Z_-^1}\p_1)f=\ud{Z^2_-}=\ud{\Lambda^2_-},
%\quad (\p_t+\ud{Z_+^1}\p_z)f=\ud{\Lambda_+^2}.
(\p_t +\ud{Z_\pm^1}\p_1)f=\ud{Z^2_\pm}=\ud{\Lambda^2_\pm}.
\end{align*}
To solve the above equations,
%we introduce the following characteristics %defined in \eqref{C2}
we draw two backward characteristics
$\phi_\pm(\tau,\sigma_\pm)$ from $(t,x_1)$ and intersect with the plane $\tau=0$ by two points $\sigma_\pm$:
\begin{align*} %\label{C2}
\begin{cases}
\f{\d}{\d\tau} \phi_\pm(\tau,\sigma_\pm)=Z^1_\pm(\tau,\phi_\pm (\tau, \sigma_\pm),f(\tau,\phi_\pm(\tau,\sigma_\pm)) ),\\
\phi_\pm (0,\sigma_\pm)=\sigma_\pm \in\BR ,\quad  \phi_\pm (t, \sigma_\pm)=x_1.
\end{cases}
\end{align*}
Note that the above characteristics are different from the ones defined by \eqref{C2}.
Along the above characteristics, % $\phi_\pm(s,\sigma_\pm,x_2)$,
%$f(t,x_1)$ satisfies
we have
\begin{align*}
\f{\d }{\d\tau} f(\tau,\phi_\pm(\tau,\sigma_\pm) )
&=(\p_tf+Z^1_\pm\p_1f) (\tau, \phi_\pm(\tau,\sigma_\pm),x_2 )|_{x_2=f(\tau,\phi_\pm(\tau,\sigma_\pm))} \\
&=\Lambda^2_\pm(\tau,\phi_\pm (\tau,\sigma_\pm),x_2)|_{x_2=f(\tau,\phi_\pm(\tau,\sigma_\pm))}.
\end{align*}
%Hence
%\begin{align}\label{C3}
%w^\mp(t,r,z )=w^\mp(t,r,\phi_\pm(t,r,\sigma_\pm) )=w^\mp(0,r,\phi_\pm(0,r,\sigma_\pm) )
%= w^\mp(0,r, \sigma_\pm )=\sigma_\pm.
%\end{align}
Taking integrals along the characteristics over $[0,t]$,
 respectively, one has
\begin{align}\label{G4}
f(t,x_1)=
f(0,\sigma_\pm)+\int_0^t \Lambda_\pm^2(\tau,\phi_\pm(\tau,\sigma_\pm),x_2)|_{x_2=f(\tau,\phi_\pm(\tau,\sigma_\pm))} \d\tau.
\end{align}
By Lemma \ref{lemC2}, we derive
\begin{align*}
&\int_0^t |(\Lambda_\pm^2)(\tau,\phi_\pm(\tau,\sigma_\pm),x_2)|_{x_2=f(\tau,\phi_\pm(\tau,\sigma_\pm))} | \d\tau\\
&\leq  \sup_{0\leq \tau\leq t}\| \langle w^\pm \rangle^{2\mu}   \Lambda_\pm^2(\tau,\cdot) \|_{L^\infty}
\int_0^t  \langle w^\pm (\tau,x) \rangle^{-2\mu} |_{x_1= \phi_\pm (\tau,\sigma_\pm),x_2=f(\tau,\phi_\pm(\tau,\sigma_\pm))} \d\tau\\
&\leq \sup_{0\leq \tau\leq t}\| \langle w^\pm \rangle^{2\mu}   \Lambda_\pm^2(\tau,\cdot) \|_{L^\infty}
\int_{-\infty}^{\infty} \langle w^\pm  \rangle^{-2\mu} \d w^\pm
\cdot\sup_{\tau}\f{1}{\big|\f{\d w^\pm(\tau,\phi_\pm (\tau,\sigma_\pm),f(\tau,\phi_\pm(\tau,\sigma_\pm))) }{\d\tau}\big|},
\end{align*}
where
\begin{align*}
&\big| \f{\d w^\pm(\tau,\phi_\pm (\tau,\sigma_\pm),f(\tau,\phi_\pm(\tau,\sigma_\pm))) }{\d\tau} \big| \\
&=\big| \big(\p_tw^\pm(\tau,x) +\p_1w^\pm(\tau,x) Z_\pm^1 (\tau,x)
 +\p_2w^\pm(\tau,x)  \Lambda^2_\pm(\tau,x) \big)|_{x_1= \phi_\pm (\tau,\sigma_\pm) ,x_2=f(\tau,\phi_\pm(\tau,\sigma_\pm))} \big|\\
&\geq 1.
\end{align*}
Consequently, \eqref{G4} yields the desired bound \eqref{G3}.
This yields \eqref{AA3}.

\subsection{Weighted energy estimate of the free surface}\label{sec-energy-f}
%In order to be compatible with the weight functions inside the domain occupied by plasmas,
Recalling that the weight functions on the free surface are defined as the trace of $w^\pm$ on $\Ga_f$:
%we define the weight functions on the free surface as follows
$$\wu^\pm (t,x_1)=w^\pm(t,x)\big|_{x_2=f(t,x_1)}.$$
For positive integer $s\geq 4$ and $1/2<\mu \leq 3/4$,
due to the bootstrap assumption $E_s\ll 1$,
employing the weight functions,  the weighted energy of the free surface
and the weighted ghost weight energy defined in Section \ref{energy} have the following equivalent form:
\begin{align*}%\label{G2}
E^f_{s+\f12} (t)\sim & \sum_{+,-} \sum_{|a|\leq s-1} \big\| \langle\wu^\pm\rangle^{2\mu}(\p_t + \ud{Z_\pm^1}\p_1)\langle\p_1\rangle^{\f12}\p_1^a f \big\|_{L^2(\bR)}^2,\\
G^f_{s+\f12} (t)\sim &\sum_{+,-}\sum_{|a|\leq s-1}
\Big\|  \f{\langle\wu^\pm\rangle^{2\mu}(\p_t + \ud{Z_\pm^1}\p_1)\langle\p_1\rangle^{\f12}\p_1^a f }
{\langle\wu^\mp\rangle^{\mu}} \Big\|_{L^2(\bR)}^2\,.\nonumber
\end{align*}
We will use these latter norms in the weighted energy estimate for the free surface.

Now we turn to the energy estimate of the free surface.
Due to the similarity of the left Alfv\'{e}n waves and the right Alfv\'{e}n waves,
we only present the estimate for
\begin{align*}
&  \sum_{|a|\leq s-1} \Big(\big\| \langle\wu^-\rangle^{2\mu}(\p_t + \ud{Z_-^1}\p_1)\langle\p_1\rangle^{\f12}\p_1^a f \big\|_{L^2(\bR)}^2
 +\Big\|  \f{\langle\wu^-\rangle^{2\mu}(\p_t + \ud{Z_-^1}\p_1)\langle\p_1\rangle^{\f12}\p_1^a f }
{\langle\wu^+\rangle^{\mu}} \Big\|_{L^2(\bR)}^2\Big)\,.\nonumber
\end{align*}
The estimate for the other weighted energy and weighted ghost weight energy of the free surface
is the same.

By Lemma \ref{lemf}, we write the equation for the surface  as follows:
\begin{align}\label{G5}
(\p_t + \ud{Z_+^1}\p_1)(\p_t + \ud{Z_-^1}\p_1)f=-\ud{\p_2 p}.
\end{align}
Now we introduce the ghost weight on the free surface.
Let $1/2<\mu \leq 3/4$ and $q(z)=\int_0^z \langle\tau\rangle^{-2\mu} d\tau$,
so that $|q(z)|\lesssim 1$.
Let $\sigu^-=-\wu^+$, $\sigu^+=\wu^-$.
Denote $e^{\qu^+}=e^{q(\sigu^{+})}=e^{q(\wu^-)}$ and $e^{\qu^-}=e^{q(\sigu^{-})}=e^{q(-\wu^+)}$ for simplicity.

Let $s\geq 4$ be an integer, $|a|\leq s-1$.
Applying the derivative
$\langle\p_1\rangle^{\f12}\p_1^a $ onto \eqref{G5}, then one has
\begin{align}\label{G8}
&(\p_t + \ud{Z_+^1}\p_1)(\p_t + \ud{Z_-^1}\p_1) \langle\p_1\rangle^{\f12}\p_1^a f
 \\[-4mm]\nonumber\\\nonumber
&= -[\langle\p_1\rangle^{\f12}\p_1^a,
\ud{\Lambda_-^1} \ud{\Lambda_+^1}] \p_1^2 f
-[\langle\p_1\rangle^{\f12}\p_1^a,\ud{\Lambda_+^1}]
  \p_1 (\p_t-\p_1) f
-[\langle\p_1\rangle^{\f12}\p_1^a,\ud{\Lambda_-^1}]
  \p_1 (\p_t+\p_1) f \\[-4mm]\nonumber\\\nonumber
&\quad
-\langle\p_1\rangle^{\f12}\p_1^a (N_f \cdot\ud{\nabla p})
-\ud{\p_1 p} \cdot \langle\p_1\rangle^{\f12}\p_1^{a+1} f
:=F_{s+\f12}.
\end{align}

A simple application of chain rules yields
\begin{align}\label{G11}
&(\p_t + \ud{Z_+^1} \p_1)\wu^-
%=\ud{(\p_t + Z_+^1\p_1)w^-}
%+\ud{\p_2 w}^- (\p_t + \ud{Z_+^1} \p_1)f
=\ud{\p_2 w}^- (\p_t + \ud{Z_+^1} \p_1)f
,\\\nonumber
&(\p_t + \ud{Z_-^1}\p_1)\wu^+
%=\ud{(\p_t + Z_-^1\p_1)w^+}+\ud{\p_2 w}^+ (\p_t + \ud{Z_-^1} \p_1)f
=\ud{\p_2 w}^+ (\p_t + \ud{Z_-^1} \p_1)f.
%\label{G12}
\end{align}
Multiplying \eqref{G8} by  $\langle \ud{w}^-\rangle^{2\mu}$,  by \eqref{G11}, one has
\begin{align*}
&(\p_t + \ud{Z_+^1}\p_1)
\big[\langle \ud{w}^-\rangle^{2\mu}(\p_t + \ud{Z_-^1}\p_1) \langle\p_1\rangle^{\f12}\p_1^a f \big] \\[-4mm]\nonumber\\\nonumber
&=\langle \ud{w}^-\rangle^{2\mu} F_{s+\f12}
 +2\mu \langle\wu^-\rangle^{2\mu-2}\wu^-
\ud{\p_2 w}^- (\p_t + \ud{Z_+^1} \p_1)f \cdot (\p_t + \ud{Z_-^1} \p_1)\langle\p_1\rangle^{\f12}\p_1^a f .
\end{align*}
Furthermore, multiplying the above equation by $e^{\qu^-} \langle\wu^-\rangle^{2\mu}(\p_t + \ud{Z_-^1}\p_1) \langle\p_1\rangle^{\f12}\p_1^a f $, we obtain
\begin{align}\label{G9}
&\f12 e^{\qu^-} (\p_t + \ud{Z_+^1}\p_1)
\big[\langle \ud{w}^-\rangle^{2\mu} (\p_t + \ud{Z_-^1}\p_1) \langle\p_1\rangle^{\f12}\p_1^a f \big]^2\\\nonumber
&=e^{\qu^-} \langle\wu^-\rangle^{2\mu} F_{s+\f12}\cdot \big[\langle \ud{w}^-\rangle^{2\mu} (\p_t + \ud{Z_-^1}\p_1)\langle\p_1\rangle^{\f12}\p_1^a f \big] \\[-4mm]\nonumber\\\nonumber
&\quad +e^{\qu^-} 2\mu \langle\wu^-\rangle^{2\mu-2}\wu^-
\ud{\p_2 w}^- (\p_t + \ud{Z_+^1} \p_1)f \cdot \langle\wu^-\rangle^{2\mu}
\big[(\p_t + \ud{Z_-^1} \p_1)\langle\p_1\rangle^{\f12}\p_1^a f\big]^2 .
\end{align}
Then taking integral of \eqref{G9} in the $x_1$ variable over $\BR$,
 by integration by parts, we have
\begin{align}\label{G10}
&\frac12\frac{\d}{\dt}\int_{\mathbb{R}} e^{\qu^-}
\big[\langle\wu^-\rangle^{2\mu} (\p_t + \ud{Z_-^1}\p_1)\langle\p_1\rangle^{\f12}\p_1^a f \big]^2 \dx_1\\ \nonumber
&-\frac12\int_{\mathbb{R}}
\big[\langle\wu^-\rangle^{2\mu} (\p_t + \ud{Z_-^1}\p_1)\langle\p_1\rangle^{\f12}\p_1^a f \big]^2
(\p_t + \ud{Z_+^1}\p_1)e^{\qu^-}  \dx_1\\ \nonumber
&=\frac12\int_{\mathbb{R}} e^{\qu^-} \p_1\ud{Z_+^1}
\big[\langle\wu^-\rangle^{2\mu} (\p_t + \ud{Z_-^1}\p_1)\langle\p_1\rangle^{\f12}\p_1^af \big]^2 \dx_1\\ \nonumber
&\quad+\int_{\mathbb{R}}
e^{\qu^-} 2\mu \langle\wu^-\rangle^{2\mu-2}\wu^-
\ud{\p_2 w}^- (\p_t + \ud{Z_+^1} \p_1)f \cdot \langle\wu^-\rangle^{2\mu}
\big[(\p_t + \ud{Z_-^1} \p_1)\langle\p_1\rangle^{\f12}\p_1^a f\big]^2
\dx_1\\\nonumber
&\quad+\int_{\mathbb{R}} e^{\qu^-}\langle\wu^-\rangle^{2\mu}
F_{s+\f12} \cdot \big[\langle\wu^-\rangle^{2\mu} (\p_t + \ud{Z_-^1}\p_1) \langle\p_1\rangle^{\f12}\p_1^a f \big]\dx_1.
\end{align}
The second line of \eqref{G10} corresponds to the ghost weight energy.
By Lemma \ref{lemC3} and the bootstrap assumption $|\Lambda_\pm|\leq C\epsilon$,
we compute
\begin{align*}
&-(\p_t + \ud{Z_+^1}\p_1) e^{\qu^-}=
\f{e^{\qu^-}}{\langle \wu^+ \rangle^{2\mu}}
 (\p_t + \ud{Z_+^1}\p_1) \wu^+\geq
 \f 32 \f{e^{\qu^-}}{\langle \wu^+ \rangle^{2\mu}}\,.
\end{align*}
Consequently, the second line of \eqref{G10} has the following lower bound
\begin{align*}
&-\f12\int_{\mathbb{R}}
\big[\langle\wu^-\rangle^{2\mu} (\p_t + \ud{Z_-^1}\p_1)\langle\p_1\rangle^{\f12}\p_1^a f \big]^2
(\p_t + \ud{Z_+^1}\p_1) e^{\qu^-}  \dx_1\\ \nonumber
&\geq \f34
\int_{\mathbb{R}}
\f{\big[\langle\wu^-\rangle^{2\mu} (\p_t + \ud{Z_-^1}\p_1)\langle\p_1\rangle^{\f12}\p_1^a f \big]^2}{\langle\wu^+\rangle^{2\mu} }  e^{\qu^-}  \dx_1.
\end{align*}
For the first two terms on the right hand side of \eqref{G10},
by the weighted Sobolev inequalities Lemma \ref{Sobo2}, the weighted chain rules Lemma \ref{lemG3}
and the weighted trace Lemma \ref{lemD2},
they can be bounded as follows:
\begin{align*}
&\frac12\int_{\mathbb{R}} e^{\qu^-} \p_1\ud{Z_+^1}
\big[\langle\wu^-\rangle^{2\mu} (\p_t + \ud{Z_-^1}\p_1)\langle\p_1\rangle^{\f12}\p_1^af \big]^2 \dx_1\\ \nonumber
&-\int_{\mathbb{R}}
e^{\qu^-} 2\mu \langle\wu^-\rangle^{2\mu-2}\wu^-
\ud{\p_2 w}^- (\p_t + \ud{Z_+^1} \p_1)f \cdot \langle\wu^-\rangle^{2\mu}
\big[(\p_t + \ud{Z_-^1} \p_1)\langle\p_1\rangle^{\f12}\p_1^a f\big]^2
\dx_1\\\nonumber
&\leq \f12\int_{\BR} e^{\qu^-} \langle\wu^+\rangle^{2\mu} |\p_1\ud{\Lambda_+^1}|
\f{\big[\langle\wu^-\rangle^{2\mu} (\p_t + \ud{Z_-^1}\p_1)\langle\p_1\rangle^{\f12}\p_1^af \big]^2}{\langle\wu^+\rangle^{2\mu}} \dx_1\\ \nonumber
&\quad +\int_{\BR}
e^{\qu^-} 2\mu \langle\wu^+\rangle^{2\mu}
 |(\p_t + \ud{Z_+^1} \p_1)f| \cdot
\f{\big[\langle\wu^-\rangle^{2\mu} (\p_t + \ud{Z_-^1} \p_1)\langle\p_1\rangle^{\f12}\p_1^a f\big]^2}{\langle\wu^+\rangle^{2\mu}}
\dx_1\\\nonumber
&\lesssim E^{\f12}_s G_s.
\end{align*}

%\newpage
For the last line of \eqref{G10}, it has the following direct bound
%which involves the commutators and the pressure term:
\begin{align}\label{G20}
&\int_{\BR} e^{\qu^-} \langle\wu^-\rangle^{2\mu}F_{s+\f12}
\cdot \big[\langle\wu^-\rangle^{2\mu}(\p_t + \ud{Z_-^1} \p_1)\langle\p_1\rangle^{\f12}\p_1^af \big]\dx_1\\\nonumber
&\lesssim \| \langle\wu^+\rangle^{\mu} \langle\wu^-\rangle^{2\mu}F_{s+\f12}\|_{L^2(\BR)}
G_s^{\f12} .
\end{align}
By the definition of $F_{s+\f12}$ in \eqref{G8},
the estimate of \eqref{G20} reduces to
\begin{align}\label{G21}
&\| \langle\wu^+\rangle^{\mu} \langle\wu^-\rangle^{2\mu}F_{s+\f12}\|_{L^2(\BR)}
\\[-4mm]\nonumber\\\nonumber
&\leq  \| \langle\wu^+\rangle^{\mu} \langle\wu^-\rangle^{2\mu} \big[\langle\p_1\rangle^{\f12}\p_1^a,\ud{\Lambda^1_+}\ \big] \p_1(\p_t-\p_1) f\|_{L^2(\BR)} \\[-4mm]\nonumber\\\nonumber
&\quad+  \| \langle\wu^+\rangle^{\mu} \langle\wu^-\rangle^{2\mu} \big[\langle\p_1\rangle^{\f12}\p_1^a,\ud{\Lambda^1_-}\ \big] \p_1(\p_t+\p_1) f\|_{L^2(\BR)} \\[-4mm]\nonumber\\\nonumber
&\quad+  \| \langle\wu^+\rangle^{\mu} \langle\wu^-\rangle^{2\mu} \big[\langle\p_1\rangle^{\f12}\p_1^a, \ud{\Lambda_-^1} \ud{\Lambda_+^1}\ \big] \p_1^2 f\|_{L^2(\BR)} \\[-4mm]\nonumber\\\nonumber
&\quad+  \| \langle\wu^+\rangle^{\mu} \langle\wu^-\rangle^{2\mu}
\big(  \langle\p_1\rangle^{\f12}\p_1^a ( N_f \cdot\ud{\nabla p} )
+ \langle\p_1\rangle^{\f12}\p_1^{a+1} f \ud{\p_1 p} \big) \|_{L^2(\BR)} .
\end{align}

\textbf{Estimate of the second and the third line of \eqref{G21}:}

For the second and the third line of \eqref{G21},
%involving commutators.
they have the same structure.
Hence we only present the estimate for the second line of \eqref{G21}.
The third line can be estimated similarly.

For $|a|\leq s-1$, if $|a|=0$, by the weighted commutator estimate Lemma \ref{lemD8},
the weighted Sobolev inequalities Lemma \ref{Sobo2},
the weighted chain rule Lemma \ref{lemG3} and
the weighted trace Lemma \ref{lemD2}, we calculate
\begin{align*}
&\big\| \langle\wu^+\rangle^{\mu} \langle\wu^-\rangle^{2\mu}
\big[ \langle\p_1\rangle^{\f12},\ud{\Lambda^1_+}\ \big]  \p_1(\p_t-\p_1) f
\big\|_{L^2(\BR)} \\[-5mm]\nonumber\\\nonumber
&\lesssim \big\| \langle\wu^+\rangle^{2\mu} \langle \p_1\rangle^{\f12} \ud{\Lambda^1_+}
\big\|_{L^\infty(\BR)}
\Big\|\f{\langle\wu^-\rangle^{2\mu}}{\langle\wu^+\rangle^{\mu}}
 \p_1(\p_t-\p_1) f \Big\|_{L^2(\BR)}\\\nonumber
&\lesssim E_2^{\f12} G_2^{\f12} .
\end{align*}
If $a=1$, by the weighted commutator estimate Lemma \ref{lemD10},
the weighted Sobolev inequalities Lemma \ref{Sobo2},
the weighted chain rules Lemma \ref{lemG3} and the weighted trace Lemma \ref{lemD3}, we deduce that
\begin{align*}
&\| \langle \wu^+\rangle^{\mu} \langle \wu^-\rangle^{2\mu} \big[\langle\p_1\rangle^{\f12}\p_1,\ud{\La_{+}^1 } \big] \p_1(\p_t-\p_1) f\|_{L^2(\BR)} \\
&\lesssim \big\| \langle \wu^+\rangle^{2\mu} \langle \p_1\rangle^{\f12}\p_1\ud{\La_{+}^1 }
\big\|_{L^2(\BR)}
\Big\|\f{\langle \wu^-\rangle^{2\mu}}{\langle \wu^+\rangle^{\mu}}
 \p_1 (\p_t-\p_1) f \Big\|_{L^\infty(\BR)}  \\\nonumber
&\quad+ \big\| \langle \wu^+\rangle^{2\mu} \langle \p_1\rangle^{\f12}\ud{\La_{+}^1}
\big\|_{L^\infty(\BR)}
\Big\|\f{\langle \wu^-\rangle^{2\mu}}{\langle \wu^+\rangle^{\mu}}
 \p_1^2 (\p_t-\p_1) f \Big\|_{L^2(\BR)}  \\\nonumber
%&\lesssim \big( \| \langle \wu^+\rangle^{2\mu}\p^{\leq 2}  \La_{+}^1\|_{L^2(\Om_f)}
% + \| \langle \wu^+\rangle^{2\mu} \p^{\leq 2}  \hat{\La}_{+}^1\|_{L^2(\hat{\Om}_f)} \big)  G_3^{\f12} \\\nonumber
&\lesssim E_2^{\f12} G_3^{\f12} .
\end{align*}
If $2\leq a\leq s-1$, by weighted commutator estimate Lemma \ref{lemD11}, weighted chain rule Lemma \ref{lemG3} and the weighted trace Lemma \ref{lemD3}, there holds
\begin{align*}
&\| \langle \wu^+\rangle^{\mu} \langle \wu^-\rangle^{2\mu} \big[\langle\p_1\rangle^{\f12}\p_1^a,\ud{\La_{+}^1 } \big] \p_1(\p_t-\p_1) f\|_{L^2(\BR)} \\
&\lesssim \big\| \langle \wu^+\rangle^{2\mu} \langle \p_1\rangle^{\f12}
 \p_1^{\leq a} \ud{\La_{+}^1 }
\big\|_{L^2(\BR)}
\Big\|\f{\langle \wu^-\rangle^{2\mu}}{\langle \wu^+\rangle^{\mu}} \langle \p_1\rangle^{\f12}
 \p_1^{\leq a}(\p_t-\p_1) f \Big\|_{L^2(\BR)}  \\\nonumber
&\lesssim  \| \langle w^+\rangle^{2\mu}\p^{\leq s}  \La_{+}^1\|_{L^2(\Om_f)}
 \mathcal{F}(1+\|\p_1 f\|_{H^{s-\f32}})
   G_s^{\f12} \\\nonumber
&\lesssim E_s^{\f12} G_s^{\f12} \mathcal{F}(1+\|\p_1 f\|_{H^{s-\f32}}).
\end{align*}

\textbf{Estimate of the fourth line of \eqref{G21}:}
\begin{align*}
& \| \langle\wu^+\rangle^{\mu} \langle\wu^-\rangle^{2\mu} \big[\langle\p_1\rangle^{\f12}\p_1^a, \ud{\Lambda_-^1} \ud{\Lambda_+^1}\ \big] \p_1^2 f\|_{L^2(\BR)} .
\end{align*}
Similar to the estimate of the second line of \eqref{G21}, by employing the weighted Sobolev inequalities Lemma \ref{Sobo2},
the weighted commutator estimate Lemma \ref{lemD8}, Lemma \ref{lemD10}, Lemma \ref{lemD11},
the weighted chain rules Lemma \ref{lemG3} and the weighted trace Lemma \ref{lemD3},
we infer
\begin{align*}
&\| \langle \wu^+\rangle^{\mu} \langle \wu^-\rangle^{2\mu} \big[\langle\p_1\rangle^{\f12}\p_1^a, \ud{\La_{-}^1} \ud{\La_{+}^1}\ \big] \p_1^2 f\|_{L^2(\BR)} \\
&\lesssim \big\| \langle \wu^+\rangle^{\mu} \langle \wu^-\rangle^{2\mu} \langle \p_1\rangle^{\f12}
 \p_1^{\leq s-1} \big( \ud{\La_{-}^1} \ud{\La_{+}^1} \big)
\big\|_{L^2(\BR)}
\|\langle \p_1\rangle^{\f12} \p_1^{\leq s} f \|_{L^2(\BR)}\\\nonumber
&\lesssim E_s G_s^{\f12} \mathcal{P}(1+\| f'\|_{H^{s-\f32}}).
\end{align*}

\textbf{Estimate of the fifth line of \eqref{G21} concerning pressure term:}
%The remaining terms involving pressure is
\begin{align*}
& \| \langle\wu^+\rangle^{\mu} \langle\wu^-\rangle^{2\mu}
\big(  \langle\p_1\rangle^{\f12}\p_1^a ( N_f \cdot\ud{\nabla p} )
+ \langle\p_1\rangle^{\f12}\p_1^a\p_1 f \ud{\p_1 p} \big) \|_{L^2(\BR)} .
\end{align*}
%The last two
%terms concerning pressure can be written as follows:
Note that we have the rearrangement
\begin{align}\label{G30}
& \langle\p_1\rangle^{\f12}\p_1^a \big( N_f \cdot\ud{\nabla p}\big)
+ \langle\p_1\rangle^{\f12}\p_1^{a+1} f \ud{\p_1 p} \\[-5mm]\nonumber\\\nonumber
&= \langle\p_1\rangle^{\f12}\p_1^a ( -\p_1f\ud{\p_1 p}+\ud{\p_2 p})
+ \langle\p_1\rangle^{\f12}\p_1^{a+1} f \ud{\p_1 p} \\[-5mm]\nonumber\\\nonumber
&=-\big[\langle\p_1\rangle^{\f12}\p_1^a,\ud{\p_1 p}\ \big] \p_1 f
+ \langle\p_1\rangle^{\f12}\p_1^a \ud{\p_2 p}.
\end{align}
We first estimate the last term of \eqref{G30}.
By the weighted chain rules of Lemma \ref{lemG3}
and the weighted trace Lemma \ref{lemD3}, we obtain
\begin{align}\label{G31}
&\big\|\langle\wu^+\rangle^{\mu} \langle\wu^-\rangle^{2\mu}
 \langle\p_1\rangle^{\f12}\p_1^a  \ud{\p_2 p} \big\|_{L^2(\BR)} \\\nonumber
&\lesssim \sum_{|\alpha|\leq a}\big\| \langle\wu^+\rangle^{\mu} \langle\wu^-\rangle^{2\mu} \langle\p_1\rangle^{\f12}  \ud{\p^\alpha \p_2 p}\big\|_{L^2(\BR)} \\\nonumber
&\lesssim \sum_{|\alpha|\leq a+1}\big\| \langle w^+\rangle^{\mu} \langle w^-\rangle^{2\mu}   \p^\alpha\p_2 p \big\|_{L^2(\Om_f)}.
\end{align}
By the weighted  estimate of pressure in Lemma \ref{lemPre}, it is further bounded by
$E_s^{\f12} G_s^{\f12}$.

Next, we estimate the first term on the right hand side of \eqref{G30}.
By the Leibniz rule, we write
\begin{align*}
\langle\p_1\rangle^{\f12}\p_1^a ( \p_1f\ud{\p_1 p} )
=\langle\p_1\rangle^{\f12} \sum_{b+c=a}
C_a^b \p_1^{b+1}f \p_1^c \ud{\p_1 p} .
\end{align*}
Consequently, by Lemma \ref{lemD8} %Lemma \ref{lemD10}
 and Lemma \ref{Sobo2}, we obtain
\begin{align}\label{G33}
& \| \langle\wu^+\rangle^{\mu} \langle\wu^-\rangle^{2\mu}
\big[\langle\p_1\rangle^{\f12}\p_1^a,\ud{\p_1 p}\ \big] \p_1 f \|_{L^2(\BR)} \\[-4mm]\nonumber\\\nonumber
&\lesssim  \sum_{b+c=a} \| \langle\wu^+\rangle^{\mu} \langle\wu^-\rangle^{2\mu}
\langle\p_1\rangle^{\f12} \big( \p_1^{b+1}f \p_1^c \ud{\p_1 p}\big)  \|_{L^2(\BR)} \\\nonumber
&\quad+ \| \langle\wu^+\rangle^{\mu} \langle\wu^-\rangle^{2\mu}
\langle\p_1\rangle^{\f12}\p_1^{a+1}f   \ud{\p_1 p}  \|_{L^2(\BR)} \\\nonumber
&\lesssim
(\|f\|_{\dot{H}^1(\BR)}+\|
  f\|_{\dot{H}^{s+\f12}(\BR)}) \cdot
\sum_{c\leq s-1}\| \langle\wu^+\rangle^{\mu} \langle\wu^-\rangle^{2\mu}
\langle\p_1\rangle^{\f12}\p_1^c  \ud{\p_1 p}  \|_{L^2(\BR)}.
\end{align}
Then similar to the estimate of
$\big\|\langle\wu^+\rangle^{\mu} \langle\wu^-\rangle^{2\mu}
 \langle\p_1\rangle^{\f12}\p_1^a  \ud{\p_2 p} \big\|_{L^2(\BR)}$ in \eqref{G31},
\eqref{G33}  can be further bounded by
$E_s  G_s^{\f12}$.

In summary, combining all the above estimates, \eqref{G10} yields
\begin{align*}
&\frac12\frac{\d}{\dt}\int_{\mathbb{R}} e^{\qu^-}
\big[\langle\wu^-\rangle^{2\mu} (\p_t + \ud{Z_-^1}\p_1)\langle\p_1\rangle^{\f12}\p_1^a f \big]^2 \dx_1\\
&+\f34
\int_{\mathbb{R}}
\f{\big[\langle\wu^-\rangle^{2\mu} (\p_t + \ud{Z_-^1}\p_1)\langle\p_1\rangle^{\f12}\p_1^a f \big]^2}{\langle\wu^+\rangle^{2\mu} }  e^{\qu^-}  \dx_1\\
& \lesssim (E_s^{\f12} + E_s ) G_s.
\end{align*}
%Taking the sum for all $|a|\leq s-1$. Note similar estimates also holds for the right Alfv\'{e}n waves
%and $e^{\qu^\pm} \sim 1$, we obtain
%%\begin{align*}
%%\frac{d}{dt} E_s^{f}+G_s^f \lesssim (E_s^{\f12} + E_s )  G_s.
%%\end{align*}
Summing over all $|a|\leq s-1$ and taking the integral over $[0,t]$,
and noting that the same estimate also holds for the right Alfv\'{e}n waves and $e^{q^\pm}\sim 1$, we obtain
\begin{align*}
E^f_s(t)+\int_0^t G^f_s(\tau)\d\tau \lesssim E^f_s(0)+\int_0^t E_s^{\f12}  G_s (\tau)\d\tau.
\end{align*}
This yields %the \emph{a priori} estimate
\eqref{AA2}.

\begin{rem}\label{rem-hw}
%If we use the energy in the form \eqref{G1}. % instead of \eqref{G2}.
%Then we would derive that....
If we treat the energy estimate using \eqref{B13} %which we write it
 in the following form:
\begin{align}\label{GG1}
\p_t^2 f-\p_1^2 f
%=&- (\ud{\Lambda_+^1} +\ud{\Lambda_-^1})
%\p_t\p_1 f
%-\ud{\Lambda_-^1}\p^2_1 f + \ud{\Lambda_+^1}\p^2_1 f \nonumber\\
%&- \ud{\Lambda_-^1}  \ud{\Lambda_+^1}  \p_1^2 f
%- N_f \cdot\ud{\nabla p} \nonumber\\
=- \ud{\Lambda_+^1}\p_1(\p_t-\p_1)f
  - \ud{\Lambda_-^1}\p_1(\p_t+\p_1)f
-\ud{\Lambda_-^1}  \ud{\Lambda_+^1} \p_1^2 f
- N_f \cdot\ud{\nabla p}\, .
\end{align}
Then it is natural to use the following weighted energy
\begin{align*}
E^f_{s+\f12} (t)= \sum_{+,-} \sum_{|a|\leq s-1} \big\| \langle v^\pm \rangle^{2\mu}(\p_t\pm\p_1)\langle\p_1\rangle^{\f12}\p_1^a f \big\|_{L^2(\bR)}^2,
\end{align*}
where $v^\pm=x_1\pm t$.
Let us formally treat the energy estimate using the trick in \cite[Section 6.2]{CL-24}.
 Let $e^{q^+}=e^{q(v^-)}$ and $e^{q^-}=e^{q(-v^+)}$.
% without weights $\wu^\pm$ and ghost weights $e^{\qu^\pm}$.
%and without ghost weight.
Apply the derivative $\langle\p_1\rangle^{s-\f12}$ to \eqref{GG1}, then
taking the $L^2$ inner product of the resulting equation with
%$(\p_t-\p_1)\langle\p_1\rangle^{s-\f12} f$,
%(the $L^2$ inner product $(\p_t-\p_1)f$ is similar),
$2\langle v^-\rangle^{4\mu}(\p_t-\p_1)\langle \p_1\rangle^{\f12}\p_1^af\cdot e^{q^-}$,
%and  $2\langle v^+\rangle^{4\mu}(\p_t+\p_1)\langle \p_1\rangle^{\f12}\p_1^af \cdot e^{q^+}$), respectively,
we obtain
\begin{align}\label{GG2}
& \f{\d}{\dt}
\int_{\BR}\big| \langle v^-\rangle^{2\mu}(\p_t-\p_1) \langle \p_1\rangle^{s-\f12}  f
\big|^2 e^{q^-}\d x_1
+2\int_{\BR}\f{\big| \langle v^-\rangle^{2\mu}(\p_t-\p_1) \langle \p_1\rangle^{s-\f12} f
\big|^2}{\langle v^+\rangle^{2\mu}} e^{q^-}\d x_1\\\nonumber
&= \f14 \f{\d}{\dt}\int_{\BR}
|\p_1 \langle \p_1\rangle^{s-\f12} f|^2
e^{q^-}\cdot\langle v^-\rangle^{4\mu} \ud{\La^1_{+}} \ud{\La^1_{-}}\d x_1\\\nonumber
&\quad-\f14\int_{\BR} \big| \p_1 \langle \p_1\rangle^{s-\f12} f  \big|^2 e^{q^-}\cdot (\p_t-  \p_1)
\big(\langle v^-\rangle^{4\mu} \ud{\La_{+}^1} \ud{\La_{-}^1} \big)
  \d x_1\\\nonumber
&\quad+....
\end{align}
To deal with the second line and the third line of \eqref{GG2},
we need to estimate
%(or slightly higher $L^2$ based Sobolev norms with bigger power weight functions.)
\begin{align*}
\sum_{+,-}\sum_{|a|\leq 1}
\| \langle v^\pm\rangle^{5\mu} \p_1^a \ud{\La_{\pm}^1} \|_{L^\infty}.
\end{align*}
Thus we need to apply the weights $\langle v^\pm\rangle^{5\mu}$ on $\La_{\pm}$
rather than $\langle v^\pm\rangle^{2\mu}$.
We refer to \cite{CL-24} for more details concerning this approach.
\end{rem}

\section{Local solutions}
In this last section, we show the existence of local solutions of the plasma-vacuum free boundary problem
for the two dimensional ideal incompressible MHD under strong magnetic field background in the plasma region.
The  local solution  to  the problem  will be obtained by the method of successive approximations.
Precisely, we show that a sequence of solutions
$\{\Lambda_+^{(n)},\Lambda_-^{(n)}, f^{(n)}, p^{(n)}\}_{n=0}^{\infty}$
with initial value
 $\{ \Lambda_+^{(n)},\Lambda_-^{(n)}, f^{(n)}\}_{n=0}
=\{ \Lambda_+^{0},\Lambda_-^{0}, f^{0}\}$
%$p^{(0)}$ will be constructed by \eqref{R12}.
will finally converge to the solutions of the plasma-vacuum free boundary problem.

\subsection{Approximate solutions}

Recalling that the equation for the free surface \eqref{B13} can be written as follows:
\begin{align} \label{R0}
(\p_t^2-\p_1^2) f
+ (\ud{\Lambda_+^1} +\ud{\Lambda_-^1})
\p_t\p_1 f
+\big(\ud{\Lambda_-^1}  - \ud{\Lambda_+^1}
+ \ud{\Lambda_-^1}  \ud{\Lambda_+^1}  \big)\p_1^2 f
=- N_f \cdot\ud{\nabla p} ,
\end{align}
and recalling  the equations for the incompressible MHD \eqref{MHD1}-\eqref{A19} in the fixed domain $\Om_L$ as follows (see Section \ref{Sec-Lag}):
\begin{equation} \label{R1}
\begin{cases}
\p_t \wt{\Lambda}_\pm
+(\wt{\Lambda}_\mp  \mp e_1 )\cdot (A^\top\nabla) \wt{\Lambda}_\pm  +A^\top \nabla \wt{p}
-(0,\p_t\psi/J)\cdot \nabla \wt{\Lambda}_\pm =0 \quad \mathrm{in}\,\, \Om_L, \\[-4mm]\\
\div^{\Psi} \wt{\Lambda}_\pm=0\quad \mathrm{in}\,\,  \Om_L, \\[-5mm]\\
   \wt{p}=0\quad \mathrm{on}\,\,  \Gamma, \\[-5mm]\\
\wt{\Lambda}_+\cdot e_2=0,\,\, \wt{\Lambda}_-\cdot e_2=0
\quad \text{on}\,\, \Gamma^- ,\\[-5mm]\\
\wt{\Lambda}_\pm(0,x)= \wt{\Lambda}_\pm^0(x)=\Lambda_\pm( \Psi_{f^0}(x) ).
\end{cases}
\end{equation}
Moreover, by \eqref{com}, the initial data satisfy
\begin{equation} \label{R-com}
\left\{\begin{aligned}
&\div^{\Psi} \wt{\Lambda}_+^{0} = 0, \,\, \div^{\Psi} \wt{\Lambda}_-^{0} = 0\,\, \text{ in } \Om_L,\,\,  \\
& v^0=(\wt{\Lambda}_\pm^{0} \pm e_1)\cdot N_{f^0} \,\, \text{ on } \Gamma,\\
&  %u_{20},\, h_{20} = 0
\wt{\Lambda}_+^{0}\cdot e_2=0,\,\, \wt{\Lambda}_-^{0}\cdot e_2=0 \,\,\text{ on }\Gamma^- .
\end{aligned}\right.
\end{equation}
where $N_{f^0}=(-\p_1 f^0,1)$.

Now we construct approximate solutions $\{\Lambda_+^{(n)},\Lambda_-^{(n)}, f^{(n)}, p^{(n)}\}_{n=0}^{\infty}$.
Denote
\begin{align*}
{Z}_\pm^{(n)}
&={\Lambda}_\pm^{(n)}\pm e_1\quad \mathrm{in}\ \Om_{f^{(n)}},\\
\ud{\Lambda_+^1}^{(n)}
&={\Lambda_+^1}^{(n)}\big|_{\Gamma_{f^{(n)}}},\\
N_{f^{(n)}}&=(-\p_1f^{(n)},1),\\
{p}^{(n+1)}:&\,\, [0,T]\times\Om_{f^{(n)}}\rightarrow \BR.
\end{align*}
where
\begin{align*}
\Om_{f^{(n)}}:&=\left\{ x=(x_1,x_2) | -1<x_2 < f^{(n)}(t,x_1),\, x_1\in\BR \right\},\\
\Gamma_{f^{(n)}}:&=\left\{ x=(x_1,x_2) | x_2 = f^{(n)}(t,x_1),\, x_1\in\BR \right\}.
\end{align*}
Note that the domains of ${Z}_\pm^{(n)},\, {\Lambda}_\pm^{(n)},\, {p}^{(n)}$ are different
for varying $n$.
To treat the problem in the fixed domain $\Om_L$, we construct a series of maps $\Psi_{f^{(n)}}$ as follows:
% We make the following change of variables $\Om_L \mapsto\Om_f$ to flatten the boundary of  $\Omega_f$ by a regularized mapping \cite{Lannes, ABZ, CMST}:
\begin{align*}
%(x_1,x_2)\in \BR\times [-1,0]=\Om_L
%\mapsto
\BR\times [-1,0]=\Om_L \ni (x_1,x_2)
\mapsto
\Psi_{f^{(n)}}(t,x_1,x_2)=(x_1,x_2+\psi^{(n)}(t,x_1,x_2)) \in \Om_{f^{(n)}},
\end{align*}
where
\begin{align*}
%&\BR\times [-1,0]=\Om_L \ni (x_1,x_2)
%\mapsto   [-1, f^{(n)}(t,x_1)],\\
&\psi^{(n)}(t,x_1,x_2)=(1+x_2) \varphi(\rho x_2|\p_1|) f^{(n)}(t,x_1).
\end{align*}
Here $\rho$ is a small parameter, $\varphi \in C_c^\infty(\BR)$ is a smooth cut-off function satisfying:
%($\varphi$ need to be define in $\BR$ since we will
%calculate the Fourier transform,
% and need check this cut-off and the above construction)
\begin{equation*}
\varphi(x_2)=
\begin{cases}
1,\quad \textrm{if}\quad -\f31 \leq x_2 \leq 0,\\
0,\quad \textrm{if}\quad  -1\leq x_2 \leq  -\f23.
\end{cases}
\end{equation*}

%Denote the inverse of $\nabla \Psi_{f^{(n)}}$   by $A_n$:
Later we will show $\|f^{(n)}\|_{L^\infty}\leq 1-\delta$ for all $n\geq 0$.
Thus there holds $\f14\delta\leq J_n\leq 2-\f14\delta$, which ensures that all the maps $\Psi_{f^{(n)}}$ are  diffeomorphisms.

The properties of $\Psi_{f^{(n)}}$ are similar to $\Psi$ in Section \ref{Lag-stra}.
Let us use the following notations:
\begin{align*}
J_n=\det(\nabla \Psi_{f^{(n)}}),\quad
A_n=[\nabla \Psi_{f^{(n)}}]^{-1}=\frac{1}{J_n}
\left(
\begin{matrix}
1+\p_2\psi^{(n)} \quad 0\\
-\p_1\psi^{(n)} \qquad 1
\end{matrix}
\right).
\end{align*}
%It is easy to see that $\textrm{det}(A)=J^{-1}=(1+\p_2\psi)^{-1}$.
%By the series of maps, we
Set
\begin{align*}
&\wt{Z}_\pm^{(n)} (t,x)=Z_\pm^{(n)}  (t,\Psi_{f^{(n)}}(x)),\quad
\wt{\Lambda}_\pm^{(n)} (t,x)=\Lambda_\pm^{(n)} (t,\Psi_{f^{(n)}}(x) ),\quad\\
&\wt{p}^{(n)} (t,x)=p^{(n)} (t,\Psi_{f^{(n-1)}}(x) ).
\end{align*}
Consequently,
\begin{align*}
\ud{\wt{\Lambda}_\pm}^{(n)}:=\wt{\Lambda}_\pm^{(n)} \big|_{\Gamma}
%=\Lambda_\pm^{(n)} (t,\Psi_{f^{(n)}}(x) )\big|_{\Gamma_{f^{(n)}}} .
=\Lambda_\pm^{(n)}\big|_{\Gamma_{f^{(n)}}} .
\end{align*}
%where $\Gamma$ is the upper boundary of $\Om_L$: $\Gamma=\BR\times\{0\}$.
%and the lower boundary by $\Gamma_-=\BR\times\{-1\}$.
Now for all $n\geq 0$,
$\wt{Z}_\pm^{(n)}(t),\, \wt{\Lambda}_\pm^{(n)}(t),\, \wt{p}^{(n)}(t)\,$
have the same domain $\Om_L$.

In order to simplify the presentation,
let us denote
\begin{align*}
\nabla^{\Psi_n}=A_n^\top \nabla.
\end{align*}
For integer $n\geq 0$, we construct the following approximate solutions for the free surface:
%Problem I:
\begin{equation}\label{R10}
\begin{cases}
(\p_t^2-\p_1^2) f^{(n+1)} + ({\ud{\wt{\Lambda}_+^1}^{(n)}}+{\ud{\wt{\Lambda}_-^1}^{(n)}})\p_t\p_1 f^{(n+1)} \\[-4mm]\\
\quad +\big( {\ud{\wt{\Lambda}_-^1}^{(n)}}-{\ud{\wt{\Lambda}_+^1}^{(n)}}+ {\ud{\wt{\Lambda}_-^1}^{(n)}} {\ud{\wt{\Lambda}_+^1}^{(n)}} \big)\p_1^2 f^{(n+1)}\\[-4mm]\\
\quad=-N_{f^{(n)}}\cdot \ud{\nabla p}^{(n+1)},  \\[-4mm]\\
%(\p_t + \ud{Z_+^1}\p_1)(\p_t + \ud{Z_-^1}\p_1)g
%=-\ud{\p_2 p} &\text{in } \mathbb R\\
f^{(n+1)} (0) = f^0,\,\p_t f^{(n+1)} (0)=v^0,
%\p_t f^{(n+1)} (0)= (\Lambda_\pm^0 \pm e_1)|_{\Ga_{f^0}}\cdot N_{f^0}
\end{cases}
\text{in } \BR.
\end{equation}
Moreover, we construct the following approximate solutions for the plasma in the fixed domain $\Om_L$:
%In the fixed domain $\Om_L$, the equations for the incompressible MHD reads as follows:
\begin{equation}\label{R11}
\begin{cases}
\p_t \wt{\Lambda}_\pm^{(n+1)}
+(\wt{\Lambda}_\mp^{(n)}  \mp e_1 )\cdot (A_n^\top\nabla) \wt{\Lambda}_\pm^{(n+1)}  \\[-4mm]\\
\quad+A^\top_n \nabla \wt{p}^{(n+1)}
-(0,\p_t\psi^{(n)}/J_n)\cdot \nabla \wt{\Lambda}_\pm ^{(n+1)} =0 \quad \mathrm{in}\,\, \Om_L ,\\[-4mm]\\
\div^{\Psi_n} \wt{\Lambda}_\pm^{(n+1)}=0,\quad \mathrm{in}\,\, \Om_L, \\[-4mm]\\
\wt{p}^{(n+1)}=0\,\,\mathrm{on}\,\,  \Gamma, \\[-4mm]\\
\wt{\Lambda}_+^{(n+1)} \cdot e_2=0,\,\, \wt{\Lambda}_-^{(n+1)} \cdot e_2=0
\quad \text{on}\,\, \Gamma^- .
\end{cases}
\end{equation}
By applying $\div^{\Psi_n}$  to the first equation of \eqref{R11}, we obtain the elliptic equation for the pressure as follows:
\begin{equation}\label{R12}
\begin{cases}
 -\nabla^{\Psi_n}\cdot \nabla^{\Psi_n} \wt{p}^{(n+1)}
=\p^{\Psi_n}_j \wt{\Lambda}^{i\,(n)}_- \p^{\Psi_n}_i \wt{\Lambda}^{j\,(n+1)}_+  \quad  {\rm in} \; \Omega_L,\\
\wt{p}^{(n+1)} \big|_{\Gamma}=0,\quad  \p_2 \wt{p}^{(n+1)} \big|_{\Gamma_-}=0.
\end{cases}
\end{equation}

\subsection{Estimate for the approximate equations}
For integer $s\geq 0,$ denote
\begin{align*}
\mathcal{E}_s^b (\wt{\Lambda}_+,\wt{\Lambda}_-)
&= \|\wt{\Lambda}_+\|^2_{H^s(\Om_L)}+\|\wt{\Lambda}_-\|^2_{H^s(\Om_L)}, \\
\mathcal{E}_s^f (\p_tf,f)
&= \|\p_t f\|^2_{H^{s-\f12}(\BR)}+\|\p_1 f\|^2_{H^{s-\f12}(\BR)},\\
\mathcal{E}_s (\wt{\Lambda}_+,\wt{\Lambda}_-, f)
&= \mathcal{E}_s^b (\wt{\Lambda}_+,\wt{\Lambda}_-)
  +\mathcal{E}_s^f (\p_tf,f),\\
\mathcal{E}_s^{(n+1)}(t)
&=\mathcal{E}_s (\wt{\Lambda}_+^{(n)},\wt{\Lambda}_-^{(n)}, f^{(n)}),\\
%\end{align*}
%And denote
%\begin{align*}
A_1&=\mathcal{E}_s^b (\wt{\Lambda}_+^0,\wt{\Lambda}_-^0)+\mathcal{E}_s^f (v^0, f^0)
.
\end{align*}

\begin{lem}\label{}
Let $s\geq 4$, $A_1>0, 0<\delta<\f12$.
Assume that
%\begin{align*}
%&\mathcal{E}_s^b (\wt{\Lambda}_+^0,\wt{\Lambda}_-^0)+\mathcal{E}_s^f (v^0, f^0)=A_1 ,\\
%&\quad \| f^0\|_{L^\infty} \leq 1-2\delta.
%\end{align*}
\begin{align*}
\mathcal{E}_s^b (\wt{\Lambda}_+^0,\wt{\Lambda}_-^0)+\mathcal{E}_s^f (v^0, f^0)\leq A_1 ,\,\,
 \| f^0\|_{L^\infty} \leq 1-2\delta.
\end{align*}
There exists a time $T>0$, depending on $A_1$ and $\delta$, such that the systems \eqref{R10}, \eqref{R11}, \eqref{R12} admit unique solutions
$(f^{(n+1)},\p_tf^{(n+1)})\in C\big([0,T];H^{s+\f12}\times H^{s-\f12}(\bR)\big)$,
$\wt{\Lambda}_\pm^{(n)}\in C\big([0,T];H^s(\Om_L)\big)$,
$\wt{p}^{(n)}\in C\big([0,T];H^{s+1}(\Om_L)\big)$,
 such that
%\begin{align*}%\label{}
%&\sup_{t\in[0,T_*]} E_{s}^w(\omega_\pm^{(n)},\hat{\omega}_\pm^{(n)})
%+\sup_{t\in[0,T_*]}E_{s+\f12}^f(\p_tf^{(n)},f^{(n)}) \leq A, \\
%&\qquad \sup_{t\in[0,T_*]}\| f^{(n)}(t) \|_{L^\infty}\leq 1-\delta.
%\end{align*}
\begin{equation}\label{R13}
\begin{cases}
\mathcal{E}_s (\wt{\Lambda}_+^{(n)},\wt{\Lambda}_-^{(n)}, f^{(n)})\leq 2A_1,\\
\| \wt{p}^{(n)} \|_{H^{s+1}(\Om_L)}\leq C_1A_1,\\
 \| f^{(n)}(t,\cdot)\|_{L^\infty} \leq 1-\delta,
\end{cases}
 \,\, \textrm{for all}\,\, n\geq 0,
\end{equation}
where $C_1$ is a constant depending on $A_1$ and $\delta$.
\end{lem}
%Let $A_1<\f14$, $\delta<\f12$ be two constants.
% In the sequel, we will show that for some $T>0$, there holds
\begin{proof}
We can perform the energy estimate for \eqref{R11} similar to the energy estimate of the plasma without weights in Section \ref{sec-energy-b}
and derive that
%the energy estimate as follows:
\begin{align}\label{R14}
& \mathcal{E}_s^b (\wt{\Lambda}_+^{(n+1)},\wt{\Lambda}_-^{(n+1)})\\\nonumber
& \leq    \mathcal{E}_s^b (\wt{\Lambda}_+^0,\wt{\Lambda}_-^0)
   +C \int_0^t \mathcal{F}(1+\|f^{(n)}(\tau)\|_{L^\infty}+\|\p_1 f^{(n)}(\tau)\|_{H^{s-\f12}(\BR)})\\\nonumber
&\qquad \cdot   \Big(
   \big(\mathcal{E}_s^b (\wt{\Lambda}_+^{(n+1)}(\tau),\wt{\Lambda}_-^{(n+1)}(\tau))\big)^{\f12}
   \cdot \| \nabla \wt{p}^{(n+1)} (\tau)\|_{H^s(\Om_L)} \\[-4mm]\nonumber\\\nonumber
&\qquad\quad  + \mathcal{E}_s^b (\wt{\Lambda}_+^{(n+1)}(\tau),\wt{\Lambda}_-^{(n+1)}(\tau))
   \cdot \big(\mathcal{E}_s^b (\wt{\Lambda}_+^{(n)}(\tau),\wt{\Lambda}_-^{(n)}(\tau))\big)^{\f12} \Big)\d\tau ,
\end{align}
%For the detailed calculation of \eqref{R14},
where $\mathcal{F}: \BR^+\rightarrow\BR^+$ is a generic non-decreasing function.
We also refer to Wang-Xin \cite{Wang_20} (taking the resistivity and the surface tension to be zero), Gu \cite{Gu_17} etc.
for a similar calculation of \eqref{R14}.

Next, apply the derivative
$\langle\p_1\rangle^{s-\f12}  $ to \eqref{R10}
and then taking the inner product with $\p_t \langle\p_1\rangle^{s-\f12} f^{(n+1)}$,
by integration by parts, the Sobolev inequalities, the trace theorem, commutator estimate, it implies
\begin{align}\label{R15}
&   \mathcal{E}_s^f (f^{(n+1)})
   \leq \mathcal{E}_s^f (v^0,f^0)\\\nonumber
&\quad   +C \int_0^t
\mathcal{F}(1+\|f^{(n)}(\tau)\|_{L^\infty(\BR)}+\|\p_1 f^{(n)}(\tau)\|_{H^{s-\f12}(\BR)})\\\nonumber
&\qquad\quad  \Big(\mathcal{E}_s^f (f^{(n+1)}(\tau))\cdot
\big(\mathcal{E}_s^{\f12} (\wt{\Lambda}_+^{(n)}(\tau),\wt{\Lambda}_-^{(n)}(\tau), f^{(n)}(\tau))
+\mathcal{E}_s (\wt{\Lambda}_+^{(n)}(\tau),\wt{\Lambda}_-^{(n)}(\tau), f^{(n)}(\tau))
\big)\\\nonumber
&\qquad\quad +\big(\mathcal{E}_s^f (f^{(n+1)}(\tau))\big)^{\f12}
   \| \nabla \wt{p}^{(n+1)} (\tau)\|_{H^s(\Om_L)} \Big) \d\tau .
\end{align}
We also refer to Sun-Wang-Zhang \cite{SWZ18, SWZ19} for a similar calculation of \eqref{R15}.

On the other hand,
by the expression for the one dimensional wave equation for \eqref{R10}, we obtain
\begin{align}\label{R16}
&\| f^{(n+1)}(t,\cdot)\|_{L^\infty(\BR)}
\leq \| f_0 \|_{L^\infty(\BR)} +t\| v^0\|_{L^\infty(\BR)}\\\nonumber
&\quad +t \sup_{0\leq \tau\leq t}\mathcal{E}_s^{\f12} (f^{(n+1)}(\tau))
\cdot\Big(\mathcal{E}_s^{\f12} (\wt{\Lambda}_+^{(n)}(\tau),\wt{\Lambda}_-^{(n)}(\tau))
+\mathcal{E}_s (\wt{\Lambda}_+^{(n)}(\tau),\wt{\Lambda}_-^{(n)}(\tau))
\Big)\\\nonumber
&\quad+t \sup_{0\leq \tau\leq t} \mathcal{F}(1+\|f^{(n)}(\tau)\|_{L^\infty}+\|\p_1 f^{(n)}(\tau)\|_{H^{s-\f12}(\BR)})
\| \nabla \wt{p}^{(n+1)} (\tau)\|_{H^s(\Om_L)} .
\end{align}

For the pressure $ \wt{p}^{(n+1)}$ of \eqref{R12},
by choosing $\rho\leq \f18\delta (A_1\tilde{C})^{-1}$, we have
 $|\p_1\psi^{(n)}|\leq \|f^{(n)}(t,\cdot)\|_{L^\infty}+\rho\tilde{C}\|\p_1f^{(n)}(t,\cdot) \|_{H^1(\bR)}\leq 1-\f14\delta$, then $\f14\delta\leq J_n=1+\p_2\psi^{(n)}\leq 2-\f14\delta$.
Similar to the estimate of $ \wt{p}$ in Lemma \ref{lemPreL}, we obtain that
\begin{align}\label{R18}
\|\wt{p}^{(n+1)}\|^2_{H^{s+1}(\Om_L)}
&\leq \mathcal{F}(1+\|f^{(n)}\|_{L^\infty(\BR)}+\|\p_1 f^{(n)}\|_{H^{s-\f12}(\BR)})\\ \nonumber
&\quad\cdot \|\wt{\Lambda}_+^{(n+1)}\|_{H^s(\Om_L)}  \|\wt{\Lambda}_-^{(n)}\|_{H^s(\Om_L)}.
\end{align}

Combining \eqref{R14}-\eqref{R18},
for $t\leq \f14$, we obtain
\begin{align*} %\label{R19}
&\mathcal{E}_s (\wt{\Lambda}_+^{(n+1)}(t),\wt{\Lambda}_-^{(n+1)}(t), f^{(n+1)}(t)) \\\nonumber
& \leq \f32 A_1+Ct \sup_{0\leq \tau\leq t}\Big(
\mathcal{F}(1+\|f^{(n)}\|_{L^\infty(\BR)}+\mathcal{E}_s (\wt{\Lambda}_+^{(n)},\wt{\Lambda}_-^{(n)}, f^{(n)}))
\cdot \mathcal{E}_s (\wt{\Lambda}_+^{(n+1)},\wt{\Lambda}_-^{(n+1)}, f^{(n+1)})
\Big),\\\nonumber
&\| f^{(n+1)}(t,\cdot)\|_{L^\infty(\BR)}
\leq \| f_0 \|_{L^\infty(\BR)} +t\| v^0\|_{L^\infty(\BR)}\\\nonumber
&\quad+Ct \sup_{0\leq \tau\leq t}\Big(
%\mathcal{F}(\mathcal{E}_s (\wt{\Lambda}_+^{(n)},\wt{\Lambda}_-^{(n)}, f^{(n)}))
\mathcal{F}(1+\|f^{(n)}\|_{L^\infty(\BR)}+\mathcal{E}_s (\wt{\Lambda}_+^{(n)},\wt{\Lambda}_-^{(n)}, f^{(n)}))
\cdot \mathcal{E}_s (\wt{\Lambda}_+^{(n+1)},\wt{\Lambda}_-^{(n+1)}, f^{(n+1)})
\Big).
\end{align*}
Thus we obtain
\begin{align}
\label{R20}
&\sup_{0\leq t\leq T} \mathcal{E}_s^{(n+1)}(t)\leq \f32 A_1+T\cdot
\sup_{0\leq t\leq T} \mathcal{E}_s^{(n+1)}(t)
\sup_{0\leq t\leq T}  \mathcal{F}\big(1+\|f^{(n)}(t)\|_{L^\infty(\BR)}+\mathcal{E}_s^{(n)}(t)\big) ,\\
\label{R21}
&\sup_{0\leq t\leq T}\| f^{(n+1)}(t,\cdot)\|_{L^\infty(\BR)}
\leq \| f_0 \|_{L^\infty(\BR)} +T  CA_1\\\nonumber
&\qquad\quad +T\cdot
\sup_{0\leq t\leq T} \mathcal{E}_s^{(n+1)}
\sup_{0\leq t\leq T}  \mathcal{F}\big(1+\|f^{(n)}\|_{L^\infty(\BR)}+\mathcal{E}_s^{(n)}\big).
\end{align}
Taking $T=\min\{\f14 \big({\mathcal{F}(2+2A_1)}\big)^{-1} ,\f14 ,\f12(CA_1)^{-1}\delta , \f12 (A_1\mathcal{F}(2+2A_1))^{-1}\delta \} $, \eqref{R20} yields  $\eqref{R13}_1$, \eqref{R21} yields  $\eqref{R13}_3$. Combining $\eqref{R13}_1$, $\eqref{R13}_3$ and  \eqref{R18} yields $\eqref{R13}_2$.
%\begin{align}\label{R21}
%\sup_{0\leq t\leq T} \mathcal{E}_s^{(n+1)}\leq 2 A_1 .
%\end{align}
\end{proof}

\subsection{Convergence of the approximate solutions}

Let us show that $\{ \wt{\Lambda}_+^{(n)},\wt{\Lambda}_-^{(n)}, f^{(n)}, \wt{p}^{(n)}\}$ is a Cauchy sequence in some Sobolev spaces.
Indeed, it follows from \eqref{R10} and \eqref{R11} that
\begin{equation}\label{R31}
\begin{cases}
\p_t \big(\wt{\Lambda}_\pm^{(n+1)}-\wt{\Lambda}_\pm^{(n)} \big)
+A_n\wt{Z}_\mp^{(n)}\cdot \nabla \big(\wt{\Lambda}_\pm^{(n+1)}-\wt{\Lambda}_\pm^{(n)} \big) \\[-4mm]\\
\quad +\big(A_n\wt{Z}_\mp^{(n)}-A_{n-1}\wt{Z}_\mp^{(n-1)}\big)\cdot\nabla  \wt{\Lambda}_\pm^{(n)}% \\[-4mm]\\
\quad + A^\top_n \nabla \wt{p}^{(n)}-A^\top_{n-1}\nabla \wt{p}^{(n-1)}   \\[-4mm]\\
=(0,\p_t\psi^{(n)}/J_n)\cdot \nabla \big(\wt{\Lambda}_\pm ^{(n+1)}-\wt{\Lambda}_\pm ^{(n)} \big)\\[-4mm]\\
\quad+(0,\p_t\psi^{(n)}/J_n-\p_t\psi^{(n-1)}/J_{n-1})\cdot \nabla \wt{\Lambda}_\pm ^{(n)},  \\[-4mm]\\
\div^{\Psi_n} \big( \wt{\Lambda}_\pm^{(n+1)}-\wt{\Lambda}_\pm^{(n)} \big)
+\big(\div^{\Psi_n}-\div^{\Psi_{n-1}} \big) \wt{\Lambda}_\pm^{(n)}=0, \\[-4mm]\\
\wt{\Lambda}_+^{(n+1)} - \wt{\Lambda}_+^{(n)} =0, \quad
\wt{\Lambda}_-^{(n+1)} - \wt{\Lambda}_-^{(n)} =0,
\end{cases}
\text{in } \Omega_L.
\end{equation}
and
\begin{equation} \label{R32}
\begin{cases}
(\p_t^2-\p_1^2) \big(f^{(n+1)}-f^{(n)} \big)+ ({\ud{\wt{\Lambda}_+^1}^{(n)}}+{\ud{\wt{\Lambda}_-^1}^{(n)}})\p_t\p_1 \big(f^{(n+1)}-f^{(n)} \big) \\[-4mm]\\
+ ({\ud{\wt{\Lambda}_+^1}^{(n)}}+{\ud{\wt{\Lambda}_-^1}^{(n)}} -{\ud{\wt{\Lambda}_+^1}^{(n-1)}}-{\ud{\wt{\Lambda}_-^1}^{(n-1)}} )\p_t\p_1 f^{(n)} \\[-4mm]\\
+\big( {\ud{\wt{\Lambda}_-^1}^{(n)}}-{\ud{\wt{\Lambda}_+^1}^{(n)}}+ {\ud{\wt{\Lambda}_-^1}^{(n)}} {\ud{\wt{\Lambda}_+^1}^{(n)}}\big)
 \p_1^2 \big( f^{(n+1)}-f^{(n)} \big)
  \\[-4mm]\\
+\big( {\ud{\wt{\Lambda}_-^1}^{(n)}}-{\ud{\wt{\Lambda}_+^1}^{(n)}}+ {\ud{\wt{\Lambda}_-^1}^{(n)}} {\ud{\wt{\Lambda}_+^1}^{(n)}}
  - {\ud{\wt{\Lambda}_-^1}^{(n-1)}}+{\ud{\wt{\Lambda}_+^1}^{(n-1)}}- {\ud{\wt{\Lambda}_-^1}^{(n-1)}} {\ud{\wt{\Lambda}_+^1}^{(n-1)}}\big)
 \p_1^2  f^{(n)}
  \\[-4mm]\\
%\quad= -N^{(n)}\cdot \ud{\nabla p}^{(n+1)}+N^{(n-1)}\cdot \ud{\nabla p}^{(n)} \, , \\[-4mm]\\
\quad= -N_{f^{(n)}} \cdot \big(\ud{\nabla p}^{(n+1)}-\ud{\nabla p}^{(n)}\big)
-(N_{f^{(n)}}-N_{f^{(n-1)}})\cdot \ud{\nabla p}^{(n)} \, , \\[-4mm]\\
\big(f^{(n+1)}-f^{(n)} \big) (0) = 0,\big(\p_tf^{(n+1)}-\p_tf^{(n)} \big) (0)= 0,
\end{cases}
\text{in } \BR.
\end{equation}

Similar to the estimate for \eqref{R14}-\eqref{R18}, for $s'=s-1$,
%(or $s'=s-2$)
we treat the $H^{s'}(\Om_L)$ energy estimate for \eqref{R31} and the $H^{s'+\f12}(\BR)$
energy estimate for \eqref{R32} with slight modifications. For $0<T'\leq T$, we conclude that
\begin{align}\label{R33}
&\sup_{[0,T']}\mathcal{E}_{s'} (\wt{\Lambda}_+^{(n+1)}-\wt{\Lambda}_+^{(n)},
\wt{\Lambda}_-^{(n+1)}-\wt{\Lambda}_-^{(n)}, f^{(n+1)}-f^{(n)}) \\\nonumber
&\leq T' \mathcal{F}(2+2A_1) \sup_{[0,T']}
\Big(\mathcal{E}_{s'} (\wt{\Lambda}_+^{(n+1)}-\wt{\Lambda}_+^{(n)},
\wt{\Lambda}_-^{(n+1)}-\wt{\Lambda}_-^{(n)}, f^{(n+1)}-f^{(n)}) \\\nonumber
&\qquad\quad +\mathcal{E}_{s'} (\wt{\Lambda}_+^{(n)}-\wt{\Lambda}_+^{(n-1)},
\wt{\Lambda}_-^{(n)}-\wt{\Lambda}_-^{(n-1)}, f^{(n)}-f^{(n-1)})
+\|\wt{p}^{(n+1)}-\wt{p}^{(n)}\|^2_{H^{s+1}(\Om_L)}\Big)
\end{align}
and
%\begin{align}\label{R34}
%&\| f^{(n+1)}(t,\cdot)-f^{(n)}(t,\cdot)\|_{L^\infty(\BR)} \\\nonumber
%&\leq \sup_{[0,T']}\mathcal{E}_{s'} (\wt{\Lambda}_+^{(n)}-\wt{\Lambda}_+^{(n-1)},
%\wt{\Lambda}_-^{(n)}-\wt{\Lambda}_-^{(n-1)}, f^{(n)}-f^{(n-1)}) \\\nonumber
%&\quad +TC \| f^{(n+1)}(t,\cdot)-f^{(n)}(t,\cdot)\|_{L^\infty(\BR)}.
%\end{align}
\begin{align}\label{R34}
&\|\wt{p}^{(n+1)}-\wt{p}^{(n)}\|^2_{H^{s+1}(\Om_L)} \\ \nonumber
&\leq \mathcal{F}(2+2A_1)
\cdot \big(\mathcal{E}_{s'} (\wt{\Lambda}_+^{(n+1)}-\wt{\Lambda}_+^{(n)},
\wt{\Lambda}_-^{(n+1)}-\wt{\Lambda}_-^{(n)}, f^{(n+1)}-f^{(n)}) \\[-5mm]\nonumber\\\nonumber
&\quad +\mathcal{E}_{s'} (\wt{\Lambda}_+^{(n)}-\wt{\Lambda}_+^{(n-1)},
\wt{\Lambda}_-^{(n)}-\wt{\Lambda}_-^{(n-1)}, f^{(n)}-f^{(n-1)}) \big)
\end{align}
and
\begin{align}\label{R33-2}
&\sup_{[0,T']}\|f^{(n+1)}-f^{(n)}\|^2_{L^\infty(\bR)} \\\nonumber
&\leq T'^2 \mathcal{F}(2+2A_1) \sup_{[0,T']}
\Big(\mathcal{E}_{s'} (\wt{\Lambda}_+^{(n+1)}-\wt{\Lambda}_+^{(n)},
\wt{\Lambda}_-^{(n+1)}-\wt{\Lambda}_-^{(n)}, f^{(n+1)}-f^{(n)}) \\\nonumber
&\qquad\quad +\mathcal{E}_{s'} (\wt{\Lambda}_+^{(n)}-\wt{\Lambda}_+^{(n-1)},
\wt{\Lambda}_-^{(n)}-\wt{\Lambda}_-^{(n-1)}, f^{(n)}-f^{(n-1)})
+\|\wt{p}^{(n+1)}-\wt{p}^{(n)}\|^2_{H^{s+1}(\Om_L)}\Big),
\end{align}

Taking $T' \mathcal{F}(A_1)\leq \f13$, \eqref{R33} and \eqref{R34} yield
\begin{align}\label{R35}
&\sup_{[0,T']}\mathcal{E}_{s'} (\wt{\Lambda}_+^{(n+1)}-\wt{\Lambda}_+^{(n)},
\wt{\Lambda}_-^{(n+1)}-\wt{\Lambda}_-^{(n)}, f^{(n+1)}-f^{(n)}) \\\nonumber
&\leq \f12 \sup_{[0,T']}
\mathcal{E}_{s'} (\wt{\Lambda}_+^{(n)}-\wt{\Lambda}_+^{(n-1)},
\wt{\Lambda}_-^{(n)}-\wt{\Lambda}_-^{(n-1)}, f^{(n)}-f^{(n-1)}).
\end{align}

We are now ready to state and prove the local well-posedness result for \eqref{R0} and \eqref{R1}.
\begin{thm}\label{thm2}
Let $s\ge 4$  be an integer, $A_1>0, 0<\delta<\f12$. %$\kappa>0$ and $\sigma>0$ and
%Assume that $\wt{\Lambda}_\pm\in H^s(\Omega_L)$  and
%$f_0\in H^{s+\f12}(\BR )$
%Assume $(\Lambda_+^0, \Lambda_-^0) \in H^s(\Omega_{f^0})$,
Assume that $\wt{\Lambda}^0_\pm\in H^s(\Omega_L)$,
%$f^0 \in  H^{s+\frac12}(\BR)$,
$f^0 \in  \dot{H}^1(\BR) \cap \dot{H}^{s+\frac12}(\BR) \cap L^{\infty}(\BR)$
and $v^0 \in  H^{s-\frac12}(\BR)  $
are given such that $\mathcal{E}_s^b (\wt{\Lambda}_+^0,\wt{\Lambda}_-^0)+\mathcal{E}_s^f (v^0, f^0)\leq A_1$,
 $\|f^0\|_{L^\infty(\bR)}\leq 1-2\delta$
  and  the  compatibility conditions \eqref{R-com}  are satisfied.
There exist universal positive constants $\delta_0$ and $T_0$ depending only on $A_1$ such that if $\mathcal{E}_s (0) \le \delta_0$ and $0<T\leq T_0$,  then there exists a  unique solution $(\Lambda_+,\Lambda_-,p,f )$  to  \eqref{R0} and \eqref{R1} on $[0,T]$ satisfying
\begin{equation*} %\label{bound110}
\sup_{t\in [0,T]}\mathcal{E}_s(t) \lesssim \mathcal{E}_s(0),\, \sup_{t\in[0,T]}\| f(t,\cdot)\|_{L^\infty} \leq 1-\delta.
\end{equation*}
\end{thm}
\begin{proof}[Proof of Theorem \ref{thm2}]
%Take  $\delta_0=\tilde \delta_2$ and $T_0=T_2$ as given in Proposition  \ref{ }.
Let $\{(\wt{\Lambda}_+^{(n)},\wt{\Lambda}_-^{(n)}, f^{(n)}, \wt{p}^{(n)})\}_{n=1}^\infty $
 on $[0,T]$ with $0<T\le T_0$ be  the sequence constructed in \eqref{R10} and \eqref{R11}.
The uniform estimate  \eqref{R13}
and the contractive estimate \eqref{R34}-\eqref{R35}
  show that the whole sequence  $\{(\wt{\Lambda}_+^{(n)},\wt{\Lambda}_-^{(n)}, f^{(n)}, \wt{p}^{(n)})\} $ converges strongly to the limit
$\{(\wt{\Lambda}_+,\wt{\Lambda}_-, f, \wt{p})\}$  % in the  norms of $C^2$,
which is sufficient for passing to the limit in  \eqref{R10} and \eqref{R11}.  Then one finds that the limit $\{(\wt{\Lambda}_+,\wt{\Lambda}_-, f, \wt{p})\}$ is a
strong solution to \eqref{R0} and \eqref{R1} on $[0,T]$. The
uniqueness of   solutions to \eqref{R0} and \eqref{R1}
%satisfying \eqref{}
 follows  similarly to that in the proof of the contraction.
\end{proof}

Finally, we show the identity \eqref{A21} $\p_t f=\ud{Z_\pm}\cdot N_f$.
Due to the local existence of solutions to \eqref{R0}, \eqref{R1}
and the diffeomorphism of the inverse map $\Psi$, we obtain the MHD equations of the plasma in the
 Eulerian formulation \eqref{MHD1}.
Restricting \eqref{MHD1} to the free surface in the trace sense,  by chain rules, we derive
\begin{align}\label{RR1}
(\p_t +\ud{Z_+^1}\p_1)\ud{Z_-}
&=\ud{(\p_t +{Z_+^1}\p_1)Z_-}
+\ud{\p_2 Z_-}  (\p_t +\ud{Z_+^1}\p_1)f \\\nonumber
&=-\ud{{Z_+^2}\p_2 Z_-}-\ud{\nabla p}
+\ud{\p_2 Z_-}  (\p_t +\ud{Z_+^1}\p_1)f\\\nonumber
&=-\ud{\nabla p}+\ud{\p_2 Z_-} \big( (\p_t +\ud{Z_+^1}\p_1)f- \ud{Z_+^2}\big).
\end{align}
On the other hand, there naturally holds
\begin{align} \label{RR2}
&(\p_t + \ud{Z_+^1}\p_1)(\p_t + \ud{Z_-^1}\p_1)f \\ \nonumber
&=\p_t^2 f+(\ud{Z_+^1}+\ud{Z_-^1})
\p_1\p_t f
+(\ud{Z_-^1}\ud{Z_+^1})\p_1^2 f\\ \nonumber
&\quad+(\p_t + \ud{Z_+^1}\p_1)\ud{Z_-^1}\p_1f.
\end{align}
Plugging \eqref{RR1} into \eqref{RR2} yields
\begin{align*} %\label{RR3}
&(\p_t + \ud{Z_+^1}\p_1)(\p_t + \ud{Z_-^1}\p_1)f \\ \nonumber
&=\p_t^2 f+(\ud{Z_+^1}+\ud{Z_-^1})
\p_1\p_t f
+(\ud{Z_-^1}\ud{Z_+^1})\p_1^2 f\\ \nonumber
&\quad+(\p_t + \ud{Z_+^1}\p_1)\ud{Z_-^1}\p_1f \\ \nonumber
&=-N_f \cdot \ud{\nabla p}
+\big(-\ud{{Z_+^2}\p_2 Z_-^1}-\ud{\p_1 p}
+\ud{\p_2 Z_-^1}  (\p_t +\ud{Z_+^1}\p_1)f \big)\p_1f \\ \nonumber
&=- \ud{\p_2 p}
+\ud{\p_2 Z_-^1} \big( (\p_t +\ud{Z_+^1}\p_1)f -\ud{Z_+^2} \big)\p_1f.
\end{align*}
In the above deduction, we have used \eqref{R0}.
Consequently, we calculate
\begin{align} \label{RR4}
&(\p_t + \ud{Z_+^1}\p_1)\big((\p_t + \ud{Z_-^1}\p_1)f-\ud{Z_-^2} \big)\\ \nonumber
&=- \ud{\p_2 p}
+\ud{\p_2 Z_-^1} \big( (\p_t +\ud{Z_+^1}\p_1)f -\ud{Z_+^2} \big)\p_1f \\\nonumber
&\quad -\Big( -\ud{\p_2 p}+\ud{\p_2 Z_-^2} \big( (\p_t +\ud{Z_+^1}\p_1)f- \ud{Z_+^2}\big) \Big)\\\nonumber
&=\big(\ud{\p_2 Z_-^1} \p_1f- \ud{\p_2 Z_-^2} \big) \big( (\p_t +\ud{Z_+^1}\p_1)f -\ud{Z_+^2} \big) .
%&\quad -\ud{\p_2 Z_-^2} \big( (\p_t +\ud{Z_+^1}\p_1)f- \ud{Z_+^2} \big) \\\nonumber
\end{align}
Similarly we obtain
\begin{align} \label{RR5}
&(\p_t + \ud{Z_-^1}\p_1)\big((\p_t + \ud{Z_+^1}\p_1)f-\ud{Z_+^2} \big)\\ \nonumber
&=\big(\ud{\p_2 Z_+^1} \p_1f- \ud{\p_2 Z_+^2} \big) \big( (\p_t +\ud{Z_-^1}\p_1)f -\ud{Z_-^2} \big) .
\end{align}
Then by Gronwall's inequality, the energy estimate for \eqref{RR4}, \eqref{RR5}   yields \eqref{A21}.

\section*{Acknowledgment.}
This work was supported by NSFC grant (No. 12571247).

\end{document}